\documentclass[]{gtpart}   
\usepackage{geometry}                		
\usepackage{graphicx}				

\usepackage[mathscr]{euscript}

\usepackage{amsmath, amssymb, bm}
\usepackage{enumitem}

\usepackage[nameinlink]{cleveref}

\usepackage[title]{appendix}
\usepackage{import}
\usepackage{multirow}
\usepackage{array}

\newcommand{\ol}{\overline}
\newcommand{\mr}{\mathring}
\newcommand{\wt}{\widetilde}
\newcommand{\wh}{\widehat}
\renewcommand{\Z}{\mathbb{Z}}
\newcommand{\hyp}{\mathbb{H}}
\renewcommand{\R}{\mathbb{R}}

\newcommand{\U}{\mathscr{U}}
\renewcommand{\J}{\mathscr{J}}

\newcommand{\F}{\mathcal{F}}

\newcommand{\LL}{\mathcal{L}}

\newcommand{\mc}{\mathcal}
\newcommand{\ms}{\mathscr}

\newcommand{\into}{\hookrightarrow}
\newcommand{\onto}{\twoheadrightarrow}
\newcommand{\sut}{\rightsquigarrow}
\newcommand{\cut}{\backslash\!\backslash}

\renewcommand{\phi}{\varphi}

\newcommand{\p}{\mathbf{p}}
\newcommand{\m}{\mathbf{m}}
\renewcommand{\b}{\mathbf{b}}
\renewcommand{\u}{\mathbf{u}}
\newcommand{\s}{\mathbf{s}}

\newcommand{\del}{\partial}

\DeclareMathOperator{\intr}{int}

\DeclareMathOperator{\ind}{index}

\DeclareMathOperator{\brloc}{brloc}

\DeclareMathOperator{\graft}{graft}
\DeclareMathOperator{\pers}{pers}
\DeclareMathOperator{\sink}{sink}

\numberwithin{equation}{section}

\newtheorem*{thm:preperiodic}{\Cref{thm:preperiodic}}
\newtheorem*{thm:depthonetovbs}{\Cref{thm:depthonetovbs}}
\newtheorem*{thm:hmvbsunique}{\Cref{thm:hmvbsunique}}
\newtheorem*{thm:folcone}{\Cref{thm:folcone}}
\newtheorem*{thm:VBStoDP}{\Cref{thm:VBStoDP}}
\newtheorem*{cor:hmveeringdynamicpair}{\Cref{cor:hmveeringdynamicpair}}

\usepackage{thmtools}
\declaretheorem[style=definition,qed=$\lozenge$,sibling=equation]{construction}
\declaretheorem[style=definition,qed=$\lozenge$,sibling=equation]{example}
\declaretheorem[style=definition,qed=$\lozenge$,sibling=equation]{definition}
\declaretheorem[style=definition,qed=$\lozenge$,sibling=equation]{remark}
\declaretheorem[style=definition,qed=$\lozenge$,sibling=equation]{convention}
\declaretheorem[style=plain,sibling=equation]{lemma}
\declaretheorem[style=plain,sibling=equation]{proposition}
\declaretheorem[style=plain,sibling=equation]{theorem}
\declaretheorem[style=plain,sibling=equation]{corollary}

\declaretheorem[style=plain,sibling=equation,name=Proposition/Definition]{propdefn}

\theoremstyle{definition}

\newtheorem*{question*}{Question}

\theoremstyle{remark}

\begin{document}
\title{Endperiodic maps, splitting sequences, and branched surfaces}
\author{Michael P. Landry}
\address{Saint Louis University}
\email{michael.landry@slu.edu}
\author{Chi Cheuk Tsang}
\address{University of California Berkeley}
\email{chicheuk@math.berkeley.edu}

\begin{abstract}
We strengthen the unpublished theorem of Gabai and Mosher that every depth one sutured manifold contains a very full dynamic branched surface by showing that the branched surface can be chosen to satisfy an additional property we call veering. To this end we prove that every endperiodic map admits a periodic splitting sequence of train tracks carrying its positive Handel-Miller lamination.  
This completes step one of Gabai-Mosher's unpublished two-step proof that every taut finite depth foliation of a compact, oriented, atoroidal  3-manifold is almost transverse to a pseudo-Anosov flow.
Further, a veering branched surface in a sutured manifold is a generalization of a veering triangulation, and we extend some of the theory of veering triangulations to this setting. In particular we show that the branched surfaces we construct are unique up to a natural equivalence relation, and give an algorithmic way to compute the foliation cones of Cantwell-Conlon. 
\end{abstract}

\maketitle

\tableofcontents
\section{Introduction}

\subsection{Context}
Let $\mc F$ be a finite depth cooriented foliation obtained from a sutured hierarchy, i.e. a sequence of sutured manifold decompositions
\[
M=Q_0 \sut Q_1\sut \cdots \sut Q_n\sut Q_{n+1}
\]
where $M$ is a compact $3$-manifold with toral boundary and $Q_{n+1}$ is a product sutured manifold (see \cite{Gab83}). The last decomposition $Q_n \sut Q_{n+1}$ induces a \emph{depth one foliation} on $Q_n$ by the well-known ``spinning" construction.

A depth one foliation of a sutured manifold is a cooriented foliation whose only compact leaves are the positive and negative tangential boundary $R_\pm$, and whose noncompact leaves accumulate only on $R_\pm$. In this situation, the complement of $R_\pm$ fibers over $S^1$ with fibers the noncompact leaves $L$ of the foliation, and the monodromy of the fibration determines an \textit{endperiodic map} $f:L \to L$. 
The foliated sutured manifold is a compactification of the mapping torus of $f$.

In the 1980s, Handel and Miller developed a theory of \textit{positive} and \textit{negative laminations} for endperiodic maps, analogous to the Nielsen-Thurston picture of laminations for pseudo-Anosov mapping classes; this theory was definitively exposited only recently by Cantwell, Conlon, and Fenley \cite{CCF19}. Morally, an endperiodic map expands the leaves of the positive lamination $\Lambda_+$ and contracts the leaves of the negative lamination $\Lambda_-$ up to isotopy. Moreover, the isotopy class of the endperiodic map contains a \textit{Handel-Miller representative} which preserves the laminations on the nose. 
By suspending the 1-dimensional positive and negative Handel-Miller laminations to the compactified mapping torus, one obtains 2-dimensional \textit{unstable} and \textit{stable Handel-Miller laminations}.

Subsequent to the work of Handel and Miller, Gabai developed a theory to build a pseudo-Anosov flow in $M$ almost transverse to the original finite depth foliation $\mc F$, essentially by extending the unstable/stable Handel-Miller laminations in $Q_n$ up through the hierarchy. This theory was not written down.

In the 1990s, Mosher set out to prove the existence of Gabai's flows. He introduced the notion of a \textit{dynamic pair}, which is a combinatorial analogue of a pseudo-Anosov flow together with its unstable and stable foliations. Roughly speaking, a dynamic pair consists of a pair of branched surfaces $B^u$ and $B^s$ intersecting transversely whose complementary regions are particularly nice. In Part I of the monograph \cite{Mos96}, Mosher proves that a dynamic pair in $M$ transverse to $\mc F$ gives rise to a pseudo-Anosov flow almost transverse to $\mc F$, thus reducing the construction of a flow in $M$ to the construction of a dynamic pair from a sutured manifold hierarchy.

Mosher outlined his plans for Part II of the monograph in his introduction. Essentially Part II would have been an inductive argument consisting of two parts: a \emph{base step} and a \emph{gluing step}. In the base step one builds a dynamic pair in $Q_n$ combinatorializing the Handel-Miller laminations. In the {gluing step}, one aims to show that for each $i$, a dynamic pair in $Q_{i+1}$ can be promoted to a dynamic pair in $Q_i$. Inducting up the sutured hierarchy, this would produce a dynamic pair on $M=Q_0$.

The Gabai--Mosher construction has proven useful. For example, in \cite{Mos92b} Mosher uses it to show that there exist pseudo-Anosov flows dynamically representing top-dimensional, nonfibered faces of the Thurston norm ball; in the same paper he constructs a pseudo-Anosov flow which does not represent an entire face of the Thurston norm ball. 
In \cite{FM01} Fenley and Mosher prove that pseudo-Anosov flows almost transverse to finite depth foliations in hyperbolic 3-manifolds are quasigeodesic, and appeal to the construction to say such flows are abundant.
In \cite{CD03}, Calegari and Dunfield use the construction to show that certain Dehn fillings of torally bounded hyperbolic 3-manifolds admit pseudo-Anosov flows, allowing them to apply their results on universal circles. In \cite{Cal06} Calegari uses the construction for one direction of his proof that the unit ball of the dual Thurston norm of a closed hyperbolic 3-manifold is exactly the convex hull of the Euler classes of quasigeodesic flows. 

Unfortunately, thus far a complete proof of Gabai and Mosher's result has not appeared in the literature.

\subsection{Summary of results}

This paper is part of an attempt to revisit Mosher's program using ideas from the theory of \textit{veering triangulations}. Veering triangulations were introduced by Agol \cite{Ago10}, and were recently shown to correspond robustly to pseudo-Anosov flows \cite[Introduction]{FSS19}. The second author introduced the notion of \textit{(unstable) veering branched surfaces} in compact 3-manifolds with toral boundary, and showed that these are dual to veering triangulations \cite[Proposition 3.2]{Tsa23}. Inspired by Mosher's dynamic pairs, we generalize veering branched surfaces to the setting of sutured manifolds. See \Cref{defn:vbs} for the precise formulation.

Using this tool of veering branched surfaces, we prove the following theorem that substitutes for Mosher's base step.

\begin{thm:depthonetovbs}
Let $Q$ be an atoroidal sutured manifold with a depth one foliation $\mathcal{F}$. Then $Q$ contains an unstable veering branched surface carrying the unstable Handel-Miller lamination associated to $\mathcal{F}$.
\end{thm:depthonetovbs}

In \Cref{thm:VBStoDP} we show that \Cref{thm:depthonetovbs} in fact implies Mosher's base step.
We remark that one can symmetrically define a stable veering branched surface and prove the symmetric version of \Cref{thm:depthonetovbs}, and in fact of all the results that we state. That we mainly work with unstable veering branched surfaces in this paper is just a choice.

This theorem can be thought of as a generalization of Agol's work in \cite{Ago10} showing that the mapping torus of a pseudo-Anosov map (on a finite-type surface) contains an unstable veering branched surface carrying the unstable foliation of the suspension flow. Indeed, our strategy of proof resembles Agol's proof, though there are some essential differences.
In the finite type case, Agol uses a projectively invariant positive measure for the expanding lamination of the pseudo-Anosov monodromy to get a measured train track carrying the lamination. He then performs \textit{maximal splittings}, i.e. splitting moves on the branches of maximal weight. He shows that the resulting sequence of train tracks is eventually periodic modulo the monodromy, hence can be suspended to obtain a branched surface in the mapping torus. This branched surface carries the unstable lamination of the monodromy and satisfies our definition of veering. 

In our setting, this idea does not work since Handel-Miller laminations do not in general carry projectively invariant measures of full support (see \cite[\S 5]{Fen97} and \Cref{eg:fenleyexample}).
Instead, we perform \textit{core splittings} on endperiodic train tracks carrying the positive Handel-Miller lamination. 
This means performing every possible splitting within a certain \emph{core} of the surface. 
To show that this operation is well-defined and gives an eventually periodic sequence, we develop the theory of \textit{spiraling train tracks} (\Cref{defn:spiraltt}) and utilize properties of the complementary regions of Handel-Miller laminations. In \Cref{thm:preperiodic} we prove the resulting sequence is eventually periodic modulo the monodromy.

\begin{thm:preperiodic}
Let $f\colon L\to L$ be endperiodic, and let $\tau_0$ be an efficient $f$-endperiodic train track carrying the positive Handel-Miller lamination. Consider the sequence of train tracks $(\tau_n)$ where $\tau_{n+1}$ is a core split of $\tau_{n}$. For sufficiently large $n$, we have $\tau_{n+1}=f(\tau_n)$.
\end{thm:preperiodic}

 To prove \Cref{thm:depthonetovbs} we essentially suspend a period of this sequence to get a branched surface, which we prove is veering.

We also investigate the uniqueness of the veering branched surface. In the finite type case, it is known that the veering branched surface carrying the unstable lamination is unique up to isotopy (this follows from results in \cite{Ago10} and \cite{LMT24}; see \Cref{thm:closeduniqueness}). In the endperiodic case the situation is more subtle because two veering branched surfaces $B_1$ and $B_2$ carrying the unstable Handel-Miller lamination can have different boundary train tracks $B_1 \cap R_+(Q)$ and $B_2 \cap R_+(Q)$. However, this is the only obstruction to uniqueness.

\begin{thm:hmvbsunique} 
Let $Q$ be an atoroidal sutured manifold with depth one foliation $\mathcal{F}$, and let $\mathcal{L}$ be the unstable Handel-Miller lamination associated to $\mathcal{F}$. Any veering branched surface $B$ compatibly carrying $\mc L$ is determined up to isotopy by $B\cap R_+(Q)$.
\end{thm:hmvbsunique}

Moreover, from our study of spiraling train tracks, we know that any two veering branched surfaces carrying $\mc L$ must intersect $R_\pm(Q)$ in two train tracks differing by a sequence of shifts. From this sequence of shifts we have an explicit description of how $B_1$ and $B_2$ are related. See \Cref{subsec:shiftmoves} for details.

Finally, we generalize the theory of cones associated to veering triangulations. We briefly review this theory for context. A veering triangulation has an associated \textit{dual graph} and \textit{flow graph}. These are finite directed graphs embedded inside of the triangulation. The dual graph is easily describable as the branch locus of the dual veering branched surface. The flow graph, introduced in \cite{LMT24}, has a more complicated description.

In \cite{LMT23a}, it is shown that the cone in $H_1$ generated by directed cycles of the dual graph equals that of the flow graph. In \cite{Lan22} and \cite{LMT24}, it is shown that this cone is the cone over a (not necessarily top-dimensional) face of the Thurston norm unit ball. When the veering triangulation is one whose dual veering branched surface carries the unstable foliation of a pseudo-Anosov mapping torus as above, then its associated cone is dual to the cone over the Thurston fibered face containing the fibration. In particular this provides an algorithmic way to compute fibered faces.

In \Cref{subsec:dualgraphflowgraph} we generalize the dual graph and flow graph to the sutured setting. We then prove the following theorem.

\begin{thm:folcone}
Let $Q$ be an atoroidal sutured manifold with depth one foliation $\mathcal{F}$, and let $\mathcal{L}$ be the unstable Handel-Miller lamination associated to $\mathcal{F}$. Let $\mathcal{C}_{\mathcal{F}} \subset H_2(Q, \partial Q)$ be the foliation cone containing the class of $\mathcal{F}$. Meanwhile, let $B$ be an unstable veering branched surface compatibly carrying $\mathcal{L}$, with dual graph $\Gamma$ and flow graph $\Phi$. Let $\mathcal{C}_\Gamma$ and $\mathcal{C}_\Phi$ be the cones in $H_1(M)$ positively generated by the cycles of $\Gamma$ and $\Phi$, respectively.
Then 
\[
\mathcal{C}_{\mathcal{F}}^\vee = \mathcal{C}_\Gamma = \mathcal{C}_\Phi.
\]
\end{thm:folcone}

We refer to \Cref{subsec:reviewfolcone} for a review of the theory of foliation cones. Here it suffices to say that these are finitely many rational polyhedral cones in $H_2(Q, \partial Q)$ that contain the classes of depth one foliations transverse to Handel-Miller semiflows, mirroring Fried's theory of Thurston fibered faces in compact 3-manifolds \cite{Fri79}. Our result provides an algorithmic way to compute foliation cones.

The equality $\mathcal{C}_\Gamma = \mathcal{C}_\Phi$ is shown by generalizing the notion of dynamic planes, introduced in \cite{LMT23a}, to the sutured setting. The condition `\textit{compatibly carrying}' arises naturally when considering the extra complications in the combinatorics of these dynamic planes, see \Cref{defn:compatiblycarry}. When the unstable veering branched surface compatibly carries an unstable Handel-Miller lamination $\mathcal{L}$ as in the theorem, the dynamic planes are in one-to-one correspondence with the lifts of leaves of $\mathcal{L}$ to the universal cover of $Q$. This allows us relate the dynamics of the Handel-Miller semiflow to that of the dual graph and flow graph, and show that $\mathcal{C}_\Gamma = \mathcal{C}_\Phi = \mathcal{C}_{\mathcal{F}}^\vee$.

In general, even if an unstable veering branched surface does not carry the unstable lamination associated to a depth one foliation, it is still true that $\mathcal{C}_\Gamma = \mathcal{C}_\Phi$. 

There are other aspects of veering triangulation theory that we expect to admit generalizations to the sutured setting. We mention a few in \Cref{sec:questions}.

We also include an appendix, the purpose of which is to explain the connection between our veering branched surfaces and Mosher's base step, and to prove results that will be used in our sequel paper \cite{LT23}. The main result is the following:

\begin{thm:VBStoDP}
If an atoroidal sutured manifold $Q$ admits an unstable veering branched surface, then it admits a dynamic pair.
\end{thm:VBStoDP}

In the case when the unstable veering branched surface carries the unstable lamination associated to a depth one foliation, we can in fact pick the unstable branched surface of the dynamic pair to be the given unstable veering branched surface. Hence we have the following corollary, which recovers Mosher's base step with a little extra information.

\begin{cor:hmveeringdynamicpair}
Let $Q$ be the compactified mapping torus of an endperiodic map $f:L \to L$. If $Q$ is atoroidal, then there is a dynamic pair $(B^u, B^s)$ on $Q$ such that:
\begin{itemize}
    \item $B^u$ is an unstable veering branched surface
    \item $B^u$ compatibly carries the unstable Handel-Miller lamination
    \item The boundary train track $B^u \cap R_+$ is efficient
\end{itemize}
\end{cor:hmveeringdynamicpair}

\subsection{Future work}
The next step of our project is to revisit Mosher's gluing step using veering branched surfaces. The basic idea is that given a sutured decomposition $Q \sut Q'$ and a veering branched surface $B'$ on $Q'$, one would like to construct a veering branched surface $B$ on $Q$.

Applying such a construction to a sutured hierarchy $M=Q_0 \sut \cdots 
\sut Q_{n+1}$, one would get a veering branched surface on $M=Q_0$, which is dual to a veering triangulation. One could then apply \cite[Theorem A]{Lan22} to recover an entire face of the Thurston norm ball on $M$, the cone over which contains $[S]$; there is also a combinatorial Euler class associated to the triangulation that computes the norm in this cone.
Alternatively, using the correspondence between veering triangulations and pseudo-Anosov flows mentioned before, one could obtain a pseudo-Anosov flow $\phi$ \textit{without perfect fits} on $M$. By \cite[Flows represent faces]{Mos92b}, one could then recover an entire face of the Thurston norm ball using $\phi$ and compute the norm in the cone over this face using $\phi$'s normal Euler class. While the Gabai-Mosher construction implies that the Thurston norm may be computed by pseudo-Anosov flows, it is a longstanding open question of Mosher whether finitely many pseudo-Anosov flows suffice. 

However, there is a nontrivial obstruction to the construction of such a flow for a given hierarchy: the hierarchy must satisfy a property we refer to as No Oppositely Oriented Parallel Orbits, or \emph{NOOPO} for short. Our program aims to show this is the only obstruction. We discuss this in more detail in \Cref{sec:questions}.

\subsection{Outline of paper} \Cref{sec:ttlam}, \Cref{sec:endperiodic}, and \Cref{sec:splitseq} deal exclusively with surfaces. In \Cref{sec:ttlam} we develop the theory of \emph{spiraling train tracks} and \emph{spiraling laminations}. This includes existence and uniqueness statements for splitting sequences of spiraling  train tracks carrying spiraling laminations, which play a key role in \Cref{sec:splitseq}. In \Cref{sec:endperiodic} we give a condensed treatment of Handel-Miller theory and prove some useful lemmas. 
In \Cref{sec:splitseq} we construct periodic train track splitting sequences for endperiodic maps using the notion of \emph{core splittings}. We also investigate the uniqueness of such a splitting sequence. 

In \Cref{sec:endperiodictosutured} we recall the relation between endperiodic maps and depth one foliations on sutured manifolds. 
In \Cref{sec:vbs} we define veering branched surfaces on sutured manifolds. We also define the dual graph and flow graph, and extend the theory of dynamic planes from \cite{LMT23a} to our setting. These dynamic planes play a key role in \Cref{sec:folcone}.

In \Cref{sec:hmvbs} we prove \Cref{thm:depthonetovbs} by suspending the splitting sequences constructed in \Cref{sec:splitseq}. We also translate the uniqueness statements about the splitting sequence into uniqueness statements about the veering branched surfaces obtained via this construction.
In \Cref{sec:folcone} we prove \Cref{thm:folcone}. In \Cref{sec:hmvbsunique} we prove \Cref{thm:hmvbsunique}. Given the uniqueness results in \Cref{sec:hmvbs}, the main task is to show that a veering branched surface as in the theorem must be `layered', i.e. come from the suspension of a splitting sequence of train tracks.

In \Cref{sec:eg} we discuss some examples of veering branched surfaces. Some examples are obtained from suspending a splitting sequence as in \Cref{sec:hmvbs}, while others are constructed directly and may not carry the Handel-Miller lamination associated to a depth one foliation.
In \Cref{sec:questions} we discuss some future directions coming out of this paper, including some more details on what we anticipate for the gluing step. 

In \Cref{sec:dynamicpairs} we define dynamic pairs and prove \Cref{thm:VBStoDP} and \Cref{cor:hmveeringdynamicpair}, showing that our results imply Mosher's base step.

\medskip

\noindent\textbf{Notational conventions.} In this paper we adopt the following notational conventions.
\begin{itemize}
    \item $X \cut Y$ denotes the metric completion of $X \backslash Y$ with respect to the induced path metric from $X$. We refer to this operation as cutting $X$ along $Y$. In addition, we will call the components of $X \cut Y$ the \textbf{complementary regions} of $Y$ in $X$.
    \item All coefficients for homology and cohomology groups are in $\R$ unless otherwise stated.
    \item A ``cycle" in a directed graph refers to a \emph{directed} cycle unless otherwise stated.
\end{itemize}

\section{Train tracks and laminations} \label{sec:ttlam}

\subsection{Review: train tracks, laminations, and index} \label{subsec:reviewttlam}

In this section, a surface will mean an orientable surface with (possibly empty) boundary.

Let $S$ be a surface. A \textbf{train track} on $S$ is an embedded graph $\tau$ in $S$ whose vertices have degree 3 or degree 1, such that the vertices of degree 1 are mapped to $\del S$, the edges are $C^1$ embedded, every point in $\tau$ has a well-defined tangent space, and at each vertex of degree 3 there are two edges tangent to one side and one edge tangent to the other side. The vertices of degree 3 are called \textbf{switches}, the vertices of degree 1 are called \textbf{stops}, and the edges are called \textbf{branches}.

We define a vector field on the set of switches and stops of $\tau$ by requiring that it points into the side of the tangent space meeting only one end of a branch at each switch, and that it points out of the surface at each stop. This is called the \textbf{maw vector field}. If $b$ is a non-circular branch such that the maw vector field points into $b$ at both its ends, we say $b$ is \textbf{large}. If the maw vector field points out of $b$ at both ends, $b$ is \textbf{small}. Otherwise, $b$ is \textbf{mixed}.

A \textbf{lamination} $\Lambda$ on a surface $S$ is a partition of a closed subset of $S$ into connected 1-manifolds, such that each point $x$ on $S$ has a neighborhood $\mathbb{R} \times \mathbb{R}$ with elements of the partition intersecting the neighborhood of the form $\mathbb{R} \times C$ (if $x$ is in the interior of $S$) or a neighborhood $[0,\infty) \times \mathbb{R}$ with elements of the partition intersecting the neighborhood in sets of the form $[0,\infty) \times C$ for some closed set $C$ (if $x$ is on the boundary of $S$). The elements of the partition are called the \textbf{leaves} of $\Lambda$. We will often conflate a lamination with the union of its leaves.

Let $\tau$ be a train track in $S$ and let $\Lambda$ be a lamination on $S$. A \textbf{standard neighborhood} of $\tau$ is a closed regular neighborhood $N$ of $\tau$ which is foliated by line segments called \textbf{ties} such that each line segment meets the branches of $\tau$ transversely. By collapsing the ties, we get a projection map $N \to \tau$. We say that $\tau$ \textbf{carries} $\Lambda$ if $\tau$ has a standard neighborhood $N$ such that $\Lambda$ is embedded in $N$ in a way so that its leaves are transverse to the ties. In this case we say that the map $\Lambda \hookrightarrow N \to \tau$ is the \textbf{carrying map} and we say that $N$ is a \textbf{standard neighborhood} of $\Lambda$. Further, we say that $\tau$ \textbf{fully carries} $\Lambda$ if the carrying map is onto, i.e. $\Lambda$ intersects every tie of $N$.

\begin{figure}
    \centering
    \fontsize{6pt}{6pt}\selectfont
    \resizebox{!}{2.3in}{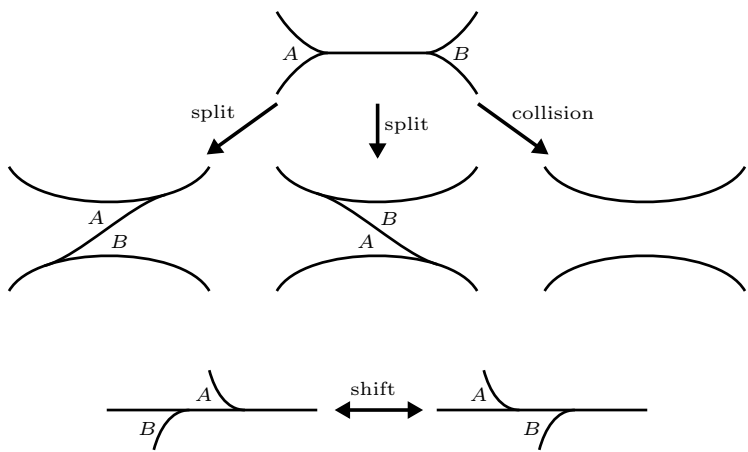}
    \caption{Illustration of split, collision, shift.}
    \label{fig:ttmoves}
\end{figure}

There are a few standard local operations on train tracks which are described below and pictured in \Cref{fig:ttmoves}. Let $b$ be a branch of $\tau$ that meets two switches $A$ and $B$ of $\tau$ at its ends.
\begin{itemize}
\item \textbf{Split}: If $b$ is a large branch, then a split alters a small neighborhood of $b$ by either pushing $A$ and $B$ past each other, or resolving $A$ and $B$ as shown. The latter type of split is also called a \textbf{collision}. Including collisions, there are three types of splits that can be performed on $b$ up to isotopy.
\item \textbf{Fold}: A fold is the inverse of a split.
\item \textbf{Shift}: If $b$ is mixed, then a shift moves $A$ and $B$ past each other.
\end{itemize}
Note that in our descriptions of these moves we are implicitly making use of a natural identification between the switches of a train track before and after all moves except collisions and their inverses. We will use this identification in the future.
In practice our train tracks will carry laminations, and the data of the laminations will control which types of splits we perform. 

\begin{figure}
\centering
\resizebox{!}{1.8in}{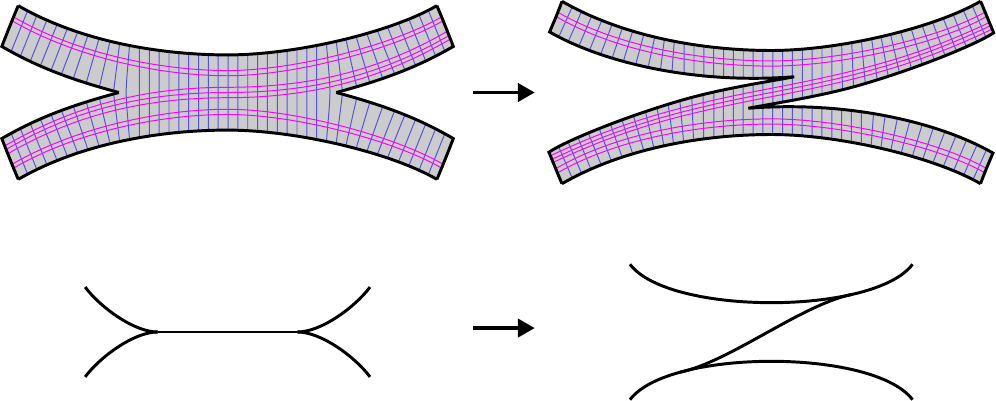}
\caption{An example of a $\Lambda$-split. Note that collisions are also possible.}
\label{fig:lambdasplit}
\end{figure}

Let $\tau$ be a train track, with standard neighborhood $N$. A \textbf{train route} is a sequence of branches of $\tau$ traversed by a $C^1$-immersed copy of $\R$, $[0,\infty)$ or $[0,1]$. Any curve immersed in $N$ transverse to the ties of $T$ naturally determines a train route. If the sequences associated to two such curves $\gamma, \gamma'$ are equal, or if one is a subsequence of the other, we say that $\gamma$ and $\gamma'$ \textbf{fellow travel} in $N$. When there is no danger of confusion, we may say that $\gamma$ and $\gamma'$ fellow travel in $\tau$. If the sequences associated to $\gamma,\gamma'$ are \emph{eventually} equal, we say they \textbf{eventually fellow travel}.

If $\tau$ is a train track fully carrying a lamination $\Lambda$, any large branch of $\tau$ can be split in a unique way so that the splitting also fully carries $\Lambda$. We call this a \textbf{$\Lambda$-split}; see \Cref{fig:lambdasplit}. 
Any switch of $\tau$ naturally determines a train route, which may be finite or infinite, as follows. Let $c$ be a switch of $\tau$, let $I$ be the tie through $c$ in a standard neighborhood of $\tau$, and let $p_1, p_2$ be the two endpoints of the component of $I\cut\Lambda$ containing $c$. Let $\rho_1$ and $\rho_2$ be the train routes determined by rays in $\Lambda$-leaves starting at $p_1, p_2$ respectively and traveling in the direction determined by the maw vector field at $c$. If $\rho_1$ and $\rho_2$ are equal, define $\rho_c:=\rho_1=\rho_2$. Otherwise, define $\rho_c$ to be the longest route which is an initial subroute of $\rho_1$ and $\rho_2$. See \Cref{fig:cusproute}. We call $\rho_c$ the \textbf{$\Lambda$-route of $c$}. 

More generally, a \textbf{$\Lambda$-route} is any train route determined by a parameterization of a part of a leaf of $\Lambda$. The \textbf{length} of a train route is the number of branches the route traverses counted with multiplicity.

\begin{figure}
    \centering
    \resizebox{!}{2in}{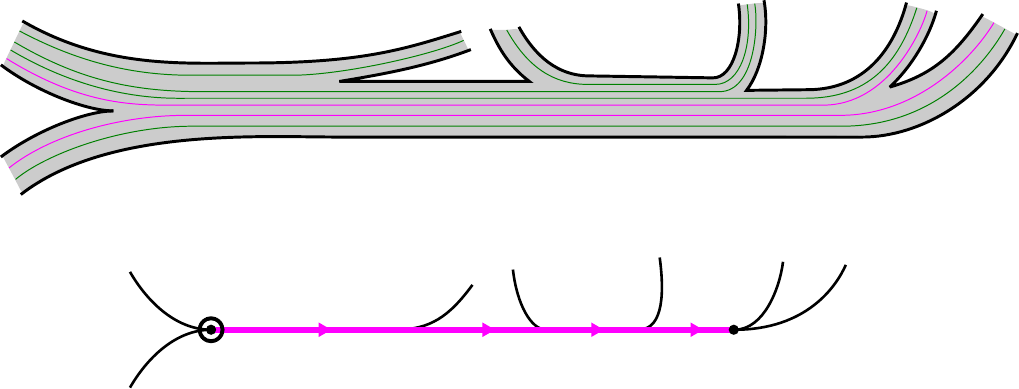}
    \caption{If the train track on the bottom carries $\Lambda$ as shown on the top, the $\Lambda$-route from the circled cusp is the pink segment.}
    \label{fig:cusproute}
\end{figure}

\begin{definition}[Cusps, corners]
A \textbf{surface with cusps and corners} is a surface $S$ along with the data of two disjoint finite sets of points on $\partial S$, called \textbf{cusps} and \textbf{corners}. The motivation for considering surfaces with cusps and corners comes from considering a complementary region $C$ of a train track on a surface, where we take the set of the cusps to be the nonsmooth points of $\del C$ corresponding to switches, and the set of corners to be the nonsmooth points corresponding to stops. Henceforth we will consider complementary regions of train tracks to be surfaces with cusps and corners in this way, unless specified otherwise. Since the switches of a train track are natural bijection with the cusps of its complementary regions, we will sometimes conflate the two objects.
\end{definition}

\begin{definition}[Index of a surface with corners and cusps]
The \textbf{index} of a compact surface with cusps and corners $S$ is defined to be
\[
\ind(S)=\chi^{\text{top}}(S)-\frac{1}{2}\text{\#cusps}-\frac{1}{4}\text{\#corners}
\]
where $\chi^{\text{top}}$ denotes the Euler characteristic of the underlying topological surface.
\end{definition}

For example, the (orientable) surfaces with cusps and corners that have index $0$ are exactly: tori, annuli, \textbf{cusped bigons}, \textbf{1-cusped triangles}, and \textbf{rectangles}. See \Cref{fig:ind0surfaces} for an illustration of the last three.

Another common surface with cusps and corners is the \textbf{bigon}, which is a topological disk with two corners in its boundary, and has index $\frac{1}{2}$.

Notice that the index is additive, that is, if $\tau$ is a train track on a surface $S$, then $\ind(S)=\ind(S \cut \tau)$.

\begin{definition}[Index of complementary regions of a lamination]
The complementary regions of a lamination on a compact surface are non-compact surfaces with boundary, but here the non-compactness can be cut off. Precisely, for such a complementary region $C$, there exists a finite collection of arcs that cut $C$ up into a union $K \cup V_1 \cup ... \cup V_n$, such that $K$ is a surface with corners and each $V_i \cong [0,\infty) \times [0,1]$. See, for example, \cite[Section 5.2]{CanCon00}. We define the \textbf{index} of $C$ to be the index of $K$.

Alternatively, one can complete $C$ by adding a point at infinity for each of its ends, and treating the completion $\overline{C}$ as a surface with cusps and corners by taking the added points as cusps. Then the index of $C$ will be equal to the index of $\overline{C}$.
\end{definition}

\begin{figure}
    \centering
    \fontsize{12pt}{12pt}\selectfont
    \resizebox{!}{1.6cm}{
\begingroup%
  \makeatletter%
  \providecommand\color[2][]{%
    \errmessage{(Inkscape) Color is used for the text in Inkscape, but the package 'color.sty' is not loaded}%
    \renewcommand\color[2][]{}%
  }%
  \providecommand\transparent[1]{%
    \errmessage{(Inkscape) Transparency is used (non-zero) for the text in Inkscape, but the package 'transparent.sty' is not loaded}%
    \renewcommand\transparent[1]{}%
  }%
  \providecommand\rotatebox[2]{#2}%
  \newcommand*\fsize{\dimexpr\f@size pt\relax}%
  \newcommand*\lineheight[1]{\fontsize{\fsize}{#1\fsize}\selectfont}%
  \ifx\svgwidth\undefined%
    \setlength{\unitlength}{358.34322391bp}%
    \ifx\svgscale\undefined%
      \relax%
    \else%
      \setlength{\unitlength}{\unitlength * \real{\svgscale}}%
    \fi%
  \else%
    \setlength{\unitlength}{\svgwidth}%
  \fi%
  \global\let\svgwidth\undefined%
  \global\let\svgscale\undefined%
  \makeatother%
  \begin{picture}(1,0.14997776)%
    \lineheight{1}%
    \setlength\tabcolsep{0pt}%
    \put(0,0){\includegraphics[width=\unitlength,page=1]{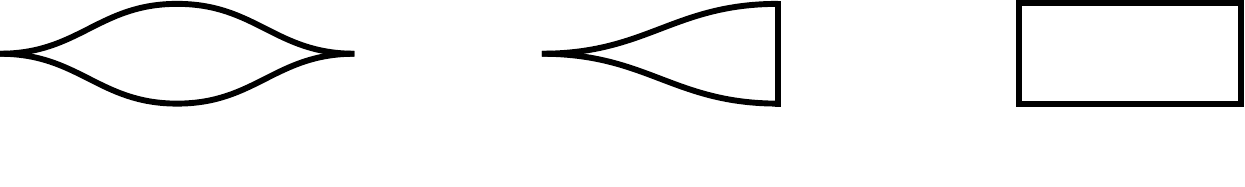}}%
    \put(0.04297323,0.0069166){\color[rgb]{0,0,0}\makebox(0,0)[lt]{\lineheight{1.25}\smash{\begin{tabular}[t]{l}cusped bigon\end{tabular}}}}%
    \put(0.41600984,0.0069166){\color[rgb]{0,0,0}\makebox(0,0)[lt]{\lineheight{1.25}\smash{\begin{tabular}[t]{l}1-cusped triangle\end{tabular}}}}%
    \put(0.83838583,0.0069166){\color[rgb]{0,0,0}\makebox(0,0)[lt]{\lineheight{1.25}\smash{\begin{tabular}[t]{l}rectangle\end{tabular}}}}%
  \end{picture}%
\endgroup%
}
    \caption{The cusped bigon, the 1-cusped triangle, and the rectangle are among the surfaces with cusps and corners that have index 0. Note that if these arise as complementary regions of a train track on $S$, the unique side of a 1-cusped triangle not incident to a cusp must lie along $\del S$, and that any rectangle has two nonadjacent sides lying along $\del S$.}
    \label{fig:ind0surfaces}
\end{figure}
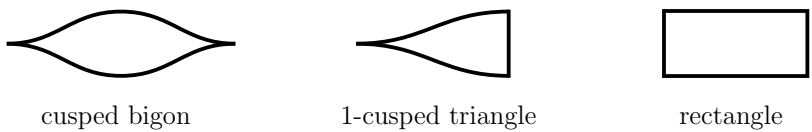

\begin{definition}[Reeb annuli] \label{defn:Reebannulus}
Let $S$ be a surface. Let $A$ be a subsurface of $S$ that is homeomorphic to an annulus. 
We say that $A$ is \textbf{carried} by a train track $\tau$ on $S$ if $\partial A$ is smoothly embedded along $\tau$. 
We say that $A$ is a \textbf{Reeb annulus} if $A \cap \tau$ is the union of $\partial A$ and a nonempty collection of arcs with endpoints attached onto both components of $\partial A$ along a consistent direction. (The arcs need not be disjoint from one another.)
Similarly, we say that an annulus $A \subset S$ is a \textbf{half Reeb annulus} if one component $l_1$ of $\partial A$ is smoothly embedded along $\tau$ while the other component $l_2$ lies along $\partial S$, and if $A \cap \tau$ is the union of $l_1$ and a nonempty collection of arcs with one endpoint attached onto $l_1$ along a consistent direction and with the other endpoint on $l_2$. (Here, again, the arcs need not be mutually disjoint.) See \Cref{fig:reeb} top row.

Similarly, we say that $A$ is carried by a lamination $\Lambda$ on $S$ if $\partial A$ is smoothly embedded along $\Lambda$. We say that $A$ is a \textbf{Reeb annulus} if $A \cap \Lambda$ is the union of $\partial A$ and a nonempty collection of non-compact leaves with ends spiraling onto both components of $\partial A$ along a consistent direction. Similarly, we say that an annulus $A \subset S$ is a \textbf{half Reeb annulus} if one component $l_1$ of $\partial A$ is smoothly embedded along $\Lambda$ while the other component $l_2$ lies along $\partial S$, and if $A \cap \Lambda$ is the union of $l_1$ and a nonempty collection of non-compact leaves with an end spiraling onto $l_1$ along a consistent direction and with an endpoint on $l_2$. See \Cref{fig:reeb} bottom row. 
\end{definition}

\begin{figure}
    \centering
    \fontsize{12pt}{12pt}\selectfont
    \resizebox{!}{6cm}{
\begingroup%
  \makeatletter%
  \providecommand\color[2][]{%
    \errmessage{(Inkscape) Color is used for the text in Inkscape, but the package 'color.sty' is not loaded}%
    \renewcommand\color[2][]{}%
  }%
  \providecommand\transparent[1]{%
    \errmessage{(Inkscape) Transparency is used (non-zero) for the text in Inkscape, but the package 'transparent.sty' is not loaded}%
    \renewcommand\transparent[1]{}%
  }%
  \providecommand\rotatebox[2]{#2}%
  \newcommand*\fsize{\dimexpr\f@size pt\relax}%
  \newcommand*\lineheight[1]{\fontsize{\fsize}{#1\fsize}\selectfont}%
  \ifx\svgwidth\undefined%
    \setlength{\unitlength}{217.52742582bp}%
    \ifx\svgscale\undefined%
      \relax%
    \else%
      \setlength{\unitlength}{\unitlength * \real{\svgscale}}%
    \fi%
  \else%
    \setlength{\unitlength}{\svgwidth}%
  \fi%
  \global\let\svgwidth\undefined%
  \global\let\svgscale\undefined%
  \makeatother%
  \begin{picture}(1,0.95086094)%
    \lineheight{1}%
    \setlength\tabcolsep{0pt}%
    \put(0,0){\includegraphics[width=\unitlength,page=1]{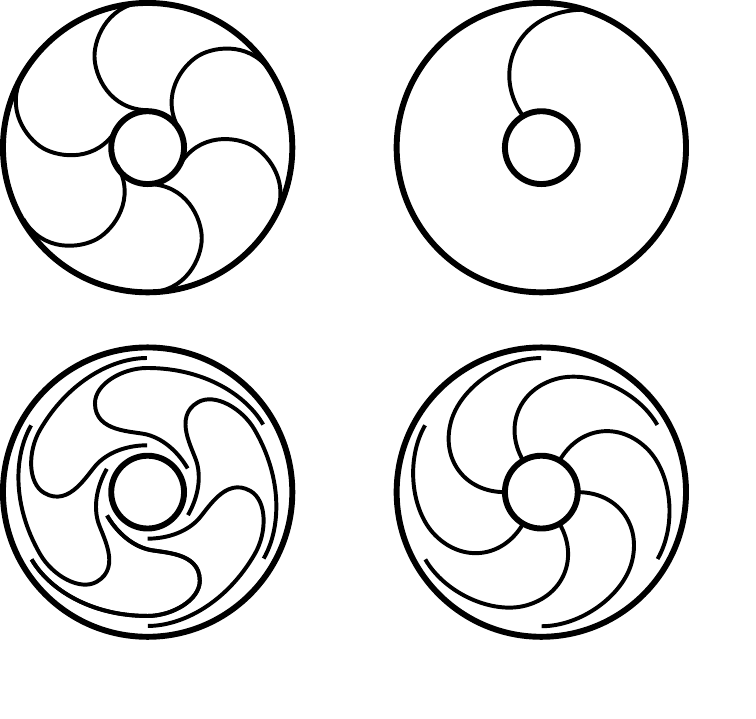}}%
    \put(0.04480413,0.00083502){\color[rgb]{0,0,0}\makebox(0,0)[lt]{\lineheight{1.25}\smash{\begin{tabular}[t]{l}Reeb annulus\end{tabular}}}}%
    \put(0.50687792,0.00083502){\color[rgb]{0,0,0}\makebox(0,0)[lt]{\lineheight{1.25}\smash{\begin{tabular}[t]{l}half Reeb annulus\end{tabular}}}}%
    \put(0,0){\includegraphics[width=\unitlength,page=2]{reeb.pdf}}%
  \end{picture}%
\endgroup%
}
    \caption{Reeb and half-Reeb annuli.}
    \label{fig:reeb}
\end{figure}

There are various definitions of `nice' train tracks in the literature. In this paper, it is convenient for us to make the following two definitions.

\begin{definition}[Good train tracks and laminations] \label{defn:essentialttlam}
A train track $\tau$ on $S$ is \textbf{good} if all of its complementary regions have nonpositive index, except possibly for bigons, and if it does not have Reeb annuli nor half Reeb annuli. 

A lamination $\Lambda$ on $S$ is \textbf{good} if all of its complementary regions have nonpositive index, except possibly for bigons, and if it does not have Reeb annuli nor half Reeb annuli. 
\end{definition}

\begin{definition} [Efficient train tracks]
\label{defn:efficienttt}
A train track $\tau$ is \textbf{efficient} if it is good, has no complementary regions that are cusped bigons, and does not carry an annulus.
\end{definition}

The motivation for the above terminology is the fact that a curve or a properly embedded arc can be carried by an efficient train track in at most one way, up to homotopy (rel endpoints).

\begin{remark}
In particular, an efficient train track can have only two types of index 0 complementary regions: rectangles and 1-cusped triangles. See \Cref{fig:ind0surfaces}.
\end{remark}

\subsection{Spiraling train tracks and laminations} \label{subsec:spiralttlam}

In this paper we will encounter two special types of laminations on compact surfaces: (i) laminations such that each end of a noncompact leaf limits on one of finitely many closed leaves, and (ii) laminations on surfaces with boundary for which all leaves are properly embedded compact intervals. These two types of laminations share many important properties, so we will consider them to be in the same class of `spiraling laminations' (see \Cref{defn:spirallam}).

Beyond this paper, spiraling laminations will also be important in our future work described in the introduction.

\begin{definition}[Source orientation] \label{defn:sourceorientation}
A \textbf{source orientation} on a 1-manifold (possibly with boundary) $l$ is a choice of points $p_1,...,p_k$ and an orientation on each component of $l \backslash \{p_1,...,p_k\}$ which points away from $p_1,...,p_k$. The points $p_1,...,p_k$ are called the \textbf{sources}.
\end{definition}

\begin{definition}[Spiraling train tracks] \label{defn:spiraltt}
Let $\tau$ be a good train track, and let $V$ be a continuous vector field tangent to $\tau$. We say $(\tau, V)$ is \textbf{spiraling} if $V$ restricts to the maw vector field on the set of switches and stops, and induces a source orientation on each branch of $\tau$. 
\end{definition}

\begin{lemma} \label{lemma:nolargebranchmeansspiraling}
Let $\tau$ be a good train track. Then $\tau$ has no large branches if and only if there exists a vector field $V$ on $\tau$ such that $(\tau,V)$ is spiraling.
\end{lemma}
\begin{proof}
If $\tau$ has no large branches, we can define a vector field $V$ as follows. For each branch $b$ which is not a circle, we can extend the maw vector field on its endpoints to a vector field on $b$ which has at most one singularity. Since $\tau$ has no large branches, a singularity arises only when the maw vector field points out of $b$ at each endpoint, in which case the singularity will be a source. For each circular branch, we may choose any nonsingular vector field on that branch. This vector field makes $\tau$ a spiraling train track.

Conversely, suppose $\tau$ has a large branch $b$. If $V$ is a vector field on $\tau$ extending the maw vector field, then the vector field $V|_b$ points into $b$ at both endpoints and therefore cannot induce a source orientation on $b$. It follows that there is no vector field making $\tau$ spiraling.
\end{proof}

\begin{definition}
Let $(\tau,V)$ be a spiraling train track. Notice that the forward trajectory of any point $x$ in $\tau$ under $V$ is well-defined (possibly) up to the point where it exits the surface through a stop. If $x$ is not a singularity of $V$ and its forward trajectory does not exit the surface, then the trajectory is an infinite ray carried by $\tau$, with limit set equal to a circle carried by $\tau$, the circle being oriented by $V$. 
The \textbf{sink} of $(\tau,V)$ is defined to be the union of the stops of $\tau$ and the oriented circles which are the limit sets of the infinite forward trajectories.
\end{definition}

Our choice of the term ``spiraling" is motivated by the case when $\tau$ has no stops, where any ray smoothly immersed in $\tau$ eventually periodically traverses a loop in the sink. If a spiraling track $\tau$ has stops, then this particular behavior is not necessarily present but we will see in \Cref{subsec:spiralfacts} that other important properties remain.

\begin{definition}[Spiraling laminations] 
\label{defn:spirallam}
A good lamination $\Lambda$ is said to be \textbf{spiraling} if:
\begin{enumerate}
    \item There are finitely many closed leaves $\ell_1,...,\ell_k$ in $\Lambda$, each with a fixed orientation.
    \item For all noncompact $\ell\subset \Lambda$, 
    if $\gamma\colon[0,\infty)\to \ell$ is a proper embedding then $\gamma$ spirals onto some $\ell_i$ in the direction specified by the orientation of $\ell_i$.\qedhere
\end{enumerate}
\end{definition}

It may be useful to think of these as ``depth $\le1$ laminations" in analogy with finite depth foliations.

The following shows that spiraling laminations naturally decompose into sets of two types, which we define after stating the result.

\begin{proposition} \label{prop:spirallam}
Let $S$ be a compact surface and $\Lambda$ a spiraling lamination on $S$. Then $\Lambda$ can be decomposed into a finite union of sets of the following form:
\begin{enumerate}
    \item A good lamination restricted to a (possibly degenerate) annulus carried by $\Lambda$
    \item A pocket of non-circular leaves
\end{enumerate}
\end{proposition}

We have defined an annulus carried by $\Lambda$ in \Cref{defn:Reebannulus}. A \textbf{degenerate annulus} carried by $\Lambda$ is just a circular leaf of $\Lambda$. A \textbf{pocket} of noncircular leaves is a subset $\Lambda'\subset \Lambda$ such that there exists a (not necessarily proper) embedding $R\times[0,1]\to S$ where $R$ is equal to either $\R$, $[0,\infty)$, or $[0,1]$, and a closed $C\subset [0,1]$ such that $R\times C$ maps one-to-one onto $\Lambda'$.

\begin{proof}
Consider the family $F$ of closed curves in $S$ comprised of the closed leaves of $\Lambda$. A maximal collection of these curves in the same isotopy class must be contained in a (possibly degenerate) annulus carried by $\Lambda$.

All remaining leaves of $\Lambda$ are either compact or spiral onto compact leaves. Since $S$ is compact, there can only be finitely many proper homotopy classes of these leaves. Taking maximal collections of properly homotopic leaves, we obtain finitely many pockets.
\end{proof}

\begin{definition} [Consistent lamination] \label{defn:finitespirallam}
A spiraling lamination $\Lambda$ is \textbf{consistent} if no two (oriented) closed leaves cobound an annulus. 
\end{definition}

Equivalently, a spiraling lamination $\Lambda$ is consistent if whenever $\Lambda$ carries an annulus $A$, $\Lambda \cap A$ is a disjoint union of parallel closed leaves.

In this paper, we will be concerned with efficient spiraling train tracks carrying consistent spiraling laminations. Observe that if a train track with no large branches carries a consistent spiraling lamination, there is an (essentially) unique vector field making it into a spiraling train track, so that the spiraling of forward trajectories is compatible with the orientations of the closed leaves of the lamination. We will thus freely refer to such train tracks as spiraling train tracks.

In \Cref{lemma:spiralttexists} below, we  show that any consistent spiraling lamination is carried by an efficient spiraling train track. In \Cref{subsec:spiralfacts}, we will see that an efficient train track carrying a consistent spiraling lamination can be split into a spiraling train track in an essentially unique way.

We make an observation which we will use repeatedly for the rest of this paper. Suppose $\tau$ is an efficient spiraling train track fully carrying a consistent spiraling lamination $\Lambda$. The image of each closed leaf of $\Lambda$ in $\tau$ under the carrying map is embedded. Also the images of parallel leaves coincide, while the images of non-parallel leaves are disjoint. Hence if we let $[\lambda_1],...,[\lambda_k]$ be the parallel classes of the closed leaves of $\Lambda$, each $[\lambda_i]$ will correspond to some circular component of the sink of $\tau$. Conversely, a circular component of the sink of $\tau$ must carry some closed leaf of $\Lambda$. Hence there is a natural one-to-one correspondence between the parallel classes of the closed leaves of $\Lambda$ and the circular components of the sink of $\tau$.

Working towards \Cref{lemma:spiralttexists}, we make the following definition.

\begin{definition}
We say that a finite collection of points $\{x_i\}$ on $\partial S$ \textbf{carries} $\Lambda \cap \partial S$ if there exists a closed regular neighborhood $N$ of $\{x_i\}$ on $\partial S$ such that $\Lambda \cap \partial S \subset N$. In this case we say that the map $\Lambda \cap \partial S \hookrightarrow N \to \{x_i\}$ is the \textbf{carrying map}. Further, we say that $\{x_i\}$ fully carries $\Lambda \cap \partial S$ if the carrying map is surjective.
\end{definition}

\begin{lemma} \label{lemma:spiralttexists}
Let $\Lambda$ be a consistent spiraling lamination. Let $\{x_i\}$ be a finite collection of points on $\partial S$ fully carrying $\Lambda \cap \partial S$, with carrying map $\rho:\Lambda \cap \partial S \to \{x_i\}$. Then there exists an efficient spiraling train track $\tau$ with the set of stops equal to $\{x_i\}$ such that $\tau$ fully carries $\Lambda$ and the restriction of the carrying map $\Lambda \to \tau$ extends $\rho$ on the boundary.
\end{lemma}
\begin{proof}
Let $\Lambda$ be a spiraling lamination. We construct a spiraling train track $(\tau,V)$ according to \Cref{prop:spirallam} in the following way.

From each parallel class of closed leaves of $\Lambda$, take one closed leaf to be part of $\tau$ with a nonsingular vector field defining its orientation as in item (1) of \Cref{defn:spirallam}. This uses the consistency of $\Lambda$.

By subdivision of pockets, we can assume that each end of a pocket spirals around some closed leaf of $\Lambda$ or is mapped to a single point in the finite collection $\{x_i\}$ by the carrying map, and that each pocket is maximal with respect to this property. For each pocket, we add a branch between the corresponding closed leaves and/or boundary points. When attaching a branch to one of the closed leaves, we smooth so that the branching is compatible with the orientation.

Finally we slightly zip up the branches at the stops and perturb so that switches are trivalent to get the desired spiraling train track. 
\end{proof}

\subsection{Splitting train tracks carrying spiraling laminations} \label{subsec:spiralfacts}

For this subsection $\Lambda$ will denote a consistent spiraling lamination, $S$ a compact surface, and $\tau$ a train track fully carrying $\Lambda$.

We say a cusp $c$ of $\tau$ is \textbf{persistent} if the maximal $\Lambda$-route from $c$ does not terminate in $\intr(S)$, and denote the set of persistent cusps of $\tau$ by $\pers(\tau)$.

Let $\lambda_1,...,\lambda_k$ be the closed leaves of $\Lambda$. Each $\lambda_i$ determines a bi-infinite periodic train route.
If the maximal $\Lambda$-route from a cusp $c$ fellow travels $\lambda_i$ in $\tau$, then we call $c$ a \textbf{graft point} for $\lambda_i$. We let $\graft(\lambda_i)$ be the (finite) set of graft points of $\lambda_i$ and circularly order this set according to the order in which $\lambda_i$ meets the cusps.

If the maximal $\Lambda$-route from a cusp $c$ \emph{eventually} fellow travels $\lambda_i$, we define the \textbf{reduced $\Lambda$-route} for $c$ to be the initial segment before its maximal $\Lambda$-route fellow travels $\lambda_i$. If $c$ is a graft point, then by convention the reduced $\Lambda$-route for $c$ is just the constant path $c$.

If instead the maximal $\Lambda$-route from a cusp $c$ terminates at a stop, we define the \textbf{reduced $\Lambda$-route} for $c$ to be the same as its maximal $\Lambda$-route.

We define a partial order $\preceq$ on the set $\pers(\tau)$ as follows. If $c_1$ and $c_2$ are two persistent cusps with reduced $\Lambda$-routes $\rho_{c_1}$ and $\rho_{c_2}$, then $c_2\succeq c_1$ if and only if $\rho_{c_2}$ is a concatenation 
\[
\rho_{c_2}=\gamma*\rho_{c_1},
\]
where $\gamma$ is some initial train route. 
See \Cref{fig:partialorder}.

Suppose $\tau$ and $\tau'$ are train tracks carrying $\Lambda$. If there is a poset isomorphism $\pers(\tau)\to \pers(\tau')$ which also respects the circular orders on $\graft(l_i)$ for each $i$, then we say that the map $\pers(\tau)\to \pers(\tau')$ is \textbf{order preserving}.

\begin{figure}
    \centering
    \fontsize{8pt}{8pt}\selectfont
    \resizebox{!}{2in}{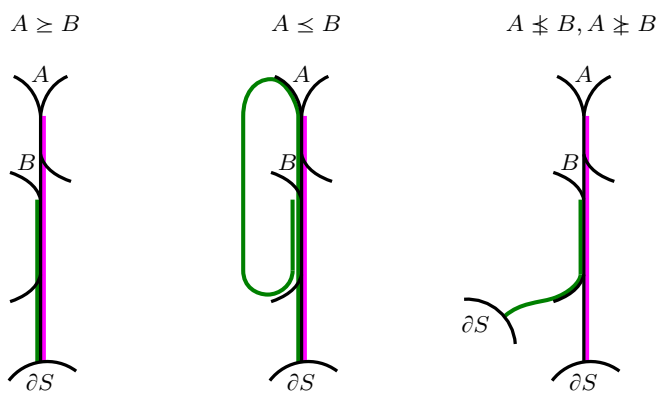}
    \caption{The pink lines represent the $\Lambda$-route from $A$, and the green lines represent possible $\Lambda$-routes from $B$. On the left we have $A\succeq B$, and in the center $A\preceq B$. On the right $A$ and $B$ are not comparable.}
    \label{fig:partialorder}
\end{figure}

\begin{lemma} \label{lemma:splitpreserveorder}
Any $\Lambda$-split is order preserving. More precisely, if $\tau$ carries a spiraling lamination $\Lambda$ and $\tau'$ is obtained from $\tau$ by a $\Lambda$-split, then the natural identification $\pers(\tau)\to \pers(\tau')$ is order preserving.
\end{lemma}
\begin{proof}
Suppose that $\tau'$ is obtained by performing a $\Lambda$-split on a branch of $\tau$. Let $\wt S$ be the universal cover of $S$, and let $\wt \Lambda$ and $\wt \tau$ be the lifts of $\Lambda$ and $\tau$ to $\wt S$.

First we explain why the property of being a graft point is preserved by a $\Lambda$-split, as is the circular order on the graft points of any circular leaf of $\Lambda$. If $\gamma$ is a closed loop in $\tau$ following the same route as a closed leaf $\lambda_i$, let $\wt\gamma$ be a lift to $\wt S$. By the argument in \cite[Cor 1.1.2]{PenHar92}, $\wt\tau$ is a tree so $\wt \gamma$ is an oriented line embedded in $\wt \tau$. Let $\wt{\graft(\lambda_i)}$ be the set of lifts of points in $\graft(\lambda_i)$ lying on $\wt \gamma$. The orientation of $\wt \gamma$ induces a linear order $\wt{\graft(\lambda_i)}$. Any $\Lambda$-split on $\tau$ lifts to infinitely many disjoint $\wt\Lambda$-splits, and each point of $\wt{\graft(\lambda_i)}$ is involved in at most one of these splits. It is easy to see that no finite collection of splits can reverse the linear order of any two points in $\wt{\graft(\lambda_i)}$, or remove any point in $\wt{\graft(\lambda_i)}$ from $\wt \gamma$. It follows that any $\Lambda$-split of $\tau$ preserves graft points and the circular order on $\graft(\lambda_i)$ for each circular leaf $\lambda_i$.

Next, let $A$ and $B$ be persistent cusps of $\tau$, and suppose that $A\succeq B$. Hence if $\gamma_A$ and $\gamma_B$ are the reduced $\Lambda$-routes from $A$ and $B$ respectively we have 
\[
\rho_A=\gamma*\rho_B
\]
for some initial train route $\gamma$. Let $C$ be the terminal point of $\rho_A$ and $\rho_B$. As above, a lift of $\rho_A$ to the universal cover $\wt S$ is embedded. This makes it clear that no $\Lambda$-split can reverse the order of $A$ and $B$. The same ideas show that any fold on $\tau'$ preserves the relation $A\succeq B$. Since any $\Lambda$-split can be reversed by a fold, this shows that $A\succeq B$ in $\pers(\tau)$ if and only if $A\succeq B$ in $\pers(\tau')$.
\end{proof}

We will show in \Cref{lemma:maximalsplitting} that a train track fully carrying a spiraling lamination can be $\Lambda$-split to a spiraling train track. To do so, we need a measure of complexity that decreases after each $\Lambda$-split. To this end we introduce the following definition.

\begin{definition}
If $A$ and $B$ are cusps of $\tau$, a \textbf{large $\Lambda$-biroute} connecting $A$ and $B$ is a train route $\gamma$ in $\tau$ from $A$ to $B$ such that: 
\begin{itemize}
    \item the maw vector field points into $\gamma$ at both $A$ and $B$, and
    \item $\gamma$ is an initial subroute of the $\Lambda$-route from $A$, and the reverse of $\gamma$ is an initial subroute of the $\Lambda$-route from $B$.
\end{itemize}
We consider two large $\Lambda$-biroutes to be the same if they differ by reversing orientation.
\end{definition}

\begin{lemma}\label{lem:finitebiroutes}
Let $\Lambda$ be a consistent spiraling lamination. Suppose $\tau$ is a train track fully carrying $\Lambda$. Then the number of large $\Lambda$-biroutes is finite.
\end{lemma}

\begin{proof}
Fix a cusp $A$ of $\tau$. It is enough to show that there are finitely many large $\Lambda$-biroutes with an endpoint on $A$. 

Since $\Lambda$ is spiraling, the $\Lambda$-route $\rho_A$ from $A$ either has finite length or eventually periodically traverses a closed oriented route in $\tau$. 
If $\rho_A$ has finite length, it is clear there are only finitely many large $\Lambda$-biroutes ending at $A$. 

If $\rho_A$ eventually fellow travels a closed route $\rho$ in $\tau$, then there are at most finitely many large $\Lambda$-biroutes connecting $A$ to any cusp not lying along $\rho$.
If $B$ is any cusp on the closed route $\rho$ such that there is a large $\Lambda$-biroute connecting $A$ and $B$, then the maw vector field at $B$ points against the orientation of $\rho$. 
Since $\Lambda$ is consistent, there cannot a closed leaf of $\Lambda$ traversing $\rho$ in the opposite direction, so the $\Lambda$-route $\rho_B$ from such a cusp $B$ can traverse $\rho$ at most finitely many times before leaving $\rho$. 
In particular, there are only finitely many initial subroutes of $\rho_A$ which are initial subroutes of $\rho_B$ when reversed, so there are only finitely many large $\Lambda$-biroutes connecting $A$ and $B$.
\end{proof}

\begin{lemma} \label{lemma:maximalsplitting}
Let $\Lambda$ be a consistent spiraling lamination. 
Suppose $\tau$ is a train track fully carrying $\Lambda$. Then any sequence of $\Lambda$-splits $\tau=\tau_0 \to\tau_1\to\tau_2\to\dots$ must terminate in a spiraling train track. 
In particular, there exists a sequence of $\Lambda$-splits $\tau=\tau_0 \to \cdots \to \tau_n$ such that $\tau_n$ is spiraling.
\end{lemma}
\begin{proof}
Let $BR(\tau)$ be the set of large $\Lambda$-biroutes of $\tau$. By \Cref{lem:finitebiroutes}, $\# BR(\tau)<\infty$. 
It is clear that $BR(\tau)$ is nonempty if and only if $\tau$ has large branches. Let $b$ be a large branch of $\tau$, and $\tau'$ be the track obtained by performing a $\Lambda$-split on $b$. Note that $b$ itself is an element of $BR(\tau)$, and there is a natural bijection 
\[
\sigma: (BR(\tau)-\{b\})\to BR(\tau').
\]
(See \Cref{fig:biroutesplit} for a visual description of $\sigma$ and $\sigma^{-1}$).
\begin{figure}
    \centering
    \fontsize{8pt}{8pt}\selectfont
    \resizebox{!}{1.5in}{
\begingroup%
  \makeatletter%
  \providecommand\color[2][]{%
    \errmessage{(Inkscape) Color is used for the text in Inkscape, but the package 'color.sty' is not loaded}%
    \renewcommand\color[2][]{}%
  }%
  \providecommand\transparent[1]{%
    \errmessage{(Inkscape) Transparency is used (non-zero) for the text in Inkscape, but the package 'transparent.sty' is not loaded}%
    \renewcommand\transparent[1]{}%
  }%
  \providecommand\rotatebox[2]{#2}%
  \newcommand*\fsize{\dimexpr\f@size pt\relax}%
  \newcommand*\lineheight[1]{\fontsize{\fsize}{#1\fsize}\selectfont}%
  \ifx\svgwidth\undefined%
    \setlength{\unitlength}{265.06650346bp}%
    \ifx\svgscale\undefined%
      \relax%
    \else%
      \setlength{\unitlength}{\unitlength * \real{\svgscale}}%
    \fi%
  \else%
    \setlength{\unitlength}{\svgwidth}%
  \fi%
  \global\let\svgwidth\undefined%
  \global\let\svgscale\undefined%
  \makeatother%
  \begin{picture}(1,0.37143848)%
    \lineheight{1}%
    \setlength\tabcolsep{0pt}%
    \put(0,0){\includegraphics[width=\unitlength,page=1]{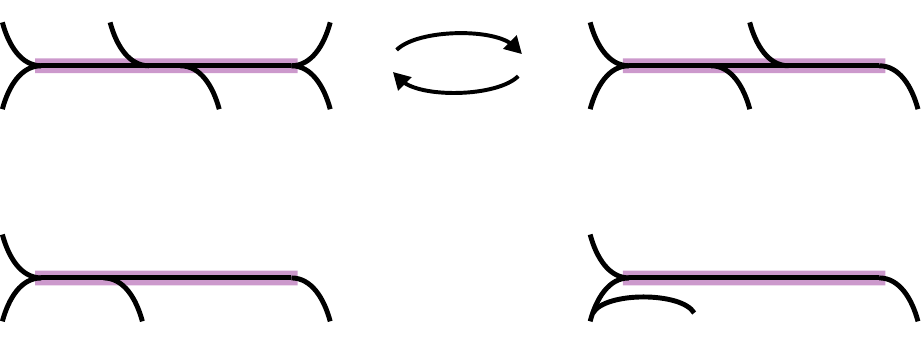}}%
    \put(0.47798888,0.34821318){\color[rgb]{0,0,0}\makebox(0,0)[lt]{\lineheight{1.25}\smash{\begin{tabular}[t]{l}$\sigma$\end{tabular}}}}%
    \put(0.47172183,0.23825112){\color[rgb]{0,0,0}\makebox(0,0)[lt]{\lineheight{1.25}\smash{\begin{tabular}[t]{l}$\sigma^{-1}$\end{tabular}}}}%
    \put(0,0){\includegraphics[width=\unitlength,page=2]{fig_biroutesplit.pdf}}%
    \put(0.47798888,0.11619571){\color[rgb]{0,0,0}\makebox(0,0)[lt]{\lineheight{1.25}\smash{\begin{tabular}[t]{l}$\sigma$\end{tabular}}}}%
    \put(0.47172183,0.0062337){\color[rgb]{0,0,0}\makebox(0,0)[lt]{\lineheight{1.25}\smash{\begin{tabular}[t]{l}$\sigma^{-1}$\end{tabular}}}}%
    \put(0,0){\includegraphics[width=\unitlength,page=3]{fig_biroutesplit.pdf}}%
  \end{picture}%
\endgroup%
}
    \caption{If $\tau'$ is obtained by $\Lambda$-splitting a branch $b$, then there is a bijection $\sigma\colon BR(\tau)-\{b\}\to BR(\tau)$. On the top and bottom we see $\Lambda$-biroutes highlighted in purple before and after a $\Lambda$-split.}
    \label{fig:biroutesplit}
\end{figure}
Hence a $\Lambda$-split decreases the number of large $\Lambda$-biroutes by one, so any sequence of $\Lambda$-splits must terminate in a train track with no large $\Lambda$-biroutes, i.e. one with no large branches.
\end{proof}

A splitting sequence $\tau_0\to\cdots \to \tau_n$ where $\tau_n$ has no large branches, as in \Cref{lemma:maximalsplitting}, is said to be a \textbf{maximal splitting sequence} for $\tau_0$. The train track $\tau_n$ is a \textbf{maximal splitting} of $\tau_0$.
It turns out that a maximal splitting is unique up to isotopy, and a maximal splitting sequence is unique up to a natural equivalence. 

If $b_1$ and $b_2$ are disjoint large branches of $\tau$, then the moves splitting $b_1$ and $b_2$ commute with each other, and we call the operation of swapping their order a \textbf{commutation}.

\begin{lemma}
 \label{lemma:differbycommutations}
If $\tau$ is a train track fully carrying a consistent spiraling lamination $\Lambda$, then any two maximal splitting sequences for $\tau$ are related by commutations. In particular any two maximal splitting sequences have the same length and end in the same train track.
\end{lemma}

\begin{proof}
Let the two splitting sequences be $\tau=\tau_0 \to \cdots \to \tau_k$ and $\tau=\tau'_0 \to\cdots\to \tau'_{k'}$. We induct on $\max\{k, k'\}$. When $\max\{k,k'\}=0$, the statement is clear.

Suppose branch $b$ is split in $\tau_0 \to \tau_1$. Locate the term in the other splitting sequence where $b$ is split. The terms before that are performed on branches disjoint from $b$, so we can move the splitting of $b$ to the beginning of the sequence via commutations, then apply the inductive hypothesis to $\tau_1$. 
\end{proof}

\begin{definition}[$\Lambda$-compatible]\label{def:lambdacomp}
Let $\Lambda$ be a consistent spiraling lamination, and let $\tau$ be an efficient train track carrying $\Lambda$. Each complementary region $C$ of $\Lambda$ is a surface with boundary, possibly noncompact. Its boundary $\partial C$ can be decomposed into $\partial_v C$, which lies along $\partial S$, and $\partial_h C$, which lies along leaves of $\Lambda$. We say that a component of $\partial_v C$ is \textbf{associated} to a stop $v$ of $\tau$ if both of its endpoints are mapped by the carrying map into $v$. Note that not every component of $\partial_v C$ need be associated to a stop; see \Cref{fig:verticalcomponent}.

If $\tau'$ is another efficient train track carrying $\Lambda$, then we say $\tau$ and $\tau'$ are \textbf{$\Lambda$-compatible} if the set of stops of $\tau$ equals the set of stops of $\tau'$, and for each complementary region of $\Lambda$, a component of $\partial_v C$ is associated to a stop of $\tau$ if and only if it is associated to the same stop of $\tau'$.
\end{definition}

Note that if $\del S=\varnothing$, then the $\Lambda$-compatibility condition is vacuous.

\begin{definition}[$\Lambda$-identification]
Suppose $\tau_1$ and $\tau_2$ are $\Lambda$-compatible. Each complementary region $T$ of $\tau_i$ corresponds to a complementary region $C$ of $\Lambda$. Since $\tau_i$ is efficient, $C$ either has negative index, or some component of $\partial_v C$ is not associated to any stop. Each persistent cusp of $\tau_i$ in $T$ corresponds to an end of $C$ or a component of $\partial_v C$ that is associated to a stop of $\tau_i$. Conversely, if $C$ is a complementary region of $\Lambda$ of negative index or has some component of $\partial_v C$ not associated to a stop of $\tau_i$, then each end of $C$ and each component of $\partial_v C$ that is associated to a stop of $\tau_i$ corresponds to some persistent cusp of $\tau_i$. From this, we get a natural identification between $\pers(\tau_1)$ and $\pers(\tau_2)$, which we call the \textbf{$\Lambda$-identification}.
\end{definition}

\begin{figure}
    \centering
    \fontsize{12pt}{12pt}\selectfont
    \resizebox{!}{4cm}{
\begingroup%
  \makeatletter%
  \providecommand\color[2][]{%
    \errmessage{(Inkscape) Color is used for the text in Inkscape, but the package 'color.sty' is not loaded}%
    \renewcommand\color[2][]{}%
  }%
  \providecommand\transparent[1]{%
    \errmessage{(Inkscape) Transparency is used (non-zero) for the text in Inkscape, but the package 'transparent.sty' is not loaded}%
    \renewcommand\transparent[1]{}%
  }%
  \providecommand\rotatebox[2]{#2}%
  \newcommand*\fsize{\dimexpr\f@size pt\relax}%
  \newcommand*\lineheight[1]{\fontsize{\fsize}{#1\fsize}\selectfont}%
  \ifx\svgwidth\undefined%
    \setlength{\unitlength}{146.47959606bp}%
    \ifx\svgscale\undefined%
      \relax%
    \else%
      \setlength{\unitlength}{\unitlength * \real{\svgscale}}%
    \fi%
  \else%
    \setlength{\unitlength}{\svgwidth}%
  \fi%
  \global\let\svgwidth\undefined%
  \global\let\svgscale\undefined%
  \makeatother%
  \begin{picture}(1,0.79521805)%
    \lineheight{1}%
    \setlength\tabcolsep{0pt}%
    \put(0,0){\includegraphics[width=\unitlength,page=1]{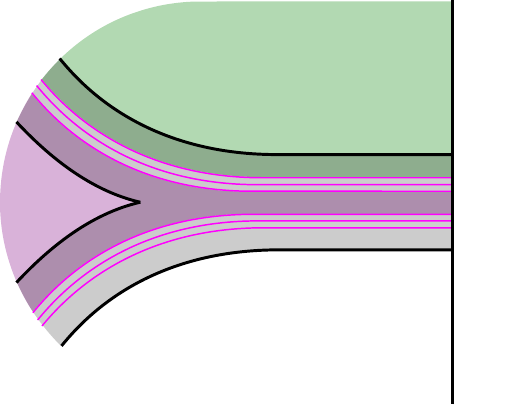}}%
    \put(0.04812412,0.37121263){\color[rgb]{0,0,0}\makebox(0,0)[lt]{\lineheight{1.25}\smash{\begin{tabular}[t]{l}$A$\end{tabular}}}}%
    \put(0.39929842,0.62669241){\color[rgb]{0,0,0}\makebox(0,0)[lt]{\lineheight{1.25}\smash{\begin{tabular}[t]{l}$B$\end{tabular}}}}%
    \put(0.91954403,0.17227757){\color[rgb]{0,0,0}\makebox(0,0)[lt]{\lineheight{1.25}\smash{\begin{tabular}[t]{l}$\partial S$\end{tabular}}}}%
  \end{picture}%
\endgroup%
}
    \caption{We have drawn a standard neighborhood (gray) of the pink lamination, and shaded two complementary regions $A$ and $B$. There is a component of $\del_vA$ which is associated to the stop obtained by collapsing the standard neighborhood, while the component of $\del_vB$ shown lying on $\del S$ is not associated to a stop.}
    \label{fig:verticalcomponent}
\end{figure}
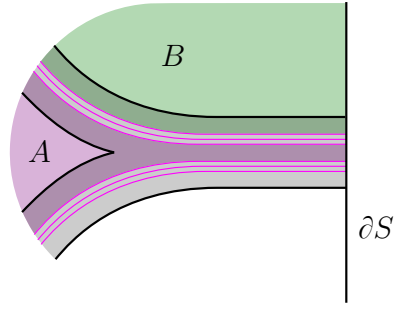

\begin{lemma}\label{lem:diffbyshifts}
Let $\Lambda $ be a consistent spiraling lamination on $S$ and let $\tau_1$ and $\tau_2$ be two efficient spiraling train tracks fully carrying $\Lambda$. 
\begin{enumerate}[label=(\alph*)]
\item If $\tau_1$ and $\tau_2$ are $\Lambda$-compatible, then up to isotopy they are related by a collection of shifts.

\item If the $\Lambda$-identification between their cusps preserves orders, then up to isotopy $\tau_1$ and $\tau_2$ are related by Dehn twists around curves isotopic to closed leaves of $\Lambda$. The effect of these Dehn twists can be achieved by shifts along the circular components of $\sink(\tau_i)$.
\end{enumerate}
\end{lemma}

\begin{figure}
     \centering
     \fontsize{10pt}{10pt}\selectfont
     \resizebox{!}{1.5in}{
\begingroup%
  \makeatletter%
  \providecommand\color[2][]{%
    \errmessage{(Inkscape) Color is used for the text in Inkscape, but the package 'color.sty' is not loaded}%
    \renewcommand\color[2][]{}%
  }%
  \providecommand\transparent[1]{%
    \errmessage{(Inkscape) Transparency is used (non-zero) for the text in Inkscape, but the package 'transparent.sty' is not loaded}%
    \renewcommand\transparent[1]{}%
  }%
  \providecommand\rotatebox[2]{#2}%
  \newcommand*\fsize{\dimexpr\f@size pt\relax}%
  \newcommand*\lineheight[1]{\fontsize{\fsize}{#1\fsize}\selectfont}%
  \ifx\svgwidth\undefined%
    \setlength{\unitlength}{224.26007568bp}%
    \ifx\svgscale\undefined%
      \relax%
    \else%
      \setlength{\unitlength}{\unitlength * \real{\svgscale}}%
    \fi%
  \else%
    \setlength{\unitlength}{\svgwidth}%
  \fi%
  \global\let\svgwidth\undefined%
  \global\let\svgscale\undefined%
  \makeatother%
  \begin{picture}(1,0.48294546)%
    \lineheight{1}%
    \setlength\tabcolsep{0pt}%
    \put(0,0){\includegraphics[width=\unitlength,page=1]{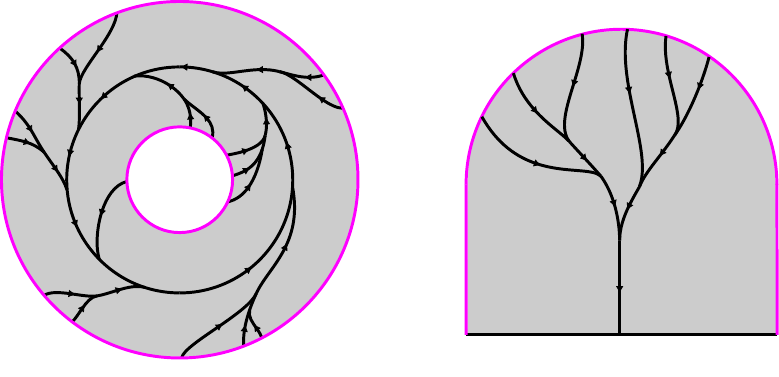}}%
    \put(0.68258404,0.00897047){\color[rgb]{0,0,0}\makebox(0,0)[lt]{\lineheight{1.25}\smash{\begin{tabular}[t]{l}$\partial S$\end{tabular}}}}%
    \put(0,0){\includegraphics[width=\unitlength,page=2]{fig_spiralingpieces.pdf}}%
  \end{picture}%
\endgroup%
}
     \caption{Picture from the proof of \Cref{lem:diffbyshifts}.}
     \label{fig:spiralingpieces}
\end{figure}
\begin{proof}
Since $\tau_1$ and $\tau_2$ are $\Lambda$-compatible spiraling train tracks carrying $\Lambda$, there is a natural identification of their sinks (recall that the sink of $\tau_i$ is a collection of circles and stops).

Let $N$ be the closure of a small neighborhood of the sink of $\tau_1$ and $\tau_2$. Thus $N$ is a collection of annuli and disks, where each disk has two corners and half of its boundary on $\del S$ and half of its boundary in $\intr(S)$. See \Cref{fig:spiralingpieces}. Let $S'= S\cut N$.

By initial isotopies, we can arrange so that all switches of $\tau_1$ and $\tau_2$ lie in $N$.

We observe that $S'$ is a surface with boundary, and that $N\cap S'$ is a collection of intervals and circles contained in $\del S'$. Each leaf of the lamination $\Lambda\cap S'$ is a properly embedded interval, both of whose endpoints lie in a component of $N\cap S'$. 

Let $i=1$ or $2$. Each component of the train track $\tau_i \cap S'$ is also a properly embedded segment with endpoints in $N\cap S'$. Note that $\tau_i \cap S'$ has no components which are isotopic rel $\del N$. If such components existed, there would be a rectangle $R$ with cyclically ordered edges $a,b,c,d$ such that $a, c\subset \tau_i$ and $b,d\subset S'\cap N$. See \Cref{fig:norects}. 
 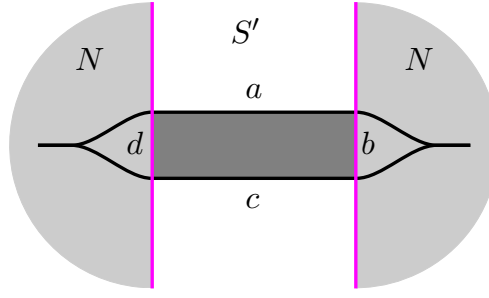
\begin{figure}
     \centering
     \resizebox{!}{1.5in}{
\begingroup%
  \makeatletter%
  \providecommand\color[2][]{%
    \errmessage{(Inkscape) Color is used for the text in Inkscape, but the package 'color.sty' is not loaded}%
    \renewcommand\color[2][]{}%
  }%
  \providecommand\transparent[1]{%
    \errmessage{(Inkscape) Transparency is used (non-zero) for the text in Inkscape, but the package 'transparent.sty' is not loaded}%
    \renewcommand\transparent[1]{}%
  }%
  \providecommand\rotatebox[2]{#2}%
  \newcommand*\fsize{\dimexpr\f@size pt\relax}%
  \newcommand*\lineheight[1]{\fontsize{\fsize}{#1\fsize}\selectfont}%
  \ifx\svgwidth\undefined%
    \setlength{\unitlength}{165.67880004bp}%
    \ifx\svgscale\undefined%
      \relax%
    \else%
      \setlength{\unitlength}{\unitlength * \real{\svgscale}}%
    \fi%
  \else%
    \setlength{\unitlength}{\svgwidth}%
  \fi%
  \global\let\svgwidth\undefined%
  \global\let\svgscale\undefined%
  \makeatother%
  \begin{picture}(1,0.58615147)%
    \lineheight{1}%
    \setlength\tabcolsep{0pt}%
    \put(0,0){\includegraphics[width=\unitlength,page=1]{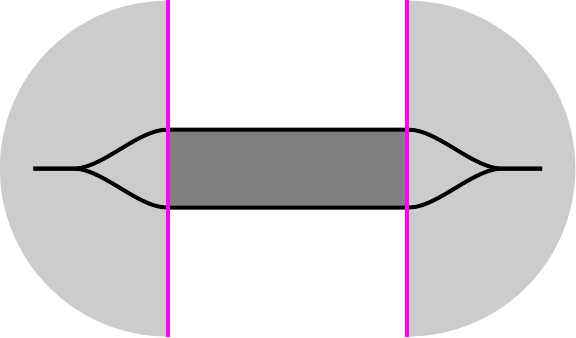}}%
    \put(0.13456305,0.44760534){\color[rgb]{0,0,0}\makebox(0,0)[lt]{\lineheight{1.25}\smash{\begin{tabular}[t]{l}$N$\end{tabular}}}}%
    \put(0.45229584,0.50259531){\color[rgb]{0,0,0}\makebox(0,0)[lt]{\lineheight{1.25}\smash{\begin{tabular}[t]{l}$S'$\end{tabular}}}}%
    \put(0.80755458,0.44759946){\color[rgb]{0,0,0}\makebox(0,0)[lt]{\lineheight{1.25}\smash{\begin{tabular}[t]{l}$N$\end{tabular}}}}%
    \put(0.4825065,0.38883029){\color[rgb]{0,0,0}\makebox(0,0)[lt]{\lineheight{1.25}\smash{\begin{tabular}[t]{l}$a$\end{tabular}}}}%
    \put(0.48250729,0.16527917){\color[rgb]{0,0,0}\makebox(0,0)[lt]{\lineheight{1.25}\smash{\begin{tabular}[t]{l}$c$\end{tabular}}}}%
    \put(0.23953834,0.27431666){\color[rgb]{0,0,0}\makebox(0,0)[lt]{\lineheight{1.25}\smash{\begin{tabular}[t]{l}$d$\end{tabular}}}}%
    \put(0.71893788,0.27431725){\color[rgb]{0,0,0}\makebox(0,0)[lt]{\lineheight{1.25}\smash{\begin{tabular}[t]{l}$b$\end{tabular}}}}%
  \end{picture}%
\endgroup%
}
     \caption{Components of $\tau_i\cap S'$ which are isotopic in $S'$ rel $S'\cap N$ force the existence of cusped bigon complementary components of $\tau_i$ for $i=1,2$.}
     \label{fig:norects}
 \end{figure}
Since $\Lambda$ is spiraling, this would give rise to a cusped bigon complementary region of $\tau_i$, contradicting efficiency.
It follows that the branches of $\tau_i \cap S'$ are in one-to-one correspondence with the isotopy classes rel $S'\cap N$ of the leaves of $\Lambda\cap S'$. 
Fix such an isotopy class, $P$. The leaves of $P$ lie in a closed rectangle $R_P$ with two opposite sides on $S'\cap N$, and two opposite sides formed by leaves in $P$. Up to isotopy we can assume that the corresponding component of $\tau_i \cap S'$ lies inside $R_P$.
Hence we can perform an isotopy of $S$ supported in a neighborhood of $\Lambda$ so that $\tau_1\cap S'=\tau_2\cap S'$.

Now we turn our attention to $\tau_1\cap N$ and $\tau_2\cap N$.

Let $C$ be a disk component of $N$. Then $\tau_1\cap C$ and $\tau_2\cap C$ are both tracks with no large branches whose sets of stops are equal. Their cusps are also identified via the  $\Lambda$-identification; we will call these cusps $c_1,\dots, c_n$. There is a distinguished stop $v\in \del C$ such that the $\Lambda$-route from each $c_j$ ends at $v$. It is not hard to show that $\tau_1\cap C$ and $\tau_2\cap C$ are related by shifts and isotopy rel boundary; we will do so by showing that each $\tau_i$ can be shifted to obtain the same track.

Each $c_j$ is associated to a complementary region $R_j$ of $\Lambda\cap C$ that has a vertical component associated to the stop $v$. If we choose an orientation of $\del C$, we may assume up to relabeling that the vertical components of $R_1,\dots, R_n$ move from left to right in a small neighborhood of $v$. By performing shifts on each $\tau_i\cap C$, we can arrange so that $c_1\preceq c_2\preceq\cdots\preceq c_n$. See \Cref{fig:forexample}.
\begin{figure}
    \centering
    \fontsize{8pt}{8pt}\selectfont
    \resizebox{!}{1.5in}{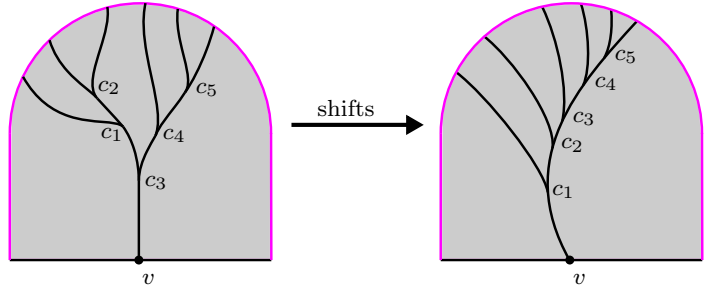}
    \caption{When $C$ is a disk, we can shift both $\tau_1\cap C$ and $\tau_2\cap C$ so that, for example,  $c_1\preceq\cdots\preceq c_n$. Hence the two tracks differ by shifts and isotopy rel stops.}
    \label{fig:forexample}
\end{figure}
We conclude that $\tau_1\cap C$ and $\tau_2\cap C$ are related by shifts.

If the $\Lambda$-identification between the cusps of $\tau_1$ and $\tau_2$ preserves orders, then an induction on the number of cusps of $\tau_1\cap C$ and $\tau_2\cap C$ shows that the two tracks are isotopic rel stops; that is, no shifts are necessary above.

Next, suppose that $A$ is an annulus component of $N$ containing a circular sink component $\ell$. For $i=1,2$ we perform shifts on $\tau_i$ to obtain a track $\tau_i'$ whose only cusps lie on $\ell$. We then perform shifts and isotopy on $\tau_1'$ and $\tau_2'$, obtaining new tracks $\tau_1''$ and $\tau_2''$, so that the circular order on the cusps of $\tau_1''$ and $\tau_2''$ is the same. A further isotopy can arrange so that the sinks and corresponding cusps of $\tau_1''$ and $\tau_2''$ agree as points in $A$. At this point $\tau_1''$ and $\tau_2''$ differ only by the application of some number of Dehn twists around curves parallel to $\ell$. The effect of applying these Dehn twists can be achieved by shifts and isotopy. We have completed the proof of (a).

Next we suppose that the $\Lambda$-identification between the cusps of $\tau_1$ and $\tau_2$ preserves orders. We assume that we have already isotoped $\tau_1$ and $\tau_2$ so that all their cusps lie in $N$ and so that they agree outside $N$.
In a disk component $C$ of $N$, we have already observed in the proof of (a) that the partial order $\preceq$ determines each $\tau_i\cap C$ up to isotopy rel stops. Hence we need consider only the annulus components of $N$.

Now let $A$ be an annulus component of $N$ containing a circular sink $\ell$. Up to isotopy we can assume that $\tau_1\cap A$ and $\tau_2\cap A$ agree on the circular sink $\ell$. Since the $\Lambda$-identification preserves orders and the graft points are minimal in the partial order, $\tau_1$ and $\tau_2$ have the same collection of graft points along $\ell$. Let $v$ be one of these graft points, and let $\tau_i^v$ be the component of $(\tau_i\cap A)\cut v$ not containing $\ell$. Since the $\Lambda$-identification preserves orders and $\tau_1^v$, $\tau_2^v$ have the same sets of stops, we see that (up to isotopy rel stops) they can differ only in how many times they wrap around $A$. Applying this analysis to each cusp along $\ell$ gives that $\tau_1\cap A$ and $\tau_2\cap A$ con only differ by isotopy and by twisting around a curve parallel to $\ell$. The twisting can clearly be achieved by performing shift moves involving branches incident to $\ell$.
\end{proof}

When restricted to a subclass of spiraling laminations called $I$-laminations, this lemma has the important \Cref{cor:anyorder}, which will play a role later. We first define $I$-laminations.

\begin{definition}[$I$-laminations]
Let $K$ be a compact surface with boundary. An \textbf{$I$-lamination} is a lamination in which every leaf is a compact, properly embedded arc in $K$. 
\end{definition}

Notice that every $I$-lamination is spiraling: condition (2) in \Cref{defn:spirallam} concerning noncompact leaves is vacuous. Furthermore, any $I$-lamination is consistent since it does not carry any annuli.

Let $\tau$ be a train track fully carrying an $I$-lamination $\lambda$, and suppose $a$ and $b$ are cusps of $\tau$ at two ends of a mixed branch $m$ as shown in \Cref{fig:incomparableshift}. We say that $a$ and $b$ are \textbf{divergent neighbors} if either $a$ or $b$ is nonpersistent, or if both $a$ and $b$ are persistent and $\rho_a\ne m*\rho_b$, where as in \Cref{subsec:spiralfacts} $\rho_c$ denotes the maximal $\Lambda$-route from a cusp $c$. If $a$ and $b$ are divergent neighbors then we call the shift along $m$ a \textbf{shift of divergent neighbors}.

\begin{corollary}\label{cor:anyorder}
Let $S$ be a compact surface, $\Lambda$ an $I$-lamination in $S$, and $\sigma$, $\tau$ two $\Lambda$-compatible efficient train tracks which fully carry $\Lambda$.

Let $\sigma'$ and $\tau'$ be maximal splittings of $\sigma$ and $\tau$ respectively. Then $\sigma'$ and $\tau'$ are isotopic rel stops if and only if the $\Lambda$-identification $\pers(\sigma)\to \pers(\tau)$ is an isomorphism of partially ordered sets.

Furthermore, if $\sigma$ and $\tau$ differ by a sequence of shifts of divergent neighbors, then $\sigma'$ and $\tau'$ are isotopic rel stops.
\end{corollary}
\begin{proof}
By \Cref{lemma:splitpreserveorder}, the natural identifications of $\pers(\sigma)$ with $\pers(\sigma')$ and of $\pers(\tau)$ with $\pers(\tau')$ are poset isomorphisms. If $\sigma'$ and $\tau'$ are isotopic, then the $\Lambda$-identification $\pers(\sigma') \cong \pers(\tau')$ is a poset isomorphism, thus the $\Lambda$-identification $c: \pers(\sigma) \to \pers(\tau)$, being a composition of these identifications, is a poset isomorphism.

Conversely, if the $\Lambda$-identification $\pers(\sigma)\to \pers(\tau)$ is a poset isomorphism, then the $\Lambda$-identification $\pers(\sigma') \cong \pers(\tau')$ is a poset isomorphism by the above reasoning. Hence by \Cref{lem:diffbyshifts}, $\sigma'$ and $\tau'$ are isotopic since $\Lambda$ has no closed leaves. This proves the biconditional statement.

For the furthermore statement, in light of the above, it suffices to prove that a single shift of divergent neighbors induces a poset isomorphism. We do this by cases, after setting some notation. As before, $\rho_v$ denotes the maximal $\Lambda$-route from a cusp $v$. Suppose that prior to the shift, $a$ and $b$ are as shown in \Cref{fig:incomparableshift}. If $c$ is a cusp such that  $c\succeq a$ (resp. $c\succeq b$), we say that $\rho_c$ ``joins $a$ (resp. $b$) through" $x$, $y$, or $z$ if $x$, $y$, or $z$ is the last of $\{x,y,z\}$ traversed by $\rho_c$ before it fellow travels $\rho_a$ (resp. $\rho_b)$.

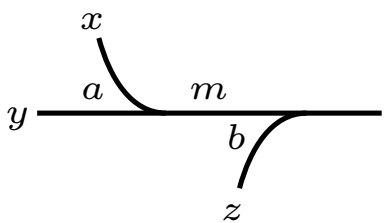
\begin{figure}
    \centering
    \fontsize{6pt}{6pt}\selectfont
    \resizebox{!}{3cm}{
\begingroup%
  \makeatletter%
  \providecommand\color[2][]{%
    \errmessage{(Inkscape) Color is used for the text in Inkscape, but the package 'color.sty' is not loaded}%
    \renewcommand\color[2][]{}%
  }%
  \providecommand\transparent[1]{%
    \errmessage{(Inkscape) Transparency is used (non-zero) for the text in Inkscape, but the package 'transparent.sty' is not loaded}%
    \renewcommand\transparent[1]{}%
  }%
  \providecommand\rotatebox[2]{#2}%
  \newcommand*\fsize{\dimexpr\f@size pt\relax}%
  \newcommand*\lineheight[1]{\fontsize{\fsize}{#1\fsize}\selectfont}%
  \ifx\svgwidth\undefined%
    \setlength{\unitlength}{62.11279069bp}%
    \ifx\svgscale\undefined%
      \relax%
    \else%
      \setlength{\unitlength}{\unitlength * \real{\svgscale}}%
    \fi%
  \else%
    \setlength{\unitlength}{\svgwidth}%
  \fi%
  \global\let\svgwidth\undefined%
  \global\let\svgscale\undefined%
  \makeatother%
  \begin{picture}(1,0.60441571)%
    \lineheight{1}%
    \setlength\tabcolsep{0pt}%
    \put(0,0){\includegraphics[width=\unitlength,page=1]{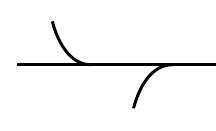}}%
    \put(0.58407163,0.20838815){\color[rgb]{0,0,0}\makebox(0,0)[lt]{\lineheight{1.25}\smash{\begin{tabular}[t]{l}$b$\end{tabular}}}}%
    \put(0.19384148,0.34225624){\color[rgb]{0,0,0}\makebox(0,0)[lt]{\lineheight{1.25}\smash{\begin{tabular}[t]{l}$a$\end{tabular}}}}%
    \put(0.19086974,0.53008018){\color[rgb]{0,0,0}\makebox(0,0)[lt]{\lineheight{1.25}\smash{\begin{tabular}[t]{l}$x$\end{tabular}}}}%
    \put(-0.00731092,0.27698407){\color[rgb]{0,0,0}\makebox(0,0)[lt]{\lineheight{1.25}\smash{\begin{tabular}[t]{l}$y$\end{tabular}}}}%
    \put(0.56953324,0.01702736){\color[rgb]{0,0,0}\makebox(0,0)[lt]{\lineheight{1.25}\smash{\begin{tabular}[t]{l}$z$\end{tabular}}}}%
    \put(0.48436937,0.34102054){\color[rgb]{0,0,0}\makebox(0,0)[lt]{\lineheight{1.25}\smash{\begin{tabular}[t]{l}$m$\end{tabular}}}}%
  \end{picture}%
\endgroup%
}
    \caption{Notation from the definition of divergent neighbors and the proof of \Cref{cor:anyorder}.}
    \label{fig:incomparableshift}
\end{figure}

Suppose $c$ and $d$ are comparable, say $c\succeq d$. Then either:
    \begin{enumerate}[label=(\alph*)]
        \item $\{c,d\}\cap \{a,b\}=\varnothing$, or 
        \item $c=a$, or
        \item $c=b$, or
        \item $d=a$, or
        \item $d=b$.
    \end{enumerate}
In each of (1a)-(1d), it is clear that shifting $a$ and $b$ does not affect $c\succeq d$. If $d=b$, then $c\succeq d$ holds after shifting $a$ and $b$ unless $\rho_c$ joins $a$ and $b$ through $x$, which cannot happen because $a$ and $b$ are divergent neighbors.
    
Since the inverse of a shift of divergent neighbors is also a shift of divergent neighbors, this shows that $c\succeq d$ before such a shift if and only if $c\succeq d$ after such a shift. Hence the shift induces a poset isomorphism.
\end{proof}

\section{Endperiodic maps} \label{sec:endperiodic}

In this section, we give a condensed treatment of the parts of Handel-Miller theory we need. Other than a few lemmas which we prove, all of the material presented here can be found in more detail in \cite{CCF19}.

\subsection{Surfaces and ends}

Let $L$ be an orientable surface.
Consider a sequence $A_1\supset A_2\supset A_3\supset\cdots $ such that there exists a compact exhaustion $K_1\subset K_2\subset K_3\subset\cdots$ of $L$ and $A_i$ is a connected component of $L-K_i$. We consider two such sequences $\{A_i\}$ and $\{B_i\}$ to be equivalent if each term of one sequence contains all but finitely many terms of the other. An \textbf{end} of $L$ is an equivalence classes of such sequences. If $e$ is an end of $L$ and $\{A_i\}$ is a sequence in the equivalence class $e$, we will call $\{A_i\}$ a \textbf{regular neighborhood basis} for $e$. We denote the set of ends of $L$ by $\ms E(L)$. 

It will be useful for our treatment to single out two specific types of ends of $L$. We say that an end $e$ is an \textbf{infinite strip end} if it has a regular neighborhood basis in which each set is homeomorphic to $[0,1]\times (0,\infty)$. 
We say that $e$ is an \textbf{infinite cyclinder end} if it has a regular neighborhood basis in which each set is homeomorphic to an open annulus.

There is a natural topology on $L\sqcup \ms E(L)$ which compactifies $L$: a base for this topology is given by the open sets of $L$ together with all sets of the form $\{e\}\cup V$ where $V$ is open in $L$ and contains a term of a regular neighborhood basis for $e$. Together with this topology, $L\sqcup \ms E(L)$ is called the \textbf{end compactification} of $L$; the subspace topology on $\ms E(L)$ makes it into a totally disconnected set. 
A \textbf{neighborhood} of the end $e$ is an open set $A\subset L$ such that $A\cup e$ is an honest neighborhood of $e$ in the end compactification of $L$. The end compactification motivates our use of the term ``regular neighborhood basis" for elements of $e$: if $A_1,A_2,A_3,\dots$ is a regular neighborhood basis for $e$ in our sense then $A_1\cup\{e\}, A_2\cup\{e\},A_3\cup\{e\},\dots$ is a regular neighborhood basis for $e$ in the end compactification, in the traditional point set topological sense.

\subsection{Endperiodic maps}
Let $f\colon L\to L$ be a homeomorphism. Then $f$ induces a homeomorphism of the end compactification which restricts to a homeomorphism of $\ms E(L)$. We say an end $e\in \ms E(L)$ is \textbf{periodic} if there exists an integer $p$ such that $f^p(e)=e$. If $p$ is the smallest such integer, we call it the \textbf{period} of $e$.

Let $e$ be a periodic end of $L$ with period $p$. Then $e$ is \textbf{positive} if there exists a neighborhood $U$ of $e$ such that $U, f^p(U), f^{2p}(U),\dots$ is a regular neighborhood basis for $e$. Symmetrically, $e$ is \textbf{negative} if $e$ is a positive end of $f^{-1}$. We think of the positive ends of $L$ as attracting and the negative ends as repelling. 

The set $\{e, f(e),\dots, f^{p-1}(e)\}$ is called a \textbf{positive end-cycle} or \textbf{negative end-cycle} depending on whether $e$ is positive or negative, respectively.

\begin{definition}[Endperiodic map, Reeb endperiodic map]\label{def:endperiodic}
Let $L$ be an oriented surface with finitely many ends, none of which are infinite cylinder ends. Let $f\colon L\to L$ be an orientation-preserving homeomorphism. We say that $f$ is \textbf{endperiodic} if
\begin{enumerate}[label=(\alph*)]
    \item all ends of $L$ are positive or negative, and
\end{enumerate}
if $L$ has a noncompact boundary component $\ell$ with period $p$, then
\begin{enumerate}[label=(\alph*),resume]
    \item $\ell$ runs between two ends of $L$ with the opposite sign and $f^p|_\ell$ has no fixed points.
\end{enumerate}
We say that $f$ is \textbf{Reeb endperiodic} if $f$ satisfies (a) and if $\ell$ is a noncompact component of $\del L$ with period $p$ then either $\ell$ satisfies (b) or
\begin{enumerate}[label=(\alph*),resume]
    \item $\ell$ runs between two ends of $L$ with the same sign and $f^p|_\ell$ has a single fixed point $x_0$, which is a source or sink for $f^p$ in $L$ depending on whether the ends are both positive or negative, respectively.\qedhere
\end{enumerate}
\end{definition}

We remark that if $f$ satisfies (a) then we can homotope $f$ so that each noncompact boundary component running between two ends of the opposite sign satisfies (b) and each noncompact boundary component running between two ends of the same sign satisfies (c).

Reeb endperiodic maps are a convenient generalization of endperiodic maps. We will see in \Cref{subsec:reebsutured} that they are associated to generalized sutured manifolds called Reeb sutured manifolds. These will play a part in our future work \cite{LT23}.

\subsection{Junctures}\label{sec:junctures}

The construction of Handel-Miller laminations starts by producing a collection of ``junctures" associated to an endperiodic map $f\colon L\to L$, which are cooriented 1-manifolds in $L$ ``dual to a cohomology class at infinity" in a certain sense.

Let $e$ be an end of $L$, and let $\{U_i\}_{i\in \Z_+}$ be any regular neighborhood basis for $e$. We say a sequence of sets $\{A_i\}_{i\in \Z_+}$  \textbf{escapes to $e$} if each $U_j$ contains all but finitely many $A_i$.
If $e$ is a positive end of $L$ with period $p$, let 
\[
\U_e=\left\{q\in L\mid \{f^{np}(q)\}_{n\in \Z^+} \text{ escapes to }e\right\}.
\]
Letting $e_0=e$, if $Z=\{e_0,e_1,\dots, e_{p-1}\}$ is the $f$-cycle of $e_0$, let
\[
\U_Z=\bigcup_{i=0}^{p-1} \U_{e_i}.
\]
Then $\U_Z$ is a surface with connected components $\U_{e_0},\dots,\U_{e_{p-1}}$, and is in fact a regular covering space of a compact, connected surface $F_Z=\U_Z/\langle f\rangle$ with cyclic deck group generated by $f$. 

The \textbf{positive escaping set} $\U_+$ is the union of $\U_Z$ where $Z$ ranges over all positive end-cycles:
\[
\U_+=\bigcup_{\text{pos. end-cycles $Z$}}\U_Z.
\]
The \textbf{negative escaping set} $\U_-$ is the positive escaping set of $f^{-1}$.

\begin{construction}[Juncture components and tilings]\label{construction:junctures}
Choose a base point $b\in F_Z$, and let $\gamma$ be an oriented loop based at $b$. If $\wt \gamma$ is any lift of $\gamma$ traveling from $\wt b_1$ to $\wt b_2$ in $\U_Z$, there is a unique $n\in \Z$ such that $f^n(\wt b_1)=\wt b_2$. This defines a map $\pi_1(F_Z,b)\to \Z$, which gives a cohomology class $u\in H^1(F_Z;\Z)$. It is clear from the definition that $u$ must take values only in $p\Z$. Moreover the smallest positive value taken by $u$ is $p$. This can be seen as follows: let $x_0$ be a lift of $b$ to $\U_{e_0}$, and define $x_i=f^{i}(x_0)$. We can find a path in $\U_{e_0}$ from $x_0$ to $x_p$. This projects to a loop in $F_Z$ evaluating to $p$ under $u$.

We can choose a weighted, cooriented 1-manifold $J_Z$ in $F_Z$ which is dual to $u$. The weights come from collapsing parallel components of a representative of the Lefschetz dual of $u$. Further, we can require that $J_Z$ be nonseparating in $F_Z$. In this situation we see that all the weights on components of $J_Z$ will be divisible by $p$. We require that $J_Z$ be disjoint from our basepoint $b\in F_Z$.

Let $F'_Z$ be $F_Z\cut J_Z$, and let $\J_Z$ be the preimage of $J_Z$ in $\U_Z$. The surface $\U_Z$ decomposes into compact, connected surfaces called \textbf{tiles} which are glued along components of $\J_Z$, each of which has a weight in $p\Z_+$ induced by $J_Z$. Because $J_Z$ is nonseparating, for each $n\in\Z$ there is a unique component $t_n$ of $\U_Z-\J_Z$ that contains $x_n$. Here are some relevant facts about the tiles $t_n$:
\begin{itemize}
\item $f$ carries $t_n$ to $t_{n+1}$ for all $n$
\item If $j$ is a component of $\J_Z$ with weight $w$, then there is a unique $n$ such that $j$ joins $t_n$ to $t_{n+w}$ and the coorientation on $j$ points out of $t_n$ and into $t_{n+w}$.
\end{itemize}

A decomposition of $\U_Z$ into tiles as above is called a \textbf{tiling}.
Given a tiling $\{t_n\mid n\in \Z\}$ of $\U_Z$, a \textbf{tiled neighborhood} of $Z$ is any set of the form $\bigcup_{i=n}^\infty t_i$. This is a union of neighborhoods of each of the ends in $Z$, which are themselves unions of tiles. We call each of these end-neighborhoods a tiled neighborhood of the corresponding end. The boundary of a tiled neighborhood of an end $e$ is called a \textbf{juncture} of $e$. 

Now let $\J_+=\bigcup_Z \J_Z$, where the union is over all $f$-cycles of positive ends of $L$. The components of $\J_+$ are called \textbf{positive juncture components}. A juncture component $j$ is called \textbf{escaping} if $\{f^{np}(j)\}_{n<0}$ escapes to a negative end and \textbf{nonescaping} otherwise.
\end{construction}

\begin{example}\label{ex:tiling}
This example is a special case of \cite[Example 2.29]{CCF19}. Let $L$ be the two-ended surface shown in \Cref{fig:tiling}, which is made by gluing together countably many X-shaped pieces like the one shaded on the top of the figure. Let $f\colon L\to L$ be translation to the right by one X-shaped piece. 
Let $e$ be the positive end of $L$. Then $\U_e=L$, $F$ is as shown on the right of \Cref{fig:tiling}, and we can take $J$ to be the weighted cooriented 1-manifold shown there. With this $J$, the tiles corresponding to the singleton end-cycle $e$ are exactly the X-shaped pieces. On the bottom of \Cref{fig:tiling} we see a tiled neighborhood of $e$. Note that because the weights on components of $J$ are greater than the period of $e$ (in this case the period is 1) each juncture component is a part of multiple junctures.
\end{example}

\begin{remark} \label{rmk:connectedjuncture}
If all boundary components of the surface $F_Z$ have zero pairing with the cohomology class $u$ in the above construction of junctures, then $J$ can be chosen to have one component disjoint from $\del F_Z$, with weight equal to the period of $Z$. See the discussion in \cite[\S 2]{FKLL23}, in which $\del F=\varnothing$. In this case each juncture component will belong to just one juncture, unlike the situation in \Cref{ex:tiling}. 
\end{remark}

\begin{figure}
\centering
\fontsize{12pt}{12pt}\selectfont
\resizebox{!}{2in}{
\begingroup%
  \makeatletter%
  \providecommand\color[2][]{%
    \errmessage{(Inkscape) Color is used for the text in Inkscape, but the package 'color.sty' is not loaded}%
    \renewcommand\color[2][]{}%
  }%
  \providecommand\transparent[1]{%
    \errmessage{(Inkscape) Transparency is used (non-zero) for the text in Inkscape, but the package 'transparent.sty' is not loaded}%
    \renewcommand\transparent[1]{}%
  }%
  \providecommand\rotatebox[2]{#2}%
  \newcommand*\fsize{\dimexpr\f@size pt\relax}%
  \newcommand*\lineheight[1]{\fontsize{\fsize}{#1\fsize}\selectfont}%
  \ifx\svgwidth\undefined%
    \setlength{\unitlength}{375.94076334bp}%
    \ifx\svgscale\undefined%
      \relax%
    \else%
      \setlength{\unitlength}{\unitlength * \real{\svgscale}}%
    \fi%
  \else%
    \setlength{\unitlength}{\svgwidth}%
  \fi%
  \global\let\svgwidth\undefined%
  \global\let\svgscale\undefined%
  \makeatother%
  \begin{picture}(1,0.37016987)%
    \lineheight{1}%
    \setlength\tabcolsep{0pt}%
    \put(0,0){\includegraphics[width=\unitlength,page=1]{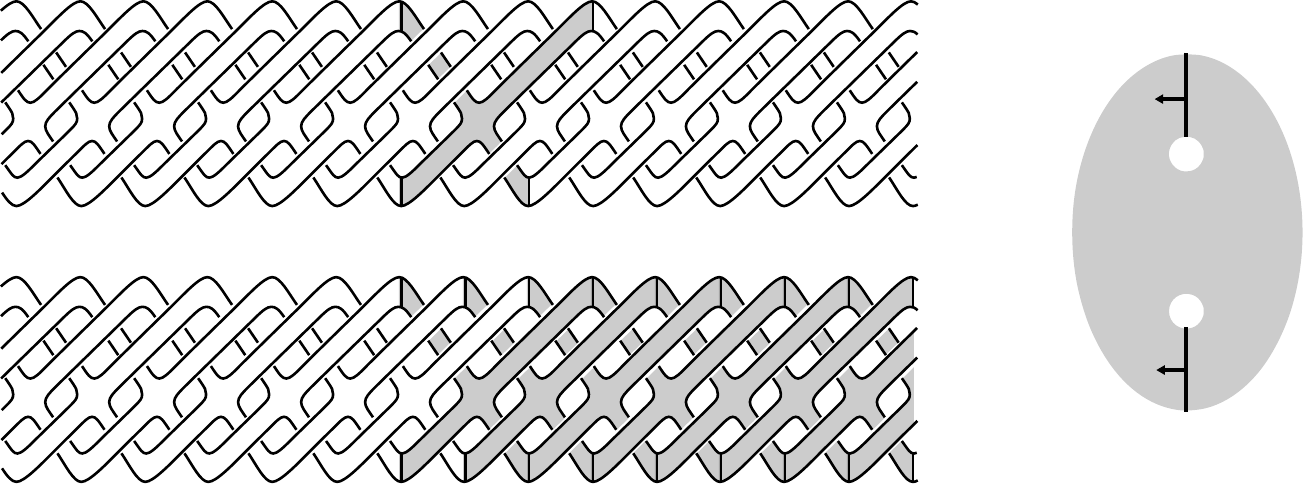}}%
    \put(0.89621642,0.17989529){\color[rgb]{0,0,0}\makebox(0,0)[lt]{\lineheight{1.25}\smash{\begin{tabular}[t]{l}$F$\end{tabular}}}}%
    \put(0.91761324,0.28191731){\color[rgb]{0,0,0}\makebox(0,0)[lt]{\lineheight{1.25}\smash{\begin{tabular}[t]{l}3\end{tabular}}}}%
    \put(0.91770897,0.07710159){\color[rgb]{0,0,0}\makebox(0,0)[lt]{\lineheight{1.25}\smash{\begin{tabular}[t]{l}2\end{tabular}}}}%
    \put(0,0){\includegraphics[width=\unitlength,page=2]{fig_tiling.pdf}}%
  \end{picture}%
\endgroup%
}
\caption{This figure accompanies \Cref{ex:tiling}.}
\label{fig:tiling}
\end{figure}

\subsection{The Handel-Miller laminations}\label{sec:HMlams}

A \textbf{geodesic half plane}, or \textbf{half plane}, is a closed subset of $\hyp^2$ bounded by a single geodesic. Following \cite{CCF19}, we say a complete hyperbolic metric on a surface $L$ is \textbf{standard} if there is no isometric embedding of a half plane in $L$ and if all components of $\del L$ are geodesics.

Recall that $\ms J_+$ is the set of positive juncture components. There is a symmetric definition of the set $\ms J_-$ of \textbf{negative juncture components}.

The following theorem is proved in \cite{CCF19} and is the foundation of Handel-Miller theory.

\begin{theorem}[Handel-Miller, Cantwell-Conlon-Fenley]\label{thm:HMrep}
Let $L$ be endowed with a standard hyperbolic metric and suppose that $f:L \to L$ is an endperiodic map. The geodesic representatives of non-escaping components of $\ms J_+$ limit on a geodesic lamination $\Lambda_-$. Similarly the geodesic representatives of non-escaping components of $\J_-$ limit on a geodesic lamination $\Lambda_+$ which is transverse to $\Lambda_-$.

The geodesic tightenings of the negative (positive) juncture components are mutually disjoint and disjoint from $\Lambda_+$ ($\Lambda_-$).

Moreover, up to isotopy we can assume that $f$ preserves the geodesic laminations $\Lambda_+$, $\Lambda_-$ as well as permutes the geodesic representatives of positive and negative juncture components.
\end{theorem}

\begin{remark}
This statement is slightly more general than \cite{CCF19} in that we allow infinite strip ends, but the methods of proof in \cite{CCF19} work just as well in this case.
\end{remark}

The laminations $\Lambda_\pm$ are called the \textbf{positive/negative Handel-Miller laminations} for $f$. 
Note that they are independent of the representative of the isotopy class of $f$, since they depend only on the geodesic tightenings of the $f$-images of curves. Additionally, they are independent of the choice of junctures \cite[Cor. 4.72]{CCF19}. Finally, by \cite[Cor. 10.16]{CCF19}, the union $\Lambda_+\cup \Lambda_-$ is independent of the choice of standard hyperbolic metric on $L$ up to ambient isotopy. As such we will sometimes refer to the Handel-Miller laminations without specifying a choice of metric. 

A representative of $f$ which preserves the Handel-Miller laminations as well as the geodesic representatives of the juncture components is called a \textbf{Handel-Miller representative} of the homotopy class of $f$, or simply a \textbf{Handel-Miller map}. When the metric on $L$ is not specified, a Handel-Miller map means a Handel-Miller map for some choice of metric.

Suppose we are given a Handel-Miller map $f$; by definition $f$ comes with some associated (geodesic) juncture components. If we perform the construction of junctures on $f$ from \Cref{construction:junctures} to produce some \emph{other}  collection of juncture components, then $f$ will permute these new juncture components 
as well as the leaves of $\Lambda_+$ and $\Lambda_-$. These new juncture components will not be geodesics in general.

\begin{example}[Translation]\label{eg:translate}
Suppose that $f\colon L\to L$ is endperiodic and that each point in $L$ escapes compact sets of $L$ under postive and negative iteration of $f$, i.e.  $\mathscr U_+=\mathscr U_-=L$. Such an endperiodic map is called a \textbf{translation}. In this case , $\Lambda_+=\Lambda_-=\varnothing$. See \cite[\S 4.8]{CCF19}. One can show that if $L$ is connected then it has exactly two ends. In general $f$ generates the deck group of an infinite cyclic covering $L\to L/\langle f\rangle$, and the mapping torus $(L\times I)/(x,1)\sim (f(x), 0)$ is homeomorphic to $(L/\langle f\rangle)\times (0,1)$.
\end{example}

\begin{example}\label{example:stackofchairs}
The following is one of the simplest examples in which the Handel-Miller laminations are nonempty. Let $L=\{(x,y)\in \mathbb R^2\mid xy\le1\}$, and let $f\colon L\to L$ be given by the matrix $\left(\begin{smallmatrix}2&0\\0&\frac{1}{2}\end{smallmatrix}\right)$. There are two positive and two negative ends of $L$, all of which are infinite strip ends. The lamination $\Lambda_+$ consists of a single line running between the two positive ends, while $\Lambda_-$ consists of a single line running between the negative ends. See \Cref{fig:chairmonodromy}.
\end{example}

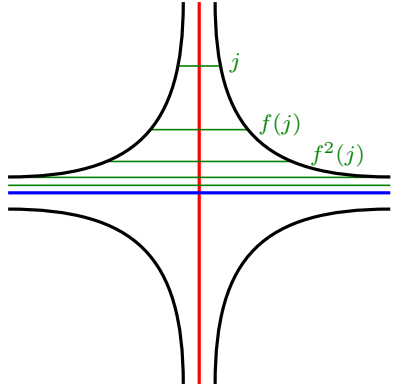
\begin{figure}
    \centering
    \fontsize{8pt}{8pt}\selectfont
    \resizebox{!}{2in}{
\begingroup%
  \makeatletter%
  \providecommand\color[2][]{%
    \errmessage{(Inkscape) Color is used for the text in Inkscape, but the package 'color.sty' is not loaded}%
    \renewcommand\color[2][]{}%
  }%
  \providecommand\transparent[1]{%
    \errmessage{(Inkscape) Transparency is used (non-zero) for the text in Inkscape, but the package 'transparent.sty' is not loaded}%
    \renewcommand\transparent[1]{}%
  }%
  \providecommand\rotatebox[2]{#2}%
  \newcommand*\fsize{\dimexpr\f@size pt\relax}%
  \newcommand*\lineheight[1]{\fontsize{\fsize}{#1\fsize}\selectfont}%
  \ifx\svgwidth\undefined%
    \setlength{\unitlength}{141.61492177bp}%
    \ifx\svgscale\undefined%
      \relax%
    \else%
      \setlength{\unitlength}{\unitlength * \real{\svgscale}}%
    \fi%
  \else%
    \setlength{\unitlength}{\svgwidth}%
  \fi%
  \global\let\svgwidth\undefined%
  \global\let\svgscale\undefined%
  \makeatother%
  \begin{picture}(1,0.96079559)%
    \lineheight{1}%
    \setlength\tabcolsep{0pt}%
    \put(0,0){\includegraphics[width=\unitlength,page=1]{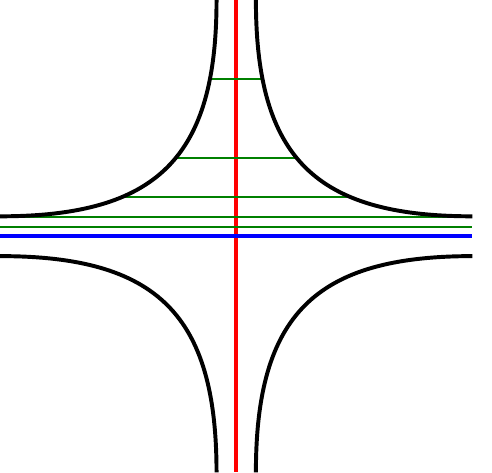}}%
    \put(0.55719051,0.79481899){\color[rgb]{0,0.50196078,0}\makebox(0,0)[lt]{\lineheight{1.25}\smash{\begin{tabular}[t]{l}$j$\end{tabular}}}}%
    \put(0.62946929,0.63894987){\color[rgb]{0,0.50196078,0}\makebox(0,0)[lt]{\lineheight{1.25}\smash{\begin{tabular}[t]{l}$f(j)$\end{tabular}}}}%
    \put(0.75913996,0.56045485){\color[rgb]{0,0.50196078,0}\makebox(0,0)[lt]{\lineheight{1.25}\smash{\begin{tabular}[t]{l}$f^2(j)$\end{tabular}}}}%
  \end{picture}%
\endgroup%
}
    \caption{A picture of $\Lambda_\pm$ for the map $f=\left(\begin{smallmatrix}2&0\\0&\frac{1}{2}\end{smallmatrix}\right)$ from \Cref{example:stackofchairs}. The blue line denotes $\Lambda_+$ and the red denotes $\Lambda_-$. The green lines depict part of the $f$-orbit of a negative juncture $j$.}
    \label{fig:chairmonodromy}
\end{figure}

\subsection{Principal regions}

We will now assume that $f\colon L\to L$ is Handel-Miller. The positive and negative escaping sets $\U_+$ and $\U_-$ (defined in \Cref{sec:junctures}) of $f$ are related quite simply to the Handel-Miller laminations $\Lambda_+$, $\Lambda_-$ by the following lemma, which is \cite[Lemma 4.71]{CCF19}.

\begin{lemma}[Cantwell-Conlon-Fenley]
For a Handel-Miller map $f$, we have $\Lambda_+=\del \U_-$ and $\Lambda_-=\del \U_+$.
\end{lemma}

\begin{definition}[Principal regions]\label{def:principalregion}
Let $\ms P_+=L-(\Lambda_+\cup \U_-)$. Each connected component of $\ms P_+$ is called a \textbf{positive principal region}. Symmetrically, a \textbf{negative principal region} is a positive principal region of $f^{-1}$.
\end{definition}

We now describe some general structure of principal regions without justification; for more details see \cite[\S\S 5.3, 6.1-6.4]{CCF19}.

Let $P_+$ be a positive principal region. Then $P_+$ is homeomorphic to the interior of a compact surface with boundary, say $\Sigma$. The metric completion $\ol P_+$ of $P_+$ is homeomorphic to $\Sigma$ minus a finite nonempty set of points on each component of $\del\Sigma$.
Hence $\del \ol P_+$ consists of finitely many lines $\lambda_1,\dots, \lambda_n$. Furthermore each $\lambda_i$ is a leaf of $\Lambda_+$. Assume for now that $\Sigma$ has only one boundary component. Then each $\lambda_i$ has the same period under $f$, say $p$. For each $\lambda_i$, there is a maximal $f$-invariant closed interval $I_i\subset \ell_i$ such that if $x\in \lambda_i-I_i$, then $\{f^{kp}(x)\mid k\ge 0\}$ escapes an end of $\lambda_i$.

Orient each $\lambda_i$ from left to right looking out from $P_+$, and let $a_i$ and $b_i$ be the left and right endpoints of $I_i$, respectively (see \Cref{fig:principal}). Then there is a leaf $\lambda_i'$ of $\Lambda_-$ passing through $b_i$ and $a_{i+1}$ (indices taken mod $n$). The leaves $\lambda_1',\dots, \lambda_n'$ are also periodic of period $p$. The leaf $\lambda_i'$ has a maximal invariant interval $I_i'$ with endpoints on $b_i$ and $a_{i+1}$, and if $x\in \lambda_i'-I_i'$ then $\{f^{kp}(x)\mid k\ge 0\}$ escapes an end of $\lambda_i'$.
The set $P_+-(I_1'\cup\cdots\cup I_n')$ consists of $n$ unbounded, simply connected components called the \textbf{arms} of $P_+$ and one bounded component called the \textbf{nucleus} of $P_+$. The nucleus is homeomorphic to the interior of $\Sigma$.

In fact, the leaves $\lambda_1',\dots, \lambda_n'$ bound a negative principal region $P_-$; the pair $P_+$, $P_-$ are said to be \textbf{dual principal regions}. See \Cref{fig:principal}. Symmetrically to $P_+$, $P_-$ has $n$ arms and one nucleus, which is equal to the nucleus of $P_+$.

When the compact surface $\Sigma$ has more than one boundary component, the description above holds with minor modifications. Namely the boundary lines of $P_+$ break into finitely many sets corresponding to the boundary components of $\Sigma$, each of these sets gives rise to finitely many arms, the complement in $P_+$ of the closure of the arms is the nucleus of $P_+$, and there exists a dual principal region $P_-$ sharing a nucleus with $P_-$.

\begin{figure}
    \centering
    \fontsize{8pt}{8pt}\selectfont
    \resizebox{!}{2.5in}{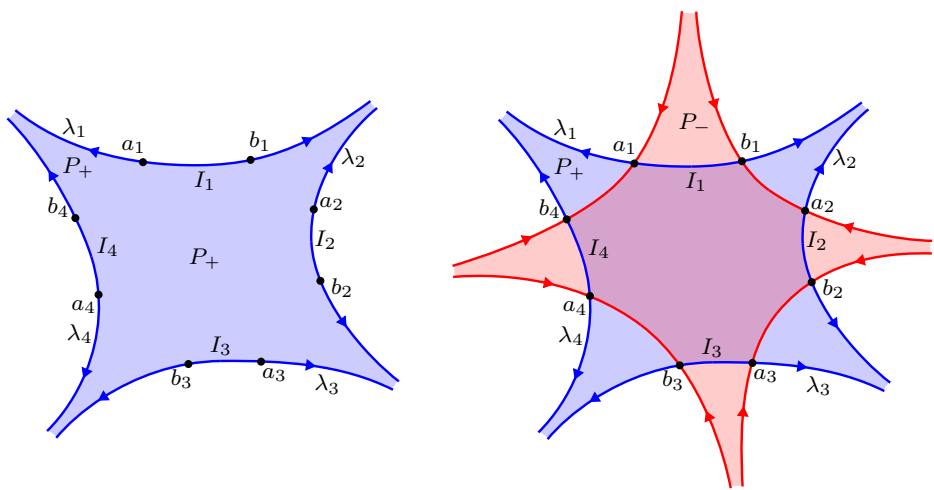}
    \caption{An example of a positive principal region $P_+$ (left) together with some notation from our description of the structure of principal regions. On the right we see the dual principal region $P_-$. Each of $P_+$ and $P_-$ has 4 arms, and the nucleus is the 8-sided region in the center. The arrows on leaves indicate the direction in which points move under application of $f^p$, where $p$ is the period of the leaves.\\
    Note that unlike in this example,  principal regions need not be simply connected in general.}
    \label{fig:principal}
\end{figure}

\subsection{Useful lemmas about the Handel-Miller representative}\label{sec:HMlems}

For this section, let $f\colon L\to L$ be a Handel-Miller map.

Since $f$ preserves the laminations $\Lambda_\pm$, it induces quotient laminations $\Lambda_\pm^\infty$ in the quotients $\U_\pm/\langle f \rangle$. 

\begin{lemma}\label{lem:bdylams}
\begin{enumerate}[label=(\alph*)]
\item The laminations $\Lambda_\pm^\infty$ are spiraling and consistent. Each compact leaf of $\Lambda_\pm^\infty$ is the image under the quotient map of a leaf of $\Lambda_\pm \cap \U_\pm$ which has a subray escaping to an end of $L$. Conversely, each such leaf of $\Lambda_\pm \cap \U_\pm$ gives a compact leaf of $\Lambda_\pm^\infty$.
\end{enumerate}

For the next two items, see the definitions in \Cref{sec:endperiodictosutured}.

\begin{enumerate}[resume,label=(\alph*)]
\item If $\lambda$ is a compact leaf of $\Lambda_+^\infty$, then $\lambda$ is the boundary component of a periodic leaf $\ell$ of the suspension $\LL^u$ of $\Lambda_+$, and $\lambda$ is homotopic in $\ell$ to a positive multiple of the closed orbit of the suspension semiflow at the core of $\ell$.

\item Further, $\Lambda_+^\infty$ contains parallel closed leaves if and only if $f$ has a positive principal region with an arm bounded by rays in $\Lambda_+$ which escape compact sets in $L$. This arm gives rise to two parallel closed leaves of $\Lambda_+^\infty$ bounding an annulus disjoint from the lamination. Symmetric statements hold for $\Lambda_-^\infty$.
\end{enumerate}
\end{lemma}

\begin{proof}
The statements in this lemma are all restatements of the material in \cite[Section 6.7]{CCF19}.
\end{proof}

\begin{lemma}\label{lem:accumonends}
Let $\lambda$ be a leaf of $\Lambda_+$, and let $p\in \lambda$. Then no component of $\lambda-\{p\}$ is contained in a compact subset of $L$.
\end{lemma}

\begin{proof}
This is \cite[Corollary 4.50]{CCF19}.
\end{proof}

\begin{lemma}\label{lem:ilam}
Let $K$ be a compact subsurface of $L$ with $\del K\pitchfork \Lambda_+$. Then $\Lambda_+\cap K$ is an $I$-lamination.
\end{lemma}
\begin{proof}
If $\lambda$ is a leaf of $\Lambda_+$, then by \Cref{lem:accumonends} neither of its ends accumulates in $K$. The lemma follows.
\end{proof}

\begin{lemma}\label{lem:nojunctures}
Let $\lambda$ be a leaf of $\Lambda_+$. If there is a side of $\lambda$ on which negative (geodesic) juncture components do not accumulate, then $\lambda$ borders a principal region on that side.
\end{lemma}

\begin{proof}
By \cite[Lemma 5.14]{CCF19} (or \Cref{def:principalregion}), $L$ is the disjoint union of the negative escaping set $\ms U_-$, the lamination $\Lambda=\Lambda_+$, and the union of positive principal regions $\ms P_+$.

If $\lambda$ has a side on which negative juncture components do not accumulate, it follows from the definition of $\Lambda_+$ that leaves of $\Lambda_+$ do not accumulate on that side either.
Hence $\lambda$ borders a component $P$ of $L-\Lambda$ on that side, which must be either an escaping component or a principal region.

By \cite[Proposition 5.11]{CCF19} negative juncture components accumulate on $\lambda$ from any side bordering $\mathscr U_-$. We conclude that $P$ is a principal region.
\end{proof}

Let $\wt L$ denote the universal cover of $L$, where $L$ is endowed with a standard hyperbolic metric. Then $\wt L$ can be identified with a subset of $\mathbb H^2$ with (possibly empty) boundary a collection of geodesics. Thus it has a natural compactification to a closed disk obtained by taking the closure in $\mathbb H^2\cup\del_\infty \mathbb H^2$. We denote the intersection of this closed disk with  $\del_\infty \mathbb H^2$ by $\del_\infty(\wt L)$. Since $\Lambda_\pm$ are geodesic laminations, each leaf of their lifts $\wt \Lambda_\pm$ to $\wt L$ determine well-defined endpoints in $\del_\infty(\wt L)$.

The following lemma says that the arms of principal regions never ``fellow travel." Its proof makes use of the fact that all leaves of $\Lambda^+$ and all juncture components are geodesics.

\begin{lemma}\label{lem:principalcusps}
Let $\lambda_1$ and $\lambda_2$ be leaves of $\Lambda_+$, with distinct lifts $\wt \lambda_1$ and $\wt\lambda_2$ sharing a point $p$ in $\del_\infty(\wt L)$. Then $\wt\lambda_1$ and $\wt\lambda_2$ border the same lifted principal region.
\end{lemma}

\begin{proof}
Let $A$ be the component of $\wt L-(\wt\lambda_1\cup \wt \lambda_2)$ bordered by both $\wt\lambda_1$ and $\wt\lambda_2$. We claim that there are no lifted negative juncture components accumulating on $\wt\lambda_1$ or $\wt\lambda_2$ from inside $A$. This is because any such lifted juncture component sufficiently far into such a sequence would have to have $p$ as an endpoint, since the negative junctures are disjoint from $\Lambda_+$. If $\wt j$ were such a lift of a juncture component $j$, then both $\lambda_1$ and $\lambda_2$ would limit on $j$ in $L$. However, by \Cref{lem:accumonends} both ends of each leaf of $\Lambda_+$ pass arbitrarily near at least one end of $L$ so this is impossible.

By \Cref{lem:nojunctures}, $\wt \lambda_1$ and $\wt \lambda_2$ must each border a lift of a principal region lying in $A$; call these lifts $\wt P_1$ and $\wt P_2$ respectively. If $\wt P_1\ne \wt P_2$, then there must be at least one lifted leaf $\wt\lambda_3\in \wt \Lambda$ lying in $A$ and having $p$ as an endpoint. However this would imply the existence of a sequence of lifted negative junctures accumulating on $\wt\lambda_3$, each having $p$ as an endpoint, which we have already seen is impossible. It follows that no such $\wt \lambda_3$ exists so that $\wt P_1=\wt P_2$.
\end{proof}

\begin{lemma}\label{lem:juncaccum}
Let $\lambda$ be a periodic leaf of $\Lambda_+$. Let $\wt\lambda$ be a lift of $\lambda$ to $\wt L$. Let $\wt f^n$ be a lift of $f^n$ that preserves $\wt \lambda$ and its ends. For each side of $\wt \lambda$ such that $\wt \lambda$ does not border a lifted principal region on that side, there exists a lift $\wt j$ of a (geodesic) juncture component $j$ such that $(\wt j, \wt f^n(\wt j),\wt f^{2n}(\wt j),\dots)$ converges to $\wt\lambda$ from that side.
\end{lemma}

\begin{proof}
This is a restatement of \cite[Lemma 6.7]{CCF19}.
\end{proof}

When infinite strip ends are present, they interact predictably with the laminations as the following lemma shows. 

\begin{lemma}\label{lem:infstripend}
Let $e$ be an infinite strip end of $L$. If $e$ is positive then there are either 1 or 2 leaves of $\Lambda_+$ escaping to $e$. These are the only leaves that see $e$, in the sense that there is a neighborhood of $e$ disjoint from all other leaves of $\Lambda_\pm$.
\end{lemma}

\begin{proof}
Let $\ell_1$ and $\ell_2$ be the two boundary components of $L$ with ends escaping to $e$. We choose a collection of junctures $\ms J$ for $L$. Note that there are infinitely many negative junctures meeting $\ell_i$ since $\ell_i$ connects a negative end to a positive end.
If $p$ is the period of $e$, then $f^p$ maps $\ell_i$ to $\ell_i$ with no fixed points, translating points from the negative end to the positive end. Hence there is a sequence of negative junctures for each $i=1,2$, such that each juncture has an endpoint on $\ell_i$ and these endpoints escape to $e$. Hence each sequence of junctures accumulates on a leaf $\lambda_i$ of $\Lambda_+$ that escapes $e$. If $\lambda_1\ne \lambda_2$, then $\lambda_1$ and $\lambda_2$ border a single principal region by \Cref{lem:principalcusps} which cannot contain any other leaves of $\Lambda_+$.
\end{proof}

\begin{construction}[Endperiodization]\label{construction:samesign}
Suppose that $f\colon L\to L$ is a Reeb endperiodic map (see \Cref{def:endperiodic}). Assume first that there is a unique component $\ell$ of $\del L$ connecting two ends of the same sign. Without loss of generality assume that $\ell$ connects two negative ends, so that  $f|_\ell$ has a unique fixed point $x_0$ and all points in $\ell-\{x_0\}$ are attracted to $x_0$ under iteration of $f$. 
By definition $x_0$ is a local sink in $L$, so there is an attracting neighborhood $A\ni x_0$ in $L$ with $\lim_{n\to \infty}f^n(x)=x_0$ for all $x\in A$. 
Note that $f|_{L-\{x_0\}}$ is now an actual endperiodic map; we have introduced a new infinite strip end corresponding to the point $x_0$ and no boundary components of $L-\{x_0\}$ connect ends of the same sign. We can therefore speak of the Handel-Miller laminations associated to $f|_{L-\{x_0\}}$. When $L$ has multiple such boundary components of various periods, this construction can be modified in the obvious way to produce an endperiodic map. We call this new map the \textbf{endperiodization} of $f$.
\end{construction}

\section{Periodic sequences of endperiodic train tracks} \label{sec:splitseq}

\subsection{Conventions and definitions}
\label{sec:tilingconventions}

In this section we assume that $f\colon L\to L$ is a Handel-Miller map. Recall that by our definition of Handel-Miller map, there is a choice of standard hyperbolic metric on $L$ and geodesic juncture components such that $\Lambda_\pm$ are geodesic laminations and such that $f$ permutes the leaves of the laminations as well as the juncture components.

Fix a tiling of each positive and negative end-cycle as in \Cref{construction:junctures}; we emphasize that the juncture components defining these tilings need not be the geodesic juncture components used to construct the Handel-Miller laminations; this freedom will be useful to us later.

Let $E_\pm$ be a union of mutually disjoint tiled neighborhoods of all end-cycles of $L$, such that $\Lambda_+$ ($\Lambda_-$) is disjoint from the tiled neighborhoods of the negative (positive) end-cycles. Let $K_0=L\cut E_\pm$. We call $K_0$ a \textbf{core} of $L$.

Suppose that the positive end-cycles are $Z_1,\dots, Z_n$, and let $t^i_j$ be the $j$th tile in the tiled neighborhood for $Z_i$, where the indexing on $j$ starts at 1.
Starting with $\ell=0$, we inductively define the \textbf{$(\ell+1)$st core} to be
\[
K_{\ell+1}=K_\ell\cup \left(t^1_{\ell+1}\cup\cdots\cup t^n_{\ell+1}\right).
\]
That is, $K_{\ell+1}$ is obtained from $K_\ell$ by absorbing a $t^i_j$ for each $i$, where $j$ is the smallest integer such that $t^i_j$ does not already lie in $K_\ell$.
Let $E_i$ be the union of components of $L\cut K_i$ which are neighborhoods of positive ends.

\subsection{Train tracks carrying $\Lambda_+$: some geometric arguments}

We continue to assume that $f\colon L\to L$ is a Handel-Miller map. 

In this subsection our arguments make use of the standard hyperbolic metric on $L$ and of the geodesic juncture components associated to $f$ by the definition of a Handel-Miller map. 
However, we emphasize that the results do not depend on this geometric structure since the Handel-Miller laminations are independent of the standard metric up to ambient isotopy.

\begin{lemma}\label{lem:fellowtravel}
Let $\tau$ be a train track fully carrying $\Lambda_+$. Let $\rho,\rho'$ be two rays in leaves of $\Lambda_+$ which eventually fellow travel in $\tau$, and let $\wt\rho,\wt\rho'$ be two lifts of these rays eventually fellow traveling in $\wt \tau$. Then the ideal points in $\del_\infty(\wt L)$ determined by $\wt \rho$ and $\wt\rho'$ are equal.
\end{lemma}

The above lemma is obvious for laminations carried by train tracks on hyperbolic surfaces of finite type, but we  should convince ourselves that it also holds in our setting. In the finite type setting it is immediate that the ties of a standard neighborhood of a lamination have bounded length, which immediately implies the claim. We give a different argument here that does not require a uniform bound on the length of ties.

\begin{proof}
By truncating, we may assume that the two rays begin on the same tie of $\tau$. Choose parameterizations of $\wt \rho$ and $\wt \rho'$ that are compatible with the ties of $\wt \tau$. That is, for all $t\in[0,\infty]$ there exists a segment $a_t$ connecting $\rho(t)$ to $\rho'(t)$ such that $a_t$ is contained in a tie of $\wt \tau$. Thus the parameterization of $\rho'$ is uniquely determined by the parameterization of $\rho$. Let $\wt A_t$ be the region bounded by the tie segments $a_0, a_t$ and the leaf segments $[\wt\rho(0),\wt\rho(t)], [\wt\rho'(0),\wt\rho'(t)]$. Let $\wt A=\bigcup_{t\in[0,\infty)} \wt A_t$. 

Suppose that the mapping of $\wt A$ to $L$ under the covering projection is not an embedding. In this case, one can check that $\rho$ and $\rho'$ must follow the same train route as a closed curve $\gamma$. Thus the two rays stay in a compact set of $L$, contradicting \Cref{lem:accumonends}.

Otherwise $\wt A$ projects homeomorphically to a set $A\subset L$ which is foliated by segments of ties of $\tau$. Note that $\wt A$ does not contain any lifts of boundary components of $L$, since $\wt A$ is foliated by tie segments disjoint from all such lifts. As a consequence, if $p\ne p'$ then $\wt A$ contains a half plane $\wt H$, which projects isometrically to $L$. See \Cref{fig:halfplane}. This contradicts that the hyperbolic metric on $L$ is standard, so $p=p'$ as desired.
\end{proof}

\begin{figure}
    \centering
    \fontsize{10pt}{10pt}\selectfont
    \resizebox{!}{2in}{
\begingroup%
  \makeatletter%
  \providecommand\color[2][]{%
    \errmessage{(Inkscape) Color is used for the text in Inkscape, but the package 'color.sty' is not loaded}%
    \renewcommand\color[2][]{}%
  }%
  \providecommand\transparent[1]{%
    \errmessage{(Inkscape) Transparency is used (non-zero) for the text in Inkscape, but the package 'transparent.sty' is not loaded}%
    \renewcommand\transparent[1]{}%
  }%
  \providecommand\rotatebox[2]{#2}%
  \newcommand*\fsize{\dimexpr\f@size pt\relax}%
  \newcommand*\lineheight[1]{\fontsize{\fsize}{#1\fsize}\selectfont}%
  \ifx\svgwidth\undefined%
    \setlength{\unitlength}{140.59240633bp}%
    \ifx\svgscale\undefined%
      \relax%
    \else%
      \setlength{\unitlength}{\unitlength * \real{\svgscale}}%
    \fi%
  \else%
    \setlength{\unitlength}{\svgwidth}%
  \fi%
  \global\let\svgwidth\undefined%
  \global\let\svgscale\undefined%
  \makeatother%
  \begin{picture}(1,0.96364372)%
    \lineheight{1}%
    \setlength\tabcolsep{0pt}%
    \put(0,0){\includegraphics[width=\unitlength,page=1]{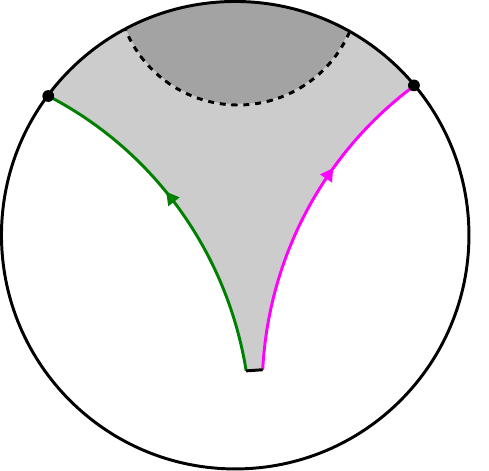}}%
    \put(0.44128158,0.84281689){\color[rgb]{0,0,0}\makebox(0,0)[lt]{\lineheight{1.25}\smash{\begin{tabular}[t]{l}$\widetilde{H}$\end{tabular}}}}%
    \put(0.43857088,0.5724248){\color[rgb]{0,0,0}\makebox(0,0)[lt]{\lineheight{1.25}\smash{\begin{tabular}[t]{l}$\widetilde{A}$\end{tabular}}}}%
    \put(0.01759263,0.77006903){\color[rgb]{0,0,0}\makebox(0,0)[lt]{\lineheight{1.25}\smash{\begin{tabular}[t]{l}$p$\end{tabular}}}}%
    \put(0.88075014,0.78690357){\color[rgb]{0,0,0}\makebox(0,0)[lt]{\lineheight{1.25}\smash{\begin{tabular}[t]{l}$p'$\end{tabular}}}}%
    \put(0.65316555,0.50130095){\color[rgb]{0,0,0}\makebox(0,0)[lt]{\lineheight{1.25}\smash{\begin{tabular}[t]{l}$\tilde{\rho}'$\end{tabular}}}}%
    \put(0.22149947,0.56466017){\color[rgb]{0,0,0}\makebox(0,0)[lt]{\lineheight{1.25}\smash{\begin{tabular}[t]{l}$\tilde{\rho}$\end{tabular}}}}%
  \end{picture}%
\endgroup%
}
    \caption{If $p\ne p'$, the metric on $L$ would be nonstandard.}
    \label{fig:halfplane}
\end{figure}

If $\tau$ is a train track fully carrying $\Lambda_+$, we say a cusp of $\tau$ is \textbf{principal} if it corresponds to a principal region of $\Lambda_+$. All other cusps are \textbf{nonprincipal}. Any cusp $c$ of $\tau$ determines two rays $\rho_1,\rho_2$ in leaves of $\Lambda_+$ bordering the corresponding complementary component of $\Lambda_+$. There are two possibilities: either $\rho_1$ and $\rho_2$ fellow travel in $\tau$, or there is some first cusp $c'$ of $\tau$ at which the two rays split away from each other. In the second case we say that the cusps $c$ and $c'$ \textbf{collide} with one another.

\begin{lemma}\label{lem:collision}
Let $\tau$ be a train track fully carrying $\Lambda_+$, and let $c$ be a nonprincipal cusp of $\tau$. Then there exists a cusp $c'$ of $\tau$ such that $c$ and $c'$ collide.
\end{lemma}

\begin{proof}
We prove the contrapositive: if there exists no such cusp $c'$ then $c$ is principal.

Let $\wt c$ be a lift of $c$ to the universal cover $\wt L$. If $c$ does not collide with any cusp then $\wt c$ determines 2 rays  of leaves in $\wt \Lambda_+$ that follow the same route in $\wt \tau$; these are two border leaves $\wt\lambda$, $\wt\lambda'$ of the complementary region determined by $\wt c$. 
Let $\lambda$ and $\lambda'$ be their respective projections to $L$. By \Cref{lem:fellowtravel} there exists a point $p\in \del_\infty \wt L$ which is an endpoint of both $\wt\lambda$ and $\wt\lambda'$, so by \Cref{lem:principalcusps} there is a lifted principal region bounded by $\wt\lambda$ and $\wt\lambda'$. Therefore $c$ is principal.
\end{proof}

\subsection{Train tracks carrying $\Lambda_+$: topological arguments}

In this subsection we prove some more facts about train tracks carrying $\Lambda_+$, but using topological arguments that do not rely on the standard hyperbolic metric on $L$. We also introduce the ``core split," an operation on train tracks that will be important going forward.

We make the convention that train tracks carrying $\Lambda_+$ will always be transverse to any junctures under consideration. We can always arrange for this to hold by a small perturbation. 

\begin{lemma}\label{lem:lambdacompatible}
Let $\tau_1$ and $\tau_2$ be efficient train tracks fully carrying $\Lambda_+$ which 
are equal in $E_i$, i.e. $\tau_1|_{E_i}=\tau_2|_{E_i}$. 
Then the train tracks $\tau_1|_{K_{j}}$ and $\tau_2|_{K_{j}}$ are $\Lambda|_{K_j}$-compatible as train tracks in $K_{j}$ for all $j\ge i$. 
\end{lemma}

Here, $\tau_1$ and $\tau_2$ being efficient implies that $\tau_1|_{K_j}$ and $\tau_2|_{K_j}$ are efficient, so the notion of $\Lambda$-compatibility makes sense.

\begin{proof}[Proof of \Cref{lem:lambdacompatible}]
Implicit in our discussion of $\tau_1$ and $\tau_2$ is the existence of standard neighborhoods witnessing the fact that $\Lambda$ is carried by both of these train tracks, and which agree on $E_j$ for $j\ge i$. Let $N(\tau_1)$ and $N(\tau_2)$ be these standard neighborhoods, respectively.
Let $A$ be a complementary region of $\Lambda_+|_{K_j}$, and let $c$ be a component of the vertical boundary $\del_v A\subset\del K_{j}$ associated to the stop $s$ of $\tau_1|_{K_j}$. We must show that $c$ is mapped to $s$ under the carrying map $N(\tau_2)\to \tau_2$. 
Let $A'$ be the complementary component of $\Lambda_+|_{E_j}$ meeting $A$ along $c$ (see \Cref{fig:compatible}). 
\begin{figure}
    \centering
    \fontsize{10pt}{10pt}\selectfont
    \resizebox{!}{2in}{
\begingroup%
  \makeatletter%
  \providecommand\color[2][]{%
    \errmessage{(Inkscape) Color is used for the text in Inkscape, but the package 'color.sty' is not loaded}%
    \renewcommand\color[2][]{}%
  }%
  \providecommand\transparent[1]{%
    \errmessage{(Inkscape) Transparency is used (non-zero) for the text in Inkscape, but the package 'transparent.sty' is not loaded}%
    \renewcommand\transparent[1]{}%
  }%
  \providecommand\rotatebox[2]{#2}%
  \newcommand*\fsize{\dimexpr\f@size pt\relax}%
  \newcommand*\lineheight[1]{\fontsize{\fsize}{#1\fsize}\selectfont}%
  \ifx\svgwidth\undefined%
    \setlength{\unitlength}{149.90488394bp}%
    \ifx\svgscale\undefined%
      \relax%
    \else%
      \setlength{\unitlength}{\unitlength * \real{\svgscale}}%
    \fi%
  \else%
    \setlength{\unitlength}{\svgwidth}%
  \fi%
  \global\let\svgwidth\undefined%
  \global\let\svgscale\undefined%
  \makeatother%
  \begin{picture}(1,0.90339672)%
    \lineheight{1}%
    \setlength\tabcolsep{0pt}%
    \put(0,0){\includegraphics[width=\unitlength,page=1]{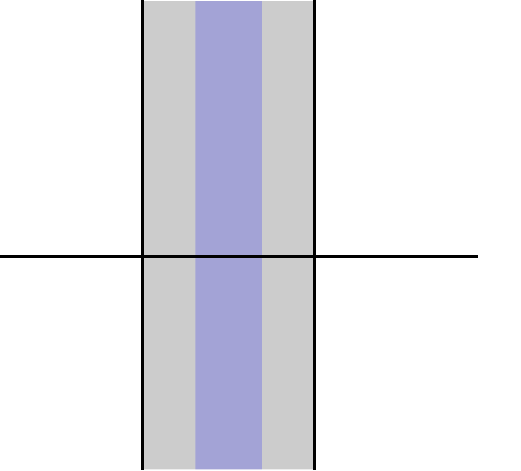}}%
    \put(0.39244146,0.20966637){\color[rgb]{0,0,0}\makebox(0,0)[lt]{\lineheight{1.25}\smash{\begin{tabular}[t]{l}$A'$\end{tabular}}}}%
    \put(0.00631748,0.54476809){\color[rgb]{0,0,0}\makebox(0,0)[lt]{\lineheight{1.25}\smash{\begin{tabular}[t]{l}$K_1$\end{tabular}}}}%
    \put(0.00873406,0.26652337){\color[rgb]{0,0,0}\makebox(0,0)[lt]{\lineheight{1.25}\smash{\begin{tabular}[t]{l}$E'_+$\end{tabular}}}}%
    \put(0.80852441,0.32833297){\color[rgb]{0,0,0}\makebox(0,0)[lt]{\lineheight{1.25}\smash{\begin{tabular}[t]{l}$\partial K_1$\end{tabular}}}}%
    \put(0.4014287,0.69364925){\color[rgb]{0,0,0}\makebox(0,0)[lt]{\lineheight{1.25}\smash{\begin{tabular}[t]{l}$A$\end{tabular}}}}%
    \put(0.41455552,0.44042538){\color[rgb]{0,0,0}\makebox(0,0)[lt]{\lineheight{1.25}\smash{\begin{tabular}[t]{l}$c$\end{tabular}}}}%
    \put(0,0){\includegraphics[width=\unitlength,page=2]{fig_compatible.pdf}}%
  \end{picture}%
\endgroup%
}
    \caption{Notation for showing $\Lambda_+|_{K_j}$-compatibility of $\tau_1|_{K_j}$ and $\tau_2|_{K_j}$ in the proof of \Cref{lem:lambdacompatible}.}
    \label{fig:compatible}
\end{figure}
Then $c$, when viewed as a component of $\del_v A'$, is also associated to $s$ (viewing $s$ as a stop of $\tau_1|_{E_j}$). 
We have that $N(\tau_1)|_{E_j}=N(\tau_2)|_{E_j}$, so $c$ is also associated to $s$ viewed as a stop of $\tau_2|_{E_j}$. Since the local picture is as in \Cref{fig:compatible}, it follows that $c$ is associated to $s$ viewed as a stop of $\tau_2|_{K_{j}}$. Therefore $\tau_1|_{K_j}$ and $\tau_2|_{K_j}$ are $\Lambda|_{K_j}$-compatible.
\end{proof}

\begin{definition}[Core split]
Let $\tau$ be a train track fully carrying $\Lambda_+$. 
If $\tau$ has no large branches, define $\kappa(\tau)=\tau$.
Otherwise, let $i$ be the least $i\in \Z_{\ge0}$ such that $\tau|_{K_i}$ contains a large branch. Let $\mu$ be the maximal splitting of $\tau|_{K_i}$, and let $\kappa(\tau)=\mu\cup(\tau\cut K_i)$(recall the maximal splitting is unique by \Cref{lemma:differbycommutations}).  We call the operation of replacing $\tau$ by $\kappa(\tau)$ a \textbf{core split}. Sometimes we will use a double-headed arrow to denote core splitting, i.e. $\tau \twoheadrightarrow \tau'$ means that $\tau'=\kappa(\tau)$.
\end{definition}

Note that the definition of a core split depends on the tiling fixed in \Cref{sec:tilingconventions}.
For an example see \Cref{fig:fenleyexample}, where if we fix the tiling shown, then the composition of the first two splitting moves is a core split.

\begin{lemma}\label{lem:coresplitcommute}
Let $\tau$ be a efficient train track fully carrying $\Lambda_+$. Then core splitting commutes with $f$, that is $\kappa(f(\tau))=f(\kappa(\tau))$.
\end{lemma}

\begin{proof}
If $\tau$ has no large branches then $f(\tau)$ also has no large branches, and the statement reduces to $f(\tau)=f(\tau)$; hence we may assume $\tau$ has large branches. Suppose that $i$ is the least $i$ such that $\tau|_{K_i}$ has large branches.

It is clear that $f(\kappa(\tau))|_{E_{i+1}}=\kappa(f(\tau))|_{E_{i+1}}$.
Moreover, $f(\kappa(\tau))|_{K_{i+1}}$ and $\kappa(f(\tau))|_{K_{i+1}}$ are $\Lambda|_{K_{i+1}}$-compatible by \Cref{lem:lambdacompatible}, and they are spiraling.
Since core splitting and the application of the map $f$ both preserve orders on persistent cusps, we conclude that $f(\kappa(\tau))|_{K_{i+1}}$ and $\kappa(f(\tau))|_{K_{i+1}}$ are isotopic in $K_{i+1}$ rel stops by \Cref{cor:anyorder}f. This gives a compactly supported isotopy from $\kappa(f(\tau))$ to $f(\kappa(\tau))$.
\end{proof}

\subsection{Endperiodic train tracks} 

We say a train track $\tau$ in $L$ fully carrying $\Lambda_+$ is \textbf{$f$-endperiodic} if there exists an $i$ such that
\begin{equation}\label{eq:periodiccondition} 
f(\tau)|_ {E_{i+1}}=\tau|_{E_{i+1}}. 
\end{equation}
(The sets $E_i$ were defined in \Cref{sec:tilingconventions}).
We say that $\tau$ is \textbf{$f$-endperiodic in $E_i$} if \Cref{eq:periodiccondition} holds.

\begin{proposition}\label{lem:trackconstruction}
There exists $i\ge0$ such that the lamination $\Lambda_+$ is fully carried by an efficient endperiodic train track $T_+$ with the following properties:
\begin{enumerate}[label=(\roman*)]
\item $T_+$ is $f$-endperiodic in $E_i$,
\item $T_+\cap K_i$ has no large branches, and 
\item $T_+\cap E_i$ has no large branches.
\end{enumerate}
\end{proposition}

We remark that conditions (ii) and (iii) together are equivalent to the statement that all large branches of $T_+$ pass through $\del K_i$. For an example, see the train track on the far left of \Cref{fig:fenleyexample}, where the core is the second piece from the top.

\begin{proof}
Recall the quotient surface $\mathscr{U}_+/\langle f \rangle$, which contains the quotient lamination $\Lambda_+^\infty$. By \Cref{lem:bdylams} (a) and (c), $\Lambda_+^\infty$ is spiraling, and has no closed leaves which are homotopic to each other with the opposite orientation. By \Cref{lemma:spiralttexists}, $\Lambda_+^\infty$ is carried by an efficient spiraling train track $T_+^\infty$.

Let $N(\Lambda^\infty_+)$ be a standard neighborhood of $T_+^\infty$ which is also a standard neighborhood of $\Lambda_+^\infty$.
By taking the preimage of $N(\Lambda^\infty_+)$ under the map $E_0 \hookrightarrow \mathscr{U}_+ \to \mathscr{U}_+/\langle f \rangle$, we obtain a standard neighborhood of $\Lambda_+|_{E_0}$. By taking a standard neighborhood of $\Lambda_+|_{K_0}$ which matches up along $\del K_0$, we can extend this to a standard neighborhood $N(\Lambda_+)$ of $\Lambda_+$.

Similarly, by taking the preimage of $T^\infty_+$, we obtain a train track $T_{E_0}$ on $E_0$. Let $T_{K_0}$ be a train track for $N(\Lambda_+)|_{K_0}$ that matches up with $T_E$ along $\del K_0$. Since $\Lambda_+|_{K_0}$ is an $I$-lamination by \Cref{lem:ilam}, we may replace $T_K$ by its maximal $\Lambda$-splitting by \Cref{lemma:maximalsplitting}. Let
\[
T_0=T_{E_0}\cup T_{K_0}.
\]

Note that $T_0$ has no annulus complementary regions since this would force the existence of a circular leaf of $\Lambda_+$. Also, $T_0$ has at most finitely many cusped bigon or cusped monogon complementary regions since $T_{E_0}$ has none such, so any cusped bigon complementary region must pass through $K_0$. Suppose $R$ is a cusped bigon or monogon complementary region of $T_0$. At least one of the cusps of $R$ must collide with another cusp, for otherwise \Cref{lem:fellowtravel} and \Cref{lem:principalcusps} would give a principal region of $\Lambda^+$ with nonnegative index, and all of the principal regions of $\Lambda_+$ have negative index ($\Lambda^+$ is geodesic). Hence there is a finite sequence of splits reducing the number of cusped bigon and monogon complementary regions by one.

Therefore after finitely many splits we obtain a track $T_1$ which agrees with $T_0$ outside some smallest core $K_i$. We perform a core split on $T_1$ and call the result $T_+$. By construction $T_+$ is efficient, and endperiodic in $E_i$.  Further, (ii) is satisfied because $T_+$ is the core split of $T_1$, and (iii) is satisfied because $T_+^\infty$ is spiraling.
\end{proof}

\begin{remark}\label{rmk:trackconstruction}
In the last paragraph of the proof of \Cref{lem:trackconstruction}, note that the uniqueness of maximal splittings implies that $T_+$ can also be obtained by performing $i$ core splits on $T_0$.
\end{remark}

\begin{theorem}\label{thm:sequenceexist}
There exists an efficient $f$-endperiodic train track $\tau$ which fully carries $\Lambda_+$, satisfies properties (i)-(iii) from \Cref{lem:trackconstruction}, and is the first term of a $\Lambda_+$-splitting sequence 
\[
\tau \to \cdots \to f(\tau).
\]
\end{theorem}

\begin{proof}
Let $T_+$ be the train track furnished by \Cref{lem:trackconstruction}. Reindex the $E_i$ so that $T_+$ and $E_0$ satisfy properties (i)-(iii) from that proposition. Set $\Lambda_i:=\Lambda_+|_{K_i}$.

Let $\tau_0=T_+$, and let $\tau_{i+1}$ be obtained by performing a core split on $\tau_i$ for $i\ge 0$. By \Cref{lem:lambdacompatible}, $\tau_0|_{K_j}$ and $f(\tau_0)|_{K_j}$ are $\Lambda_j$-compatible for $j\ge 1$. Since $\Lambda_j$-splitting preserves $\Lambda_j$-compatibility, this implies that $\tau_1|_{K_j}$ and $f(\tau_0)|_{K_j}$ are $\Lambda_j$-compatible for $j\ge 1$. In particular, there is a natural $\Lambda_j$-identification of their cusps. Note also that $\tau_1|_{K_1}$ and $f(\tau_0)|_{K_1}$ are both spiraling train tracks, so \Cref{lem:diffbyshifts} implies that they differ by at most a collection of shifts.

By \Cref{lem:principalcusps} and \Cref{lem:fellowtravel}, none of the $\Lambda_+$-routes from principal cusps of $\tau_1$ or $f(\tau_0)$ eventually fellow travel in their respective train tracks.
By \Cref{lem:collision}, the $\Lambda_+$-routes from all nonprincipal cusps all experience collisions. We can therefore choose a natural number $N$ such that the following hold:
\begin{itemize}
    \item if $a,b$ are principal cusps of $\tau_1|_{K_1}$ or $f(\tau_0)|_{K_1}$ such that their maximal $\Lambda$-routes in $K_1$ end at the same point of $\del K_1$, then their maximal $\Lambda$-routes in $K_N$ diverge at some point, and 
    \item each nonprincipal cusp of $\tau_1|_{K_1}$ and $f(\tau_0)|_{K_1}$ is nonpersistent in $K_N$.
\end{itemize}

As noted above, the train tracks $\tau_1|_{K_N}$ and $f(\tau_0)|_{K_N}$ are $\Lambda_N$-compatible and differ by a collection of shifts. By the choice of $N$, each of these is a shift of divergent neighbors. By \Cref{cor:anyorder}, the two train tracks have identical core splittings in $K_N$.

We have therefore shown that $\tau_N=\kappa^{N-1}(f(\tau_0))=f(\tau_{N-1})$ (we have used \Cref{lem:coresplitcommute} in the second equality).
Renaming $\tau=\tau_{N-1}$, we have that the core split of $\tau$ is $f(\tau)$, so there is a splitting sequence from $\tau$ to $f(\tau)$ as claimed.
\end{proof}

\Cref{thm:sequenceexist} is phrased so as to be maximally useful to us later in the paper. However, we note that the same proof actually gives the following statement, which mirrors Agol's construction of layered veering triangulations and may be of independent interest to researchers in Handel-Miller theory.

\begin{theorem}\label{thm:preperiodic}
Let $f\colon L\to L$ be endperiodic, and let $\tau_0$ be an efficient $f$-endperiodic train track carrying the positive Handel-Miller lamination. Consider the sequence of train tracks $\tau_0,\tau_1, \tau_2,\dots$ where $\tau_i$ is a core split of $\tau_{i-1}$. For sufficiently large $n$, we have $\tau_{n}=f(\tau_{n-1})$.
\end{theorem}

In particular, note that we do not assume above that the quotient train track in the positive ends of $L$ is spiraling.

\subsection{Uniqueness of the splitting sequence} \label{subsec:splitsequnique}

There were several choices involved in the construction of the splitting sequence in \Cref{thm:sequenceexist}. However, we claim that viewed through an appropriate lens, the only choice that mattered was that of the train track $T_+^\infty$. We investigate this now.

Let $\tau_0\to \tau_1 \to \tau_2\to \cdots$ be an infinite splitting sequence of train tracks in $L$ (i.e. for each $i$, $\tau_{i+1}$ is obtained from $\tau_i$ by performing a single split. Further, suppose there exists a positive integer $p$ such that $f(\tau_i)=\tau_{i+p}$ for all $i$ and that $p$ is the least such positive integer. Then we say that $(\tau_n)$ is an \textbf{$f$-periodic splitting sequence} with period $p$.

Suppose that $(\tau_n)$ is an $f$-periodic splitting sequence with period $p\ge 2$ and that for some $i\ge0$, there exists a track $\tau_i'$ such that $\tau_{i-1} \to \tau_{i} \to \tau_{i+1}$ and $\tau_{i-1}\to \tau_{i}' \to \tau_{i+1}$  differ by a commutation. Then the operation of replacing $\tau_{i+np}$ by $f^n(\tau_i')$ for all $n$ such that $i+np>0$ is called an \textbf{$f$-periodic commutation}. The result of performing an $f$-periodic commutation is another $f$-periodic splitting sequence. 
(For $p=0$ or $1$, there are no $f$-periodic commutations.)

If $\tau_0 \to \tau_1 \to \tau_2 \to \cdots$ is a splitting sequence such that some truncation $(\tau_n)_{n\ge N}$ is $f$-periodic, we say $(\tau_n)$ is \textbf{eventually $f$-periodic}. We say that two eventually $f$-periodic splitting sequences are \textbf{equivalent} if they have truncations which are related by finitely many $f$-periodic commutations.
We will see in \Cref{subsec:hmvbsunique} that an eventually $f$-periodic splitting sequence determines a certain branched surface in the compactified mapping torus $\ol M_f$ of $f$, and that equivalent sequences determine the same branched surface up to isotopy.

Recall that in the construction of the splitting sequence in \Cref{thm:sequenceexist}, we made the following choices:
\begin{enumerate}[label=(\alph*)]
    \item an efficient spiraling train track $T_+^\infty$ carrying $\Lambda_+^\infty$,
    \item a tiling of the end-cycles of $L$,
    \item an initial core $K=K_0$, and
    \item a train track $\tau_K=\tau_{K_0}$ fully carrying $\Lambda|_{K_0}$, which together with $T^\infty_+$ gives rise to a train track $\tau_0$ fully carrying $\Lambda_+$ which is endperiodic in $E_0$, and whose image in $E_0/f$ is equal to $T_+^\infty$.
\end{enumerate}
In the proofs of \Cref{lem:trackconstruction} and \Cref{thm:sequenceexist}, we showed that repeatedly performing core splits yields a sequence $\tau_0\twoheadrightarrow\tau_1\twoheadrightarrow\tau_2\twoheadrightarrow\cdots$ such that for large $i$, $\tau_i$ is efficient and $f(\tau_i)=\tau_{i+1}$. 
We can then factor the core splits $\tau_i\twoheadrightarrow \tau_{i+1}$ into sequences of individual splits such that the result is an eventually $f$-periodic splitting sequence. By \Cref{lemma:differbycommutations}, the equivalence class of this resulting $f$-periodic splitting sequence is well-defined. We denote the equivalence class by $\mathscr S(T_+^\infty,\mathcal T,  K, \tau_K)$, where $\mathcal T$ denotes our choice of tiling.

We claim that $\mathscr S(T_+^\infty,\mathcal T,  K, \tau_K)$ is determined up to equivalence by the train track $T_+^\infty$. That is, up to equivalence $\mathscr S(\cdot, \cdot, \cdot, \cdot)$ is independent of the last three arguments.

\begin{lemma}\label{lem:anytrack}
Let $\tau_K'$ be another choice of input for the function $\mathscr S(T_+^\infty,\mathcal T,  K,\cdot)$. Then
\[
\mathscr S(T_+^\infty,\mathcal T,  K, \tau_K)=\mathscr S(T_+^\infty,\mathcal T,  K, \tau_K').
\]
\end{lemma}

\begin{proof}
For notational simplicity, in this proof we write $\tau$ for a train track on $L$ constructed from $\tau_K$ and $\tau'$ for a train track constructed from $\tau'_K$.

By the proof of \Cref{lem:trackconstruction} (see \Cref{rmk:trackconstruction}), there is a number $i$ such that the train tracks obtained by performing $i$ core splits on $\tau$ and $\tau'$ are efficient. Hence by truncating and relabeling, we may assume that $\tau$ and $\tau'$ are efficient. By \Cref{lem:lambdacompatible}, $\tau|_{K_i}$ and $\tau'|_{K_i}$ are $\Lambda|_{K_i}$-compatible.

Now, as in the proof of \Cref{thm:sequenceexist}, there exists a number $N$ such that the $\Lambda_+|_{K_N}$-identification of the persistent cusps of $\tau|_{K_N}$ and $\tau'|_{K_N}$ is a poset isomorphism. Therefore $\tau$ and $\tau'$ have common core splittings, so $\mathscr S(T_+^\infty,\mathcal T,  K, \tau_K)=\mathscr S(T_+^\infty,\mathcal T,  K, \tau_K')$.
\end{proof}

In light of \Cref{lem:anytrack} above, we will drop the fourth argument of $\mathscr S(\cdot,\cdot,\cdot,\cdot)$ and write simply $\mathscr S(\cdot,\cdot,\cdot)$ going forward.

\begin{lemma}\label{lem:anyindex}
For all $i\ge0$, $\mathscr S(T_+^\infty,\mathcal T, K)=\mathscr S(T_+^\infty,\mathcal T, K_i)$ (recall our convention $K=K_0$).
\end{lemma}
\begin{proof}
Up to equivalence, to obtain a sequence representing $\mathscr S(T_+^\infty,\mathcal T, K_i)$ we may start with the train track $\kappa^i(\tau)$ and take iterative core splits. On the other hand, a sequence representing $\mathscr S(T_+^\infty,\mathcal T, K)$ is obtained by taking iterative core splits of $\tau$. The two sequences are related by truncation, so $\mathscr S(T_+^\infty,\mathcal T, K)=\mathscr S(T_+^\infty,\mathcal T, K_i)$.
\end{proof}

\begin{lemma}\label{lem:anycore}
Given our fixed tiling $\mathcal T$ of the positive and negative end-cycles of $L$, let $K'=K_0'$ be another choice of core. Then $\mathscr S(T_+^\infty,\mathcal T, K)=\mathscr S(T_+^\infty,\mathcal T, K')$.
\end{lemma}

\begin{proof}
By our definition of a core, $\Lambda_+$ is disjoint from the components of $L\cut K$ and $L\cut K'$ which are neighborhoods of negative ends. Hence it suffices to assume that $K$ and $K'$ contain exactly the same negative tiles, and differ only by the positive tiles they contain.

Let $Z$ be a positive end-cycle of $L$, and let $t_1$ be the lowest-index tile in the tiling of $\mathscr U_{Z}$ which is not already contained in $K$. We will first prove the statement of the lemma for $K'=K\cup t_1$.

For all $i\ge0$, define $K_i'$ so that $K_{i+1}'$ is to $K_i'$ as $K_{i+1}$ is to $K_i$; that is, $K_{i+1}'$ is obtained by adding to $K_i'$, for each positive end-cycle, the tile of lowest index not already lying in $K_i'$.

Let $\tau_0$ be the track chosen in (d) above. For $i\ge0$, let $\tau_i$ be the maximal splitting of $\tau_0$ in $K_i$, and $\tau_i'$ be the maximal splitting of $\tau_0$ in $K_i'$. Note that there exists a natural number $N$ such that for all $i\ge N$, we have $f(\tau_i)=\tau_{i+1}$ and $f(\tau_i')=\tau_{i+1}'$.

By factoring maximal splits into individual splits, the sequence $\tau_0,\tau_0',\tau_1,\tau_1',\tau_2,\tau_2',\dots$ may be factored into splits to give an eventually $f$-periodic splitting sequence 
\[
\tau_0\to\cdots\to
\tau_0'\to\cdots\to
\tau_1\to\cdots\to
\tau_1'\to\cdots\to
\tau_2\to\cdots\to
\tau_2'\to\cdots
\]
representing $\mathscr S(T_+^\infty, \mathcal T, K)$. By truncating the sequence above to begin with $\tau_0'$, we see that $\mathscr S(T_+^\infty, \mathcal T, K)=\mathscr S(T_+^\infty, \mathcal T, K')$.

The special case we have just proven can be applied iteratively to prove the lemma for general $K'$.
\end{proof}

In light of \Cref{lem:anycore}, we will drop the third argument of $\mathscr S(\cdot,\cdot,\cdot)$ and write simply $\mathscr S(\cdot,\cdot)$ going forward.

We recall from our construction of junctures and tilings in \Cref{sec:junctures} that an end-cycle $Z$ gives rise to a compact surface $F_Z$ and a homology class $u\in H^1(F_Z;\Z)$. Any nonseparating $\Z$-weighted 1-manifold $J$ which represents $u$ and intersects $\del F_Z$ with consistent coorientation then defines a tiling $\mathcal T_Z$ of $\U_Z$.

Let $Z$ be an end-cycle of $L$. We say that two tilings $\mathcal T_Z$, $\mathcal T_Z'$ of $Z$ respectively determined by weighted cooriented 1-manifolds $J, J'\subset F_Z$ are \textbf{interleaved} if $J$ and $-J'$ cobound an embedded subsurface in $F_Z$.

\begin{lemma}\label{lem:interleaved}
Let $Z$ be a positive end-cycle of $L$. Suppose that $\mathcal T'$ is obtained from $\mathcal T$ by replacing the tiling $\mathcal T_Z$ of $Z$ by an tiling $\mathcal T_Z'$ of $Z$ which is interleaved with $\mathcal T_Z$. Then $\mathscr S(T_+^\infty, \mathcal T)=\mathscr S(T_+^\infty, \mathcal T')$.
\end{lemma}
\begin{proof}
Suppose that $\mathcal T$ and $\mathcal T'$ are induced by $J$ and $J'\subset F_Z$. Then because $\mathcal T$ and $\mathcal T'$ are interleaved, $F_Z-(J\cup J')$ consists of two surfaces, into one of which $J$ points and into one of which $J'$ points. Let these surfaces be called $W_{0.5}$ and $W_1$.

Let $\{t_i\mid i\in\mathbb Z\}$ be the tiles of $\mathcal T_Z$, each of which is naturally identified with $F_Z\cut J$. Let $t_{i.5}$ be the subsurface of $t_i$ corresponding to $W_{0.5}$.

By relabeling we can assume that $t_0$ is the lowest index tile of $\mathcal T$ not contained in $K_0$.

Let $K_{i.5}=K_{i}\cup t_{i.5}$. Then we have that $K_0\subset K_{0.5}\subset K_{1}\subset K_{1.5}\subset\cdots$. Let $\tau_{i.5}$ be the maximal splitting of $\tau$ in $K_{i.5}$. Now similarly to the proof of \Cref{lem:anycore}, the sequence $\tau_0,\tau_{0.5},\tau_1, \tau_{1.5}, \tau_2,\tau_{2.5},\dots$ can be factored to give an eventually periodic splitting sequence
\[
\tau_0\to\cdots\to
\tau_{0.5}\to\cdots\to
\tau_1\to\cdots\to
\tau_{1.5}\to\cdots\to
\tau_{2}\to\cdots\to
\tau_{2.5}\to\cdots
\]
representing $\mathscr S(T_+^\infty, \mathcal T)$. Truncating gives a sequence representing $\mathscr S(T_+^\infty, \mathcal T')$. Hence 
$\mathscr S(T_+^\infty, \mathcal T)=\mathscr S(T_+^\infty, \mathcal T')$.
\end{proof}

The following is a fairly well-known lemma.

\begin{lemma}\label{lem:subsurfaceseq}
Let $F$ be a compact oriented surface, and let $C$ and $C'$ be two cooriented multicurves such that for each boundary component $\alpha$ of $F$, all components of $C$ and $C'$ which meet $\alpha$ do so with consistent coorientation.

If $[C]=[C']$ in $H_1(F,\del F)$, then there exists a sequence of cooriented multicurves
\[
C=C_0, C_1,\dots, C_n=C'
\]
such that:
\begin{itemize}
    \item for each boundary component $\alpha$ of $F$, all components of $C_0,...,C_n$ which meet $\alpha$ do so with consistent coorientation, and
    \item for $0\le i\le n-1$, $C_i$ and $-C_{i+1}$ are the boundary of an embedded subsurface $W_i$.
\end{itemize}
\end{lemma}

\begin{proof}
This is asserted in \cite{Gab87a}, and we provide a proof here for completeness. We make the convention that the boundary of an oriented manifold is cooriented into the manifold.

In the case where $F$ is closed, the lemma follows from \cite{Hat08}. We will use this to treat the case when $F$ has nonempty boundary.
We can double $C$ and $C'$ across $\partial F$ to get collections of closed curves $DC$ and $DC'$ on the doubled surface $F$, such that $[DC]=[DC']$ in $H_1(DF)$. Then applying \cite{Hat08}, we have a sequence of cooriented multicurves
\[
DC=\widehat{C_0}, \widehat{C_1},\dots, \widehat{C_n}=DC'
\]
such that for $0\le i\le n-1$, $\widehat{C_i}$ and $-\widehat{C_{i+1}}$ are the boundary of an embedded subsurface $\widehat{W_i}$. What we will do is to restrict $\widehat{C_i}$ and $\widehat{W_i}$ to $F$, and then perform some operations to obtain the desired sequence $C=C_0, C_1,\dots, C_n=C'$ on $F$.

The details are as follows. For each $\widehat{C_i}$, consider the restriction $\widehat{C_i} \cap F$. Up to a small perturbation, we can assume that this is a multicurve on $F$. However, its components might not meet components of $\partial F$ with consistent coorientations. To fix this, we inductively perform cut-and-paste along innermost pairs of intersection points $\widehat{C_i} \cap \partial F$ that are cooriented towards each other.
See \Cref{fig:interleave1} top. Call the resulting multicurve $C_i$. Notice that $DC$ and $DC'$ meet each boundary component of $F$ with consistent coorientations, so in this case the cut-and-paste operation is not necessary and we have $C_0=C$ and $C_n=C'$.

\begin{figure}
    \centering
    \fontsize{12pt}{12pt}\selectfont
    \resizebox{!}{6cm}{
\begingroup%
  \makeatletter%
  \providecommand\color[2][]{%
    \errmessage{(Inkscape) Color is used for the text in Inkscape, but the package 'color.sty' is not loaded}%
    \renewcommand\color[2][]{}%
  }%
  \providecommand\transparent[1]{%
    \errmessage{(Inkscape) Transparency is used (non-zero) for the text in Inkscape, but the package 'transparent.sty' is not loaded}%
    \renewcommand\transparent[1]{}%
  }%
  \providecommand\rotatebox[2]{#2}%
  \newcommand*\fsize{\dimexpr\f@size pt\relax}%
  \newcommand*\lineheight[1]{\fontsize{\fsize}{#1\fsize}\selectfont}%
  \ifx\svgwidth\undefined%
    \setlength{\unitlength}{286.2992126bp}%
    \ifx\svgscale\undefined%
      \relax%
    \else%
      \setlength{\unitlength}{\unitlength * \real{\svgscale}}%
    \fi%
  \else%
    \setlength{\unitlength}{\svgwidth}%
  \fi%
  \global\let\svgwidth\undefined%
  \global\let\svgscale\undefined%
  \makeatother%
  \begin{picture}(1,0.59406166)%
    \lineheight{1}%
    \setlength\tabcolsep{0pt}%
    \put(0,0){\includegraphics[width=\unitlength,page=1]{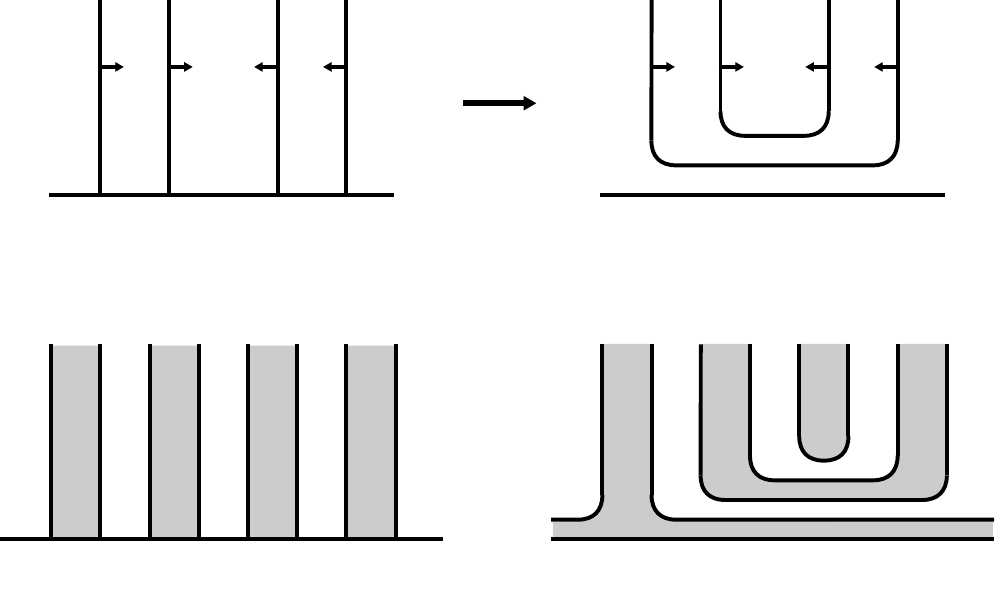}}%
    \put(0.04580062,0.00843196){\color[rgb]{0,0,0}\makebox(0,0)[lt]{\lineheight{1.25}\smash{\begin{tabular}[t]{l}(3)\end{tabular}}}}%
    \put(0.14481052,0.00843196){\color[rgb]{0,0,0}\makebox(0,0)[lt]{\lineheight{1.25}\smash{\begin{tabular}[t]{l}(2)\end{tabular}}}}%
    \put(0.24382042,0.00843196){\color[rgb]{0,0,0}\makebox(0,0)[lt]{\lineheight{1.25}\smash{\begin{tabular}[t]{l}(1)\end{tabular}}}}%
    \put(0.34283032,0.00843196){\color[rgb]{0,0,0}\makebox(0,0)[lt]{\lineheight{1.25}\smash{\begin{tabular}[t]{l}(2)\end{tabular}}}}%
    \put(0,0){\includegraphics[width=\unitlength,page=2]{interleave1.pdf}}%
  \end{picture}%
\endgroup%
}
    \caption{A sequence on $DF$ restricts to a sequence on $F$, up to performing cut-and-paste along $\partial F$ if necessary.}
    \label{fig:interleave1}
\end{figure}

We claim that for each $0\le i\le n-1$, $C_i$ and $-C_{i+1}$ bound an embedded subsurface $W_i$. To define $W_i$, consider the restriction $\widehat{W_i} \cap F$. For each boundary component $\alpha$ of $F$, if $\widehat{W_i} \cap \alpha$ is a union of intervals, then each such interval $I$ is one of 3 types:
\begin{enumerate}
    \item The endpoints of $I$ lie on $\widehat{C_i}$, hence are both cooriented inwards.
    \item One endpoint of $I$ lies on $\widehat{C_i}$ and the other lies on $\widehat{C_{i+1}}$, hence they are cooriented in the same direction. 
    \item The endpoints of $I$ lie on $\widehat{C_{i+1}}$, hence are both cooriented outwards.
\end{enumerate}

We first move $\widehat{W_i} \cap \alpha$ away from $\alpha$ near intervals of type (1). Then, inductively, for innermost pairs of intervals of type (2) whose endpoints are cooriented towards each other, we join $\widehat{W_i} \cap \alpha$ along the pair. Finally, if there are any intervals of type (3), we append a small collar neighborhood of $\alpha$ to $\widehat{W_i} \cap \alpha$. See \Cref{fig:interleave1} bottom. Call the resulting surface $W_i$.

It is straightforward to check that the boundary of $W_i$ is the union of $C_i$ and $-C_{i+1}$ as desired.
\end{proof}

\begin{lemma}\label{lem:anytiling}
Let $\mathcal T'$ be another tiling of $L$. Then $\mathscr S(T_+^\infty, \mathcal T)$ is equivalent to $\mathscr S(T_+^\infty, \mathcal T')$.
\end{lemma}

\begin{proof}
By applying \Cref{lem:subsurfaceseq} to one end-cycle at a time, $\mathcal T$ and $\mathcal T'$ are related by a sequence of tilings such that each one is interleaved with the next. The lemma then follows from \Cref{lem:interleaved}.
\end{proof}

Combining the sequence of lemmas in this subsection gives the following theorem.

\begin{theorem}\label{thm:splitsequnique}
Up to equivalence, the splitting sequence $\mathscr S(T_+^\infty, \mathcal T, K, \tau_K)$, which is defined by factoring repeated core splits of an efficient $f$-endperiodic train track $\tau_0$ carrying the positive Handel-Miller lamination, depends only on the train track $T_+^\infty$ induced by $\tau_0$ on $\U_+/\langle f\rangle$.
\end{theorem}

Hence we are justified in dropping the last three arguments of $\mathscr S(\cdot, \cdot, \cdot, \cdot)$ and simply writing $\mathscr S(\cdot)$ to denote the equivalence class of any sequence obtained from the core splitting construction of \Cref{thm:sequenceexist}.

\section{Endperiodic maps and sutured manifolds} \label{sec:endperiodictosutured}

Up to this point in the paper, we have been dealing with surfaces and automorphisms on surfaces. An equivalent way of studying this data is to consider the mapping tori of automorphisms and their associated suspension flows. This will be our perspective from this point forward. To this end, in this section we will describe how one passes from the 2-dimensional to the 3-dimensional picture, and prove some lemmas for later use.

\subsection{Sutured manifolds and compactified mapping tori} \label{subsec:compactifiedmappingtorus}

A \textbf{sutured manifold} $(Q,\gamma)$ is an oriented, compact 3-manifold $Q$ with decorated boundary. We have $\gamma\subset \del Q$ and $\gamma=A(\gamma)\cup T(\gamma)$, where $A(\gamma)$ is a union of annuli and $T(\gamma)$ is a union of tori. Each component of $A(\gamma)$ contains an oriented curve at its core called a \textbf{suture}. Let $R(\gamma)=\del Q\cut \gamma$; sometimes $R(\gamma)$ is called the \textbf{tangential boundary} of $Q$ and $\gamma$ the \textbf{transverse boundary}. We require that each component of $R(\gamma)$ is oriented, and that each component of $\partial R(\gamma)$, when given the boundary orientation, has the homology class of a suture in $H_1(\gamma)$. Since $Q$ is oriented, each component of $\del R(\gamma)$ has a well-defined coorientation either pointing out of or into $Q$. The sets of components whose coorientations point out of and into $Q$ are denoted $R_+(\gamma)$  and $R_-(\gamma)$, respectively. A consequence of this definition is that for every component $A$ of $A(\gamma)$, one component of $A$ lies on $R_+$ and the other lies on $R_-$. Often when there is no chance of confusion we omit reference to $\gamma$, for example writing $Q$ instead of $(Q,\gamma)$ or $R_\pm$ instead of $R_\pm(\gamma)$. 

A sutured manifold is \textbf{atoroidal} if any essential torus is boundary parallel.

A \textbf{foliation} $\mathcal{F}$ of a sutured manifold $(Q,\gamma)$ is a 2-dimensional cooriented foliation of $Q$ which is transverse to $\gamma$ and tangent to $R(\gamma)$ in such a way that the coorientation of $\mathcal{F}$ restricts to the coorientation of $R(\gamma)$. We say $\mathcal{F}$ is \textbf{taut} if each leaf of $\mathcal{F}$ intersects either a closed curve transverse to $\mathcal{F}$ or an interval transverse to $\mathcal{F}$ with one endpoint on $R_-$ and the other on $R_+$. We say that $\mathcal{F}$ is \textbf{depth one} if $Q-(R_+\cup R_-)$ fibers over $S^1$ with fibers the noncompact leaves of $\mathcal{F}$. 

A \textbf{depth one sutured manifold} is a sutured manifold admitting a depth one foliation and having no torus components in $R_\pm$.

For us, a \textbf{semiflow} on a sutured manifold $(Q,\gamma)$ is a 1-dimensional oriented foliation $\phi$, whose leaves we call \textbf{orbits}, that points inward along $R_-$, outward along $R_+$, and which is tangent to $\gamma$ in such a way that each orbit contained in $A(\gamma)$ is a properly embedded oriented interval with initial endpoint on $R_-(\gamma)$ and terminal endpoint on $R_+(\gamma)$. For us it will not be important to explicitly parameterize a semiflow. However, if one chooses a parameterization, orbits are not generally defined for all forward or backward time.

The following is well-known and appears as \cite[Lemma 12.5]{CCF19}.
\begin{lemma}
Let $f\colon L\to L$ be an endperiodic map. Then there exists a  sutured manifold $\ol M_f$, a taut depth one foliation $\mathcal{F}$ of $\ol M_f$, and a semiflow $\phi_f$ of $\ol M_f$ such that $L$ is homeomorphic to each noncompact leaf of $\mathcal{F}$ and the first return map induced by $\phi_f$ is equal to $f$.
\end{lemma}

We call the semiflow $\phi_f$ the \textbf{suspension semiflow} of $f$.

The manifold $\ol M_f$ is called the \textbf{compactified mapping torus} of $f$ because it is constructed by compactifying the mapping torus $M_f=L\times [0,1]/((x,1)\sim(f(x),0))$. The compactification works by attaching copies of the surfaces $\U_+/\langle f \rangle$ and $\U_-/\langle f \rangle$ to $M_f$. (Recall that $\U_+$ and $\U_-$ are the positive and negative escaping sets of $f$, respectively. See \Cref{sec:junctures}). Thus this compactification is obtained by gluing on one ideal point for each escaping end of an $f$-orbit. For the details of this construction see \cite[\S 3]{FKLL23}; they work only with ``irreducible" endperiodic maps on boundaryless surfaces but their construction goes through in our setting also.

The sutured structure on $\ol M_f$ is as follows. The tangential boundary $R_\pm$ arises from the copy of $\U_\pm/\langle f \rangle$ added to $M_f$ during the compactification. A component of $T(\gamma)$ arises from a compact boundary component of $L$ whose $f$-orbit has finitely many components. A component of $A(\gamma)$ can arise in two ways: from a noncompact boundary component of $L$, or from a compact boundary component of $L$ whose $f$-orbit has infinitely many components.

If $f$ is a Handel-Miller map, then $f$ preserves the Handel-Miller laminations and induces the spiraling laminations $\Lambda_\pm^\infty$ on $\U_\pm/\langle f \rangle$ (see \Cref{lem:bdylams}). We can canonically identify $R_\pm$ with $\U_\pm/\langle f \rangle$ (see \cite[Lemma 12.36]{CCF19}, so we can think of $\Lambda_\pm$ as a subset of $R_\pm$. Moreover, if we consider the union of all $\phi_f$-orbits passing through points in $\Lambda_\pm\subset L\subset \ol M_f$, we obtain a pair of 2-dimensional laminations $\LL^u$, $\LL^s$ such that $\LL^u\cap R_+=\Lambda_+^\infty$ and $\LL^s\cap R_-=\Lambda_-^\infty$. We call these the \textbf{unstable} and \textbf{stable Handel-Miller laminations} respectively.

\begin{example}\label{ex:stackofchairsflow}
Let $L=\{(x,y)\mid xy\le1\}$ and let $f\colon L\to L$ be the map $\left(\begin{smallmatrix}2&0\\0&\frac{1}{2}\end{smallmatrix}\right)$ from \Cref{example:stackofchairs}. Then $\ol M_f$ is a solid torus with four longitudinal sutures, each one homotopic to the core of the solid torus. Each of $R_+$ and $R_-$ has 2 components, both of which are annuli. The depth 1 foliation of $\ol M_f$ is known as a ``stack of chairs." Each of $\LL^u$ and $\LL^s$ consists of a single annulus connecting the two components of $R_+$ and the two components of $R_-$, respectively. The laminations $\Lambda_+^\infty$ and $\Lambda_-^\infty$ each have two components, each of which is a circle. See \Cref{fig:stackofchairsflow}.
\end{example}

\begin{figure}
    \centering
    \resizebox{!}{2in}{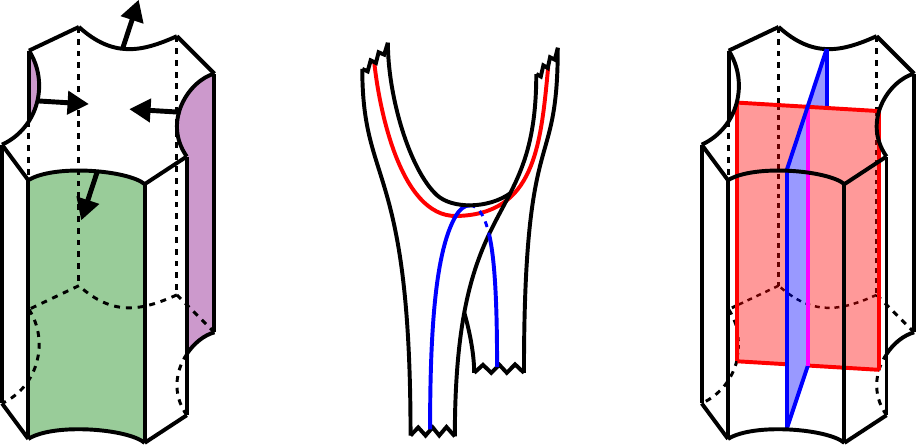}
    \caption{Left: the compactified mapping torus $\ol M_f$ for the map $f=\left(\begin{smallmatrix}2&0\\0&\frac{1}{2}\end{smallmatrix}\right)$, where the top should be identified with the bottom to give a solid torus. The green and purple annuli are $R_+$ and $R_-$ respectively. Center: the surface $L$ sits inside $\ol M_f$ as shown, which two ends spiraling onto $R_+$ and two ends spiraling onto $R_-$. Right: the laminations $\LL_+$ and $\LL_-$.}
    \label{fig:stackofchairsflow}
\end{figure}

\subsection{Useful lemmas about the unstable lamination} \label{subsec:unstablelamfacts}

In this subsection, we will state and prove some lemmas about the unstable Handel-Miller lamination $\mathcal{L}^u$ in the setting above. These facts will play a role in \Cref{sec:folcone} and \Cref{sec:hmvbsunique}. Symmetric statements hold for the stable Handel-Miller lamination, even though those will not play a role in this paper.

\begin{lemma} \label{lemma:hmlamnoproductregions}
$\mathcal{L}^u$ has no $I$-fibered complementary regions whose boundary components lie along $\mathcal{L}^u$. Here a complementary region $C$ is $I$-fibered if $C$ fibers over some surface $F$ with $I$ fibers.
\end{lemma}
\begin{proof}
Suppose otherwise. Let us identify the surface $L$ with a fixed leaf of $\mathcal{F}$. $L$ intersects $\partial C$ in leaves of $\Lambda_+$, hence each component of this intersection is a copy of $\mathbb{R}$. Meanwhile, $L$ is incompressible in $\overline{M_f}$ hence in $C$. Since $L$ is incompressible in $\overline{M_f}$, $L \cap C$ is incompressible in $C$, from which it can deduced that each component of $L \cap C$ is homeomorphic to $\mathbb{R} \times [0,1]$, meaning that there is a complementary region of $\Lambda_+$ in $L$ homeomorphic to $\mathbb{R} \times [0,1]$. But this is impossible since the leaves of $\Lambda_+$ are geodesics for some standard hyperbolic metric on $L$.
\end{proof}

\begin{lemma} \label{lemma:hmlamholonomy}
Let $\gamma$ be a closed orbit of $\phi_h$ on a leaf $A$ of $\mathcal{L}^u$. Suppose $A$ is non-isolated from some side. Then the holonomy of $\mathcal{L}^u$ along $\gamma$ is topologically contracting on that side, i.e. there is an immersion of a rectangle $\alpha: [0,1]_t \times [0,1]_s \to \overline{M_f}$ such that $\alpha([0,1] \times \{0\})$ traverses $\gamma$ for increasing $t$, $\alpha([0,1] \times \{1\})$ lies on a leaf of $\mathcal{L}^u$, and $\alpha(\{1\} \times [0,1]) \subsetneq \alpha(\{0\} \times [0,1])$.
\end{lemma}
\begin{proof}

The leaf $A$ is the suspension of some periodic leaf $l_+$ of $\Lambda_+$. Applying \cite[Corollary 6.11]{CCF19} to $l_+$, there is a leaf $l_-$ of $\Lambda_-$ and a periodic point $x \in  l_+ \cap l_-$.

Notice that $l_+$ is non-isolated from the side that suspends to the non-isolated side of $A$. 
Now applying \cite[Corollary 6.11]{CCF19} again but to $l_-$ this time (and also using \cite[Corollary 6.15]{CCF19}), we see that the holonomy is contracting on our fixed side of $l_+$, which implies the lemma.
\end{proof}

\subsection{Reeb sutured manifolds} \label{subsec:reebsutured}

We now define Reeb sutured manifolds, which generalize sutured manifolds in the same way that Reeb endperiodic maps generalize endperiodic maps. We will also explain a construction in \Cref{construction:reebmaptorus} which converts a Reeb sutured manifold into a sutured manifold, analogous to endperiodization (\Cref{construction:samesign}).

A \textbf{Reeb sutured manifold} $(P,\gamma)$ is defined similarly to a sutured manifold. The only difference is that $A(\gamma)$ is allowed to contain components that have both boundary components on $R_+$ or both boundary components on $R_-$. Such an annulus is called a \textbf{Reeb annulus} and does not have a suture at its core. 

A \textbf{foliation} of a Reeb sutured manifold is defined just as for foliations of sutured manifolds, with the additional requirement that the restriction of the foliation to each Reeb annulus is a Reeb foliation.

A \textbf{semiflow} on a Reeb sutured manifold is defined like a semiflow on a sutured manifold except on the Reeb annuli. If $A$ is a Reeb annulus touching only $R_+$ ($R_-$), orbits of the semiflow do not end on $R_+$ ($R_-$) in the backward (forward) direction.

If $(P,\gamma)$ is a Reeb sutured manifold, there is a naturally associated sutured manifold obtained by performing the following procedure for each Reeb annulus of $A(\gamma)$. Suppose without loss of generality that $A$ connects $R_+$ to $R_+$. Subdivide $A$ into 3 annuli $A_1, A_2, A_3$, labeled so that $A_1$ and $A_3$ touch $\del A$ and $A_2\subset \intr A$. Then modify $A(\gamma)$ and $R(\gamma)$ by adding $A_2$ to $R_-(\gamma)$, and replacing $A$ by $A_1$ and $A_3$ in $A(\gamma)$. Place sutures in $A_1$ and $A_3$, oriented so as to be compatible with the orientations of $R_+$ and $R_-$. The result is an honest sutured manifold $(P', \gamma')$ called the \textbf{de-Reebification} of $(P,\gamma)$.

\begin{construction}\label{construction:reebmaptorus}[Mapping tori of Reeb endperiodic maps] Let $f\colon L\to L$ be a Reeb endperiodic map. There is a compactified mapping torus $\ol M_f$ of $f$ defined just as for endperiodic maps, but $\ol M_f$ is in general only a Reeb sutured manifold. However, there is still a depth one foliation $\mathcal{F}$ of $\ol M_f$ whose depth one leaves are homeomorphic to $L$, and a semiflow on $\ol M_f$ whose first return map is conjugate to $f$ under this identification.

Let $f'\colon L'\to L'$ be the endperiodization of $f$ (\Cref{construction:samesign}), where $L'$ is $L$ minus finitely many boundary points. Let $\mathcal{F}'$ be the associated depth one foliation of $\ol M_{f'}$. Then $\ol M_{f'}$ is naturally identified with the de-Reebification of $\ol M_f$, and $\mathcal{F}'$ is obtained from ``spinning" $\mathcal{F}$ around the annuli added to the tangential boundary of $\ol M_f$ in the de-Reebification.
\end{construction}

\section{Veering branched surfaces} \label{sec:vbs}

In this section we start working with branched surfaces and laminations in 3-manifolds. The prototype of the laminations we consider is the unstable Handel-Miller lamination in a compactified mapping torus. In this context, the natural type of branched surfaces to consider are (unstable) dynamical branched surfaces. In the first two subsections, we will recall the definition of these and some related ideas. Essentially all the definitions presented in these subsections are due to Mosher. 

Then we define veering branched surfaces in sutured manifolds, the main objects of study in the rest of the paper. The rest of the section develops some of the theory of these veering branched surfaces, most of it being adapted from the theory in non-sutured manifolds.

\subsection{Dynamic branched surfaces} \label{subsec:dynbranchedsurfaces}

A \textbf{branched surface} $B$ is a 2-complex embedded in a 3-manifold such that every point in $B$ has a neighborhood smoothly modeled on a point in the space shown on the top left of \Cref{fig:bsdef}. In particular, every point in $B$ has a well-defined tangent space. 

If $(Q,\gamma)$ is a sutured manifold, a \textbf{vertical branched surface} in $Q$ is a 2-complex $B\subset Q$ such that $B\cap \intr Q$ is a branched surface, and each point in $B\cap \del Q$ has a neighborhood modeled on a point in the space on the top right of \Cref{fig:bsdef}. 

\begin{figure}
    \centering
    \fontsize{12pt}{12pt}\selectfont
    \resizebox{!}{2.5in}{
\begingroup%
  \makeatletter%
  \providecommand\color[2][]{%
    \errmessage{(Inkscape) Color is used for the text in Inkscape, but the package 'color.sty' is not loaded}%
    \renewcommand\color[2][]{}%
  }%
  \providecommand\transparent[1]{%
    \errmessage{(Inkscape) Transparency is used (non-zero) for the text in Inkscape, but the package 'transparent.sty' is not loaded}%
    \renewcommand\transparent[1]{}%
  }%
  \providecommand\rotatebox[2]{#2}%
  \newcommand*\fsize{\dimexpr\f@size pt\relax}%
  \newcommand*\lineheight[1]{\fontsize{\fsize}{#1\fsize}\selectfont}%
  \ifx\svgwidth\undefined%
    \setlength{\unitlength}{269.18595084bp}%
    \ifx\svgscale\undefined%
      \relax%
    \else%
      \setlength{\unitlength}{\unitlength * \real{\svgscale}}%
    \fi%
  \else%
    \setlength{\unitlength}{\svgwidth}%
  \fi%
  \global\let\svgwidth\undefined%
  \global\let\svgscale\undefined%
  \makeatother%
  \begin{picture}(1,0.65449484)%
    \lineheight{1}%
    \setlength\tabcolsep{0pt}%
    \put(0,0){\includegraphics[width=\unitlength,page=1]{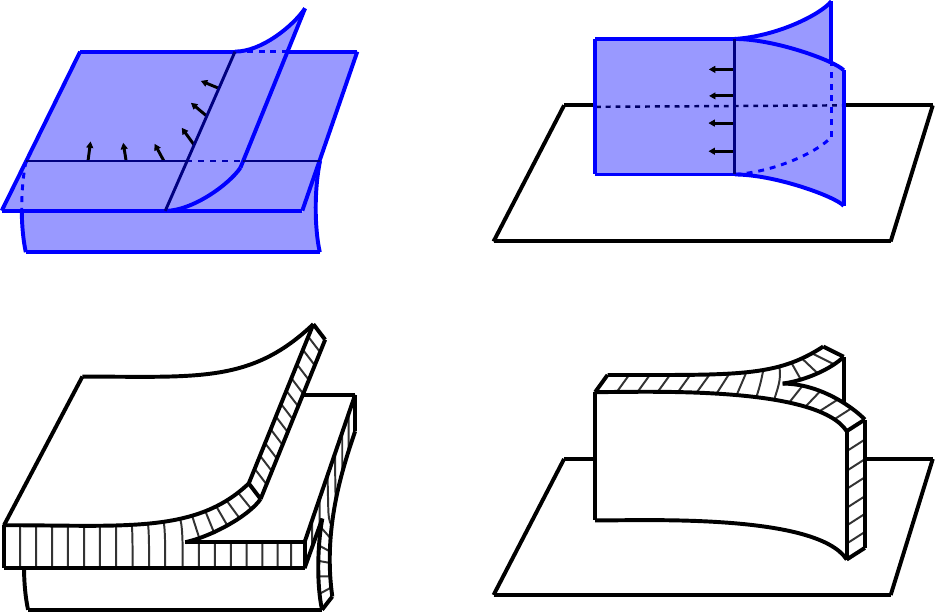}}%
    \put(0.08966515,0.54157244){\color[rgb]{0,0,0}\makebox(0,0)[lt]{\lineheight{1.25}\smash{\begin{tabular}[t]{l}$B$\end{tabular}}}}%
    \put(0.67886105,0.57074971){\color[rgb]{0,0,0}\makebox(0,0)[lt]{\lineheight{1.25}\smash{\begin{tabular}[t]{l}$B$\end{tabular}}}}%
    \put(0.57412084,0.420168){\color[rgb]{0,0,0}\makebox(0,0)[lt]{\lineheight{1.25}\smash{\begin{tabular}[t]{l}$\partial Q$\end{tabular}}}}%
  \end{picture}%
\endgroup%
}
    \caption{Top: The local models for vertical branched surfaces at points in $\intr(Q)$ and in $\del Q$. The maw vector field is indicated using arrows along the branch locus. Bottom: The local models for a standard neighborhood of a vertical branched surface.}
    \label{fig:bsdef}
\end{figure}

The union of the nonmanifold points of a vertical branched surface $B$ is called the \textbf{branch locus} of $B$ and denoted $\brloc(B)$. This set decomposes as a union of smooth, properly immersed curves and arcs called \textbf{branch curves} and \textbf{branch arcs}, respectively. A \textbf{component} of $\brloc(B)$ refers to a branch curve or branch arc (note that in general this is not a connected component of $\brloc(B)$). A \textbf{branch segment} is the image of a smooth immersion $I\to \brloc(B)$. The nonmanifold points of $\brloc(B)$ are called the \textbf{triple points} of $B$. A mnemonic for this definition is that near such a point, the branched surface is the quotient of a stack of three disks such that each disk is $C^1$ embedded.

There is a continuous vector field on $\brloc(B)$ called the \textbf{maw vector field} defined up to homotopy by the property that away from triple points it always points from the 2-sheeted side to the 1-sheeted side of $\brloc(B)$. The maw vector field induces a coorientation on each branch arc and branch loop that we call the \textbf{maw coorientation}.

A \textbf{sector} is a component of $B \cut \brloc(B)$.
Note that each sector is naturally a surface with corners, with the sides lying along $\brloc(B)$ and $\partial Q$.

$B \cap R_+$ and $B \cap R_-$ are train tracks on $R_+$ and $R_-$ respectively. We refer to these as the \textbf{boundary train tracks} of $B$, and denote their union by $\partial B$.

A \textbf{lamination} $\Lambda$ in a sutured manifold $Q$ is a partition of a closed subset of $Q$ into connected 2-manifolds, such that each point $x\in Q$ has a neighborhood $\mathbb{R}^2 \times \mathbb{R}$ with elements of the partition intersecting the neighborhood of the form $\mathbb{R}^2 \times C$ (if $x$ is in the interior of $Q$) or a neighborhood $[0,\infty) \times \mathbb{R} \times \mathbb{R}$ with elements of the partition intersecting the neighborhood in sets of the form $[0,\infty) \times \mathbb{R} \times C$ for some closed set $C$ (if $x$ is on the boundary of $Q$). The elements of the partition are called the \textbf{leaves} of $\Lambda$. As with one-dimensional laminations, we will often conflate a lamination with the union of its leaves.

If $\Lambda$ is a lamination in $Q$, then $\Lambda \cap R_+$ and $\Lambda \cap R_-$ are laminations on $R_+$ and $R_-$ respectively. We refer to these as the \textbf{boundary laminations} of $\Lambda$. 

Let $B$ be a branched surface in $Q$ and let $\Lambda$ be a lamination on $Q$. A \textbf{standard neighborhood} of $B$ is a closed regular neighborhood $N$ of $\tau$ which is foliated by line segments called \textbf{ties}
such that each line segment meets the sectors of $B$ transversely. See \Cref{fig:bsdef} bottom.

The \textbf{vertical boundary} of $N$, denoted by $\partial_v N$, is the complement of the union of endpoints of the ties in $\partial N \backslash \partial Q$. The \textbf{horizontal boundary} of $N$, denoted by $\partial_h N$, is the complementary region of $\partial_v N$ in $\partial N \backslash \partial Q$.
Notice that in this definition, $\partial_v N$ is a union of 1-manifolds, which may differ from some conventions in the literature. We have chose to define $N$ in this way for better analogy with the definitions for train tracks in \Cref{sec:ttlam}.

We remark that $\partial N \cap \partial Q$ is neither in the horizontal boundary nor in the vertical boundary of $N$. Also notice that $\partial N \cap R_\pm$ is a standard neighborhood of the boundary train tracks $B \cap R_\pm$ respectively.

By collapsing the ties, we get a projection map $N \to B$. We say that $B$ \textbf{carries} $\Lambda$ if $B$ has a standard neighborhood $N$ such that $\Lambda$ is embedded in $N$ in a way so that its leaves are transverse to the ties. In this case we say that the map $\Lambda \hookrightarrow N \to B$ is the \textbf{carrying map} and we say that $N$ is a \textbf{standard neighborhood} of $\Lambda$. Further, we say that $\tau$ \textbf{fully carries} $\Lambda$ if $\Lambda$ intersects every tie of $N$.

We now bring dynamics into the picture.
An \textbf{unstable dynamic branched surface} in $Q$ is an ordered pair $(B,V)$ where $B$ is a vertical branched surface and $V$ is a nonvanishing $C^0$ vector field on $M$ such that 
\begin{itemize}
\item $V$ is tangent to $\gamma$, inward pointing along $R_-$ and outward pointing along $R_+$,
\item $V$ is tangent to $B$, and
\item $V|_{\brloc(B)}$ is a maw vector field for $B$.
\end{itemize}
$V$ in this definition is said to be \textbf{smooth} if it is smooth on $B \backslash \brloc(B)$ and has a unique forward trajectory starting at each point of $B$. 

Symmetrically, a \textbf{stable dynamic branched surface} in $Q$ is an ordered pair $(B,V)$ where $B$ is a vertical branched surface and $V$ is a nonvanishing $C^0$ vector field on $M$ such that 
\begin{itemize}
\item $V$ is tangent to $\gamma$, inward pointing along $R_-$ and outward pointing along $R_+$,
\item $V$ is tangent to $B$, and
\item $-V|_{\brloc(B)}$ is a maw vector field for $B$.
\end{itemize}
$V$ in this definition is said to be \textbf{smooth} if it is smooth on $B \backslash \brloc(B)$ and has a unique backward trajectory starting at each point of $B$.

In this paper, except for \Cref{sec:dynamicpairs}, $V$ will always be smooth. Hence for the sake of brevity, we will implicitly include $V$ being smooth as part of the definition of an unstable or stable branched surface. In \Cref{sec:dynamicpairs} we will need to relax the definitions slightly in order to have a single vector field $V$ such that $(B^u, V)$ and $(B^s,V)$ are unstable and stable dynamic branched surfaces respectively.

\begin{remark} \label{rmk:vf=combinatorialdata}
As Mosher points out in \cite[\S 1.5]{Mos96}, the existence of a dynamic vector field $V$ for a given branched surface $B\subset M$ is a purely combinatorial property of $B$. Indeed, such a vector field can always be constructed along $\brloc (B)$. Whether it can be extended to all of $B$ and then to $M$ depends only on the combinatorics of the sectors of $B$ and the combinatorics of the components of $M\cut B$, respectively. As such we generally think of $V$ as a placeholder for some combinatorial data, unless we are explicitly using it for one of our arguments.
\end{remark}

Finally, we recall the definition of dynamically splitting an unstable dynamic branched surface $B$.

Let $N(B)$ be a standard neighborhood of $B$. Let $F$ be a surface embedded in $N(B)$ such that $\partial F = \partial_v F \cup \partial_i F$, where $\partial_v F \subset \partial_v N(B)$ and $\partial_i F \subset \intr N(B)$, and such that $F$ is transverse to the ties of $N(B)$. Then we can pull back the vector field $V$ on $B$ to $F$ via the composition $F \to N(B) \to B$. If the image of $\partial_i F$ under $N(B) \to B$ is transverse to $\brloc(B)$, and if the pulled back vector field points outwards along $\partial_i F$, then we call $F$ a \textbf{dynamic splitting surface}. By \textbf{dynamically splitting along $F$}, we refer to the operation of cutting $N(B)$ along $F$, then collapsing the remaining intervals to get a branched surface $B_F$. There is a natural choice of vector field making $B_F$ an unstable branched surface. See \Cref{fig:dynsplit}.

\begin{figure}
    \centering
    \fontsize{12pt}{12pt}\selectfont
    \resizebox{!}{4cm}{
\begingroup%
  \makeatletter%
  \providecommand\color[2][]{%
    \errmessage{(Inkscape) Color is used for the text in Inkscape, but the package 'color.sty' is not loaded}%
    \renewcommand\color[2][]{}%
  }%
  \providecommand\transparent[1]{%
    \errmessage{(Inkscape) Transparency is used (non-zero) for the text in Inkscape, but the package 'transparent.sty' is not loaded}%
    \renewcommand\transparent[1]{}%
  }%
  \providecommand\rotatebox[2]{#2}%
  \newcommand*\fsize{\dimexpr\f@size pt\relax}%
  \newcommand*\lineheight[1]{\fontsize{\fsize}{#1\fsize}\selectfont}%
  \ifx\svgwidth\undefined%
    \setlength{\unitlength}{328.42073482bp}%
    \ifx\svgscale\undefined%
      \relax%
    \else%
      \setlength{\unitlength}{\unitlength * \real{\svgscale}}%
    \fi%
  \else%
    \setlength{\unitlength}{\svgwidth}%
  \fi%
  \global\let\svgwidth\undefined%
  \global\let\svgscale\undefined%
  \makeatother%
  \begin{picture}(1,0.35325903)%
    \lineheight{1}%
    \setlength\tabcolsep{0pt}%
    \put(0,0){\includegraphics[width=\unitlength,page=1]{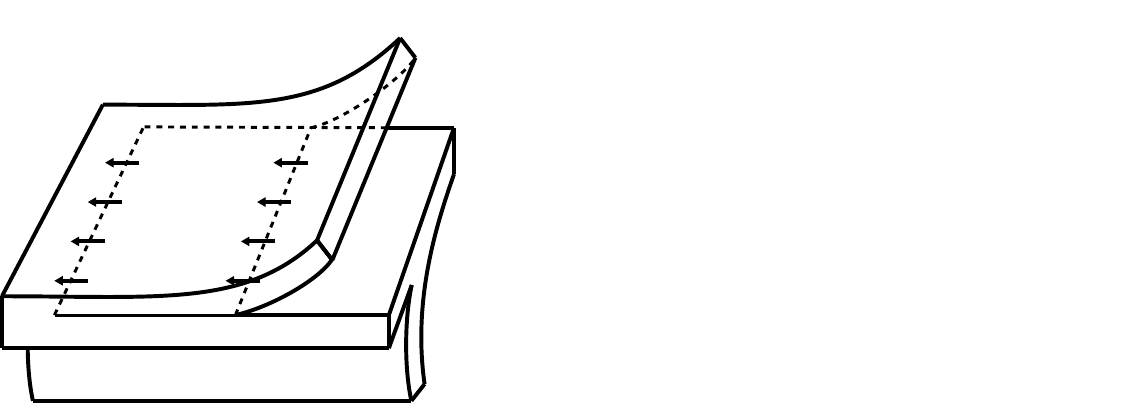}}%
    \put(0.14084499,0.15098802){\color[rgb]{0,0,0}\makebox(0,0)[lt]{\lineheight{1.25}\smash{\begin{tabular}[t]{l}$F$\end{tabular}}}}%
    \put(0,0){\includegraphics[width=\unitlength,page=2]{dynsplit.pdf}}%
  \end{picture}%
\endgroup%
}
    \caption{Dynamically splitting a dynamic branched surface along a dynamic splitting surface $F$.}
    \label{fig:dynsplit}
\end{figure}

Let $F \looparrowright B$ be an immersed surface with $\partial F = \partial_v F \cup \partial_i F$ where $\partial_v F$ lies along $\brloc(B)$. Suppose $F$ can be lifted to a dynamic splitting surface $F'\subset N(B)$. Then, as long as $F'$ is clear from context, we will refer to dynamically splitting $B$ along $F'$ as \textbf{dynamically splitting $B$ along $F$}. 

\subsection{Dynamic manifolds}

Following \cite{Mos96}, we define a \textbf{3-manifold with corners} to be a 3-manifold $M$ with boundary such that every point $p$ has a neighborhood modeled on one of the following 6 closed subsets of $\R^3$, where $p$ is identified with the origin:
\begin{itemize}
    \item Interior point: all of $\R^3$
    \item Boundary point: the upper half space $\{(x,y,z)\mid z\ge 0\}$
    \item Apex: the closed orthant $\{(x,y,z)\mid x,y,z\ge 0\}$
    \item (Convex) corner edge: $\{(x,y,z)\mid x,y\ge0\}$
    \item Gable: $\{x,y,z\mid x\ge0, z\le f(y)\}$ where $f\colon \R\to (-\infty, 0]$ is a cusp function e.g. $f(y)=-\sqrt{|y|}$
    \item Cusp edge: $\{x,y,z\mid  z\le f(y)\}$ where $f$ is as above
\end{itemize}

\begin{figure}
    \centering
    \fontsize{6pt}{6pt}\selectfont
    \resizebox{!}{2.8cm}{
\begingroup%
  \makeatletter%
  \providecommand\color[2][]{%
    \errmessage{(Inkscape) Color is used for the text in Inkscape, but the package 'color.sty' is not loaded}%
    \renewcommand\color[2][]{}%
  }%
  \providecommand\transparent[1]{%
    \errmessage{(Inkscape) Transparency is used (non-zero) for the text in Inkscape, but the package 'transparent.sty' is not loaded}%
    \renewcommand\transparent[1]{}%
  }%
  \providecommand\rotatebox[2]{#2}%
  \newcommand*\fsize{\dimexpr\f@size pt\relax}%
  \newcommand*\lineheight[1]{\fontsize{\fsize}{#1\fsize}\selectfont}%
  \ifx\svgwidth\undefined%
    \setlength{\unitlength}{375.67578016bp}%
    \ifx\svgscale\undefined%
      \relax%
    \else%
      \setlength{\unitlength}{\unitlength * \real{\svgscale}}%
    \fi%
  \else%
    \setlength{\unitlength}{\svgwidth}%
  \fi%
  \global\let\svgwidth\undefined%
  \global\let\svgscale\undefined%
  \makeatother%
  \begin{picture}(1,0.20715395)%
    \lineheight{1}%
    \setlength\tabcolsep{0pt}%
    \put(0,0){\includegraphics[width=\unitlength,page=1]{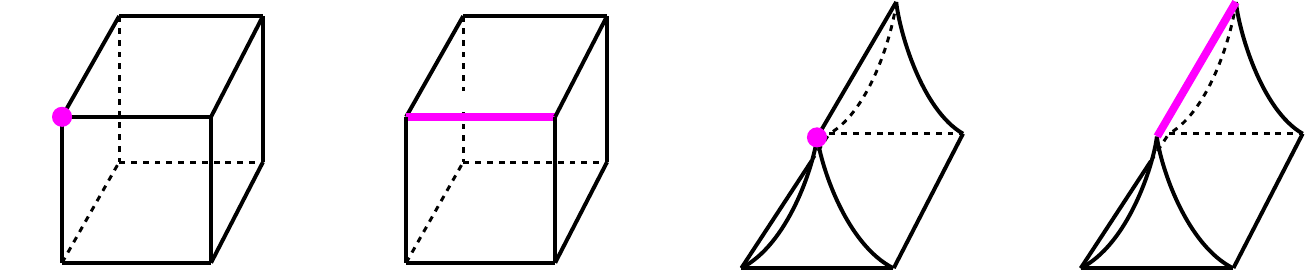}}%
    \put(-0.00081104,0.10248891){\color[rgb]{0,0,0}\makebox(0,0)[lt]{\lineheight{1.25}\smash{\begin{tabular}[t]{l}apex\end{tabular}}}}%
    \put(0.3247989,0.12534012){\color[rgb]{0,0,0}\makebox(0,0)[lt]{\lineheight{1.25}\smash{\begin{tabular}[t]{l}corner edge\end{tabular}}}}%
    \put(0.58316571,0.11509862){\color[rgb]{0,0,0}\makebox(0,0)[lt]{\lineheight{1.25}\smash{\begin{tabular}[t]{l}gable\end{tabular}}}}%
    \put(0.82601504,0.15011472){\color[rgb]{0,0,0}\makebox(0,0)[lt]{\lineheight{1.25}\smash{\begin{tabular}[t]{l}cusp edge\end{tabular}}}}%
  \end{picture}%
\endgroup%
}
    \caption{Illustrations of an apex, corner edge, gable, and cusp edge.}
    \label{fig:corners}
\end{figure}

See \Cref{fig:corners}. A connected component of the set of boundary points is called a \textbf{face} of the manifold with corners. A connected component of the set of corner edge points or cusp edge points is called a \textbf{corner edge} or \textbf{cusp edge}, respectively.

\begin{example}
A sutured manifold $(Q,\gamma)$ can be considered as a 3-manifold with corners by taking the set of (convex) corner edges to be the curves of intersection between $R_+(\gamma)$ and $A(\gamma)$ and between $R_-(\gamma)$ and $A(\gamma)$. There are no apexes, gables, or cusp edges in this example. 
\end{example}

Going forward, we will always view sutured manifolds as 3-manifolds with corners in this way.

We next define a dynamic manifold.
Consider a triple $(D, V, \lambda)$ where $D$ is a 3-manifold with corners, $V$ is a continuous nonvanishing vector field on $D$, and $\lambda$ is a labeling 
\[
\lambda\colon \{\text{faces of $D$}\}\to \{\p,\m,\b,\s,\u\}
\]
where the labels stand for $\p$lus, $\m$inus, $\b$are, $\s$table, and $\u$nstable respectively. This assigns each cusp edge and corner edge a pair of labels and we can classify edges by this pair of labels. For example a $\u\u$-edge is one which has a $\u$-face on both of its sides.

A \textbf{dynamic manifold} is such a triple $(D,V,\lambda)$ which satisfies the following:
\begin{enumerate}[label=(\alph*)]
    \item $V$ points out of $D$ along all $\p$-faces and into $D$ along all $\m$-faces.
    \item $V$ is tangent to all $\b$-, $\s$-, and $\u$-faces
    \item all $\s\s$-, $\u\u$-, and $\p\m$-edges are cusp edges
    \item $V$ points into $D$ along all $\s\s$-edges and points out of $D$ along all $\u\u$-edges
    \item there are no $\p\p$-, $\m\m$-, $\b\b$-, $\b\s$-, or $\b\u$-edges
\end{enumerate}

The motivation behind the above axioms is that a dynamic manifold should be thought of as a complementary component of the union of a stable dynamic branched surface $B^s$ and an unstable dynamic branched surface $B^u$ intersecting transversely (the $\b$are labels correspond to sutures). For example, condition (a) corresponds to the fact that a dynamic vector field points outward along $R_+$ and inward along $R_-$. Condition (d) corresponds to the fact that a dynamic vector field restricts to the maw vector field on $\brloc(B^u)$, and to the negative of the maw vector field on $\brloc(B^s)$.

\begin{remark}\label{rmk:annulusfaces}
It follows from the axioms that all $\u$-faces which are not incident to $\s$-faces (and vice versa) are annuli or tori. Indeed, the vector field $V$ must point outward (inward) along the entire boundary of such a $\u$-face ($\s$-face), so this follows from Poincar\'e-Hopf and the fact that $V$ is nonsingular.
\end{remark}

\begin{figure}
    \centering
    \resizebox{!}{2in}{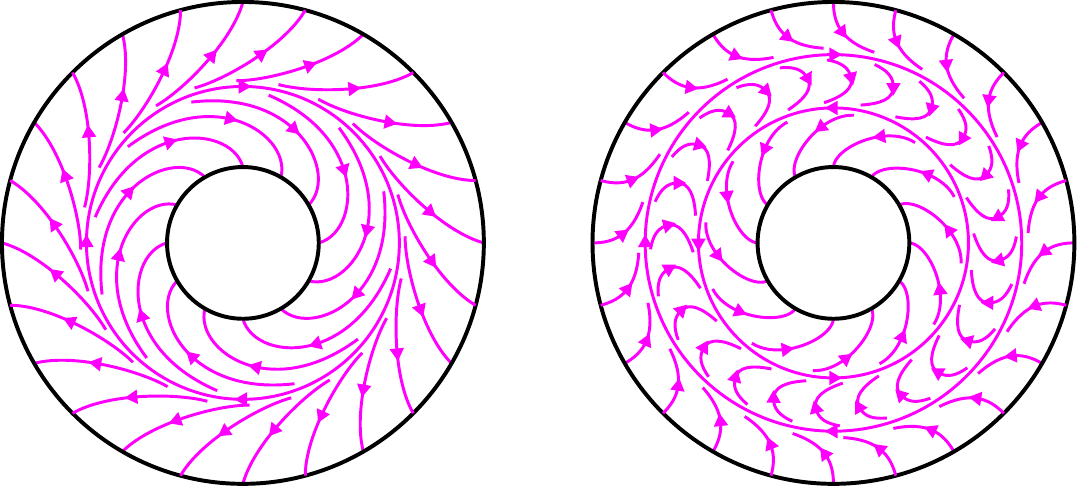}
    \caption{Here we have drawn the trajectories of two different vector fields on an annulus. The vector field on the left is circular, while the vector field on the right is not.}
    \label{fig:circular}
\end{figure}

\begin{definition}[circular, dynamic orientation, (in-)coherent cusp circle]\label{def:circular}
Let $V$ be a nonvanishing vector field generating a forward semiflow on a manifold $M$, possibly with boundary. We say $V$ is \textbf{circular} if there exists a map $g\colon M\to S^1$ such that if $\gamma(t)$ is a monotonic parameterization of a trajectory of $V$, then $g(\gamma(t))$ is a monotonic path in $S^1$. See \Cref{fig:circular} for an example and nonexample.

If $\gamma$ is a $\u\u$-cusp circle of a dynamic manifold then by \Cref{rmk:annulusfaces} the $\u$-faces $A_1$ and $A_2$ to either side of $\gamma$ are annuli.  If $V$ is circular on both of these faces, then there are induced orientations on $H_1(A_1)$ and $H_1(A_2)$. This orientation of $H_1(A_i)$ is called the \textbf{dynamic orientation}. If the dynamic orientations of $H_1(A_1)$, $H_1(A_2)$ match up along $\gamma$, we say $\gamma$ is a \textbf{coherent} cusp circle. Otherwise we say $\gamma$ is \textbf{incoherent}.
\end{definition}

The following lemma proves that a vector field of the type shown on the left of \Cref{fig:circular} is circular.

\begin{lemma}\label{lem:reebcircular}
Let $\mc R$ be a cooriented Reeb foliation of an annulus $A$, and let $V_A$ be a vector field on $A$ which is positively transverse to $\mc R$. Then $V$ is circular.
\end{lemma}

\begin{proof}
If we place a Riemannian metric on $A$, then by compactness there is some $\epsilon$ such that the vector field $V_A$ makes an angle of at least $\epsilon$ with any vector tangent to $\mc R$. Hence we can perturb the tangent distribution of $\mc R$ slightly to obtain a foliation $\mc R'$ with tangent distribution close enough to that of $\mc R$ so that $V_A$ is positively transverse to $\mc R'$, but whose leaves are properly embedded line segments. The map to the leaf space of $\mc R'$ now certifies the circularity of $V_A$.
\end{proof}

As examples, we now define two types of dynamic manifolds following Mosher. Later, these types of dynamic manifolds will feature in the definition of a ``very full" dynamic branched surface.

\begin{definition}[$\u$-cusped torus] \label{defn:cuspedtorus}
Let $\Delta$ be a closed disk whose boundary is smooth with the exception of $n\ge2$ cusps, and let $f\colon \Delta\to \Delta$ be a diffeomorphism. The mapping torus $M$ of $f$ is a 3-manifold with corners homeomorphic to a solid torus. There are $\frac{n}{p}$ circular cusp edges of $M$, where $p$ is the period of a cusp of $\Delta$ under $f$. Likewise there are $\frac{n}{p}$ annular faces of $M$. We label each of these faces $\u$. Take a circular vector field $V$ on $M$ that gives it the structure of a dynamic manifold. Equipped with such a vector field, $M$ is called a \textbf{$\u$-cusped solid torus}. See the lefthand side of \Cref{fig:ucuspedproduct}. 
Define the \textbf{index} of $M$ to be $1-\frac{n}{2}$, which is the index of a meridional disk intersecting the $\u\u$-cusp curves minimally. 

A \textbf{$\u$-cusped torus shell} is defined similarly, but replacing $\Delta$ by a closed annulus whose boundary is smooth with the exception of $n\ge1$ cusps on a single boundary component. The annulus faces are labeled $\u$, the torus face is labeled $\b$, and circularity of the vector field is defined as for the $\u$-cusped solid torus.
Notice that there is no canonical meridian in a $\u$-cusped torus shell, so there is no canonical way to assign an index as in the non-punctured case. 

A \textbf{$\u$-cusped torus} refers to a $\u$-cusped solid torus or a $\u$-cusped torus shell.
\textbf{$\s$-cusped solid tori}, \textbf{$\s$-cusped torus shells}, and \textbf{$\s$-cusped tori} are defined symmetrically.
\end{definition}

\begin{definition}[$\u$-cusped product] \label{defn:cuspedproduct}
Suppose that $D$ is homeomorphic to $S\times[0,1]$, where $S$ be a compact surface with $\ind(S) \leq 0$. We say that $(D,V,\lambda)$ is a \textbf{$\u$-cusped product} if the following hold:
\begin{enumerate}[label=(\alph*)]
    \item $S\times\{0\}$ is a $\m$-face
    \item each component of $\del S\times[0,1]$ is a $\b$-face
    \item $S\times \{1\}$ is a union of $\p$-faces and $\u$-faces
    \item each orbit of the $V$-semiflow that does not accumulate on a $\u$-face either terminates on a $\p$-face or a $\u\u$-cusp
    \item each $\p$-face has nonpositive index
    \item $V$ is circular on each $\u$-face
    \item each $\u\u$-cusp circle is incoherent
\end{enumerate}

The definition of \textbf{s-cusped product} is symmetric.
See \Cref{fig:ucuspedproduct} for an example of a $\u$-cusped product. Note that because $\b\u$-edges are prohibited, the $\u$-faces in $S\times \{1\}$ must lie entirely in $\intr(S\times \{1\})$.
\end{definition}

\begin{remark}
In this paper the $\u$-cusped products we encounter will not have $\u\u$-cusp circles, but we include condition (g) to be consistent with \cite{Mos96} and with our future work.
\end{remark}

The following lemma is inspired by \cite[Remark preceding Prop 4.6.1]{Mos96}, and allows for the recognition of $\u$-cusped products. We record it here for later use.

\begin{lemma}[\cite{Mos96}]\label{lem:ucusprecog}
Let $(Q,V,\lambda)$ be a connected dynamic manifold such that
\begin{enumerate}
    \item $Q$ has no $\s$-faces,
    \item each backward trajectory not lying in a $\u$-face terminates on an $\m$-face,
    \item each $\b$-face is an annulus with boundary consisting of one $\p\b$-circle and one $\m\b$-circle, and
    \item there are no $\p\m$-edges.
\end{enumerate}
If each $\m$-face has nonpositive index, then $Q$ is homeomorphic to $S \times [0,1]$ for some compact surface $S$ with $\ind S \leq 0$ and satisfies items (a)-(d) in the definition of a $\u$-cusped product.
\end{lemma}

\begin{proof}
We first claim that $S$ has exactly one $\m$-face. Suppose points $p, q$ lie in $\m$-faces. Take a path $\alpha$ in the interior of $Q$ from $p$ to $q$. We can flow $\alpha$ backwards so that it lies on a $\m$-face by (2), so $p$ and $q$ lie in the same $\m$-face. 

Let us denote the unique $\m$-face of $Q$ by $S$. We now claim that $Q$ is homeomorphic to the product $S \times [0,1]$ with $S$ corresponding to $S \times \{0\}$. 
Pick a basepoint $x \in S$.
Consider the map $\pi_1(S,x) \to \pi_1(Q,x)$ induced by inclusion. This map is surjective because any closed curve in the interior of $Q$ can be flowed backward to a curve on $S$ by (2). Similarly, the map is injective because any nullhomotopy of a curve on $S$ in $Q$ can be homotoped off of the $\u$-faces of $N$, and then flowed backward to lie on $S$. Thus the claim follows from, for example, \cite[Theorem 10.2]{Hem76}.

This shows (a) in \Cref{defn:cuspedproduct}. (b) follows from (4), while (c) follows from (1) and (3).

Finally for (d), suppose some orbit neither accumulates on a $\u$-face nor terminates on a $\p$-face nor a $\u\u$-cusp. Then it must have some accumulation point $x$ in the interior of $Q$. But then the backward trajectory of $x$ cannot terminate on $S$, contradicting (2).
\end{proof}

\begin{figure}
    \centering
    \fontsize{12pt}{12pt}\selectfont
    \resizebox{!}{4cm}{
\begingroup%
  \makeatletter%
  \providecommand\color[2][]{%
    \errmessage{(Inkscape) Color is used for the text in Inkscape, but the package 'color.sty' is not loaded}%
    \renewcommand\color[2][]{}%
  }%
  \providecommand\transparent[1]{%
    \errmessage{(Inkscape) Transparency is used (non-zero) for the text in Inkscape, but the package 'transparent.sty' is not loaded}%
    \renewcommand\transparent[1]{}%
  }%
  \providecommand\rotatebox[2]{#2}%
  \newcommand*\fsize{\dimexpr\f@size pt\relax}%
  \newcommand*\lineheight[1]{\fontsize{\fsize}{#1\fsize}\selectfont}%
  \ifx\svgwidth\undefined%
    \setlength{\unitlength}{89.40741159bp}%
    \ifx\svgscale\undefined%
      \relax%
    \else%
      \setlength{\unitlength}{\unitlength * \real{\svgscale}}%
    \fi%
  \else%
    \setlength{\unitlength}{\svgwidth}%
  \fi%
  \global\let\svgwidth\undefined%
  \global\let\svgscale\undefined%
  \makeatother%
  \begin{picture}(1,1.21960269)%
    \lineheight{1}%
    \setlength\tabcolsep{0pt}%
    \put(0,0){\includegraphics[width=\unitlength,page=1]{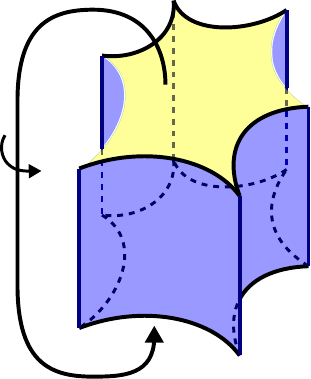}}%
    \put(0.47855637,0.41100565){\color[rgb]{0,0,0}\makebox(0,0)[lt]{\lineheight{1.25}\smash{\begin{tabular}[t]{l}$\u$\end{tabular}}}}%
    \put(0,0){\includegraphics[width=\unitlength,page=2]{fig_ucuspedtorus.pdf}}%
  \end{picture}%
\endgroup%
}
    \hspace{1.2cm}
    \resizebox{!}{4cm}{
\begingroup%
  \makeatletter%
  \providecommand\color[2][]{%
    \errmessage{(Inkscape) Color is used for the text in Inkscape, but the package 'color.sty' is not loaded}%
    \renewcommand\color[2][]{}%
  }%
  \providecommand\transparent[1]{%
    \errmessage{(Inkscape) Transparency is used (non-zero) for the text in Inkscape, but the package 'transparent.sty' is not loaded}%
    \renewcommand\transparent[1]{}%
  }%
  \providecommand\rotatebox[2]{#2}%
  \newcommand*\fsize{\dimexpr\f@size pt\relax}%
  \newcommand*\lineheight[1]{\fontsize{\fsize}{#1\fsize}\selectfont}%
  \ifx\svgwidth\undefined%
    \setlength{\unitlength}{163.95319948bp}%
    \ifx\svgscale\undefined%
      \relax%
    \else%
      \setlength{\unitlength}{\unitlength * \real{\svgscale}}%
    \fi%
  \else%
    \setlength{\unitlength}{\svgwidth}%
  \fi%
  \global\let\svgwidth\undefined%
  \global\let\svgscale\undefined%
  \makeatother%
  \begin{picture}(1,0.69614486)%
    \lineheight{1}%
    \setlength\tabcolsep{0pt}%
    \put(0,0){\includegraphics[width=\unitlength,page=1]{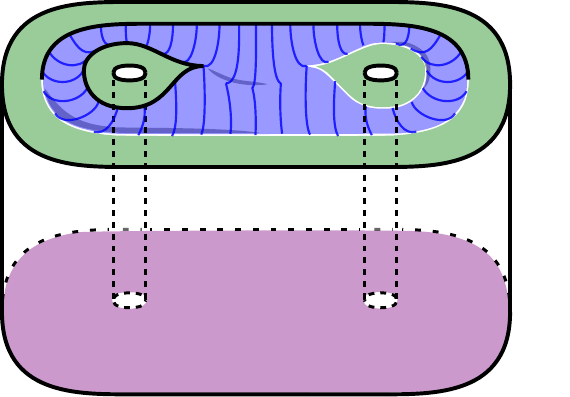}}%
    \put(0.44339292,0.14278422){\color[rgb]{0,0,0}\makebox(0,0)[lt]{\lineheight{1.25}\smash{\begin{tabular}[t]{l}$\m$\end{tabular}}}}%
    \put(0.02630673,0.32259403){\color[rgb]{0,0,0}\makebox(0,0)[lt]{\lineheight{1.25}\smash{\begin{tabular}[t]{l}$\b$\end{tabular}}}}%
    \put(0.44997575,0.57787352){\color[rgb]{0,0,0}\makebox(0,0)[lt]{\lineheight{1.25}\smash{\begin{tabular}[t]{l}$\u$\end{tabular}}}}%
    \put(0.83303163,0.53464293){\color[rgb]{0,0,0}\makebox(0,0)[lt]{\lineheight{1.25}\smash{\begin{tabular}[t]{l}$\p$\end{tabular}}}}%
    \put(0,0){\includegraphics[width=\unitlength,page=2]{fig_ucuspedproduct.pdf}}%
  \end{picture}%
\endgroup%
}
    \caption{On the left is a $\u$-cusped solid torus, and on the right a $\u$-cusped product.}
    \label{fig:ucuspedproduct}
\end{figure}

\subsection{Veering branched surfaces} \label{subsec:vbs}

In this subsection, we define veering branched surfaces, which will be the main class of objects we study in the rest of this paper.

\begin{definition}\label{def:veryfull}
An unstable dynamic branched surface $(B,V)$ in a sutured manifold $Q$ is said to be \textbf{very full} if the complementary regions of $B$ in $Q$ are all $\u$-cusped tori and $\u$-cusped products.
A very full {stable} dynamic branched surface is defined symmetrically.
\end{definition}

In particular, notice that a very full unstable dynamic branched surface $B$ does not meet $R_- \cup \gamma$ and a very full stable dynamic branched surface $B$ does not meet $R_+ \cup \gamma$.

Let $B$ be a very full unstable dynamic branched surface. Recall that $\brloc(B)$ is a union of branch loops and branch arcs. 
Suppose that we have chosen a source orientation (recall \Cref{defn:sourceorientation}) on each branch loop and each branch arc.
We say that this data specifies a \textbf{source orientation} of $\brloc(B)$.

\begin{definition} \label{defn:vbs}
A very full unstable dynamic branched surface $(B,V)$, along with the choice of a source orientation, is called an \textbf{unstable veering branched surface} if:
\begin{enumerate}
    \item For each triple point $p$ meeting branch segments $\gamma$ and $\delta$, the orientation of $\gamma$ points into the same side of $T_p\delta$ in $T_p B$ as the maw coorientation of $\delta$ does. See \Cref{fig:veeringcondition}.
    \item The \textbf{boundary train track} $\beta=B \cap R_+$ is efficient and has no large branches.
    \item For each branch loop $l$, let $c$ be the unique $\u\u$-cusp circle of $Q \cut B$ which is identified to $l$. We require that the orientation on $l$ agree with the dynamic orientations of both $\u$-faces adjacent to $c$.
    \item There are no annulus or Möbius band sectors with both boundary components having inward pointing maw coorientations.
    \item $B$ does not carry any tori or Klein bottles.
\end{enumerate}

By \Cref{lemma:nolargebranchmeansspiraling}, there exists a choice of source orientations on the boundary train track making it into a spiraling train track. In fact, a canonical choice exists here: The only freedom is how the circular branches are oriented. Each of these lie on a $\u$-face of $Q \cut B$, so we orient it according to the dynamic orientation of the $\u$-face. In the following we will implicitly assume that the boundary train track of a veering branched surface is endowed with this source orientation, thus making it a spiraling train track.
\end{definition}

\begin{figure}
    \centering
    \fontsize{8pt}{8pt}\selectfont
    \resizebox{!}{1.25in}{
\begingroup%
  \makeatletter%
  \providecommand\color[2][]{%
    \errmessage{(Inkscape) Color is used for the text in Inkscape, but the package 'color.sty' is not loaded}%
    \renewcommand\color[2][]{}%
  }%
  \providecommand\transparent[1]{%
    \errmessage{(Inkscape) Transparency is used (non-zero) for the text in Inkscape, but the package 'transparent.sty' is not loaded}%
    \renewcommand\transparent[1]{}%
  }%
  \providecommand\rotatebox[2]{#2}%
  \newcommand*\fsize{\dimexpr\f@size pt\relax}%
  \newcommand*\lineheight[1]{\fontsize{\fsize}{#1\fsize}\selectfont}%
  \ifx\svgwidth\undefined%
    \setlength{\unitlength}{81.232268bp}%
    \ifx\svgscale\undefined%
      \relax%
    \else%
      \setlength{\unitlength}{\unitlength * \real{\svgscale}}%
    \fi%
  \else%
    \setlength{\unitlength}{\svgwidth}%
  \fi%
  \global\let\svgwidth\undefined%
  \global\let\svgscale\undefined%
  \makeatother%
  \begin{picture}(1,0.74195012)%
    \lineheight{1}%
    \setlength\tabcolsep{0pt}%
    \put(0,0){\includegraphics[width=\unitlength,page=1]{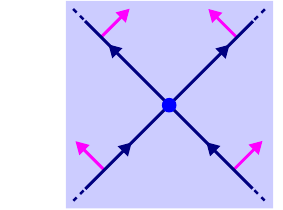}}%
    \put(0.57154853,0.27198214){\color[rgb]{0,0,0}\makebox(0,0)[lt]{\lineheight{1.25}\smash{\begin{tabular}[t]{l}$p$\end{tabular}}}}%
    \put(0.35189913,0.02625346){\color[rgb]{0,0,0}\makebox(0,0)[lt]{\lineheight{1.25}\smash{\begin{tabular}[t]{l}$T_p \gamma$\end{tabular}}}}%
    \put(0.68598916,0.02600054){\color[rgb]{0,0,0}\makebox(0,0)[lt]{\lineheight{1.25}\smash{\begin{tabular}[t]{l}$T_p \delta$\end{tabular}}}}%
    \put(-0.00559016,0.59257554){\color[rgb]{0,0,0}\makebox(0,0)[lt]{\lineheight{1.25}\smash{\begin{tabular}[t]{l}$T_p B$\end{tabular}}}}%
    \put(0,0){\includegraphics[width=\unitlength,page=2]{fig_veeringcondition.pdf}}%
  \end{picture}%
\endgroup%
}
    \caption{In a veering branched surface, the orientations (dark blue) and maw coorientations (pink) of branch segments are compatible at triple points.}
    \label{fig:veeringcondition}
\end{figure}

\begin{remark}
There is a symmetric definition for a stable veering branched surface. In this paper all the veering branched surfaces we encounter will be unstable, so we will sometimes omit the word `unstable' and refer to these as veering branched surfaces. Everything done in this paper can be performed for stable veering branched surfaces as well.
\end{remark}

Recall that sectors of a vertical branched surface are surfaces with corners. In particular, sectors of a veering branched surface are surfaces with corners, with sides along the branch locus and $R_+$. Moreover, the sides along the branch locus are naturally cooriented inwards or outwards by the vector field $V$ (equivalently, by the maw coorientation), and the sides along $R_+$ are naturally cooriented outwards by the vector field $V$. The sides of the sectors also inherit a source orientation from that of $\brloc(B)$ and $\beta$. Henceforth we will implicitly coorient and orient edges in this manner.

\begin{propdefn} \label{prop:vbssectors}
A sector $s$ of a veering branched surface on a sutured manifold must be one of the following:
\begin{itemize}
    \item A \textbf{diamond} that possibly has
    \begin{itemize}
        \item \textbf{Scalloped top}, i.e. the 2 top sides and the top vertex of the diamond can be replaced by $n \geq 3$ adjacent top sides and $n-1$ top corners, or
        \item \textbf{Rounded bottom}, i.e. the 2 bottom sides and the bottom vertex of the diamond can be replaced by 1 bottom side
    \end{itemize}
    The outermost top sides are oriented towards the top vertices, while the other top sides have sources. All the top sides are cooriented outwards. If the bottom is not rounded, the bottom sides are oriented away from the bottom vertex, otherwise the bottom side has a source. All the bottom sides are cooriented inwards. See first two columns of \Cref{fig:vbssectors}.
    \item An \textbf{annulus/Möbius band} that possibly has
    \begin{itemize}
        \item \textbf{Scalloped boundary components}, i.e. each of the boundary components can be replaced by $n \geq 1$ sides and $n$ corners.
    \end{itemize}
    
    In this case, $s$ lies on a $\u$-face $F$ of $Q \cut B$, and we have an injection $H_1(F) \hookrightarrow H_1(s)$. We call the orientation on $H_1(s)$ induced by the dynamic orientation of $F$ the \textbf{dynamic orientation} on $s$. Here $F$ may not be unique but the dynamic orientation on $s$ is well-defined as the orientation of a closed orbit on $s$ if $s$ is a source, and by \Cref{defn:vbs}(3) if $s$ is transient.

    The orientation of a non-scalloped boundary component agrees with the dynamic orientation of $s$. Non-scalloped boundary components can be cooriented inwards or outwards. Each side on a scalloped boundary component has a source. Scalloped boundary components must be cooriented outwards.
    
    We say $s$ is a \textbf{source} sector if all of $\partial s$ is cooriented outwards. See \Cref{fig:vbssectors} third column. Otherwise by \Cref{defn:vbs}(4), $s$ must be an annulus with one inward and one outward boundary component, and we say $s$ is a \textbf{transient} annulus. See \Cref{fig:vbssectors} fourth column.
\end{itemize}
\end{propdefn}

\begin{figure}
    \centering
    \fontsize{12pt}{12pt}\selectfont
    \resizebox{!}{8cm}{
\begingroup%
  \makeatletter%
  \providecommand\color[2][]{%
    \errmessage{(Inkscape) Color is used for the text in Inkscape, but the package 'color.sty' is not loaded}%
    \renewcommand\color[2][]{}%
  }%
  \providecommand\transparent[1]{%
    \errmessage{(Inkscape) Transparency is used (non-zero) for the text in Inkscape, but the package 'transparent.sty' is not loaded}%
    \renewcommand\transparent[1]{}%
  }%
  \providecommand\rotatebox[2]{#2}%
  \newcommand*\fsize{\dimexpr\f@size pt\relax}%
  \newcommand*\lineheight[1]{\fontsize{\fsize}{#1\fsize}\selectfont}%
  \ifx\svgwidth\undefined%
    \setlength{\unitlength}{354.43452083bp}%
    \ifx\svgscale\undefined%
      \relax%
    \else%
      \setlength{\unitlength}{\unitlength * \real{\svgscale}}%
    \fi%
  \else%
    \setlength{\unitlength}{\svgwidth}%
  \fi%
  \global\let\svgwidth\undefined%
  \global\let\svgscale\undefined%
  \makeatother%
  \begin{picture}(1,0.55702136)%
    \lineheight{1}%
    \setlength\tabcolsep{0pt}%
    \put(0,0){\includegraphics[width=\unitlength,page=1]{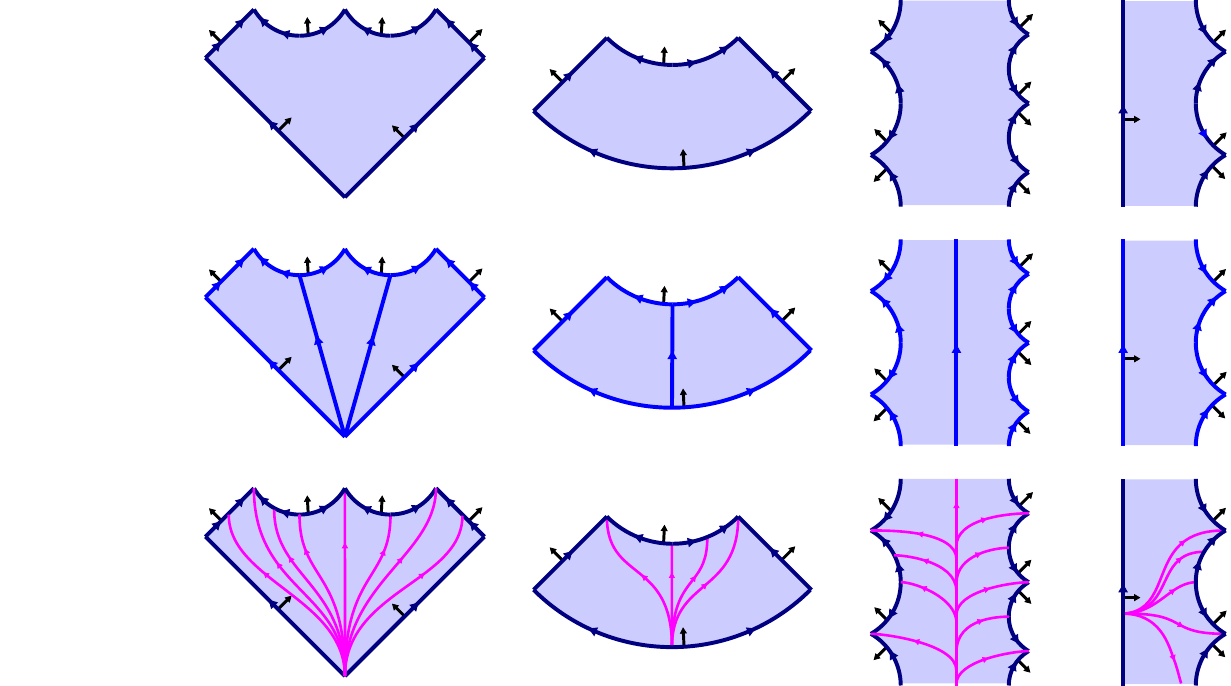}}%
    \put(-0.00221524,0.47669544){\color[rgb]{0,0,0}\makebox(0,0)[lt]{\lineheight{1.25}\smash{\begin{tabular}[t]{l}Sectors\end{tabular}}}}%
    \put(-0.00156894,0.30611322){\color[rgb]{0,0,1}\makebox(0,0)[lt]{\lineheight{1.25}\smash{\begin{tabular}[t]{l}Dual\\graph $\Gamma$\end{tabular}}}}%
    \put(-0.00156894,0.11341665){\color[rgb]{1,0,1}\makebox(0,0)[lt]{\lineheight{1.25}\smash{\begin{tabular}[t]{l}Semiflow \\graph $\Phi_+$\end{tabular}}}}%
    \put(0,0){\includegraphics[width=\unitlength,page=2]{vbssectors.pdf}}%
  \end{picture}%
\endgroup%
}
    \caption{The possible sectors of a veering branched surface and the portion of the dual graph $\Gamma$ and semiflow graph $\Phi_+$ on each of them. Column 1: A diamond with scalloped top. Column 2: A diamond with scalloped top and rounded bottom. Column 3: A source sector. Column 4: A transient annulus.}
    \label{fig:vbssectors}
\end{figure}

\begin{proof}
Since $V$ is nonsingular, by the Poincare-Hopf theorem, $\chi(s)=\mathrm{degree}(V|_{\partial s})$.
In particular, since $V$ is transverse to $\brloc(B)$ and all corners are convex, $s$ must be a disk, an annulus, or a Möbius band.

If $s$ is a disk, $\ind(V|_{\partial s})=1$, so $s$ must have some number of consecutive sides cooriented outwards followed by some number of consecutive sides cooriented inwards. But given (1) in \Cref{defn:vbs}, a side with no source must be adjacent to both an inwardly and outwardly cooriented edge; and a side with a source must be adjacent to two outwardly cooriented edges.

If $s$ is an annulus or a Möbius band, $\ind(V|_{c})=0$ for each boundary component $c$ of $s$, so each boundary component must consist of a number of consecutive edges all cooriented outwards or all cooriented inwards. But, as above, the two edges to the sides of an edge with no source cannot be cooriented in the same direction, and the two edges to the sides of an edge with a source must be cooriented outwards. 
\end{proof}

\begin{remark}
In \Cref{defn:vbs}, the requirements that $\beta$ has no large branches and (4) are not actually very restrictive, in the sense that if $B$ is a very full dynamic branched surface satisfying all the axioms except these, then they can be arranged to hold. For suppose $B$ is such a branched surface. Then a similar proof as in \Cref{prop:vbssectors} shows that a sector $s$ of $B$ must be one of the following:
\begin{itemize}
    \item A diamond that possibly has
    \begin{itemize}
        \item Scalloped top, or
        \item Rounded top, i.e. the 2 top sides and the top vertex of the diamond can be replaced by 1 top side lying on $R_+$, or
        \item Rounded bottom
    \end{itemize}
    If the top is rounded, the top side has a source. Otherwise the same statements as in \Cref{prop:vbssectors} still hold.
    \item An annulus/Möbius band that possibly has
    \begin{itemize}
        \item Scalloped boundary components
    \end{itemize}
    There are now no restrictions on the coorientations of the boundary components. The same statements about their orientations as in \Cref{prop:vbssectors} still hold.
\end{itemize}

To arrange for $\beta$ to have no large branches, we can dynamically split $B$ to get rid of all diamonds with rounded top (see \Cref{fig:splittingforgoodness} in \Cref{sec:dynamicpairs}). This might change the corner structure of some cusped product pieces, but they remain cusped product pieces, hence $B$ is still very full. The result of the splitting may contain new undesirable sectors, so we split again to get rid of those inductively. Since the splittings reduce the number of sectors, the process stops eventually, and $\beta$ now has no large branches.

Similarly, we can do dynamic splittings to arrange for (4) to hold. Suppose there is an annulus/Möbius band sector $s$ with both boundary components cooriented inwards. For a fixed boundary component $c$ of $s$, locate an immersed annulus $A$ with one boundary component on $c$ and another boundary component $c'$ on $\brloc(B) \cup \beta$ that is cooriented out of $A$, by, for example, following along the boundary of a complementary region of $B$ meeting $s$, using the fact that $B$ is very full. We can then dynamically split along $A$ minus a small neighborhood of $c'$ to reduce the number of annulus/Möbius band sectors that violate (4).
\end{remark}

\subsection{The dual graph and flow graph} \label{subsec:dualgraphflowgraph}

In this section, we define the dual graph and flow graphs of a veering branched surface. These definitions are motivated by the corresponding objects in the non-sutured setting (see e.g. \cite{LMT23a}). Within the context of this paper, the dual graph will play a part in \Cref{sec:folcone} and \Cref{sec:hmvbsunique}, while the main use of flow graphs is to prove the main result of \Cref{sec:dynamicpairs} that veering branched surfaces give rise to dynamic pairs.

\begin{definition}[Dual graph] \label{defn:dualgraph}
Let $B$ be a veering branched surface. The \textbf{dual graph} $\Gamma$ of $B$ is a directed graph embedded in $B$ determined by requiring its intersection with each sector $s$ of $B$ be as follows:
\begin{itemize}
    \item If $s$ is a diamond, then $\Gamma \cap s$ is the union of $\partial s$ with an edge from the bottom vertex (if the bottom is not rounded) or the bottom source (if the bottom is rounded) of $s$ to each source on a top side.
    \item If $s$ is a source, then $\Gamma \cap s$ is the union of $\partial s$ with an edge connecting a vertex in the interior of $s$ to itself forming a core of $s$, oriented by the dynamic orientation of $s$.
    \item If $s$ is a transient annulus, then $\Gamma \cap s = \partial s$ (with a vertex placed on each cycle) \qedhere
\end{itemize}
\end{definition}
See \Cref{fig:vbssectors} for an illustration of the dual graph.

\begin{definition} \label{defn:postransbrloc}
Let $c$ be an oriented path immersed in $B$. We say that $c$ is \textbf{positively transverse} to $\brloc(B)$ if:
\begin{itemize}
    \item $c$ intersects $\brloc(B)$ nontrivially, and at every intersection point, the orientation of $c$ agrees with the maw coorientation on $\brloc(B)$, or 
    \item $c$ is contained in an annulus or Möbius band sector, and is homotopic within the sector to a positive multiple of the core oriented by the dynamic orientation.\qedhere
\end{itemize}
\end{definition}

\begin{proposition} \label{prop:dualgraphcarries}
Any closed curve $c$ on a veering branched surface $B$ that is positively transverse to $\brloc(B)$ is homotopic to a directed cycle of the dual graph $\Gamma$.
\end{proposition}
\begin{proof}
If $c$ does not intersect $\brloc(B)$ then the proposition is clear from definition, so we can assume that $c$ intersects $\brloc(B)$.
We claim that $c$ does not intersect any annulus or Möbius band sectors in this case. Indeed, $c$ cannot enter a source sector. If $c$ intersects a transient annulus $s$, then it must enter $s$ at some point $x$ on a non-scalloped boundary component of $s$. But $c$ must lie on a transient annulus immediately before $x$ as well, so repeating this argument we eventually get a union of annuli that glue up to form a torus or Klein bottle carried by $B$ (in fact a torus because the dynamic orientations agree), contradicting \Cref{defn:vbs}(5).

Now by a small perturbation, we can assume $c$ does not meet $\brloc(B)$ in a triple point or source. We analyze the form of $c$ within each sector $s$ that it meets. By the paragraph above, $s$ must be a diamond. $c$ must enter though a bottom side and exit through a top side. For each intersection point between $c$ and $\brloc(B)$, we homotope $c$ in a neighborhood to push the intersection point against the orientation on $\brloc(B)$ until it hits a triple point or a source. After this homotopy, $c$ then enters each sector through the triple point or source in the interior of the union of bottom sides and exits through a triple point or source on the top sides. We can then clearly homotope $c$ sector-by-sector to lie on $\Gamma$ such that its orientation is compatible with that of $\Gamma$.
\end{proof}

\begin{definition}[The (semi-)flow graph] \label{defn:flowgraph}
Let $B$ be a veering branched surface. Construct an oriented train track embedded in $B$ in the following way.
First, for each component of $\brloc(B)$ with no triple points, pick a point lying on it. Similarly, for each component of $B \cap R_+$ with no switches, pick a point lying on it.
\begin{itemize}
    \item For each diamond $s$, take a union of disjoint branches going from the bottom vertex (if the bottom is not rounded) or the bottom source (if the bottom is rounded) to each corner, triple point, and source contained in the interior of the union of top sides.
    \item For each source sector $s$, take the core of $s$ oriented by the dynamic orientation. Then for each boundary component $t$ of $s$:
    \begin{itemize}
        \item If $t$ is a smooth circle with some triple points, attach to the core a branch from the core to each triple point on $t$.
        \item If $t$ is a smooth circle with no triple points, attach to the core a branch from the core to the chosen vertex on $t$.
        \item If $t$ has corners, attach to the core a branch from the core to each corner, triple point and source on $t$.
    \end{itemize}
    \item For each transient annulus $s$ with one boundary component $t_1$ cooriented inwards and the other boundary component $t_2$ cooriented outwards:
    \begin{itemize}
        \item If $t_2$ is a smooth circle with some triple points, take a union of disjoint branches going from the chosen vertex on $t_1$ to each triple point on $t_2$.
        \item If $t_2$ is a smooth circle with no triple points, take a branch going from the chosen vertex on $t_1$ to the chosen vertex on $t_2$.
        \item If $t_2$ has corners, take a union of disjoint branches going from the chosen vertex on $t_1$ to each corner, triple point and source on $t_2$.
    \end{itemize}
\end{itemize}

Notice that such an oriented train track is not uniquely defined. There is freedom in choosing the points on components of $\brloc(B)$ with no triple points and components of $B \cap R_+$ with no switches, and in the case when $s$ is a source, there is freedom in choosing how the branches connect the vertices on $\partial s$ to the core of $s$.

However, these operations do not meaningfully change the information contained by the oriented train track, so by a slight abuse of language we will refer to an oriented track defined by the above description as \emph{the} \textbf{semiflow graph} $\Phi_+$ of $B$. See \Cref{fig:vbssectors}.

Now consider the set 
$$\{x \in \Phi_+:\text{every directed train route starting at } x \text{ ends on } R_+ \} $$ 
Let $\Phi$ be the result of removing this set from $\Phi_+$, which is still an oriented train track embedded in $B$. We call $\Phi$ the \textbf{flow graph}. As with the semiflow graph, the flow graph is not uniquely defined, but the ambiguity is inconsequential.
\end{definition}

In \Cref{sec:dynamicpairs}, we will show that the flow graph is a ``dynamic train track," which will imply that one can construct a dynamic pair starting with only a veering branched surface. See \Cref{sec:dynamicpairs} for definitions of the terms. Even without this motivation, a flow graph is a natural companion object for the dual graph, as we will see in the next subsection.

\subsection{Dynamic planes} \label{subsec:dynamicplanes}

In this section we will generalize some discussion from \cite{LMT23a}. The authors of that paper introduced ``dynamic planes," which are combinatorial objects associated to a veering triangulation that correspond to different leaves of the stable/unstable foliations of pseudo-Anosov flows. Here we will define dynamic planes for veering branched surfaces that correspond to leaves of the suspension of the  Handel-Miller laminations.

Our first order of business will be to show that a veering branched surface $B$ fully carries a lamination $\mathcal{L}$ for which every leaf is $\pi_1$-injective. This is so that we can use the leaves of $\mathcal{L}$ to understand our dynamic planes. 

We remark that in the setting of \cite{LMT23a}, one starts with a pseudo-Anosov flow and gets a veering branched surface, which one can check is \emph{laminar} and hence carries an essential lamination by \cite{Li02}.
The leaves of such a lamination must be $\pi_1$-injective, so this preliminary step was automatically true in that paper. 

\begin{proposition} \label{prop:vfulldbslaminar}
A veering branched surface $B$ on a sutured manifold $Q$ carries a lamination $\mathcal{L}$ with no spherical leaves and such that every leaf is $\pi_1$-injective.
\end{proposition}
\begin{proof}
If $B$ carries a surface $F$, then we can construct a nonsingular vector field on $F$ by pulling back the vector field $V$ on $B$. In particular if $F$ is closed then it must have zero Euler characteristic. This shows that if $B$ carries a lamination $\mathcal{L}$ then no leaf of $\mathcal{L}$ can be a sphere.

Now, to construct $\mathcal{L}$ we use the tool of \textbf{laminar branched surfaces}, introduced in \cite{Li02}. Recall that a branched surface $B$ in a compact 3-manifold $M$ where $B \cap \partial M = \varnothing$ is laminar if:
\begin{enumerate}
    \item $\partial_h N(B)$ is incompressible in $M \cut N(B)$, no component of $\partial_h N(B)$ is a sphere, and $Q \cut N(B)$ is irreducible.
    \item There are no \textit{monogons} in $M \cut N(B)$.
    \item $B$ does not carry a torus that bounds a solid torus.
    \item $B$ has no \textit{trivial bubbles}.
    \item $B$ has no \textit{sink disks}.
\end{enumerate}
We refer to \cite{Li02} for the definitions of the italicized terms. By \cite[Theorem 1]{Li02}, a laminar branched surface fully carries an \textbf{essential lamination} $\mathcal{L}$. We will not go into the full definition of an essential lamination here. We only need the property that every leaf of an essential lamination is $\pi_1$-injective.

Returning to our setting, let $DQ$ be the double of $Q$ over $R_+$ and $R_-$, and let $DB \subset DQ$ be the double of $B$. We claim that $DB$ is laminar.

The complementary regions of $N(DB)$ in $DQ$ are the doubles of those of $N(B)$ in $Q$ over their $\p$ and $\m$ faces. In particular, it is straightforward to check that no component of $\partial_h N(DB)$ is a sphere, $Q \cut N(DB)$ is irreducible, and $B$ has no trivial bubbles.

To show that $\partial_h N(DB)$ is incompressible in $DQ \cut N(DB)$, assume that there is a compressing disk. Then by an innermost disk argument
using the incompressibility of $R_\pm$ and the irreducibility of $Q$, there is a compressing disk or boundary compressing disk for $\partial_h N(B)$ in $Q \cut N(B)$. But it is straightforward to check that these do not exist. Similarly, one can show that there are no monogons in $DQ \cut N(DB)$. 

If $DB$ carries a torus that bounds a solid torus, then by the same standard arguments, now using the efficiency of the boundary train track $B \cap R_+$ as well, either $B$ carries a torus bounding a solid torus $T$ or $B$ carries a properly embedded annulus that bounds a solid torus $T$ with another annulus on $R_+$. The former case is ruled out by (5) in \Cref{defn:vbs}. In the latter case, every complementary region of $B$ in $Q$ inside $T$ must be a cusped torus piece, but then this implies that $T$ has no $\p$-face, giving a contradiction.

Finally, that $DB$ has no sink disks follows from \Cref{prop:vbssectors}.

Now apply \cite[Theorem 1]{Li02} to get an essential lamination $\mathcal{L}'$ fully carried by $DB$, then restrict it to a lamination $\mathcal{L}$ fully carried by $B$. To show that every leaf $\ell$ of $\mathcal{L}$ is $\pi_1$-injective, it suffices to show that if $\ell'$ is the leaf of $\mathcal{L}'$ that contains $\ell$, then $\pi_1 (\ell) \to \pi_1 (\ell')$ is injective. Assume otherwise, then there is a disk $D$ in $\ell'$ intersecting $R_+$ in a union of circles. We take an innermost such circle which bounds a disk $D'$ in $D$. Then $D'$ is carried by $B$. However, the vector field on $B$ induces one on $D'$ which points outwards along $\partial D'$, contradicting the Poincare-Hopf theorem.
\end{proof}

\begin{remark}
The above gives a proof of a version of Li's result from \cite{Li02} in the setting of sutured manifolds. Note that Li himself generalized the result to include torally bounded manifolds in \cite{Li03}.
\end{remark}

Let $B$ be a veering branched surface on a sutured manifold $Q$ and let $\widetilde{B}$ be the lift of $B$ to the universal cover $\widetilde{Q}$. Let $\widetilde{\Gamma}$ be the lift of the dual graph $\Gamma$ to $\widetilde{B}$.

\begin{definition}
Let $\sigma$ be a sector of $\brloc(\widetilde{B})$. The \textbf{future set of $\sigma$}, which we denote as $\nabla(\sigma)$, is the union of sectors $s$ of $\widetilde{B}$ for which there is a path from $\sigma$ to $s$ that points in the same direction as the maw vector field whenever it intersects $\brloc(\widetilde{B})$. 

Let $x$ be a point of $\brloc(\widetilde{B})$. Let $\sigma(x)$ be the sector of $B$ that meets $x$ and at which the coorientation is pointing inwards. The \textbf{future set of $x$} is defined to be $\nabla(\sigma(x))$.

Let $\gamma$ be a bi-infinite directed edge path of $\widetilde{\Gamma}$ that is not a subset of $\partial \widetilde{B}$. We denote by $D(\gamma)$ the union of sectors $s$ of $\widetilde{B}$ for which there is a path from $\gamma$ to $s$ that is positively transverse to $\brloc(\widetilde{B})$.
If $\gamma$ does not eventually lie along a single component of $\brloc(\widetilde{B})$ in the backward direction, then we call $D(\gamma)$ the \textbf{dynamic plane} associated to $\gamma$.
Otherwise we call $D(\gamma)$ the \textbf{dynamic half-plane} associated to $\gamma$.
See \Cref{fig:dynplane} for a depiction of a portion of a dynamic place tiled by sectors of $\wt B$.
\end{definition}

\begin{proposition}\label{prop:dynamicplanes}
\begin{enumerate}
    \item Each future set $\nabla(\sigma)$ is a surface with corners, with interior homeomorphic to a plane and boundary along $\partial \widetilde{B}$ and $\brloc(\widetilde{B})$. Furthermore,
    \begin{itemize}
        \item If $\sigma$ is a diamond, then $\partial \nabla(\sigma) \cap \brloc(\widetilde{B})$ is precisely the rays along the components of $\brloc(\widetilde{B})$ starting from the bottom vertex (if the bottom is not rounded) or the bottom source (if the bottom is rounded) of $\sigma$.
        \item If $\sigma$ is the lift of an annulus/Möbius band, then $\partial \nabla(\sigma) \cap \brloc(\widetilde{B})$ is precisely the components of $\brloc(\widetilde{B})$ containing the sides of $\sigma$ that are cooriented inwards (if any).
    \end{itemize}
    \item Each dynamic plane $D(\gamma)$ is a surface with boundary, with interior homeomorphic to a plane and boundary along $\partial \widetilde{B}$.
    \item Each dynamic half-plane $D(\gamma)$ is a surface with boundary, with interior homeomorphic to a plane and boundary along $\partial \widetilde{B}$ and one component of $\brloc(\widetilde{B})$.
\end{enumerate}
\end{proposition}
\begin{proof}
This follows from the discussion in \cite[\S 3]{LMT23a}, which can almost be applied word for word here. We outline the idea, emphasizing the slight modifications that one has to perform in our setting.

By \Cref{prop:vfulldbslaminar}, $B$ fully carries a lamination $\mathcal{L}$ with no spherical leaves and for which every leaf is $\pi_1$-injective. Lifting this to $\widetilde{Q}$, $\widetilde{B}$ fully carries $\widetilde{\mathcal{L}}$, for which each leaf is a surface with boundary with interior homeomorphic to a plane. For each leaf $\ell$ of $\widetilde{\mathcal{L}}$, we will abuse notation and use the same name for the homeomorphic image under the collapsing map $\widetilde{\mathcal{L}} \subset N(\widetilde{B}) \to \widetilde{B}$. As such, $\ell$ inherits a decomposition into surfaces with corners that are among the sectors of $\widetilde{B}$. 

Now for a sector $\sigma$, let $\ell$ be a leaf of $\widetilde{\mathcal{L}}$ containing $\sigma$. Then $\nabla(\sigma)$ is naturally a subset of $\ell$. The boundary of $\nabla(\sigma)$ in $\ell$ will be $\partial \nabla(\sigma) \cap \brloc(\widetilde{B})$. One shows that $\nabla(\sigma)$ coincides with the region bounded by the claimed form of $\partial \nabla(\sigma) \cap \brloc(\widetilde{B})$ using the argument in \cite[Lemma 3.1]{LMT23a}.

To prove (2) and (3), notice that every directed path $\alpha$ in $\widetilde{\Gamma}$ starting at some point $x$ can be homotoped to be positively transverse to $\brloc(\widetilde{B})$ rel $x$ by pushing it along the flow on $\widetilde{B}$ slightly. Hence given a bi-infinite directed edge path $\gamma$, $D(\gamma)=\cup \nabla(x_i)$ for a sequence of points $x_i$ on $\gamma$ converging to the negative end.

We have $\nabla(x_i) \subset \intr \nabla(x_{i+1})$ unless $\gamma$ stays within a constant component $c$ of $\brloc(\widetilde{B})$ between $x_i$ and $x_{i+1}$. In the latter case, we at least have $\nabla(x_i) \backslash c \subset \intr \nabla(x_{i+1}) \backslash c$. Hence taking the union, we see that $\bigcup \nabla(x_i)$ is always a plane in its interior, and has boundary entirely along $\partial \widetilde{B}$, unless the component $c$ above eventually stays constant, in which case the boundary of $\bigcup \nabla(x_i)$ will have one side along that $c$. 

Note that the latter case occurs exactly when $\gamma$ eventually lies along a constant component of $\brloc(\widetilde{B})$ in the backwards direction, which is precisely the case when $D(\gamma)$ is a dynamic half-plane.
\end{proof}

Let $\widetilde{\Phi_+}$ be the lift of the semiflow graph $\Phi_+$ to $\widetilde{Q}$. Each dynamic plane $D(\gamma)$ inherits the restriction of $\widetilde{\Phi_+}$. This will be an oriented train track on $D(\gamma)$ with only diverging switches. In particular, every point has a infinite, unique backward trajectory.

\begin{definition} [Triangles, rectangles, tongues, following]\label{defn:dynamicplanepieces}
Let $\Phi_+$ be the semiflow graph. Notice that by construction, for each sector $s$ of $B$, the components of $s \cut \Phi_+$ are of the following forms:
\begin{itemize}
    \item \textbf{Triangles}, i.e. triangles with two sides on $\brloc(B)$ or $R_+$, one cooriented outwards and one cooriented inwards, and the remaining side on $\Phi_+$.
    \item \textbf{Rectangles}, i.e. rectangles with two opposite sides on $\brloc(B)$ or $R_+$, one cooriented outwards and one cooriented inwards, and the remaining two opposite sides on $\Phi_+$.
    \item \textbf{Tongues}, i.e. one-cusped triangles with the side opposite to the cusp on $\brloc(B)$ or $R_+$ cooriented outwards, and the remaining two sides on $\Phi_+$.
\end{itemize}
If $s$ is a diamond, then the components of $s \cut \Phi_+$ are triangles and tongues; if $s$ is an annulus/Möbius band, then the components of $s \cut \Phi_+$ are rectangles and tongues. See \Cref{fig:vbssectors}.

Let $t_1$ and $t_2$ be two such components. We say that $t_1$ is \textbf{followed by} $t_2$ if $t_1$ is adjacent to $t_2$ along an edge of $\brloc(B)$ which is cooriented from $t_1$ to $t_2$.
\end{definition}

\begin{lemma} \label{lemma:dynamicplanepieces}
We have the following two properties.
\begin{itemize}
    \item A triangle cannot be followed by a rectangle.
    \item There cannot be an infinite sequence of rectangles following one another.
\end{itemize}
\end{lemma}
\begin{proof}
Suppose that we have a triangle followed by a rectangle. Then we would have a diamond sector $s_1$ and an annulus/Möbius band sector $s_2$ that are adjacent along an edge $e$ that is cooriented from $s_1$ to $s_2$. This implies that the component of $\partial s_2$ on which $e$ lies is cooriented inwards, hence meets no triple points of $B$. This contradicts the fact that $s_1$ is a diamond.

Now suppose that there is an infinite sequence of rectangles following one another. Then we would have an infinite sequence of transient annulus sectors each adjacent to the previous one along a circular edge. This produces a torus that is carried by $B$, contradicting \Cref{defn:vbs}(5).
\end{proof}

\begin{definition} \label{defn:abstrip}
An \textbf{AB strip} on a dynamic plane $D(\gamma)$ is a region with boundary along $\widetilde{\Phi_+}$ which is homeomorphic to $[0,1] \times \mathbb{R}$ and made out of a bi-infinite sequence of triangles following one another. An \textbf{AB half-strip} on a dynamic plane $D(\gamma)$ is a region with boundary along $\widetilde{\Phi_+}$ which is homeomorphic to $[0,1] \times (-\infty,0]$ with the infinite end being in the direction of backwards flow along $\widetilde{\Phi_+}$, and made out of an infinite sequence of triangles following one another. See \Cref{fig:dynplane} for an example.
\end{definition}

\begin{lemma} \label{lemma:dynamicplaneflowgraph}
Two infinite backwards $\widetilde{\Phi_+}$ rays lying on a dynamic plane $D(\gamma)$ either eventually coincide or both eventually lie on the boundaries of AB half-strips.
\end{lemma}
\begin{proof}
This follows nearly word for word from the proof of \cite[Lemma 3.7]{LMT23a}. As in the proof of \Cref{prop:dynamicplanes}, we outline the idea here and point out small modifications for our setting.

Recall that $D(\gamma)=\bigcup \nabla(x_i)$ for a sequence $x_i$ converging to the negative end of $\gamma$. Let $\alpha_1, \alpha_2$ be two infinite backwards $\widetilde{\Phi_+}$ rays lying on $D(\gamma)$. By inspection of the complementary regions of $\widetilde{\Phi_+}$ in sectors, the distance between $\alpha_k \cap \partial \nabla(x_i)$ on $\partial \nabla(x_i)$ is larger or equal to that between $\alpha_k \cap \partial \nabla(x_j)$ on $\partial \nabla(x_j)$ for $i<j$, and equality holds if and only if the region bounded by the portion of $\alpha_k$ in between $\partial \nabla(x_i)$ and $\partial \nabla(x_j)$ is made out of triangles and rectangles. 

But by \Cref{lemma:dynamicplanepieces}, a triangle cannot be followed by a rectangle and there cannot have an infinite sequence of rectangles following one another. Hence we conclude that if $\alpha_k$ do not eventually coincide, they will eventually bound some AB half-strips.
\end{proof}

\begin{figure}
    \centering
    \fontsize{12pt}{12pt}\selectfont
    \resizebox{!}{6cm}{
\begingroup%
  \makeatletter%
  \providecommand\color[2][]{%
    \errmessage{(Inkscape) Color is used for the text in Inkscape, but the package 'color.sty' is not loaded}%
    \renewcommand\color[2][]{}%
  }%
  \providecommand\transparent[1]{%
    \errmessage{(Inkscape) Transparency is used (non-zero) for the text in Inkscape, but the package 'transparent.sty' is not loaded}%
    \renewcommand\transparent[1]{}%
  }%
  \providecommand\rotatebox[2]{#2}%
  \newcommand*\fsize{\dimexpr\f@size pt\relax}%
  \newcommand*\lineheight[1]{\fontsize{\fsize}{#1\fsize}\selectfont}%
  \ifx\svgwidth\undefined%
    \setlength{\unitlength}{260.85656905bp}%
    \ifx\svgscale\undefined%
      \relax%
    \else%
      \setlength{\unitlength}{\unitlength * \real{\svgscale}}%
    \fi%
  \else%
    \setlength{\unitlength}{\svgwidth}%
  \fi%
  \global\let\svgwidth\undefined%
  \global\let\svgscale\undefined%
  \makeatother%
  \begin{picture}(1,0.60576067)%
    \lineheight{1}%
    \setlength\tabcolsep{0pt}%
    \put(0,0){\includegraphics[width=\unitlength,page=1]{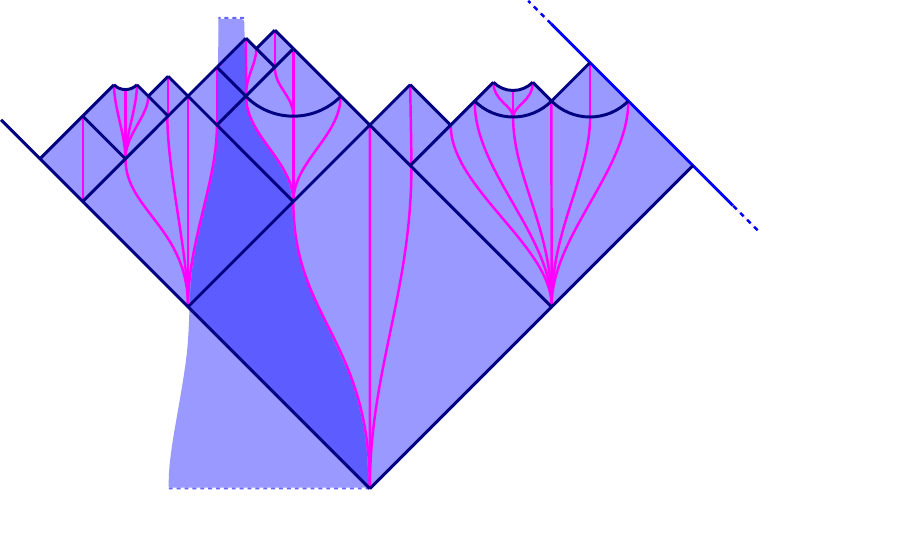}}%
    \put(0.8457306,0.35775743){\color[rgb]{0,0,1}\makebox(0,0)[lt]{\lineheight{1.25}\smash{\begin{tabular}[t]{l}$R_+$\end{tabular}}}}%
    \put(0.21226314,0.00925437){\color[rgb]{0,0,0}\makebox(0,0)[lt]{\lineheight{1.25}\smash{\begin{tabular}[t]{l}AB strip\end{tabular}}}}%
  \end{picture}%
\endgroup%
}
    \caption{A portion of a dynamic plane containing an AB strip. We remark that the entire dynamic plane could intersect $R_+$ in infinitely many components.}
    \label{fig:dynplane}
\end{figure}

In particular, just as in \cite[Proposition 3.10]{LMT23a}, this lemma easily implies that all AB half-strips in a dynamic plane are adjacent.

\begin{proposition} \label{prop:flowgraphdualgraph}
Let $B$ be a veering branched surface, let $\Gamma$ be its dual graph, and let $\Phi$ be its flow graph. If $c$ is a cycle of $\Gamma$, then $c$ or $c^2$ is homotopic to a cycle of $\Phi$ in $Q$.
\end{proposition}

\begin{proof}
Let $\widetilde{c}$ be the lift of $c$ to an infinite directed path in $\widetilde{\Gamma}$. 

Suppose first that $D(\widetilde{c})$ is a dynamic plane. Similarly to above the proposition follows essentially from the proof of \cite[Proposition 3.15]{LMT23a}, and we outline the idea here.

Let $g=[c] \in \pi_1(Q)$. Notice that the dynamic plane $D(\widetilde{c})$ is $g$-invariant. If $D(\widetilde{c})$ does not contain an AB strip, then pick some infinite backward trajectory $\alpha$ of $\widetilde{\Phi_+}$ on $D(\widetilde{c})$. From \Cref{lemma:dynamicplaneflowgraph}, we know that $\alpha$ and $g \cdot \alpha$ eventually coincide, hence $\alpha$ is eventually $g$-periodic, and the periodic part projects to a cycle of $\Phi_+$ with homotopy class $g$.

On the other hand, if $D(\widetilde{c})$ contains AB strips, then the projection of the boundary components of an AB strip is a cycle of $\Phi_+$ with homotopy class $g$ or $g^2$ (depending on whether $g$ reverses the orientation of the AB strip).

Now if $D(\widetilde{c})$ is a dynamic half-plane, then $c$ must be a component of $\brloc(B)$. Hence $c$ must lie on an annulus face of some complementary region of $B$. The restriction of $\Phi_+$ to this annulus face is a (nonempty) oriented train track with only diverging switches and with branches leaving through the boundary. Hence it must consist of some cycles and some branches from the cycles to the boundary. The cycles will have homotopy class $g$ or $g^2$.

Finally, notice that by definition a cycle of $\Phi_+$ must lie in $\Phi$.
\end{proof}

\begin{corollary}
The cone spanned by cycles of $\Gamma$ in $H_1(Q)$ agrees with that of $\Phi$.
\end{corollary}
\begin{proof}
\Cref{prop:flowgraphdualgraph} implies that the former cone is a subset of the latter cone. For the converse, observe that each cycle of $\Phi$ is positively transverse to $\brloc(B)$ in the sense of \Cref{defn:postransbrloc}, hence can be homotoped to a cycle of $\Gamma$ by \Cref{prop:dualgraphcarries}.
\end{proof}

To conclude this section, we illustrate the utility of dynamic planes by discussing the correspondence between periodic dynamic planes and periodic leaves. This will be applied in \Cref{sec:folcone}.

Suppose $B$ is a veering branched surface on $Q$ fully carrying a lamination $\mathcal{L}$ as in \Cref{prop:vfulldbslaminar}. Lift $B$ and $\mathcal{L}$ to the universal cover $\widetilde{Q}$. A dynamic plane of $\widetilde{B}$ is said to be \textbf{periodic} if it is invariant under some element of $\pi_1(Q)$ acting on $\widetilde{Q}$. Similarly, a leaf of $\widetilde{\mathcal{L}}$ is said to be \textbf{periodic} if it is invariant under some element of $\pi_1(Q)$; equivalently, a leaf of $\widetilde{\mathcal{L}}$ is said to be periodic if its image in $\mathcal{L}$ is a leaf that has interior homeomorphic to an open annulus or Möbius band.

Given a periodic leaf $\ell$ of $\widetilde{\mathcal{L}}$, consider the union of sectors that the leaf passes through. This is a surface with boundary along $\partial \widetilde{B}$ and with interior homeomorphic to a plane, which we abuse notation and call $\ell$ as well. If $\ell$ does not contains the lift of an annulus or Möbius band sector, then since each vertex of $\Gamma$ has at least one incoming edge on $\ell$, one can construct a bi-infinite $\widetilde{\Gamma}$-path $\gamma$ on $\ell$ that does not lie entirely on $\partial \widetilde{B}$.
If $\gamma$ can be chosen such that $D(\gamma)$ is a dynamic plane, then $\ell=D(\gamma)$. If not, then it must be the case that $\ell$ contains the lift of an annulus or Möbius band sector. In this case $\ell$ contains the lift of a source sector, for otherwise $B$ will carry a closed surface. We pick $\gamma$ to be the lift of the core of such a sector, then $\ell=D(\gamma)$ in this case as well. 

Conversely, given a periodic dynamic plane $D$, we can write it as a nested union $\bigcup_{i=0}^\infty \nabla(\sigma_i)$. The sets of leaves of $\widetilde{\mathcal{L}}$ passing through $\sigma_i$ gives a nested sequence of compact sets in the space of leaves of $\wt \LL$.
Taking the intersection we get a nonempty collection of leaves of $\widetilde{\mathcal{L}}$ passing through exactly the sectors that tile the dynamic plane. For general $\mathcal{L}$, this determines a packet of leaves, but when $\mathcal{L}$ is an unstable Handel-Miller lamination, then such a leaf must be unique by \Cref{lemma:hmlamnoproductregions}, hence periodic.

This discussion implies the following proposition. 

\begin{proposition} \label{prop:dynamicplanesandhmleaves}
Suppose $\mathcal{L}^u$ is an unstable Handel-Miller lamination carried by a veering branched surface $B$. Then there is a one-to-one correspondence between the periodic leaves of $\widetilde{\mathcal{L}^u}$ and the periodic dynamic planes of $B$. Moreover, this correspondence is such that a leaf of $\widetilde{\mathcal{L}^u}$ is invariant under $g \in \pi_1(\overline{M_f})$ if and only if its corresponding dynamic plane is invariant under $g$.
\end{proposition}

\section{Veering branched surfaces for Handel-Miller laminations} \label{sec:hmvbs}

In this section we construct veering branched surfaces carrying unstable Handel-Miller laminations in compactified mapping tori. The construction itself, contained in \Cref{sec:bsconstruct}, essentially consists of suspending a splitting sequence of train tracks carrying the positive Handel-Miller lamination, the existence of which is guaranteed by \Cref{thm:sequenceexist}. Checking for the axioms of a veering branched surface requires some in-depth analysis of the splitting sequence, which we perform in \Cref{subsec:sectorannulusMobius} and \Cref{subsec:3dprincipalregion}, allowing us to prove our main existence result \Cref{thm:depthonetovbs} in \Cref{subsec:depthonetovbs}. Finally in \Cref{subsec:hmvbsunique} we promote the uniqueness result of \Cref{subsec:splitsequnique} regarding splitting sequences to a uniqueness result about veering branched surfaces carrying Handel-Miller laminations.

\subsection{Staircases}
It is now convenient to consider an enlargement of the category of sutured manifolds, which by our convention have convex corner points, to include \textbf{concave corner points}, which by definition have neighborhoods modeled on the following closed set in $\R^3$:

\begin{itemize}
    \item (Concave) corner edge: $\{(x,y,z)\mid \text{$x \le 0$ or $y \le 0$}\}$.
\end{itemize}

A \textbf{sutured manifold with concavity} is defined in the same way as a sutured manifold, but now we allow the existence of annular sutures with both boundary components meeting the positive (or negative) tangential boundary, and require that for any such annular suture $A$, one component of $\del A$ consists of concave corner points and the other consists of convex corner points.

A \textbf{positive staircase} is a sutured manifold with concavity $(P,\gamma)$ satisfying the following:
\begin{enumerate}[label=(\roman*)]
    \item $P$ is homeomorphic to $S\times[0,1]$ for some compact oriented surface $S$
    \item $R_+(P)=S\times\{1\}$
    \item $T(\gamma)=\varnothing$ and $A(\gamma)=(\del S\times[0,1]) \cup A_1\cup\cdots\cup A_n$ for a collection of annuli $A_1,\dots, A_n\subset \intr(S)\times\{0\}$
    \item $R_-(\gamma)=(S\times\{0\})-\bigcup_1^n A_i$.
\end{enumerate}
A \textbf{negative staircase} is defined symmetrically by switching the roles of $R_+$ and $R_-$. See \Cref{fig:staircase}.

\begin{figure}
    \centering
    \fontsize{8pt}{8pt}\selectfont
    \resizebox{!}{2in}{
\begingroup%
  \makeatletter%
  \providecommand\color[2][]{%
    \errmessage{(Inkscape) Color is used for the text in Inkscape, but the package 'color.sty' is not loaded}%
    \renewcommand\color[2][]{}%
  }%
  \providecommand\transparent[1]{%
    \errmessage{(Inkscape) Transparency is used (non-zero) for the text in Inkscape, but the package 'transparent.sty' is not loaded}%
    \renewcommand\transparent[1]{}%
  }%
  \providecommand\rotatebox[2]{#2}%
  \newcommand*\fsize{\dimexpr\f@size pt\relax}%
  \newcommand*\lineheight[1]{\fontsize{\fsize}{#1\fsize}\selectfont}%
  \ifx\svgwidth\undefined%
    \setlength{\unitlength}{286.27442211bp}%
    \ifx\svgscale\undefined%
      \relax%
    \else%
      \setlength{\unitlength}{\unitlength * \real{\svgscale}}%
    \fi%
  \else%
    \setlength{\unitlength}{\svgwidth}%
  \fi%
  \global\let\svgwidth\undefined%
  \global\let\svgscale\undefined%
  \makeatother%
  \begin{picture}(1,0.44381695)%
    \lineheight{1}%
    \setlength\tabcolsep{0pt}%
    \put(0,0){\includegraphics[width=\unitlength,page=1]{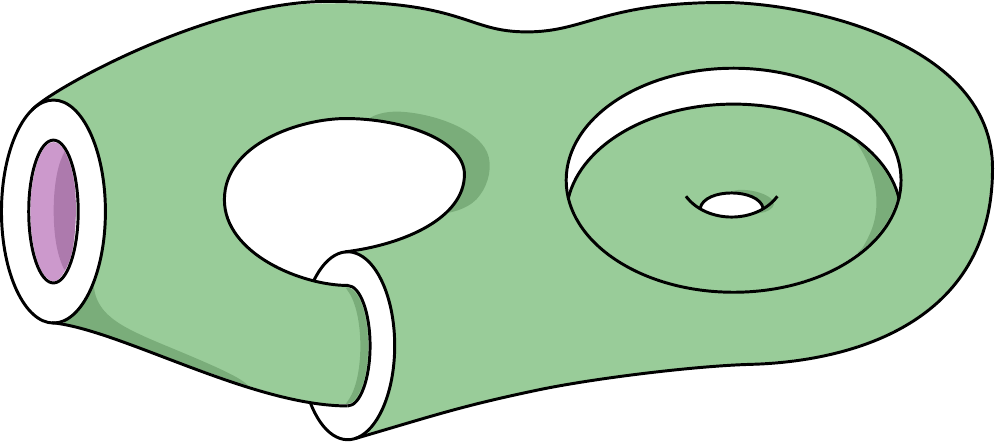}}%
    \put(0.04139211,0.31333837){\color[rgb]{0,0,0}\makebox(0,0)[lt]{\lineheight{1.25}\smash{\begin{tabular}[t]{l}$\gamma$\end{tabular}}}}%
    \put(0.2082717,0.38104774){\color[rgb]{0,0,0}\makebox(0,0)[lt]{\lineheight{1.25}\smash{\begin{tabular}[t]{l}$R_+$\end{tabular}}}}%
    \put(0.03504882,0.22953491){\color[rgb]{0,0,0}\makebox(0,0)[lt]{\lineheight{1.25}\smash{\begin{tabular}[t]{l}$R_-$\end{tabular}}}}%
    \put(0.33819968,0.16391949){\color[rgb]{0,0,0}\makebox(0,0)[lt]{\lineheight{1.25}\smash{\begin{tabular}[t]{l}$\gamma$\end{tabular}}}}%
    \put(0.65880349,0.33774314){\color[rgb]{0,0,0}\makebox(0,0)[lt]{\lineheight{1.25}\smash{\begin{tabular}[t]{l}$\gamma$\end{tabular}}}}%
  \end{picture}%
\endgroup%
}
    \caption{A negative staircase. By switching $R_+$ and $R_-$ we obtain a positive staircase.}
    \label{fig:staircase}
\end{figure}

\subsection{Construction of $(B^u,V)$} \label{sec:bsconstruct}

Let $f\colon L\to L$ be a Handel-Miller map, and let $(Q,\gamma)$ be the sutured manifold $\ol M_f$ endowed with the suspension semiflow $\phi_f$. We view $L$ as sitting inside $Q$. Let $\mc F$ be the depth one foliation of $(Q,\gamma)$ whose noncompact leaves are parallel to $L$.
Let $\mc L^u$ be the suspension of the Handel-Miller lamination $\Lambda_+$.

\begin{definition}[Staircase determined by a tiled neighborhood]
Let $E_+$ be a tiled neighborhood of a positive end-cycle of $L$ corresponding to a component $Y$ of $R_+(Q)$. The union of all forward $\phi_f$-orbits starting in $E_+$ naturally has the structure of a positive staircase $S$ with $R_+(S)=Y$. We call this the \textbf{staircase neighborhood} of $Y$ \textbf{determined by} the tiled neighborhood $E_+$.
\end{definition}

In \Cref{thm:sequenceexist} we showed that there exists a finite splitting sequence from $\tau$ to $f(\tau)$, for a certain train track $\tau$ fully carrying the lamination $\Lambda_+$. Reversing this sequence, we obtain a sequence
\[
f(\tau)=\tau_0\to \tau_1\to \cdots \to \tau_n=\tau,
\]
where $\tau_{i+1}$ is obtained from $\tau_i$ by a single fold.

We first assume that $n>0$.
As $Q$ is the compactified mapping torus of $f$, we have $Q \backslash \partial_\pm Q = (L\times[0,1])/((x,1)\sim (f(x),0))$. Let $L_i=L\times\{\frac{i}{n}\}\subset Q$ for $i=0,\dots, n-1$.

For each end-cycle $Z$, choose a tiling of $\ms U_Z$. 
Let $N(\ms E_+)$ be a tiled neighborhood of all positive end-cycles which is small enough so that $\tau_0$ is $f$-endperiodic in $N(\ms E_+)$ 
and the folding sequence from $\tau_0$ to $\tau_n$ is supported in $K=L\cut N(\ms E_+)$. Let $N(R_+)$ be the staircase neighborhood of $R_+$ determined by $N(\ms E_+)$. Note that $K$ is not compact---it is the union of a core of $L$ with neighborhoods of all negative ends.

Let $M_i=K\times [\frac{i}{n},\frac{i+1}{n}]\subset Q$ and $K_i=K\times\{\frac{i}{n}\}$. Thus $M_i$ is a region in $Q$ lying between $L_i$ and $L_{i+1}$, and $K_i\subset L_i$. Let
\[
A=R_-\cup \left(\bigcup_{i=0}^{n-1}M_i\right)=Q\cut N(R_+).
\]

Under our identification of $L$ with $L_0=L\times\{0\}$, there is a natural embedding of $\tau_0$ in $L_0$. 
Given an embedding of $\tau_i$ in $K_i$, we fix an embedding of $\tau_{i+1}$ in $K_{i+1}$ as follows: we flow $\tau_i$ forward into $K_{i+1}$ under $\phi_f$, and then perform the fold $\tau_i\to \tau_{i+1}$. 

Now for $i=0,\dots, n-2$ (if $n=1$ we skip this step) we define a dynamic branched surface $(M_i,B_i,V_i)$ with the following properties:

\begin{enumerate}[label=(\roman*)]
    \item $B_i$ intersects $K_i$ in $\tau_i$ and $K_{i+1}$ in $\tau_{i+1}$.
    \item Between the parts of $\tau_i$ and $\tau_{i+1}$ involved in the fold, there is a piece of $B_0$ that is modeled on the center or right of \Cref{fig:ttsplitting}, according to whether the fold is the reverse of a collision or not, respectively.
    \item Away from the pieces of $B_i$  described in (ii), $B_i$ is topologically a product.
    \item $\brloc(B_i)$ has a source orientation which is positively transverse to $\mc F$ at oriented points and tangent to $\mc F$ at sources.  If $\tau_i\to \tau_{i+1}$ is the reverse of a collision there is a single source of $\brloc(B_i)$(see \Cref{fig:ttsplitting}, center) and otherwise there are no sources.
    \item $V_i$ has unique forward trajectories, points forward along $\brloc(B_i)$, and is positively transverse to $\mc F|_{M_i}$. All points in $M_i-B_i$ have unique backward trajectories.
\end{enumerate}

\begin{figure}
    \centering
    \fontsize{6pt}{6pt}\selectfont
    \resizebox{!}{2.1in}{
\begingroup%
  \makeatletter%
  \providecommand\color[2][]{%
    \errmessage{(Inkscape) Color is used for the text in Inkscape, but the package 'color.sty' is not loaded}%
    \renewcommand\color[2][]{}%
  }%
  \providecommand\transparent[1]{%
    \errmessage{(Inkscape) Transparency is used (non-zero) for the text in Inkscape, but the package 'transparent.sty' is not loaded}%
    \renewcommand\transparent[1]{}%
  }%
  \providecommand\rotatebox[2]{#2}%
  \newcommand*\fsize{\dimexpr\f@size pt\relax}%
  \newcommand*\lineheight[1]{\fontsize{\fsize}{#1\fsize}\selectfont}%
  \ifx\svgwidth\undefined%
    \setlength{\unitlength}{392.438385bp}%
    \ifx\svgscale\undefined%
      \relax%
    \else%
      \setlength{\unitlength}{\unitlength * \real{\svgscale}}%
    \fi%
  \else%
    \setlength{\unitlength}{\svgwidth}%
  \fi%
  \global\let\svgwidth\undefined%
  \global\let\svgscale\undefined%
  \makeatother%
  \begin{picture}(1,0.36218587)%
    \lineheight{1}%
    \setlength\tabcolsep{0pt}%
    \put(0,0){\includegraphics[width=\unitlength,page=1]{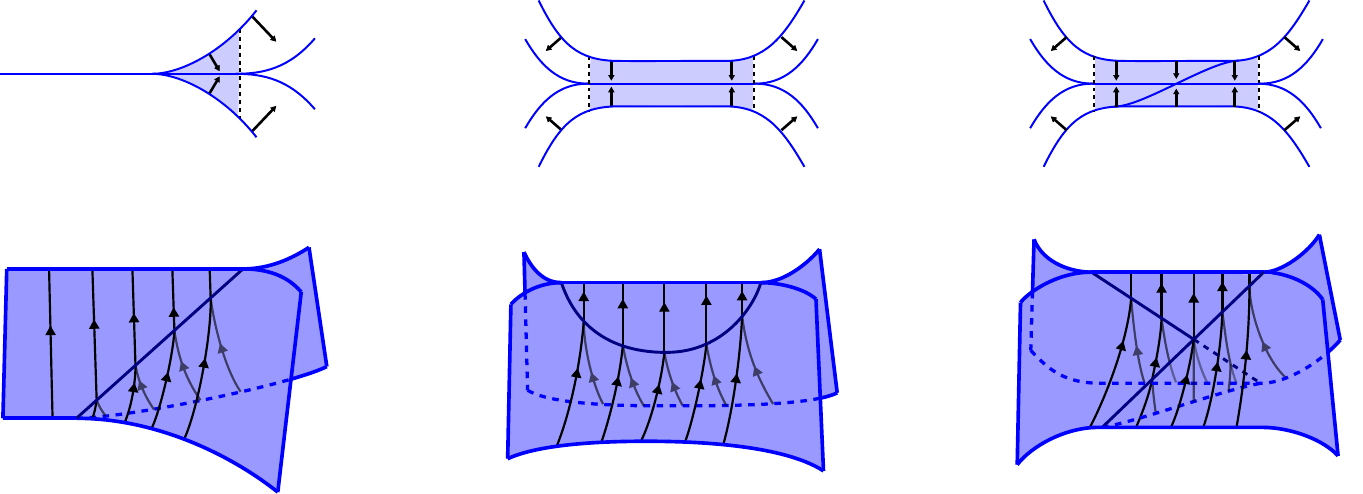}}%
    \put(0.18791837,0.25269982){\color[rgb]{0,0,0}\makebox(0,0)[lt]{\lineheight{1.25}\smash{\begin{tabular}[t]{l}$t=0$\end{tabular}}}}%
    \put(0.23071146,0.27348016){\color[rgb]{0,0,0}\makebox(0,0)[lt]{\lineheight{1.25}\smash{\begin{tabular}[t]{l}$t=1$\end{tabular}}}}%
    \put(0.58998542,0.23131583){\color[rgb]{0,0,0}\makebox(0,0)[lt]{\lineheight{1.25}\smash{\begin{tabular}[t]{l}$t=0$\end{tabular}}}}%
    \put(0.59975194,0.25965414){\color[rgb]{0,0,0}\makebox(0,0)[lt]{\lineheight{1.25}\smash{\begin{tabular}[t]{l}$t=1$\end{tabular}}}}%
    \put(0.95911792,0.23131646){\color[rgb]{0,0,0}\makebox(0,0)[lt]{\lineheight{1.25}\smash{\begin{tabular}[t]{l}$t=0$\end{tabular}}}}%
    \put(0.96888445,0.25965477){\color[rgb]{0,0,0}\makebox(0,0)[lt]{\lineheight{1.25}\smash{\begin{tabular}[t]{l}$t=1$\end{tabular}}}}%
  \end{picture}%
\endgroup%
}
    \caption{Building a dynamic branched surface from a train track folding sequence. The shaded regions in the top row indicate the sets in $K_i$ whose forward trajectories under $V_i$ merge with others (i.e. they enter $B_i$.)}
    \label{fig:ttsplitting}
\end{figure}

Between switches of $\tau_i$ and $\tau_{i+1}$ which are not involved in the fold $\tau_i\to \tau_{i+1}$, condition (iii) requires that $B_i$ be a product and (v) requires that $V$ point forward along $\brloc(B_i)$. This can be achieved by using the model shown on the left side of \Cref{fig:ttsplitting}.

We next describe $(M_{n-1}, B_{n-1}, V_{n-1})$. This is built the same way as for $i=1,\dots, n-2$, but with the additional requirement that the top of $(M_{n-1}, B_{n-1}, V_{n-1})$ line up smoothly with the bottom of $(M_0, B_0, V_0)$ on the overlap of $K_0$ and $K_n=f(K_0)$.

Let $B_A=\bigcup_{i=0}^{n-1} B_i$, and $V_A$ be the vector field on $A$ which restricts to $V_i$ on each $M_i$.

Finally, we define a vector field on $N(R_+)$. 
Let $B_N$ be the union of all forward $\phi_f$-orbits starting in $B_A\cap \del N(R_+)$. Away from $B_N$, let $V_N$ be equal to the vector field generating $\phi_f$; near $B_N$ we choose $V_N$ so that each branch line of $(B_N, V_N)$ has a neighborhood modeled on the left side of \Cref{fig:ttsplitting}.

Let $B^u=B_A\cup B_N$, and let $V$ be the vector field on $Q$ restricting to $V_A$ on $A$ and $V_N$ on $N$. Then evidently $(Q, B^u, V)$ is an unstable dynamic branched surface. Furthermore, because $B^u$ was constructed from an $f$-periodic train track folding sequence for a train track fully carrying $\Lambda_+$, we see that $B^u$ fully carries $\mc L^u$.

Now we treat the case when $n=0$. We construct a neighborhood of the positive end-cycles $N(\ms E_+)$ just as before, and similarly we construct the sets $N(R_+)$, $K=L\cut  N(R_+)$, $N(\ms E_+)$ and $A=Q\cut N(R_+)$. The construction works just as above, the only difference being that the branched surface we build in $A$ has no triple points or sources in its branch locus.

Hence we have shown the following, which is suggested by Mosher as an intermediate step in \cite{Mos96}:

\begin{proposition} \label{prop:buexist}
Let $h\colon L\to L$ be a Handel-Miller endperiodic map. There exists an unstable dynamic branched surface $(B^u,V)$ in $\ol M_f$ fully carrying $\LL^u$, the suspension of the Handel-Miller lamination $\Lambda_+$.
\end{proposition}

\subsection{Principal and nonprincipal regions of $(B^u,V)$} \label{subsec:3dprincipalregion}

We let $\phi_V$ denote the forward semiflow of $V$. Note that orbits of $\phi_V$ are uniquely determined in the forward direction but may not be in the backward direction.

Recall that there are two types of complementary region of $\Lambda_+$ in $L$: principal regions and components of the negative escaping set $\ms U_-$. Let $C_\tau$ be a complementary region of $\tau_0$. Then $C_\tau$ corresponds naturally to a complementary region $C_\Lambda$ of $\Lambda_+$. 
If $C_\Lambda$ is a principal region, we call $C_\tau$ a \textbf{principal region of $\tau_0$}. Since $\tau_0$ is efficient and fully carries $\Lambda_+$, there is only one principal region of $\tau_0$ corresponding to a given principal region of $\Lambda_+$.
In the case that $C_\tau$ is a principal region homeomorphic to an annulus with one boundary component equal to a component of $\del L$, we say that $C_\tau$ is a \textbf{peripheral principal region}.
If $C_\Lambda$ is not a principal region (and hence is a component of $\ms U_-$), then we say $C_\tau$ is a \textbf{nonprincipal region of $\tau$}.

If $U$ is a complementary region of $B^u$, then $U\cap L_0$ consists of either entirely principal regions or entirely nonprincipal regions of $\tau_0$. As such we will speak of \textbf{(peripheral) principal} and \textbf{nonprincipal} regions of $B^u$. 

The following two lemmas describe the vector fields induced on principal and nonprincipal regions.

\begin{lemma}[$V$ in principal regions]\label{lem:principalregioncircular}
Let $U$ be a principal region of $B^u$, and let $V_U$ be the vector field on $U$ induced by $V$. Then $V_U$ is circular. As a consequence, if $U$ is a solid torus or is peripheral, then $U$ is a $\u$-cusped torus or $\u$-cusped torus shell respectively.
\end{lemma}
\begin{proof}
If $U$ is a principal region of $B^u$, then the foliation of $U$ induced by $\mc F$ defines a map to $S^1$. Since $V$ is positively transverse to $\mc F$, we see that the induced vector field on $U$ is circular. The last claim follows immediately.
\end{proof}

By Handel-Miller theory, backward orbits from points in nonprincipal regions of $\LL^u$ terminate on $R_-(\ol M_f)$. We will need the corresponding fact for $(B^u, V)$.

\begin{lemma}[$V$ in nonprincipal regions]\label{lem:nonprincipal regionescaping}
Let $U$ be a nonprincipal region of $B^u$, and let $V_U$ be the vector field on $U$ induced by $V$. 
The branched surface in \Cref{prop:buexist} can be constructed so that the backward trajectory from each point in $U$ not lying in a $\u$-face ends on $R_-(\ol M_f)$.
\end{lemma}

\begin{proof}
In this proof we refer to each complementary region of a standard neighborhod of $B^u$ as ``principal" or ``nonprincipal" according to whether it is contained in a principal or nonprincipal region of $B^u$, respectively.

The hard part of the proof is the construction of a standard neighborhood $\mc N(B^u)$ of $B^u$ with the property that the Handel-Miller suspension flow points inward along each component of $\del \mc N(B^u)$ meeting a nonprincipal region. 

In this construction we make use of the fact that $f$ preserves geodesic juncture components (see \Cref{thm:HMrep}).

Recall from our construction of $B^u$ in \Cref{prop:buexist} that $N$ is a staircase neighborhood of $R_+(\ol M_f)$ and $A=\ol M_f\cut N$.

Let $U$ be a nonprincipal region of $B^u$, and let $U_\LL$ be the associated nonprincipal region of $\LL^u$.
Let $F$ be a $\u$-face of $U$. Then $F$ corresponds to a leaf $\ell$ of $\LL^u$ which borders $U_\LL$ and is carried by $F$. Let $\lambda$ be one component of $\ell\cap L$, which is evidently a semi-isolated leaf of $\Lambda_+$ (i.e.  leaves of $\Lambda_+$ do not accumulate on $\lambda$ from the side corresponding to $U_\LL$). By \cite[Theorem 6.5]{CCF19}, $\lambda$ is periodic with some period $p>0$. Since $U$ is a nonprincipal region, \Cref{lem:juncaccum} gives that there exists a negative juncture component $j$ such that the sequence $j, f^p(j), f^{2p}(j),\dots$, accumulates on $\lambda$ monotonically from the side corresponding to $U$. Let $\lambda_K$ be the component of $\lambda\cap K$ containing a periodic point; we are evidently free to assume that there is a component $j_0$ of $j\cap K$ such that $j_0$ is as close as we like to $\lambda_K$.

Let $H(F)'$ be obtained by flowing $j_0$ around $\ol M_f$ $p$ times. Since $f^p(j_0)$ is closer to $\lambda$ than $j$ is, $H(F)'$ can be homotoped slightly so that it is an annulus $H(F)$ transverse to $\phi_f$, and $\phi_f$ points through $H(F)$ toward $\ell$. See \Cref{fig:horizontalbdy}.
The letter $H$ indicates $H(F)$ will correspond to a component of the horizontal boundary of the standard neighborhood we are building.

\begin{figure}
    \centering
    \fontsize{8pt}{8pt}\selectfont
    \resizebox{!}{1.8in}{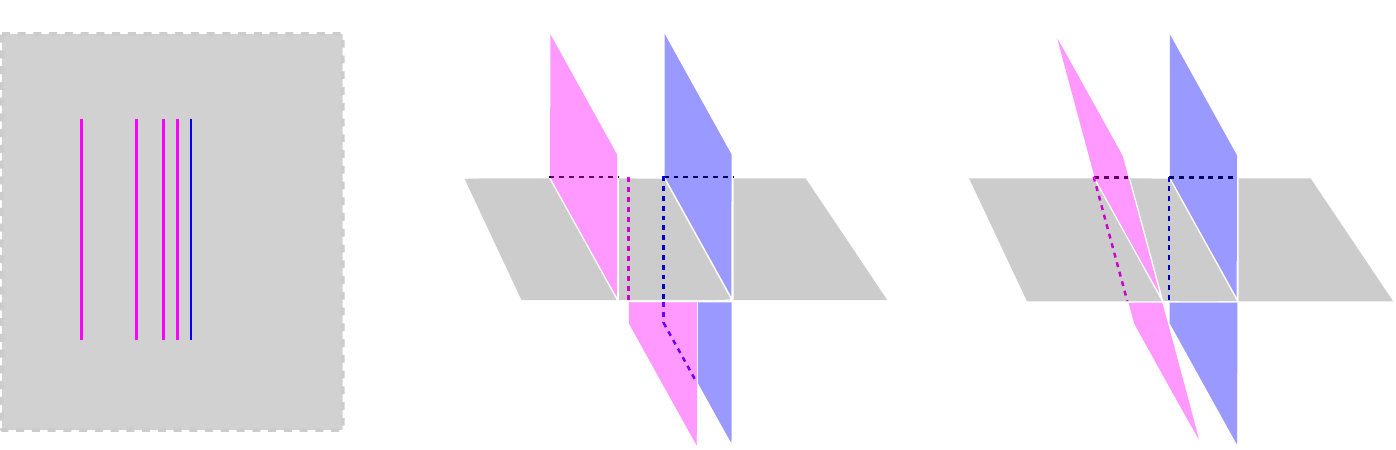}
    \caption{An illustration of the construction of $H(F)$ in the proof of \Cref{lem:nonprincipal regionescaping}. In the center and right pictures, the flow $\phi_f$ is vertical.}
    \label{fig:horizontalbdy}
\end{figure}

Next, let $c$ be a $\u\u$-cusp curve of $U$. Since $U$ is nonprincipal, $c$ is an interval and not a circle (see \Cref{lem:collision}). There are two leaves of $\LL^u$, say $\ell_1$ and $\ell_2$, which correspond to the faces of $U$ adjacent to $c$. Let $V(c)$ be a rectangle embedded in $A$ which has one edge along $\ell_1$, an opposite edge along $\ell_2$, and the remaining two edges along $\del A$. Further, choose $V(c)$ to stay close to $c$. See \Cref{fig:verticalbdy}.

\begin{figure}
    \centering
    \resizebox{!}{2.5in}{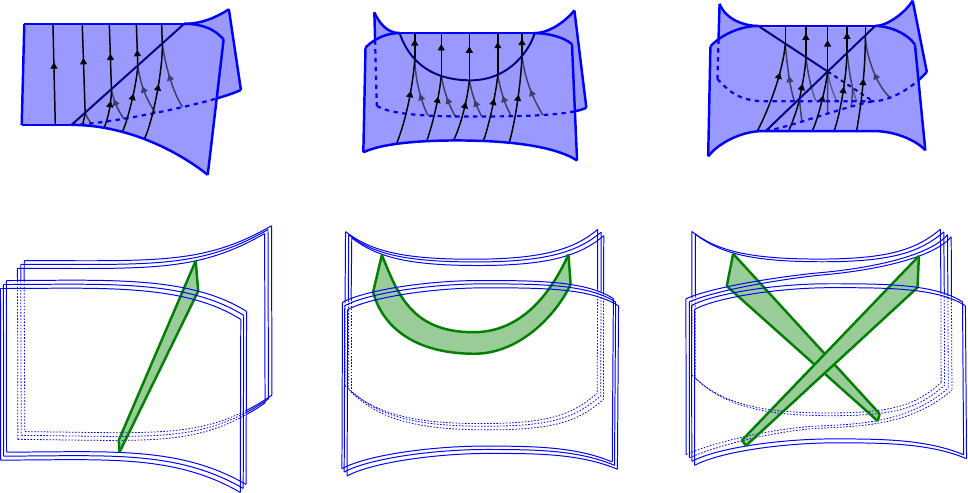}
    \caption{Locally constructing $V(c)$ for different cusp curves $c$. In the bottom row, the flow $\phi_f$ is vertical.}
    \label{fig:verticalbdy}
\end{figure}

Let $\wh U_A$ be equal to the component of $(U\cap A)\cut \big((\bigcup V(c))\cup(\bigcup H(F)\big)$ containing $U\cap R_-(\ol M_f)$, where the unions are taken over all $\u\u$-cusp curves $c$ and $\u$-faces $F$ of $U$.

For each principal region $P$ of $B^u$, let $\wh P$ be equal to $P$ minus some (any) standard neighborhood of $B^u$.

Form $A\cut ((\bigcup \wh P)\cup (\bigcup \wh U_A))$, where the unions are taken over all principal regions and all nonprincipal regions. This is a neighborhood of $B^u|_A$ which we can make into a standard neighborhood by smoothing out corners and adding cusps near the nonprincipal cusp curves of $B^u$. We call the resulting standard neighborhood $\mc N_A(B^u)$. See \Cref{fig:intothemaw}.

\begin{figure}
    \centering
    \resizebox{!}{2.5in}{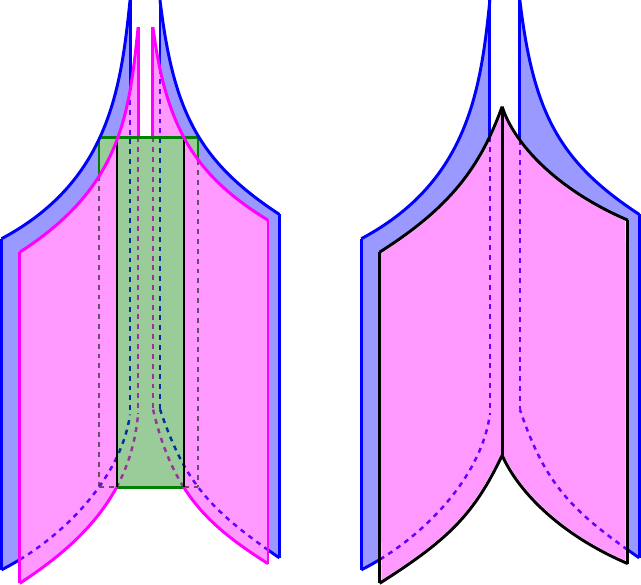}
    \caption{The view from a nonprincipal region $U$ of $B^u$, looking at two leaves of $\LL^u$ (blue) coming together near a cusp curve of $B^u$ ($B^u$ not pictured). On the left, the darker pink sheets are parts of $H(F)$ for the corresponding faces of $B^u$, and the green rectangle is part of $V(c)$ for the cusp curve $c$. The boundary of $\wh U_A$ is formed by surgering and perturbing the pieces as shown.}
    \label{fig:intothemaw}
\end{figure}

Now we extend this neighborhood to $N$. Let $\mc N_N(B^u)$ be a standard neighborhood of $\LL^u|_N$ that lines up smoothly with $\mc N_A(B^u)$, with the additional property that $\phi_f$ points into $\mc N_N(B^u)$ at every point of $\del \mc N_N(B^u)$. Such a neighborhood exists because the flow is quite simple on $N$: each flow line is an interval. See \Cref{fig:nbhdinN}.

\begin{figure}
    \centering
    \fontsize{10pt}{10pt}\selectfont
    \resizebox{!}{2in}{
\begingroup%
  \makeatletter%
  \providecommand\color[2][]{%
    \errmessage{(Inkscape) Color is used for the text in Inkscape, but the package 'color.sty' is not loaded}%
    \renewcommand\color[2][]{}%
  }%
  \providecommand\transparent[1]{%
    \errmessage{(Inkscape) Transparency is used (non-zero) for the text in Inkscape, but the package 'transparent.sty' is not loaded}%
    \renewcommand\transparent[1]{}%
  }%
  \providecommand\rotatebox[2]{#2}%
  \newcommand*\fsize{\dimexpr\f@size pt\relax}%
  \newcommand*\lineheight[1]{\fontsize{\fsize}{#1\fsize}\selectfont}%
  \ifx\svgwidth\undefined%
    \setlength{\unitlength}{311.45796425bp}%
    \ifx\svgscale\undefined%
      \relax%
    \else%
      \setlength{\unitlength}{\unitlength * \real{\svgscale}}%
    \fi%
  \else%
    \setlength{\unitlength}{\svgwidth}%
  \fi%
  \global\let\svgwidth\undefined%
  \global\let\svgscale\undefined%
  \makeatother%
  \begin{picture}(1,0.40057234)%
    \lineheight{1}%
    \setlength\tabcolsep{0pt}%
    \put(0,0){\includegraphics[width=\unitlength,page=1]{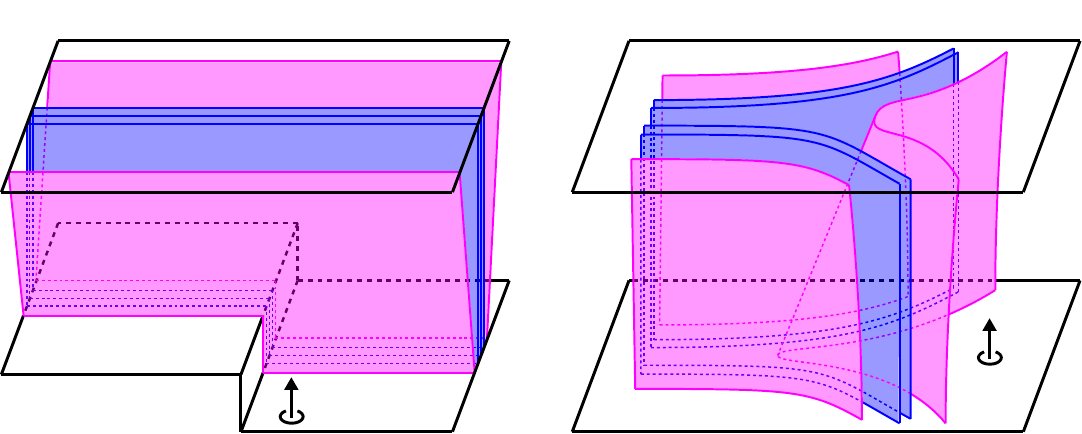}}%
    \put(0.08255573,0.07063518){\color[rgb]{0,0,0}\makebox(0,0)[lt]{\lineheight{1.25}\smash{\begin{tabular}[t]{l}$\partial A$\end{tabular}}}}%
    \put(0.56597115,0.0116447){\color[rgb]{0,0,0}\makebox(0,0)[lt]{\lineheight{1.25}\smash{\begin{tabular}[t]{l}$\partial A$\end{tabular}}}}%
    \put(0.05610894,0.37586495){\color[rgb]{0,0,0}\makebox(0,0)[lt]{\lineheight{1.25}\smash{\begin{tabular}[t]{l}$R_+(\overline{M_f})$\end{tabular}}}}%
    \put(0.58398409,0.37586495){\color[rgb]{0,0,0}\makebox(0,0)[lt]{\lineheight{1.25}\smash{\begin{tabular}[t]{l}$R_+(\overline{M_f})$\end{tabular}}}}%
    \put(0.28487742,0.01786005){\color[rgb]{0,0,0}\makebox(0,0)[lt]{\lineheight{1.25}\smash{\begin{tabular}[t]{l}$\phi_f$\end{tabular}}}}%
    \put(0.93050786,0.07243331){\color[rgb]{0,0,0}\makebox(0,0)[lt]{\lineheight{1.25}\smash{\begin{tabular}[t]{l}$\phi_f$\end{tabular}}}}%
  \end{picture}%
\endgroup%
}
    \caption{Construction of $\mc N_N$. The flow $\phi_f|_N$ is vertical in this picture.}
    \label{fig:nbhdinN}
\end{figure}

Let $\mc N(B^u)=\mc N_A(B^u)\cup \mc N_N(B^u)$. This is a standard neighborhood of $B^u$ with the promised property that $\phi_f$ points into $\mc N(B^u)$ along the boundary of each of its nonprincipal regions.

Now that we have the neighborhood $\mc N(B^u)$, we can redo the contruction of $(B^u, V)$ from \Cref{prop:buexist} so that $V$ is equal to the generating vector field for $\phi_f$ outside of $\mc N(B^u)$ and points into $\mc N(B^u)$ along the boundary of each nonprincipal region. Inside $\mc N(B^u)$ we are free to require that flow lines not lying in $B^u$ do not accumulate in $\mc N(B^u)$ in backward time.

Next we let $U$ be any nonprincipal region of $B^u$, and consider any point $p\in U$ not lying in a $\u$-face. If $p$ lies in $\mc N(B^u)$ then its backward orbit must exit $\mc N(B^u)$ by construction, so we can assume $p\notin \mc N(B^u)$. If we flow $p$ backward until it hits $L$ at a point $p'$, we see that $p'$ must lie in a nonprincipal region. Hence $p$ lies in the negative escaping set of $f$ by the definition of principal regions (\Cref{def:principalregion}), so the backward orbit must terminate on $R_-(\ol M_f)$.
\end{proof}

In light of \Cref{lem:nonprincipal regionescaping}, we will henceforth assume that $V$ has been chosen so that each backward trajectory in a nonprincipal regions not lying in a $\u$-face terminates on $R_-(\ol M_f)$.
We now begin an extended analysis of the nonprincipal regions of $(B^u,V)$, which will lead to the following proposition.

\begin{proposition}\label{lem:ucuspedproduct}
Each nonprincipal region of $(B^u, V)$ is a $\u$-cusped product with no $\u\u$-cusp circles. 
\end{proposition}

\begin{figure}
    \centering
    \resizebox{!}{1.5in}{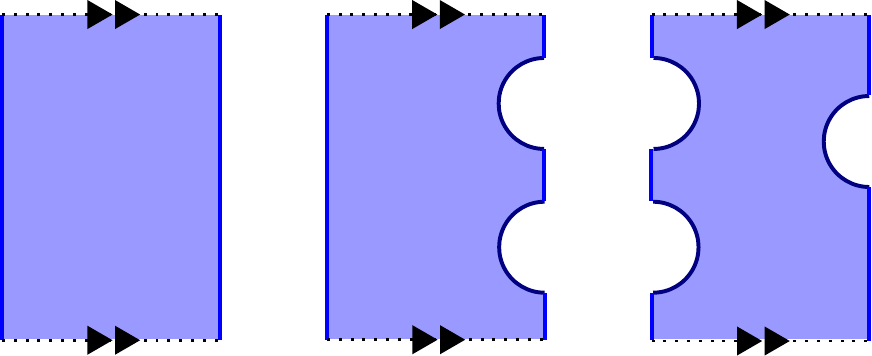}
    \caption{Possible $\u$-faces of a component of $Q\cut B^u$ containing a component of $R_-$.The blue pieces of boundary indicate parts of the branch locus of $B^u$, while the black pieces lie on $R_+$.}
    \label{fig:ufaces}
\end{figure}

Fix a nonprincipal region $U$ of $B^u$, and let
$F$ be a $\u$-face of $U$. 
By \Cref{rmk:annulusfaces}, $F$ is an annulus. 
Hence $F$ is of one of the forms shown in \Cref{fig:ufaces};
that is, each boundary component either lies entirely on $R_+$ or alternates between lying on $R_+$ and lying in $\brloc(B^u)$. Each segment of $\del F$ lying in $\brloc(B^u)$ must have a source orientation with a single source in its interior. If a component of $\del F$ does not lie completely in $R_+$ we say that component is \textbf{scalloped}.

Recall that $\mc F$ is the depth one foliation of $Q$. Let $\{L_\theta\mid \theta\in S^1\}$ denote the collection of leaves in $\intr (Q)$, each isotopic to $L$. For each $\theta\in S^1$, let $\tau_\theta$ be the intersection $B^u\cap L_\theta$.

By intersecting with $\mc F$, we obtain a cooriented foliation $\mc F_F$ of $F$. The foliation $\mc F$ is compatible with $\del F$ in the sense that 
\begin{itemize}
    \item the pieces of $F$ along $R_+$ are leaves of $\mc F_F$,
    \item $\mc F_F$ is tangent to $\del F$ at the sources, and
    \item $\mc F_F$ is positively transverse to $\del F$ at all oriented points of $\brloc(B^u)\cap F$.
\end{itemize}
Furthermore, a given leaf of $\mc F_F$ is tangent to at most one source in $\del F$. This is because in our folding sequence, folds are performed one at a time.

Fix $\theta\in S^1$, and let $\gamma$ be a component of $L_\theta\cap F$. In $L_\theta$, $\gamma$ is a piece of the boundary of a nonprincipal region of $\tau_\theta$. Note that $\gamma$ cannot be a circle, or else $\Lambda_+$ would have a circular leaf.
Since $\gamma$ is a 1-manifold, it is homeomorphic to either $[0,1]$, $[0,\infty)$, or $(-\infty,\infty)$.

We can see that $\gamma$ does not accumulate in $\intr(F)$, for otherwise $L_\theta$ would accumulate in $\intr (Q)$, a contradiction.
As a consequence, if $\gamma$ is noncompact then each end of $\gamma$ (there may be one or two) must accumulate on a component of $\del F$. It is not possible for $\gamma$ to accumulate on $\brloc(B^u)\cap F$ since $\gamma$ must be transverse to $\brloc(B^u)$ except at sources. We conclude that each neighborhood $\nu\cong [a,\infty)$ of an end of $\gamma$ limits on a boundary component of $F$ lying entirely on $R_+$, and hence that $\nu$ limits on $R_+$. We remark that when viewed as a curve in $L_\theta$, the end-neighborhood $\nu$ must be an escaping ray.

Let $\del_1$ and $\del_2$ be the two boundary components of $F$. We say that a leaf $\lambda$ of $\F_F$ \textbf{spans} $F$ if for each $\del_i$ ($i=1,2$), $\lambda$ either spirals onto $\del_i$ (if $\del_i$ is nonscalloped) or terminates on $\del_i$.

\begin{lemma} \label{lem:spanningleaf}
There exists a leaf of $\F_F$ spanning $F$.
\end{lemma}

\begin{proof}
Suppose that $\del_1$ is nonscalloped. Since $\F_F$ has no circle leaves, there must be leaves spiraling onto $\del_1$; let $\lambda$ be such a leaf. Choose an orientation for $\lambda$ and suppose that $\lambda$ spirals on $\del_1$ in the backward direction. In the forward direction, the coorientation of $\lambda$ prevents $\lambda$ from spiraling on $\del_1$. Also, as noted above $\lambda$ cannot accumulate in $\intr(F)$. Thus we see $\lambda$ must either accumulate on $\del_2$ if $\del_2$ is nonscalloped or terminate on $\del_2$ if $\del_2$ is scalloped. 

It remains to show only that there exists a spanning leaf when $\del_1$ and $\del_2$ are both scalloped. Suppose that $\del_1$ is scalloped and contains $k$ sources of $\brloc(B^u)$. Let $\lambda_1,\dots, \lambda_k$ be the leaves of $\F_F$ which are tangent to the sources. Choose an orientation of $\del_1$, which we will refer to as \emph{clockwise}. Thus every oriented portion of $\brloc(B^u)\cap F$ is either \emph{clockwise} or \emph{counterclockwise}.
If $s_i$ is the source contained in $\lambda_i$, then $\lambda_i\cut s_i$ consists of 2 components. Orient each of these components \emph{away} from $s_i$, and let  $\lambda_i^\circlearrowright$ and $\lambda_i^\circlearrowleft$ denote the components whose respective orientations at $s_i$ are clockwise and counterclockwise, respectively. 

If $\lambda_i^\circlearrowright$ terminates on $\del_2$, then it must be the case that $\lambda_i^\circlearrowleft$ terminates on $\del_1$. Indeed, if $\lambda_i$ had both its endpoints on $\del_2$ then an index argument shows that there would be a component of $F\cut \lambda_i$ on which $\mc F_F$ is singular, a contradiction. We conclude that if $\lambda_i^\circlearrowright$ terminates on $\del_2$, then $\lambda_i$ spans $F$.

Now suppose that none of $\lambda_1^\circlearrowright,\dots,\lambda_k^\circlearrowright$ terminates on $\del_2$, whence they all terminate on $\del_1$. Then because the coorientation of $\F_F$ is compatible with the orientation of $\brloc(B^u)$, each $\lambda_i^\circlearrowright$ must terminate on a counterclockwise portion of $\brloc(B^u)$ (recall that each $\lambda_i$ contains at most one source). Consider $F'=F\cut (\bigcup_{i=1}^k \lambda_i^\circlearrowright)$. By assumption, there exists an annular component of $F'$ containing $\del_2$. Without loss of generality, assume that $\lambda_1^\circlearrowright\subset \del F'$. Then $\lambda_1^\circlearrowleft$ must either terminate on a clockwise portion of $\del_1$ or terminate on $\del_2$. 
Observe that $F'$ contains no clockwise portions of $\del_1$, since each clockwise component of $\del F$ is separated from $F'$ by some $\lambda_i^\circlearrowright$. 
We conclude that $\lambda_1^\circlearrowleft$ terminates on $\del_2$, and that $\lambda_1$ spans $F$.
\end{proof}

\begin{lemma}\label{lem:circular}
The vector field $V_F$ is circular.
\end{lemma}

\begin{proof}
Let $A$ be an annulus with smooth boundary, and choose an embedding of $F$ into $A$ so that $\del F\cap R_+$ is mapped into $\del A$ and the interior of each component of $\del F\cap \brloc(B^u)$ is mapped to $\intr(A)$. See \Cref{fig:AandF}, left. 
\begin{figure}
    \centering
    \resizebox{!}{1.4in}{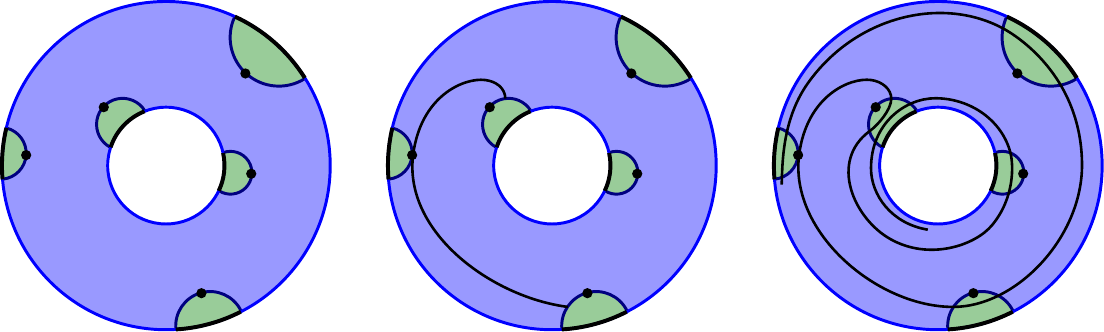}
    \caption{\emph{Left}: embedding $F$ (purple) in a larger annulus $A$ without corners. \emph{Center}: a spanning leaf $\lambda$ as furnished by \Cref{lem:spanningleaf}. \emph{Right}: extending $\lambda$ to an embedded cooriented line $\ol \lambda$ by adding segments which are alternately paths in $A-F$ and leaves of $\F_F$ so that both ends of $\ol \lambda$ spiral onto $\del A$ compatibly with the coorientation of $\del A$. Coorientations are not drawn.}
    \label{fig:AandF}
\end{figure}
By \Cref{lem:spanningleaf}, there exists a leaf $\lambda$ of $\mc F|_F$ spanning $F$. We can extend $\lambda$ to an embedded copy $\ol \lambda$ of $\R$ with both ends spiraling onto $\del A$ by adding to $\lambda$ line segments which are alternately contained in $A-F$ and in leaves of $\F_F$. See \Cref{fig:AandF}, right. 

Next, we may extend $\F_F$ to a foliation $F_A$ of $A$ such that $\ol \lambda$ is a leaf, and such that each component of $\del A$ is also. Moreover, $V$ extends to a vector field $V_A$ on $A$ transverse to $\F_A$. The spiraling of $\ol \lambda$ forces $\F_A$ to be a Reeb foliation. By \Cref{lem:reebcircular} $V_A$ is circular, so $V_F$ must also be circular.
\end{proof}

Now we can prove \Cref{lem:ucuspedproduct} as promised, which states that each nonprincipal region of $(B,V)$ is a $\u$-cusped product.

\begin{proof}[Proof of \Cref{lem:ucuspedproduct}]
We will start by using Mosher's $\u$-cusped product recognition lemma, \Cref{lem:ucusprecog}.

Let $P$ be a nonprincipal region. We will check that $P$ satisfies the hypotheses of Mosher's $\u$-cusped product recognition lemma, \Cref{lem:ucusprecog}; namely the conditions labeled (1-4) and the property that each $\m$-face has nonpositive index. Evidently $P$ contains no $\s$-faces, so condition (1) is satisfied. Recall that $V$ was chosen so that each backward orbit of the $V$-semiflow not contained in a $\u$-face terminates on $R_-$, thus condition (2) holds. Every $\b$-face of $P$ comes from the orbit of a boundary component of $L$ that accumulates on $R_-$. This orbit can either consist of finitely many line boundary components, each with one end in a positive end of $L$ and one end in a negative end of $L$; or of infinitely many compact boundary components that escape into the positive and negative ends under positive and negative iteration of $f$. In either case, the corresponding $\b$-face is an annulus with boundary consisting of one $\p\b$-circle and one $\m\b$-circle; this shows (3) holds. Also, note that there are no $\p\m$-edges since components of $R_+(\ol M_f)$ and $R_-(\ol M_f)$ are never adjacent, so (4) holds.
Further, each $\m$-face of $P$ is a component of $R_-(\ol M_f)$, so has nonpositive index. Hence we can apply \Cref{lem:ucusprecog} to conclude that $P$ is homeomorphic to $S\times [0,1]$ for some component $S$ of $R_-(\ol M_f)$, and $P$ satisfies (a-d) in the definition of $\u$-cusped product (\Cref{defn:cuspedproduct}). 

It remains to check the following conditions from \Cref{defn:cuspedproduct}: that (e) each $\p$-face has nonpositive index, (f) $V$ is circular on each $\u$-face, and (g) each $\u\u$-cusp circle is incoherent.
For (e), note that each $\p$-face is a complementary component of $B^u\cap R_+=T_+^\infty$, which is an efficient train track; hence each $\p$-face has nonpositive index. Condition (f) is satisfied by \Cref{lem:circular}. Finally, for (g), note that a $\u\u$-cusp circle of a complementary region of $B^u$ corresponds to a cusp in our construction's train track splitting sequence which experiences no collisions. By \Cref{lem:collision}, such a cusp corresponds to a principal region, so we conclude that the nonprincipal region $P$ is a $\u$-cusped product with no $\u\u$-cusp circles as claimed.
\end{proof}

\subsection{Annulus and M\"obius band sectors of $B^u$} \label{subsec:sectorannulusMobius}

In \Cref{subsec:depthonetovbs} and \Cref{subsec:vbstodp} it will be important to understand how annulus and M\"obius band sectors of $B^u$ without corners arise in our construction.

\begin{lemma}\label{lem:sectorannulus}
Suppose that $A$ is an annulus sector of $B^u$ with no corners. Then $A$ either:
\begin{enumerate}[label=(\roman*)]
    \item is a face of a  principal region of $B^u$, or
    \item has one boundary component on $R_+$ and another on a cusp circle of a principal region of $B^u$, or
    \item has both its boundary components on $R_+$.
\end{enumerate}
Furthermore, $A$ is adjacent to at least one nonprincipal complementary region of $B^u$.
\end{lemma}

\begin{proof}
Suppose that $\gamma$ is a boundary component of $A$ lying in $\brloc(B^u)$. Then in the periodic splitting sequence used to construct $B^u$, $\gamma$ corresponds to a cusp $c$ of $\tau$ which never collides with another cusp, so is a principal cusp by \Cref{lem:collision}. Let $b$ be the branch incident to $c$ and lying in $A$. One property of the splitting sequence is that every large branch is eventually split, so $b$ must not be a large branch. 

Suppose that $c$ points into $b$. If $b$ is not incident to another switch, then (ii) holds. If $b$ is incident to another switch $c'$ (necessarily principal), then the $\Lambda$-routes from $c$ and $c'$ must fellow travel, contradicting the fact that principal cusps never fellow travel (\Cref{lem:principalcusps}).

Otherwise $b$ is a small branch. Lift $b$ to the universal cover $\wt L$, call its lift $\wt b$ and suppose that the two cusps incident to $\wt b$ are $\wt c$ and $\wt c'$ corresponding to $c$ and $c'$ respectively. 
 
The reader should reference \Cref{fig:principalregions} while reading this argument.
Suppose for a contradiction that $\wt c$ and $\wt c'$ correspond to lifted principal regions of $\tau$ on each side of the small branch $\wt b$.
There are 2 lifts $P$ and $P'$ of principal regions of $\Lambda_+$ corresponding to these train track complementary regions. Note that there are multiple $\wt\Lambda_+$-leaves separating $P$ and $P'$ by \cite[Lemma 5.20]{CCF19}. Also, by  \Cref{lem:principalcusps} none of the border leaves for $P$ and $P'$ share ideal points in $S_\infty(\wt L)$. Let $\ell$ and $\ell'$ be the border leaves of $P$ and $P'$ carried by $\wt b$, respectively. Following $\ell$ and $\ell'$ from $b$ through $\wt c$ and toward $\del \wt L$, the corresponding train routes must diverge, forcing the existence of a cusp $\wt c ''$ that projects to a cusp in $L$ that eventually is involved in a split through $b$, a contradiction. 

Since $\wt c$ and $\wt c'$ are both principal cusps, the contradiction above forces them to lie on the same side of $\wt b$. We conclude that if $b$ is small, then (i) holds: $A$ is a face of a principal region of $B^u$.

\begin{figure}
    \centering
    \fontsize{6pt}{6pt}\selectfont
    \resizebox{!}{2in}{
\begingroup%
  \makeatletter%
  \providecommand\color[2][]{%
    \errmessage{(Inkscape) Color is used for the text in Inkscape, but the package 'color.sty' is not loaded}%
    \renewcommand\color[2][]{}%
  }%
  \providecommand\transparent[1]{%
    \errmessage{(Inkscape) Transparency is used (non-zero) for the text in Inkscape, but the package 'transparent.sty' is not loaded}%
    \renewcommand\transparent[1]{}%
  }%
  \providecommand\rotatebox[2]{#2}%
  \newcommand*\fsize{\dimexpr\f@size pt\relax}%
  \newcommand*\lineheight[1]{\fontsize{\fsize}{#1\fsize}\selectfont}%
  \ifx\svgwidth\undefined%
    \setlength{\unitlength}{156.03458066bp}%
    \ifx\svgscale\undefined%
      \relax%
    \else%
      \setlength{\unitlength}{\unitlength * \real{\svgscale}}%
    \fi%
  \else%
    \setlength{\unitlength}{\svgwidth}%
  \fi%
  \global\let\svgwidth\undefined%
  \global\let\svgscale\undefined%
  \makeatother%
  \begin{picture}(1,0.95840823)%
    \lineheight{1}%
    \setlength\tabcolsep{0pt}%
    \put(0,0){\includegraphics[width=\unitlength,page=1]{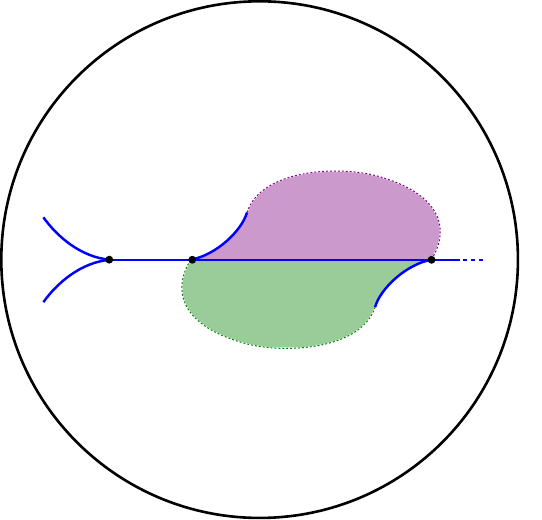}}%
    \put(0.09623599,0.46761675){\color[rgb]{0,0,0}\makebox(0,0)[lt]{\lineheight{1.25}\smash{\begin{tabular}[t]{l}$\widetilde{c}''$\end{tabular}}}}%
    \put(0.57697779,0.49083445){\color[rgb]{0,0,0}\makebox(0,0)[lt]{\lineheight{1.25}\smash{\begin{tabular}[t]{l}$\widetilde{b}$\end{tabular}}}}%
    \put(0.6175766,0.55778383){\color[rgb]{0,0,0}\makebox(0,0)[lt]{\lineheight{1.25}\smash{\begin{tabular}[t]{l}$\widetilde{P}$\end{tabular}}}}%
    \put(0.50977823,0.38854957){\color[rgb]{0,0,0}\makebox(0,0)[lt]{\lineheight{1.25}\smash{\begin{tabular}[t]{l}$\widetilde{P}'$\end{tabular}}}}%
    \put(0.42463497,0.48954601){\color[rgb]{0,0,0}\makebox(0,0)[lt]{\lineheight{1.25}\smash{\begin{tabular}[t]{l}$\widetilde{c}$\end{tabular}}}}%
    \put(0.68486556,0.43641797){\color[rgb]{0,0,0}\makebox(0,0)[lt]{\lineheight{1.25}\smash{\begin{tabular}[t]{l}$\widetilde{c}'$\end{tabular}}}}%
    \put(0,0){\includegraphics[width=\unitlength,page=2]{fig_principalregions.pdf}}%
  \end{picture}%
\endgroup%
}
    \caption{From the proof of \Cref{lem:sectorannulus} that $\wt b$ cannot abut principal regions on both its sides.}
    \label{fig:principalregions}
\end{figure}
 
The only other possibility is that $b$ is not incident to any cusps. In this case, both ends of $b$ escape, giving two closed curves in $R_+$, so possiblity (iii) holds.
 
It remains only to show that $A$ cannot border a principal region on both sides. Since any sector touching $R_+$ is incident to nonprincipal regions on both sides, we can assume that $A$ has both its boundary components in $\brloc(B^u)$ and thus corresponds to a compact branch $b$ of a train track in our splitting sequence. By the above arguments, $b$ must be incident to a nonprincipal region of the train track, forcing $A$ to border a nonprincipal region of $B^u$.
\end{proof}

\begin{lemma}\label{lem:sectormobius}
Let $M$ be a M\"obius strip sector of $B^u$ with no corners. Then $\del M\subset R_+$.
\end{lemma}
\begin{proof}
Let $b$ be the branch of $\tau$ which suspends to give $M$. As in the previous proof, the cusps incident to $b$ (if any) must be principal. Suppose there is at least one cusp incident to $b$. Since $M$ has only one boundary component, this forces $b$ to be incident to 2 cusps. Since $f$ sends $b$ to itself with the orientation reversed while preserving the orientation of $L$, these cusps must lie on opposite sides of $b$.
But this gives a similar contradiction to the previous proof by looking at the picture in $\wt L$.

We conclude that $b$ is not incident to any switches, so both of its ends must escape to positive ends and correspond to a boundary component of $M$ lying in $R_+$.
\end{proof}

\subsection{Our branched surface $(B^u, V)$ is very full and veering}
\label{subsec:depthonetovbs}

\begin{theorem}
\label{thm:depthonetovbs}
Let $Q$ be an atoroidal sutured manifold with a depth one foliation $\mathcal{F}$. Then $Q$ contains an unstable veering branched surface carrying the unstable Handel-Miller lamination associated to $\mathcal{F}$.
\end{theorem}

\begin{proof}
Let $\mc F$ be a depth one foliation of $Q$. Let $L$ be a particular noncompact leaf of $\mc F$, and let $f\colon L\to L$ be a Handel-Miller representative of the monodromy for $L$ determined by $\mc F$. Let $(B^u, V)$ be an unstable dynamic branched surface as constructed in \Cref{prop:buexist} whose nonprincipal regions are $\u$-cusped products (this exists by \Cref{lem:ucuspedproduct}).

Suppose that $f$ has a principal region $P$ which is nonperipheral and not simply connected. Then the border of the nucleus of $P$, when flowed forward under the suspension flow of $f$, sweeps out a $\pi_1$-injective nonperipheral torus in $Q$, violating atoroidality.

Otherwise all principal regions of $f$ are simply connected or peripheral, so all principal regions of $B^u$ are $\u$-cusped tori or $\u$-cusped torus shells.
By assumption all nonprincipal regions of $B^u$ are $\u$-cusped products so $B^u$ is very full in this case. It remains to show that $B^u$ is veering. Veering branched surfaces are defined in \Cref{defn:vbs} and in this proof we will refer to properties (1)-(5) from that definition.
\begin{itemize}
\item Condition (1) regarding triple points is satisfied because each triple point of $B^u$ comes from a train track folding move.

\item Condition (2) on the boundary train track is satisfied because we chose the boundary train track $T_+^\infty$ to be efficient and spiraling.

\item Condition (3) requires that the orientation on each branch loop agree with the dynamic orientation (\Cref{def:circular}) on each adjacent face of the corresponding complementary region. This is true because each branch loop is positively transverse to $\mc F$ by the construction, as is the core of each annulus face of a principal region of $B^u$. Meanwhile there are no branch loops in nonprincipal regions by \Cref{lem:ucuspedproduct}.

\item Condition (4) on annulus and Möbius band sectors holds by \Cref{lem:sectorannulus} and \Cref{lem:sectormobius}.

\item Condition (5) says that $B^u$ does not carry any tori or Klein bottles. Suppose that $T$ is a torus or Klein bottle carried by $B^u$. Since $T\subset \intr Y$, the intersection $T\cap L$ is compact. It follows that $\mc F|_T$ is a foliation of $T$ by circles. 
By reversing the folding sequence used to construct $(B^u, V)$ we get a splitting sequence 
\[
\tau_0\to \tau_1\to \cdots \tau_n\to \tau_{n+1}\to \cdots
\]
where each $\tau_i$ is a train track in $L$ carrying the Handel-Miller lamination $\Lambda_+$, and $\tau_{i+n}=f(\tau_i)$ for all $i$. Since $\mc F|_T$ is a foliation by circles, there exists a closed curve $c\subset L$ which is carried by $\tau_0$ and survives each splitting move in the above sequence. Fix a (local) side of $c$. After finitely many splits, all switches on this side of $c$ must point in the same direction around $c$. Since each $\tau_i$ fully carries $\Lambda_+$, this forces the existence of a compact leaf of $\Lambda_+$. This contradicts that all leaves of $\Lambda_+$ are noncompact; see e.g. \Cref{lem:accumonends}.
\end{itemize}

Having verified that $(B^u,V)$ satisfies all the conditions in the definition of veering branched surface, we are done. 
\end{proof}

\begin{theorem}
Let $Y$ be an atoroidal depth one Reeb sutured manifold, and let $Y'$ be the de-Reebification of $Y$. Then $Y'$ contains a very full unstable veering branched surface. For each annulus component $A$ of $R_+(Y')$ that was added in the de-Reebification process, the veering branched surface intersects $A$ in a circle.
\end{theorem}

\begin{proof}
In light of \Cref{thm:depthonetovbs}, the only case that needs our attention here is when $\del Y$ has at least one Reeb annulus. Since $Y$ is depth one, there is a depth one foliation of $Y$ with Reeb endperiodic monodromy $f$ (see \Cref{def:endperiodic}). We can replace $f$ with its endperiodization (see \Cref{construction:samesign}), obtaining a depth 1 foliation of $Y'$. Each annulus of $R_+(Y')$ corresponds to an infinite strip end. By \Cref{lem:infstripend}, the Handel-Miller lamination for the endperiodization has one or two leaves that exit each such end. When we construct the very full unstable veering branched surface ($B^u, V)$ of \Cref{prop:buexist}, each such annulus will meet $B^u$ in a single circle.
\end{proof}

\subsection{Uniqueness of the veering branched surface construction} \label{subsec:hmvbsunique}

In this subsection, we will prove some uniqueness results about the veering branched surface constructed in \Cref{thm:depthonetovbs}. Taken together with the results in \Cref{subsec:splitsequnique}, they show that our construction of such a veering branched surface ends up being quite canonical.

The construction of a dynamic branched surface in \Cref{prop:buexist} depends on three inputs, namely:
\begin{enumerate}[label=(\alph*)]
    \item a $f$-periodic splitting sequence of $f$-endperiodic train tracks $S=\{\tau_0 \to \tau_1 \to \tau_2 \to \cdots\}$,
    \item a tiling $\mathcal{T}$ of the  end-cycles of $L$, and
    \item a core $K$ large enough such that the truncated sequence $\tau_0 \to \cdots \to f(\tau_0)$ is supported in $K$.
\end{enumerate}

Our construction of the branched surface could be summarized as follows. First we reversed the splitting sequence to obtain a folding sequence from $f(\tau_0)$ to $\tau_0$. Outside a  neighborhood of $R_+(\ol M_f)$, we let the branched surface be the suspension of this folding sequence; in the neighborhood of $R_+(\ol M_f)$, we constructed the branched surface to be $T_+^\infty\times I$ (recall that $T_+^\infty$ is chosen during the construction of the splitting sequence). The choices of tiling and core simply served the purpose of fixing a neighborhood of $R_+(\ol M_f)$, so it should come as no surprise that the construction is independent of those choices. In \Cref{lemma:vbsunique-tilednbd} and \Cref{lemma:vbsunique-tiling} we make this precise.

Later, in \Cref{lemma:vbsunique-splitseq} we show that the branched surface is unchanged if we replace $S$ by an equivalent $f$-periodic splitting sequence (using the equivalence relation from \Cref{subsec:splitsequnique}). By \Cref{thm:splitsequnique}, this implies that the branched surface depends only on the choice of train track $T_+^\infty$.

Let us write $B(S, \mathcal{T}, K)$ for the branched surface constructed with the data (a), (b), (c), considered up to isotopy.

\begin{lemma} \label{lemma:vbsunique-tilednbd}
For a fixed splitting sequence $S$ and tiling $\mathcal{T}$, let $K$ and $K'$ be two choices of cores such that $\tau_0 \to ... \to f(\tau_0)$ is supported inside $K$ and $K'$. Then $B(S, \mathcal{T}, K) = B(S, \mathcal{T}, K')$.
\end{lemma}

\begin{proof}
For convenience, let $B = B(S, \mathcal{T}, K)$ and $B'= B(S, \mathcal{T}, K')$. Since there is always a core containing both $K$ and $K'$, it suffices to assume that $K \subset K'$. Let $F = K' \cut K$. Using the notation from the construction in \Cref{sec:bsconstruct}, we can write $Q$ as the union $A \cup (F \times [0,1]) \cup N'(R_+)$, with $N(R_+)=(F \times [0,1])  \cup N'(R_+)$ and $A'=A \cup (F \times [0,1])$. See \Cref{fig:vbsunique-tilednbd} for a schematic picture. 

\begin{figure}
    \centering
    \fontsize{6pt}{6pt}\selectfont
    \resizebox{!}{3cm}{
\begingroup%
  \makeatletter%
  \providecommand\color[2][]{%
    \errmessage{(Inkscape) Color is used for the text in Inkscape, but the package 'color.sty' is not loaded}%
    \renewcommand\color[2][]{}%
  }%
  \providecommand\transparent[1]{%
    \errmessage{(Inkscape) Transparency is used (non-zero) for the text in Inkscape, but the package 'transparent.sty' is not loaded}%
    \renewcommand\transparent[1]{}%
  }%
  \providecommand\rotatebox[2]{#2}%
  \newcommand*\fsize{\dimexpr\f@size pt\relax}%
  \newcommand*\lineheight[1]{\fontsize{\fsize}{#1\fsize}\selectfont}%
  \ifx\svgwidth\undefined%
    \setlength{\unitlength}{180.86760366bp}%
    \ifx\svgscale\undefined%
      \relax%
    \else%
      \setlength{\unitlength}{\unitlength * \real{\svgscale}}%
    \fi%
  \else%
    \setlength{\unitlength}{\svgwidth}%
  \fi%
  \global\let\svgwidth\undefined%
  \global\let\svgscale\undefined%
  \makeatother%
  \begin{picture}(1,0.34465923)%
    \lineheight{1}%
    \setlength\tabcolsep{0pt}%
    \put(0,0){\includegraphics[width=\unitlength,page=1]{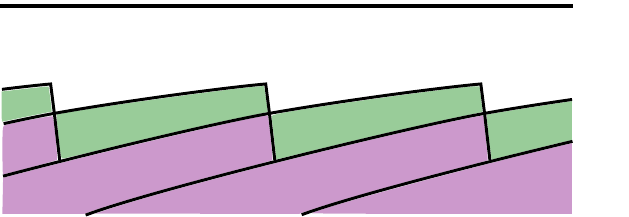}}%
    \put(0.79582685,0.12656235){\color[rgb]{0.50196078,0,0.50196078}\makebox(0,0)[lt]{\lineheight{1.25}\smash{\begin{tabular}[t]{l}$K \times [0,1]$\end{tabular}}}}%
    \put(0.78537474,0.20236035){\color[rgb]{0,0.50196078,0}\makebox(0,0)[lt]{\lineheight{1.25}\smash{\begin{tabular}[t]{l}$F \times [0,1]$\end{tabular}}}}%
    \put(0.92583491,0.32764057){\color[rgb]{0,0,0}\makebox(0,0)[lt]{\lineheight{1.25}\smash{\begin{tabular}[t]{l}$R_+$\end{tabular}}}}%
  \end{picture}%
\endgroup%
}
    \caption{A schematic picture of the proof of \Cref{lemma:vbsunique-tilednbd}.} 
    \label{fig:vbsunique-tilednbd}
\end{figure}

Notice that $B|_{F \times [0,1]}$ is the product branched surface $\tau_0|_F \times [0,1]$ and that $B|_{A'} = B|_A \cup B|_{F \times [0,1]} = B'|_{A'}$, since $\tau_0 \to ... \to f(\tau_0)$ is supported inside $K_i$. In particular $B|_{\partial A'} = B'|_{\partial A'}$, and so upon taking the forward flow completion, we have $B|_{N'(R_+)} = B'|_{N'(R_+)}$. Hence $B=B'$.
\end{proof}

In light of \Cref{lemma:vbsunique-tilednbd}, we will drop the third argument of $B(\cdot, \cdot, \cdot)$ and write simply $B(\cdot, \cdot)$.

\begin{lemma} \label{lemma:vbsunique-tiling}
For a fixed splitting sequence $S$, let $\mathcal{T}$ and $\mathcal{T}'$ be two choices of tiling of the end-cycles of $L$. Then $B(S,\mathcal{T}) = B(S,\mathcal{T}')$.
\end{lemma}
\begin{proof}
For convenience, let $B = B({S},\mathcal{T})$ and $B' = B({S},\mathcal{T}')$. Using \Cref{lem:subsurfaceseq} as in the proof of \Cref{lem:anytiling}, it suffices to prove the lemma in the case that $\mathcal{T}$ and $\mathcal{T}'$ are interleaved. Recall that this in particular means that $K_i \subset K'_i \subset K_{i+1}$ up to reindexing, where $K_i$ are the cores determined by $\mathcal{T}$ and $K'_i$ are the cores determined by $\mathcal{T}'$. Suppose $\tau_0 \to ... \to f(\tau_0)$ is supported inside $K_i$. Let $F = K'_i \cut K_i$. We will show that $B(S, \mathcal{T}, K_i) = B(S, \mathcal{T}, K'_i)$.

This proof of this is exactly as in \Cref{lemma:vbsunique-tilednbd}: We write $Q$ as the union $A \cup (F \times [0,1]) \cup N'(R_+)$, observe that $B|_{A'} = B|_A \cup B|_{F \times [0,1]} = B'|_{A'}$ then $B|_{N'(R_+)} = B'|_{N'(R_+)}$, hence $B=B'$.
\end{proof}

In light of \Cref{lemma:vbsunique-tiling}, we will drop the second argument of $B(\cdot, \cdot)$ and write simply $B(\cdot)$.

\begin{lemma} \label{lemma:vbsunique-splitseq}
Let $S$ and $S'$  be two equivalent $f$-periodic splitting sequences. Then $B(S)$ is ambient isotopic to $B(S')$.
\end{lemma}

\begin{proof}
Let $B = B(S)$ and $B' = B(S')$. It suffices to prove the lemma in the case when $S$ and $S'$ differ by one $f$-periodic commutation.

In this case, using the notation from the construction of \Cref{sec:bsconstruct}, there exist $M_i$ and $M_{i+1}$ such that $B(S)$ and $B(S')$ agree outside of $M_i\cup M_{i+1}$. In $M_i\cup M_{i+1}$, the restrictions of $B(S)$ and $B(S')$ are the suspensions (in the downward direction) of splits along the same two disjoint branches, but in different orders. This description makes it clear that the restrictions, and hence the entire branched surfaces, are isotopic.
\end{proof}

This gives the following:

\begin{corollary}\label{cor:uniqueuptotraintrack}
Up to isotopy, our construction of a veering branched surface in \Cref{thm:depthonetovbs} depends only on the choice of the boundary train track.
\end{corollary}

It turns out that an even stronger statement is true, namely that any veering branched surface in \Cref{thm:depthonetovbs} (satisfying some obviously necessary conditions) comes from our construction, see \Cref{thm:hmvbsunique}. Hence our construction produces the unique veering branched surface satisfying these conditions given a boundary train track. 
We prove this stronger statement in \Cref{sec:hmvbsunique}, where we also show that when the boundary train track is allowed to vary, the resulting branched surface transforms via moves that can be explicitly described. This will round off our investigation of uniqueness.

\section{Foliation cones} \label{sec:folcone}

In this section, we will port some results of \cite{LMT24} regarding veering triangulations and Thurston fibered faces over to the sutured setting. The main goal will be \Cref{thm:folcone}, which we need in order to complete the uniqueness discussion in \Cref{sec:hmvbsunique}. We remark that this discussion opens up some interesting questions, in particular regarding how to generalize Thurston norm faces to sutured manifolds, which we will discuss in \Cref{sec:questions}.

\subsection{Depth one foliations and foliation cones} \label{subsec:reviewfolcone}

We first review some of the theory of foliation cones for atoroidal sutured manifolds. For more detail, see the articles \cite{CC99, CC17, CCF19}. 

Let $Q$ be an atoroidal sutured manifold. 
Every depth one foliation $\mathcal{F}$ on $Q$ determines a class $[\mathcal{F}] \in H^1(Q)$ whose value on any oriented loop in $\intr(Q)$ is given by taking its algebraic intersection with a non-compact leaf of $\mathcal{F}$. An alternative way of describing $[\mathcal{F}]$ is to cut away staircase neighborhoods of $R_\pm$ and look at the restriction of a noncompact leaf, which is now a properly embedded surface; the class of this surface in $H_2(Q,\partial Q) \cong H^1(Q)$ is $[\mathcal{F}]$. It is shown by Cantwell and Conlon in \cite[Theorem 1.1]{CC94} that the isotopy class of $\mathcal{F}$ is determined by $[\mathcal{F}]$.

Cantwell and Conlon also showed that for fixed $Q$, there exist finitely many cones in $H^1(Q)$, called \textbf{foliation cones}, such that the class of any depth one foliation on $Q$ lies in the interior of one of these cones and conversely any integral point interior to a cone corresponds to a depth one foliation \cite[Theorem 4.3]{CC99}. We remark that Cantwell and Conlon refer to these as ``Handel-Miller foliation cones."

Suppose we fix a foliation $\mathcal{F}_1$ associated to $C$, with a Handel-Miller representative $f_1$ of its monodromy. By suspending the 1-dimensional positive Handel-Miller lamination of $f_1$, we obtain the 2-dimensional Handel-Miller lamination $\LL$. Now suppose $\mathcal F_2$ is another depth one foliation associated to $C$. By \cite[Thm. 4.9 and Prop. 6.20]{CC17}, $\mathcal F_2$ can be isotoped to be transverse to the suspension semiflow of $f_1$.
The proof of the Transfer Theorem of Cantwell-Conlon-Fenley \cite[Theorem 12.7]{CCF19} gives that the intersection $\Lambda_2$ of $\LL$ with a noncompact leaf of $\mathcal{F}_2$ is ambient isotopic to the 1-dimensional positive Handel-Miller lamination $\Lambda$ for the first return map to that leaf. (In fact, using this ambient isotopy to pull back a hyperbolic metric for which $\Lambda$ is geodesic shows that $\Lambda_2$ is the positive Handel-Miller lamination associated to this pullback metric). Hence as we vary the depth one foliation within the foliation cone $C$, the 2-dimensional Handel-Miller lamination is invariant up to isotopy.
Hence there exists a semiflow $\phi_C$ on $Q$ such that every depth one foliation $\mathcal{F}$ associated to $C$ is transverse to $\phi_C$ up to isotopy, and such that the first return map of $\phi_C$ preserves the Handel-Miller lamination for a given noncompact leaf of $\mathcal F$.

\begin{remark}
The papers \cite{CC17} and \cite{CCF19} assume that $R_\pm$ have no annulus or torus components. However, we have explained in \Cref{sec:endperiodic} that Handel-Miller theory works for surfaces with infinite strip ends. Reading these papers with that in mind, one sees the results also hold when annular components of $R_\pm$ are allowed. 
\end{remark}

Finally, it is explained in \cite[Section 6.2]{CC17} how one can compute foliation cones in practice using Markov partitions. We will not need to use the full knowledge of this, but rather we just need the fact that each foliation cone $C$ is dual to the cone in $H_1(Q)$ positively generated by the periodic orbits of the Handel-Miller semiflow $\phi_C$ described above.

We remark that the atoroidal case described here is simpler than the general case, where one must choose the Handel-Miller representative with an additional property called tightness to define the correct foliation cone. This complication arises in the presence of principal regions which are not disks.

Note the significant parallels between the theory of foliation cones and the flow-theoretic perspective on cones over fibered faces of the Thurston norm ball developed by Fried \cite{Fri79} in the setting of compact hyperbolic manifolds. However, the analogy is not perfect because unlike fibered cones, foliation cones cannot be interpreted as cones over faces of the unit ball of a norm on $H^1(Q)$: they are in general not invariant under multiplication by $-1$, as demonstrated by examples in \cite{CC99}.

\subsection{Cone of dual cycles} \label{subsec:folcone}

In the setting of pseudo-Anosov mapping tori, the following is true.

\begin{theorem} \label{thm:closedcones}
Let $f:S \to S$ be a pseudo-Anosov homeomorphism on a finite type surface, let $M=S\times[0,1]/(x,1)\sim (f(x),0)$ be the mapping torus of $f$ with its foliation $\mathcal{F}$ by the leaves $S\times\{t\}$, and let $\mathcal{L}$ be the unstable 2-dimensional lamination in $M$ associated to $f$. Let $\mathcal{C}_{\mathcal{F}} \subset H_2(M, \partial M)$ be the Thurston fibered cone associated to $\mathcal{F}$. 

Meanwhile, let $B$ be a veering branched surface fully carrying $\mathcal{L}$. Let $\Gamma$ be the dual graph of $B$ and let $\Phi$ be the flow graph of $B$. Let $\mathcal{C}_\Gamma \subset H_1(M)$ be the cone positively generated by the cycles of $\Gamma$, and let $\mathcal{C}_\Phi$ be the cone positively generated by the cycles of $\Phi$.

Then $\mathcal{C}_{\mathcal{F}}^\vee = \mathcal{C}_\Gamma = \mathcal{C}_\Phi$ in $H_1(M)$.
\end{theorem}

The branched surface $B$ in the statement of \Cref{thm:closedcones} is actually unique up to isotopy; see \Cref{thm:closeduniqueness}. The statement above is phrased to motivate and maximize symmetry with \Cref{thm:folcone}, which we later use to prove the corresponding uniqueness statement in our setting.

In the cases when $S$ is closed and when $S$ is fully punctured (i.e. when the singularities of $f$ all occur at punctures), \Cref{thm:closedcones} is a consequence of results in \cite{LMT24} and \cite{Lan22}. The general case does not pose any added difficulty, so we sketch a proof using the ideas of \cite{LMT23a}. Our goal for the rest of this section is to transport this proof to the setting of sutured manifolds.

\begin{proof}[Proof of \Cref{thm:closedcones}]
Let $M^\circ$ be the fully punctured mapping torus of $f$ (obtained by deleting the singular orbits of the suspension flow of $f$), and let $\Delta$ be the veering triangulation on $M^\circ$ dual to $B$. Then $\Gamma$ is the dual graph of $\Delta$ and $\Phi$ is the dual graph of $\Delta$.

By \cite[Theorem 6 and Theorem 7]{Fri79}, $\mathcal{C}_{\mathcal{F}}^\vee$ is spanned by the closed orbits of the pseudo-Anosov suspension flow. (Strictly speaking \cite{Fri79} treats only the case where $S$ is closed; a proof of the general case appears in \cite[Appendix A]{Lan23}).  Now every closed orbit of the flow, say of homotopy class $g$, lies on some annulus or Möbius band leaf of $\mathcal{L}$. The lift of such a leaf to $\widetilde{M}$ determines a $g$-invariant dynamic plane, which has to contain a $g$-invariant bi-infinite $\widetilde{\Gamma}$-line. The quotient of the line is a $\Gamma$-cycle of homotopy class $g$. Conversely, a $\Gamma$-cycle, say of homotopy class $g$, determines a $g$-invariant dynamic plane, thus a $g$-invariant leaf of $\widetilde{\mathcal{L}}$, the image of which contains a closed orbit of homotopy class $g$. See \cite{LMT23a} for more details about this argument. This shows that $\mathcal{C}_{\mathcal{F}}^\vee = \mathcal{C}_\Gamma$.
Meanwhile it is shown in \cite[Theorem 5.1]{LMT24} that $\mathcal{C}_\Gamma = \mathcal{C}_\Phi$.
\end{proof}

To generalize the above result to the sutured setting, fix, for the rest of this section, a Handel-Miller map $f: L \to L$. Let $\overline{M_f}$ be the compactified mapping torus of $f$ with depth one foliation $\mathcal{F}$ and associated unstable Handel-Miller lamination $\mathcal{L}$. Let $B$ be an arbitrary veering branched surface fully carrying $\mathcal{L}$. 
We emphasize that while such a branched surface was constructed in \Cref{prop:buexist}, we are \emph{not} assuming that $B$ comes from our construction. However, in \Cref{sec:hmvbsunique} we show that $B$ indeed arises from our construction.

We have defined the dual graph and flow graph of $B$ in \Cref{subsec:dynamicplanes}. If we define $\mathcal{C}_\Gamma$ and $\mathcal{C}_\Phi$ exactly as in \Cref{thm:closedcones}, \Cref{prop:flowgraphdualgraph} tells us that $\mathcal{C}_\Gamma=\mathcal{C}_\Phi$. Also, in \Cref{subsec:reviewfolcone}, we discussed how foliation cones serve as a natural generalization to Thurston fibered cones.

However, some thought reveals that the equality between $\mathcal{C}_\mathcal{F}^\vee$ and $\mathcal{C}_\Gamma = \mathcal{C}_\Phi$ and its proof will not carry over immediately. The branched surface may contain annulus sectors, which do not carry a canonical dynamic orientation. See \Cref{eg:bookofIbundles} for an instance where this happens. To address this, we need to impose a requirement on dynamic orientation separately. Before we do that we set up some notation.

Recall the correspondence between periodic leaves of $\widetilde{\mathcal{L}}$ and periodic dynamic planes of $\widetilde{B}$ in \Cref{prop:dynamicplanesandhmleaves}. Going forward, we will be implicitly applying this correspondence.

The quotient of a periodic leaf $A$ in $\overline{M_f}$ is a leaf of $\mathcal{L}$ whose interior is homeomorphic to an open annulus or Möbius band. Such a leaf must contain a periodic orbit of the Handel-Miller suspension flow as a core. While such a periodic orbit may not be unique, the flow direction on all of them will be oriented in the same way. With this in mind, we will refer to the dynamic orientation determined by the flow direction on this core as \textbf{the dynamic orientation on $A$}. 

Meanwhile, a periodic dynamic plane $D$, which is, say, invariant under $g \in \pi_1(M)$, contains a $g$-invariant bi-infinite $\wt \Gamma$-line. Indeed, the proof of \Cref{prop:flowgraphdualgraph} shows that unless $D$ has an odd number of AB strips, there is a $g$-invariant bi-infinite $\widetilde{\Phi}$-line, and applying \Cref{prop:dualgraphcarries} to the quotient of $D$ under $g$, we can homotope this to a $g$-invariant bi-infinite $\widetilde{\Gamma}$-line. If $D$ has even width, the zig-zag in the middle AB region is a $g$-invariant bi-infinite $\widetilde{\Gamma}$-line. The quotient of such a line is an oriented core of the quotient of $D$, and again with a slight abuse of notation, we refer to the corresponding dynamic orientation on the image of $D$ as \textbf{the dynamic orientation on $D$}. 

One should of course check that this is well-defined, since there could be more than one $g$-invariant bi-infinite line on $D$. Suppose otherwise that we can find two such lines on $D$ which descend to two $\Gamma$-cycles of opposite homotopy class. Then by \Cref{prop:flowgraphdualgraph} we would be able to find two bi-infinite $\widetilde{\Phi}$-lines which descend to two (possibly non-primitive) $\Gamma$-cycles of opposite homotopy class. But by \Cref{lemma:dynamicplaneflowgraph} any two periodic $\widetilde{\Phi}$-lines must bound a union of AB strips, thus be oriented in the same way. Hence this shows that the dynamic orientation of a periodic dynamic plane is well-defined.

If a periodic leaf $A$ corresponds to a periodic dynamic plane $D$, we say that the dynamic orientation on $A$ is \textbf{compatible} with that on $D$ if the dynamic orientations on the images of $A$ and $D$ agree.
The condition that we need for the generalization of \Cref{thm:closeduniqueness} is that the dynamic orientation on each periodic leaf is compatible with that on its associated periodic dynamic plane. 

However, such a condition is difficult to verify in practice, since there are infinitely many periodic leaves to check. Fortunately, it turns out a simpler criterion suffices here, which we take as our definition in \Cref{defn:compatiblycarry} below. In \Cref{prop:compatiblycarrydefns} we will show the equivalence between \Cref{defn:compatiblycarry} and the condition stated in the last paragraph.

To state \Cref{defn:compatiblycarry}, we need some more setup. Recall that by \Cref{lem:bdylams}(a), $\Lambda_+^\infty=\mathcal{L} \cap R_+$ is spiraling. In fact, if we use the language of this section, \Cref{lem:bdylams}(b) says further that the orientation of each closed leaf of $\Lambda_+^\infty$ is compatible with the dynamic orientation of the closed leaf of $\mathcal{L}$ that contains it. Similarly, each circular sink component of the boundary train track $T_+^\infty=B \cap R_+$ carries an orientation coming from the source orientation of $T_+^\infty$. Each circular leaf of $\Lambda_+^\infty$ corresponds to a circular sink component of the boundary train track, so we can ask if the orientations of these loops agree.

Also recall that each leaf in the boundary of a complementary region of $\mathcal{L}$ contains a periodic orbit of the Handel-Miller suspension flow, thus carries a natural dynamic orientation. Similarly, each face of a complementary region of $B$ by definition carries a dynamic orientation. Each leaf in the boundary of a complementary region of $\mathcal{L}$ corresponds to a face of a complementary region of $B$, so we can compare their dynamic orientations.

\begin{definition} \label{defn:compatiblycarry}
A veering branched surface $B=(B,V)$ in $M$ is said to \textbf{compatibly carry} $\mathcal{L}$ if:
\begin{enumerate}
    \item $B$ fully carries $\mathcal{L}$,
    \item Every circular leaf of the boundary lamination $\mathcal{L} \cap R_+$ is oriented compatibly with the corresponding circular sink component of the boundary train track $\beta=B \cap R_+$, and
    \item Every leaf in the boundary of a complementary region of $\mathcal{L}$ is oriented compatibly with the corresponding face of the corresponding complementary region of $B$.\qedhere
\end{enumerate}
\end{definition}

\begin{proposition} \label{prop:compatiblycarrydefns}
Suppose a veering branched surface $B=(B,V)$ fully carries $\mathcal{L}$. Then $B$ compatibly carries $\mathcal{L}$ if and only if the dynamic orientation on each periodic leaf is compatible with that on its associated periodic dynamic plane.
\end{proposition}
\begin{proof}
Suppose the dynamic orientation on each periodic leaf is compatible with that on its associated periodic dynamic plane. Then for every circular leaf $l$ of the boundary lamination, there is some annulus or Möbius band leaf of $\mathcal{L}$ containing $l$, whose dynamic orientation is given by the orientation of $l$. Such a leaf lifts to a periodic leaf of $\widetilde{\mathcal{L}}$, and the corresponding periodic dynamic plane of $\widetilde{B}$ must have a boundary component along $\widetilde{B} \cap \widetilde{R_+} \subset \widetilde{\Gamma}$ which projects down to the circular sink component of $\beta$ carrying $l$, and its orientation determines the dynamic orientation of the dynamic plane. Since the dynamic orientation of the leaf agrees with that of the dynamic plane, the orientation of $l$ agrees with that of the circular sink component.

For every leaf in the boundary of a complementary region of $\mathcal{L}$, consider a lift to $\widetilde{\mathcal{L}}$ and the corresponding dynamic plane. The quotient of this dynamic plane meets the corresponding face of the complementary region of $B$, and we claim that their dynamic orientations agree. This can be seen by taking a periodic orbit of $V$ on the face, which defines the dynamic orientation of the face, and applying \Cref{prop:dualgraphcarries} to homotope it to a $\Gamma$ cycle, which defines the dynamic orientation of the dynamic plane. Since the dynamic orientation of the leaf agrees with that of the dynamic plane, it in turn agrees with that of the face.

For the converse, it suffices to show that the dynamic orientation condition always holds for leaves of $\mathcal{L}$ that do not have a circular boundary component that lies on $R_+$, since for those leaves, the argument in the first paragraph goes through.

To that end, consider a periodic leaf $A$ of $\widetilde{\mathcal{L}}$ whose image in $\mathcal{L}$ does not have a circular boundary component that lies on $R_+$, and let $D$ be its corresponding periodic dynamic plane. We first consider the case when $D$ does not contain the lift of an annulus or Möbius band sector. 

Consider a bi-infinite periodic $\widetilde{\Phi}$-line on $D$. It is positively transverse to $\brloc(\wt B)$, i.e. sectors of $\widetilde{B}$ must merge with $D$ in the forward direction. Equivalently, this can be expressed by saying that the holonomy of $\mathcal{L}$ around the image of this line is contracting on at least one side. By the transversely contracting dynamics of the unstable Handel-Miller lamination (\Cref{lemma:hmlamholonomy}), the periodic orbit in the image of $A$ must be oriented in the same direction as the $\Phi$-cycle.

If $D$ contains the lift of an annulus or Möbius band sector, the sector must lie on the face of some complementary region of $B$. As reasoned in the second paragraph above, the dynamic orientation of $D$ agrees with that of the face, which by assumption agrees with that of the corresponding leaf of $\mathcal{L}$. This completes the proof of the converse.
\end{proof}

Equipped with the notion of compatibly carrying, we can proceed with the proof of our theorem.

\begin{theorem} \label{thm:folcone}
Let $f:L \to L$ be an endperiodic map, let $Q$ be the compactified mapping torus of $f$ with depth one foliation $\mathcal{F}$, and let $\mathcal{L}$ be the unstable Handel-Miller lamination in $Q$ associated to $f$. Let $\mathcal{C}_{\mathcal{F}} \subset H_2(Q, \partial Q)$ be the foliation cone associated to $\mathcal{F}$. Meanwhile, let $B$ be a veering branched surface compatibly carrying $\mathcal{L}$, with dual graph $\Gamma$ and flow graph $\Phi$. Let $\mathcal{C}_\Gamma$ and $\mc C_\Phi$ be the cones in $H_1(Q)$ positively generated by the cycles of $\Gamma$ and $\Phi$, respectively.
Then 
\[
\mathcal{C}_{\mathcal{F}}^\vee = \mathcal{C}_\Gamma = \mathcal{C}_\Phi.
\]
\end{theorem}
\begin{proof}
As pointed out before, \Cref{prop:flowgraphdualgraph} implies that $\mathcal{C}_\Gamma = \mathcal{C}_\Phi$, hence it suffices to show that $\mathcal{C}_{\mathcal{F}}^\vee = \mathcal{C}_\Gamma$. By the discussion in \Cref{subsec:reviewfolcone}, it suffices to show that every closed orbit of the Handel-Miller suspension flow is homotopic to a $\Gamma$-cycle, and that every $\Gamma$-cycle is homotopic to a closed orbit of the Handel-Miller suspension flow.

Every closed orbit of the flow lies on some leaf of $\mathcal{L}$, whose lift to $\widetilde{\mathcal{L}}$ is a periodic leaf, thus corresponds to some periodic dynamic plane. The periodic dynamic plane carries some periodic $\Gamma$-line, which quotients down to a $\Gamma$-cycle. By \Cref{prop:compatiblycarrydefns}, the $\Gamma$-cycle is homotopic to the closed orbit (as opposed to its opposite).

Conversely, every $\Gamma$-cycle determines a periodic dynamic plane, which corresponds to a periodic leaf of $\widetilde{\mathcal{L}}$ whose image contains a closed orbit. By \Cref{prop:compatiblycarrydefns}, the $\Gamma$-cycle is homotopic to the closed orbit.
\end{proof}

\section{Uniqueness of veering branched surfaces which carry Handel-Miller laminations} \label{sec:hmvbsunique}

The following theorem about veering triangulations is essentially known to experts, although perhaps with different terminology. We include a proof for completeness.

\begin{theorem} \label{thm:closeduniqueness}
Let $f: S \to S$ be a pseudo-Anosov map on a finite type surface. Let $\mathcal{L}$ be the unstable lamination of the suspension flow on the mapping torus. If $(B,V)$ and $(B', V')$ are two unstable veering branched surfaces fully carrying $\mc L$, then $B$ is isotopic to $B'$. 
\end{theorem}

That is, the unstable veering branched surface carrying $\mathcal{L}$ has a unique underlying branched surface. Note the vector field cannot be expected to be unique since there are many vector fields that make a branched surface into an unstable branched surface (This was discussed in \Cref{rmk:vf=combinatorialdata}). 

\begin{proof}[Proof of \Cref{thm:closeduniqueness}]
Let $\Gamma$ be the dual graph for $B$. By \cite[Proposition 5.11]{LMT24} there exists an oriented surface $S$ such that $M\cut S\cong S\times[0,1]$ (so $S$ is a fiber of a fibration $M\to S^1$), $S$ is positively transverse to $\Gamma$, and $S$ pairs positively with each cycle in $\Gamma$.
It follows that $\Gamma\cut S$ is an acyclic digraph. As such we can choose a function $h\colon \Gamma\cut S\to[0,1]$ which is monotonically increasing on edges of $\Gamma\cut S$, and takes the values 1 and 0 at all points of $\Gamma\cut S$ lying on $R_+(M\cut S)$ and $R_-(M\cut S)$, respectively. 

Each sector of $B\cut S$ is a disk, over which we can continuously extend $h$ so that each level set in a sector is a single properly embedded line segment transverse to $\Gamma$, and $h$ takes the values 1 and 0 exactly on points of $R_+(M\cut S)$ and $R_-(M\cut S)$, respectively. 

Finally, each component of $(M\cut S)\cut (B\cut S)$ is diffeomorphic to $A\times [0,1]$ where $A$ is either a disk with $\ge 3$ cusps on its boundary or an annulus with one smooth boundary component and $\ge 1$ cusps on the other. We can extend $h$ to each $A\times[0,1]$ region so that each level set of $h|_{A\times[0,1]}$ is diffeomorphic to $A$. Together, the level sets determine an isotopy from $S\times 1$ to $S\times 0$. Looking at the intersection with $B$ at each time slice of this isotopy, we obtain a splitting sequence of train tracks which is periodic under $f$. The branched surface $B$ is the suspension of this splitting sequence.

Note that every sector of $B$ is a disk: otherwise there would be two parallel $\u$-cusped tori. These would correspond to complementary regions of the unstable lamination with core curves $c_1$, $c_2$ such that $c_1^{k_1}$ is homotopic to $c_2^{k_2}$ for some positve integers $k_1,k_2$, a contradiction. 
Hence $B$ has a dual ideal triangulation, which we call $\Delta$. By \cite[Proposition 3.22]{Tsa23}, $\Delta$ is a veering triangulation; by the arguments above, $\Delta$ is layered. This determines $\Delta$ up to isotopy by \cite{Ago10}. As its dual, $B$ is unique up to isotopy also.
\end{proof}

We note that a shorter proof of \Cref{thm:closeduniqueness} is possible  using \cite[Theorem E]{LMT24}. 
In fact, such a proof would show a stronger version of the theorem where we just need to assume that the dual graphs of $B$ and $B'$ generate the same cone in homology. See \Cref{sec:conequestions} for more discussion. 

The goal of this section is to transport this particular proof of \Cref{thm:closeduniqueness} to our setting of endperiodic maps. Some care is needed because as mentioned earlier, when we construct a veering branched surface carrying $\mc L$ in \Cref{prop:buexist}, the resulting branched surface depends on the initial choice of boundary train track. However we will show that this is the only obstacle to uniqueness, in the sense that that a veering branched surface fully carrying $\mathcal{L}$ with a specified efficient boundary train track is unique up to isotopy. Our strategy is to mimic the proof of \Cref{thm:closeduniqueness} above: given a veering branched surface $B$ carrying $\mathcal{L}$, we will construct a surface whose intersection with $B$ gives us a splitting sequence of train tracks as we sweep it around the sutured manifold.
This will show that $B$ can be obtained from our construction in \Cref{sec:hmvbs} using this splitting sequence. We will then apply the results in \Cref{subsec:hmvbsunique} to conclude that $B$ is the only  veering branched surface fully carrying $\mc L$ with this boundary train track.

A natural question, then, is \emph{how} these veering branched surfaces differ from one another as we vary the boundary train track. We address this question in \Cref{subsec:shiftmoves} by showing that they are related by a family of moves which we describe explicitly. 

\subsection{Building surfaces from veering branched surfaces} \label{subsec:dualcelldecomp}

In this section we let $(B,V)$ be an unstable veering branched surface in a sutured manifold $Q$. 

We let $N$ be a standard neighborhood of $B$ with the property that along each component of $\del N$ lying in a $\u$-cusped product, $V$ points into $N$. This can be arranged by, in a $\u$-cusped product $P\cong \Sigma\times[0,1]$, isotoping $\Sigma=\Sigma_0$ along the orbits of $V|_P$ to obtain a family of surfaces $\{\Sigma_t\mid t\in[0,1]\}$ such that the $\u$-faces of $P$ are within some small $\epsilon>0$ from $\Sigma_1$ and the points in $\Sigma_1$ further than $\epsilon$ from the $\u$-faces lie in $R_+$. We then set the components of $\Sigma_1\cut R_+$ to be components of $\del N$. See \Cref{fig:specialneighborhood} for a schematic.

\begin{figure}
    \centering    
    \fontsize{12pt}{12pt}\selectfont
    \resizebox{!}{1.5in}{
\begingroup%
  \makeatletter%
  \providecommand\color[2][]{%
    \errmessage{(Inkscape) Color is used for the text in Inkscape, but the package 'color.sty' is not loaded}%
    \renewcommand\color[2][]{}%
  }%
  \providecommand\transparent[1]{%
    \errmessage{(Inkscape) Transparency is used (non-zero) for the text in Inkscape, but the package 'transparent.sty' is not loaded}%
    \renewcommand\transparent[1]{}%
  }%
  \providecommand\rotatebox[2]{#2}%
  \newcommand*\fsize{\dimexpr\f@size pt\relax}%
  \newcommand*\lineheight[1]{\fontsize{\fsize}{#1\fsize}\selectfont}%
  \ifx\svgwidth\undefined%
    \setlength{\unitlength}{243.75494307bp}%
    \ifx\svgscale\undefined%
      \relax%
    \else%
      \setlength{\unitlength}{\unitlength * \real{\svgscale}}%
    \fi%
  \else%
    \setlength{\unitlength}{\svgwidth}%
  \fi%
  \global\let\svgwidth\undefined%
  \global\let\svgscale\undefined%
  \makeatother%
  \begin{picture}(1,0.41971613)%
    \lineheight{1}%
    \setlength\tabcolsep{0pt}%
    \put(0,0){\includegraphics[width=\unitlength,page=1]{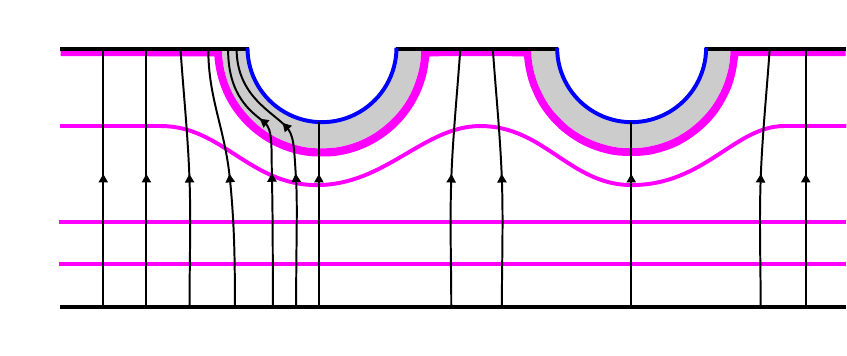}}%
    \put(0.24527489,0.01016806){\color[rgb]{0,0,0}\makebox(0,0)[lt]{\lineheight{1.25}\smash{\begin{tabular}[t]{l}$\Sigma \subset R_-$\end{tabular}}}}%
    \put(-0.00372589,0.17100948){\color[rgb]{0,0,0}\makebox(0,0)[lt]{\lineheight{1.25}\smash{\begin{tabular}[t]{l}$P$\end{tabular}}}}%
    \put(0.17929047,0.38183229){\color[rgb]{0,0,0}\makebox(0,0)[lt]{\lineheight{1.25}\smash{\begin{tabular}[t]{l}$R_+$\end{tabular}}}}%
    \put(0,0){\includegraphics[width=\unitlength,page=2]{fig_specialneighborhood.pdf}}%
  \end{picture}%
\endgroup%
}
    \caption{By using the product structure of $\u$-cusped products, we can build a standard neighborhood $N$ of $B$ with the property that along each component of $\del N$ lying in a $\u$-cusped product, $V$ points into $N$. Here the gray areas indicate parts of $N$ in the $\u$-cusped product $P$.}
    \label{fig:specialneighborhood}
\end{figure}

Let $a$ be a transient annulus sector of $B$. Since $V|_a$ points inward on one component of $\del a$ and outward on the other, we can always replace $V$ by a vector field $V'$ such that each orbit of $V'|_a$ is a properly embedded arc, with the additional properties that $V$ and $V'$ agree outside of $a$ and $(B,V')$ is a veering branched surface. 
This simplifies some upcoming definitions and constructions, so we will make the following assumption going forward:

\begin{convention}\label{conv:goodV}
For any transient annulus $a$ of $(B,V)$, all orbits of $V|_a$ are properly embedded arcs.
\end{convention}

Recall the definition of the dual graph $\Gamma$ of $B$ from \Cref{defn:dualgraph}. The complementary regions of $\Gamma$ in $B$ are disks and annuli, and it is now to our advantage to subdivide the latter into disks. There are multiple ways to do this that would work for our purposes, but we have chosen one that feels relatively natural. 

\begin{definition}
An \textbf{extended dual graph} for $B$ is a directed graph embedded in $B$, containing the dual graph $\Gamma$ as a subgraph, intersecting each sector $s$ of $B$ as follows:
\begin{itemize}
    \item if $s$ is a source sector, then for each vertex $v$ on $\del s$ we augment $\Gamma$ by adding a directed edge from the core $\Gamma$-cycle of $s$ to $v$. 
    \item if $s$ is a transient annulus, then for each vertex $v$ in the outwardly cooriented component of $\del s$, we augment $\Gamma$ by adding the $V$-trajectory connecting $v$ to the other boundary component of $s$, oriented compatibly with $V$. (This uses \Cref{conv:goodV}).
    \item For any other sector $s$, we let $\Gamma^+|_s=\Gamma|_s$.
\end{itemize}
See \Cref{fig:dualgraphextend}.
\end{definition}

\begin{figure}
    \centering
    \fontsize{12pt}{12pt}\selectfont
    \resizebox{!}{6cm}{
\begingroup%
  \makeatletter%
  \providecommand\color[2][]{%
    \errmessage{(Inkscape) Color is used for the text in Inkscape, but the package 'color.sty' is not loaded}%
    \renewcommand\color[2][]{}%
  }%
  \providecommand\transparent[1]{%
    \errmessage{(Inkscape) Transparency is used (non-zero) for the text in Inkscape, but the package 'transparent.sty' is not loaded}%
    \renewcommand\transparent[1]{}%
  }%
  \providecommand\rotatebox[2]{#2}%
  \newcommand*\fsize{\dimexpr\f@size pt\relax}%
  \newcommand*\lineheight[1]{\fontsize{\fsize}{#1\fsize}\selectfont}%
  \ifx\svgwidth\undefined%
    \setlength{\unitlength}{382.83442846bp}%
    \ifx\svgscale\undefined%
      \relax%
    \else%
      \setlength{\unitlength}{\unitlength * \real{\svgscale}}%
    \fi%
  \else%
    \setlength{\unitlength}{\svgwidth}%
  \fi%
  \global\let\svgwidth\undefined%
  \global\let\svgscale\undefined%
  \makeatother%
  \begin{picture}(1,0.42473015)%
    \lineheight{1}%
    \setlength\tabcolsep{0pt}%
    \put(0,0){\includegraphics[width=\unitlength,page=1]{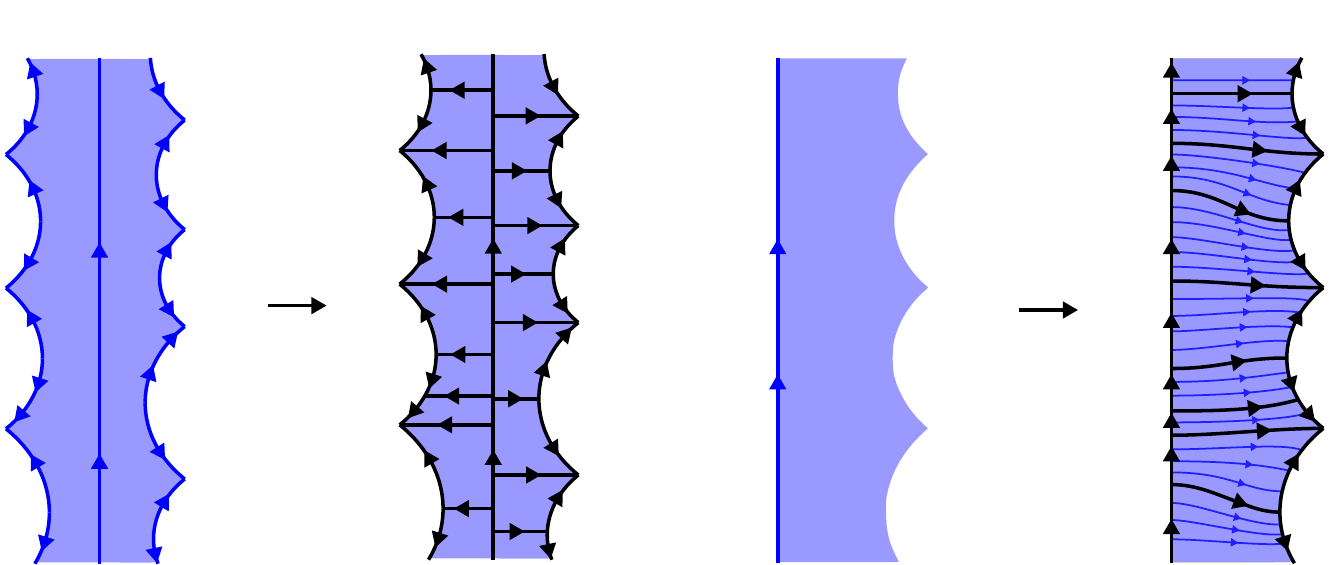}}%
    \put(0.06051085,0.40060756){\color[rgb]{0,0,0}\makebox(0,0)[lt]{\lineheight{1.25}\smash{\begin{tabular}[t]{l}$\Gamma$\end{tabular}}}}%
    \put(0.35775418,0.40060908){\color[rgb]{0,0,0}\makebox(0,0)[lt]{\lineheight{1.25}\smash{\begin{tabular}[t]{l}$\Gamma^+$\end{tabular}}}}%
    \put(0.6156764,0.40060727){\color[rgb]{0,0,0}\makebox(0,0)[lt]{\lineheight{1.25}\smash{\begin{tabular}[t]{l}$\Gamma$\end{tabular}}}}%
    \put(0.91291979,0.4006088){\color[rgb]{0,0,0}\makebox(0,0)[lt]{\lineheight{1.25}\smash{\begin{tabular}[t]{l}$\Gamma^+$\end{tabular}}}}%
    \put(0,0){\includegraphics[width=\unitlength,page=2]{fig_dualgraphextend_squish.pdf}}%
  \end{picture}%
\endgroup%
}
    \caption{Obtaining the extended dual graph $\Gamma^+$ from the dual graph $\Gamma$. For transient annulus sectors we use the data of $V$, while for source sectors we do not.}
    \label{fig:dualgraphextend}
\end{figure}

\begin{lemma}\label{lem:sourcearc}
Let $s$ be a source sector of $B$. Let $\gamma$ be a properly embedded arc in $s$ that is positively transverse to $\Gamma^+$. Then $\gamma$ is homotopic in $s$, rel $\del\gamma$, to an arc positively transverse to $V$.
\end{lemma}

\begin{figure}
    \centering
    \resizebox{!}{2.5in}{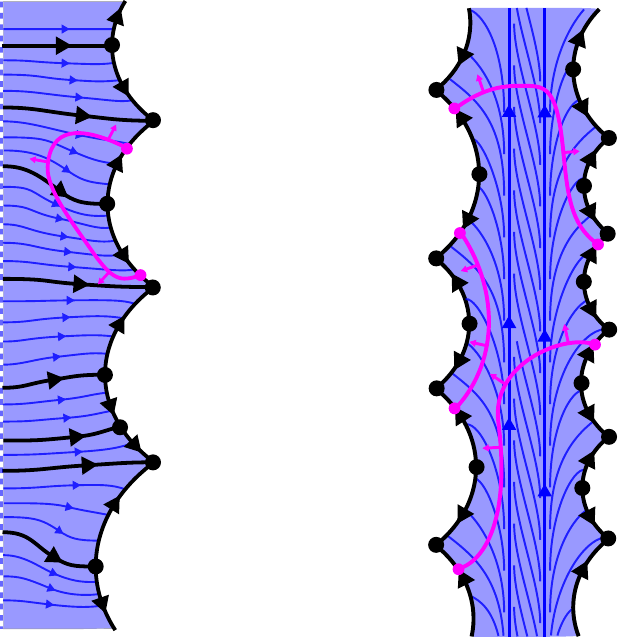}
    \caption{Left: An impossible configuration of $\gamma$ in $s$. Right: A possible configuration.}
    \label{fig:arcsource}
\end{figure}

\begin{proof}
In general, the structure of $V|_s$ is as follows. Since $V|_s$ points outward along $\del s$, there is at least one closed orbit of $V|_s$, and possibly infinitely many. All closed orbits have orientation agreeing with the dynamic orientation of $s$. Every orbit ending on $\del s$ spirals onto a closed orbit in the backward direction, and the direction of the spiraling is compatible with the closed orbit's orientation. Every non-closed orbit disjoint from $\del s$ spirals onto closed orbits in both directions such that all orientations are compatible. See the right side of \Cref{fig:arcsource}.

Let $\wt s$ be the universal cover of $s$, let $\wt \gamma$ be the lift of $\gamma$ to $\wt s$, and let $V_{\wt s}$ be the lift of $V|_s$ to $\wt s$. 

Suppose that $\gamma$ has inessential intersection with the core of $s$. In this case, since $\gamma$ is positively transverse to $\Gamma^+$, the lift $\wt \gamma$ must connect two portions of $\del \wt s$ which point toward each other. That is, the picture cannot be as shown on the left of \Cref{fig:arcsource}. Hence $\wt\gamma$ can be homotoped into a neighborhood of $\del \wt s$ where it can evidently be made positively transverse to $V_{\wt s}$. See the right side of \Cref{fig:arcsource}. This homotopy projects to one in $s$. If the resulting arc $\gamma'$ is not embedded, then there is an innermost bigon complementary region of $\gamma'$ in $s$ foliated by flow segments of $V$ joining the two sides of the bigon. Such a bigon can be eliminated while maintaining transversality by flowing part of $\gamma'$ along $V$. After finitely many steps we have produced the desired homotopy of $\gamma$.

Alternatively, suppose that $\gamma$ intersects the core of $s$ essentially. Since $\gamma$ is positively transverse to the core $\Gamma^+$-cycle in $s$, the endpoints of $\gamma$ must lie on segments of $\wt s$ that are oriented compatibly with the dynamic orientation of $s$. From the description of $V|_s$ given at the beginning of the proof, we see that we can homotope $\gamma$ so that it is positively transverse to $V$. See the right side of \Cref{fig:arcsource}.
\end{proof}

Let $e$ be an edge of $\Gamma$ which lies at the bottom of a disk sector $s$, meaning that the sector $s$ of $B$ into which the maw vector field points along $e$ is not an annulus or Möbius band. Let $p$ be the terminal vertex of $e$.

Let $A$ be the collection of points in the top of $s$ which are sinks of $s$ (these will all be corners). 
Let $A'$ be the set of points in the bottom of $s$ obtained by flowing backward along $V$ from points in $A$. Let $U(e)$ be the component of $e\setminus A'$ containing $p$ (see \Cref{fig:diskdivision}). If $e$ is an edge of $\Gamma$ that does not lie at the bottom of a disk sector let $U(e)=e$. Let
\[
U=\bigcup_e U(e)
\]
where the union is taken over all $\Gamma$-edges.

\begin{figure}
    \centering
    \fontsize{12pt}{12pt}\selectfont
    \resizebox{!}{1.5in}{
\begingroup%
  \makeatletter%
  \providecommand\color[2][]{%
    \errmessage{(Inkscape) Color is used for the text in Inkscape, but the package 'color.sty' is not loaded}%
    \renewcommand\color[2][]{}%
  }%
  \providecommand\transparent[1]{%
    \errmessage{(Inkscape) Transparency is used (non-zero) for the text in Inkscape, but the package 'transparent.sty' is not loaded}%
    \renewcommand\transparent[1]{}%
  }%
  \providecommand\rotatebox[2]{#2}%
  \newcommand*\fsize{\dimexpr\f@size pt\relax}%
  \newcommand*\lineheight[1]{\fontsize{\fsize}{#1\fsize}\selectfont}%
  \ifx\svgwidth\undefined%
    \setlength{\unitlength}{370.09674882bp}%
    \ifx\svgscale\undefined%
      \relax%
    \else%
      \setlength{\unitlength}{\unitlength * \real{\svgscale}}%
    \fi%
  \else%
    \setlength{\unitlength}{\svgwidth}%
  \fi%
  \global\let\svgwidth\undefined%
  \global\let\svgscale\undefined%
  \makeatother%
  \begin{picture}(1,0.31576248)%
    \lineheight{1}%
    \setlength\tabcolsep{0pt}%
    \put(0,0){\includegraphics[width=\unitlength,page=1]{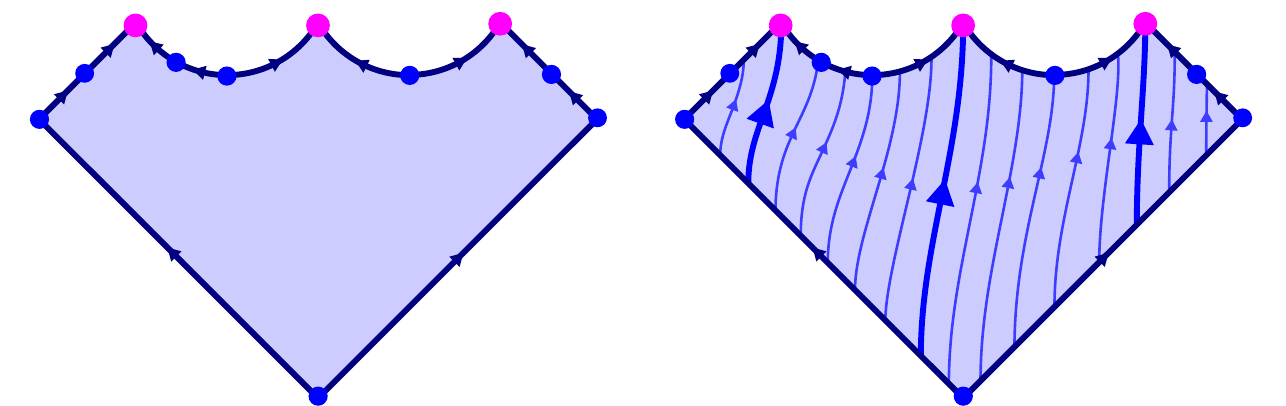}}%
    \put(-0.00245396,0.215072){\color[rgb]{0,0,0}\makebox(0,0)[lt]{\lineheight{1.25}\smash{\begin{tabular}[t]{l}$p$\end{tabular}}}}%
    \put(0.10335523,0.0958277){\color[rgb]{0,0,0}\makebox(0,0)[lt]{\lineheight{1.25}\smash{\begin{tabular}[t]{l}$e$\end{tabular}}}}%
    \put(0,0){\includegraphics[width=\unitlength,page=2]{fig_diskdivision.pdf}}%
    \put(0.47737849,0.1455127){\color[rgb]{0,0,0}\makebox(0,0)[lt]{\lineheight{1.25}\smash{\begin{tabular}[t]{l}$U(e)$\end{tabular}}}}%
    \put(0.93882512,0.29081123){\color[rgb]{1,0,1}\makebox(0,0)[lt]{\lineheight{1.25}\smash{\begin{tabular}[t]{l}$A$\end{tabular}}}}%
    \put(0.92597571,0.09264819){\color[rgb]{0,0.50196078,0}\makebox(0,0)[lt]{\lineheight{1.25}\smash{\begin{tabular}[t]{l}$A'$\end{tabular}}}}%
  \end{picture}%
\endgroup%
}
    \caption{Defining the set $U(e)$ for a $\Gamma$-edge $e$ in the bottom of a disk sector.}
    \label{fig:diskdivision}
\end{figure}

\begin{lemma}\label{lem:diskarc}
Let $\gamma$ be a cooriented arc which is properly embedded in a disk sector $s$ of $b$ and positively transverse to $\Gamma$. Further suppose that $\del s\subset U$. Then $\gamma$ can be made positively transverse to $V|_s$ by a homotopy fixing the endpoints of $\gamma$.  
\end{lemma}

\begin{figure}
    \centering
    \resizebox{!}{1.5in}{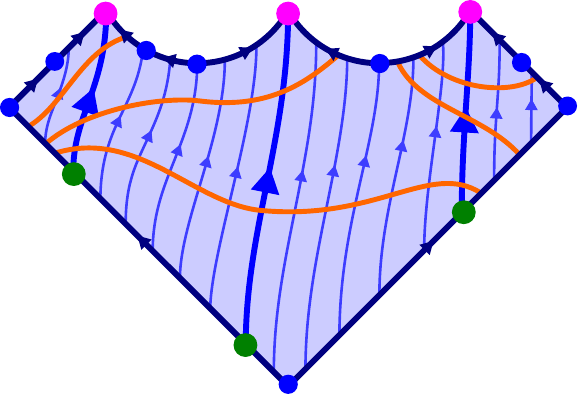}
    \caption{Any properly embedded arc in a disk sector with endpoints in $U$, positively transverse to $\Gamma$, can be made transverse to $V$ by an endpoint-fixing homotopy.}
    \label{fig:diskarcs}
\end{figure}

\begin{proof}
It is either the case that (a) both endpoints of $\gamma$ lie in the top of $s$, (b) both lie in the bottom of $s$, or (c) one lies in the top and the other in the bottom. 

In case (a) the fact that $\gamma$ is positively transverse to $\Gamma$ implies that the $\Gamma$-edges on which the endpoints of $\gamma$ lie point toward each other. Similarly in case (b) the $\Gamma$-edges on which the endpoints of $\gamma$ lie point away from each other. A picture then makes clear that $\gamma$ is positively transverse to $V$ after an endpoint-fixing homotopy (see \Cref{fig:diskarcs}).

In case (c), we use the fact that in particular the endpoint of $\gamma$ lying on the bottom of $s$ lies in the set $U$. \Cref{fig:diskarcs} shows that this is precisely the condition needed to guarantee that $\gamma$ is positively transverse to $V$ after an endpoint-fixing homotopy.
\end{proof}

\begin{figure}
    \centering
    \resizebox{!}{1.4in}{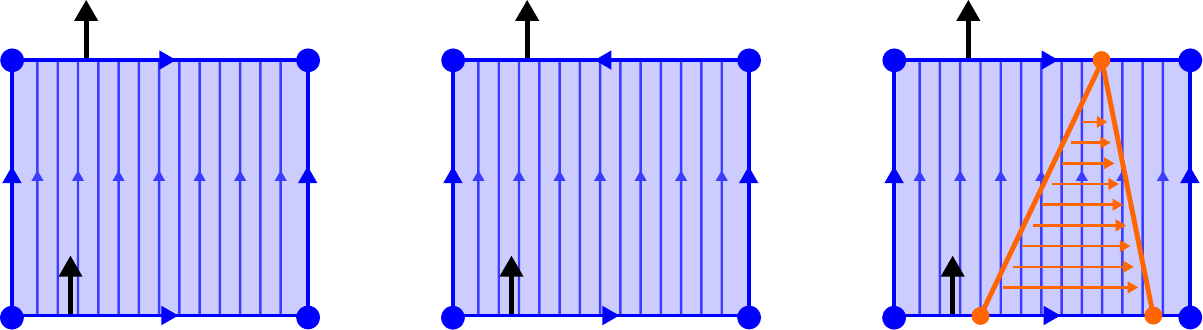}
    \caption{Left and center: for a transient annulus sector $s$, each component of $s\cut \Gamma^+$ is of one of these forms. Right: any properly embedded cooriented arc positively transverse to $\del s$ may be homotoped rel $\del s$ to be positively transverse to $v$; shown is the only case where it is necessary to move a boundary point of the arc.}
    \label{fig:transienthomotopy}
\end{figure}

\begin{lemma}\label{lem:transientarc}
Let $s$ be a transient annulus sector of $B$, and let $\gamma$ be a cooriented arc properly embedded in $s$ which is positively transverse to $\Gamma^+$. Let $p$ be the endpoint of $\gamma$ lying on the inward-pointing component of $\del s$. Then $\gamma$ can be made positively transverse to $V|_s$ by a homotopy which either fixes both endpoints of $\gamma$ or moves $p$ closer to the terminal vertex of the $\Gamma$-edge on which it lies.
\end{lemma}

\begin{proof}
Each component of $s\cut \Gamma^+$ is of the form shown in the left and center of \Cref{fig:transienthomotopy}. The arc $\gamma$ intersects each such component $c$ in an sub-arc $\gamma_c$ which is positively transverse to $\del c$. We may clearly homotope $\gamma_c$ to be postively transverse to $V$, and we can even choose the homotopy to fix the endpoints of $\gamma_c$ unless
$\gamma_c$ joins the outward and inward-pointing portions of $\del c$. In this case it may be necessary to move the endpoint on the inward pointing boundary along the orientation coming from $\Gamma^+$ (see \Cref{fig:transienthomotopy}). 
Applying this to all sub-arcs of $\gamma$ proves the lemma.
\end{proof}

Note that as remarked above, an extended dual graph is not uniquely determined. However, by the next proposition this freedom does not affect the cycles of $\Gamma^+$ (recall our convention is that cycles are directed).

\begin{proposition} \label{prop:extdualgraphnonewcycles}
Each cycle in $\Gamma^+$ lies in $\Gamma$.
\end{proposition}
\begin{proof}
Suppose there is a cycle $c$ in $\Gamma^+$ containing an edge $e$ of $\Gamma^+ \backslash \Gamma$. By the definition of $\Gamma^+$, the edge $e$ must lie interior to an annulus or Möbius band sector $s$. We see that $s$ is in fact a transient annulus, for the only cycles passing through the interior of sources are the cycles contained in their cores, which are also $\Gamma$-cycles. 

Following $c$ backward, $c$ must contain an edge entering $s$ from another annulus or Möbius band sector $s'$. Repeating the argument on $s'$, and so on, we get a cycle of transient annuli whose union is a torus, contradicting \Cref{defn:vbs}.
\end{proof}

In particular \Cref{prop:extdualgraphnonewcycles} implies that the cone in $H_1(Q)$ generated by directed cycles of $\Gamma^+$ is independent of the choice of extension.

\begin{lemma}\label{lem:cocycle}
Let $z\in H^1(\Gamma^+)$ be an integral class such that for every directed cycle $\gamma$ of $\Gamma^+$, we have $z([\gamma])\ge 0$. Then $z$ is represented by a nonnegative rational 1-cocycle on $\Gamma^+$. 
\end{lemma}
\begin{proof}
The proof directly follows ideas from \cite{McM15} and \cite {LMT24}. Let $W=\R^E$ be the  space of real weights on the edges of $\Gamma_+$ and let $A\colon W\to H^1(\Gamma^+)$ be the map to cohomology. We can choose rational bases for $W$ and $H^1(\Gamma^+)$ so that $A$ is given by an integral matrix.

By \cite[Lemma 5.10]{LMT24}, there exists a nonnegative cocycle $c$ on $\Gamma^+$ representing $z$.
Let $S$ be the set of solutions to $Ax=z$, and let $F$ be the face of $\R_{\ge0}^E$ containing $c$ in its relative interior. Since $S\cap F$ is nonempty (it contains $c$) and cut out by integer equations and inequalities, it contains a rational point; this gives the required cocycle. 
\end{proof}

We next describe a construction that takes as input a certain type of cohomology class and produces a surface properly embedded in $N$, transverse to $V$. While reading the construction, the reader should consult \Cref{fig:Nsurface}. 

Before beginning, we briefly describe the structure of complementary regions of $B\cut \Gamma^+$. By construction, each component $c$ is diffeomorphic to a square. The edges on $\del c$ are oriented so that exactly one vertex is a source and one is a sink. The source is called the \textbf{bottom vertex} of $c$ and the sink is the \textbf{top vertex} of $c$. The components of $c\cut\{\text{top and bottom vertices}\}$ are called the \textbf{sides} of $c$.

\begin{figure}
    \centering
    \resizebox{!}{3in}{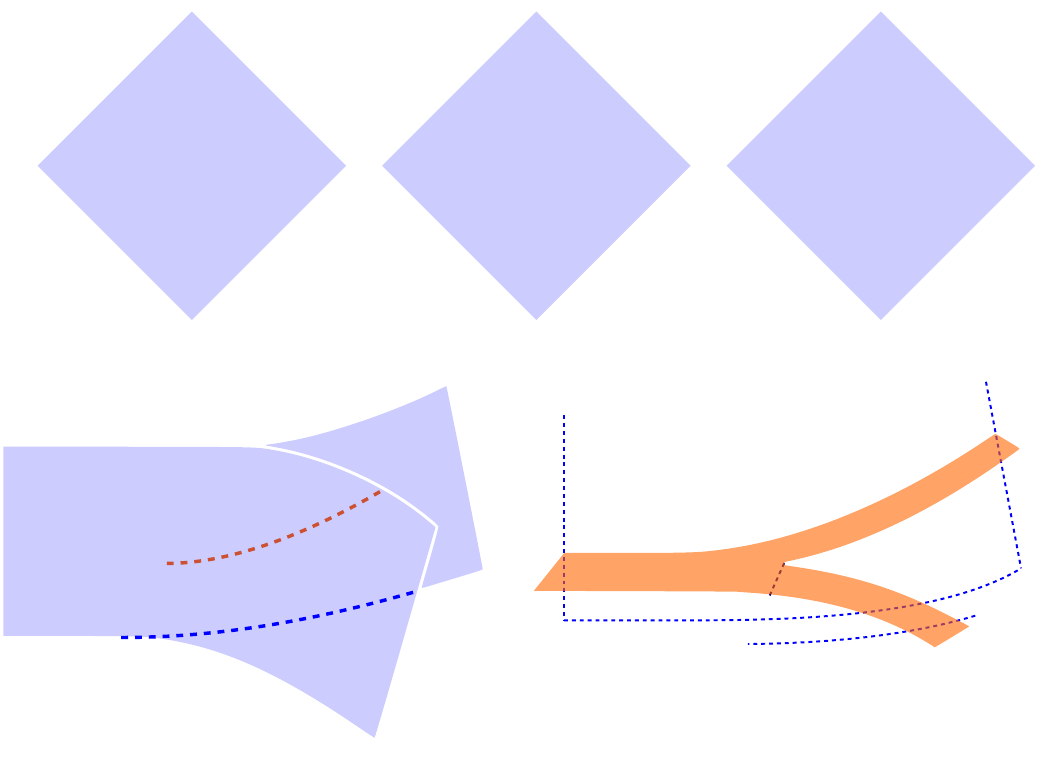}
    \caption{Some steps in \Cref{construction:Nsurface}.}
    \label{fig:Nsurface}
\end{figure}

\begin{construction}[$N$-surfaces]\label{construction:Nsurface}
Let $z\in H^1(Q)$ be an integral class such that for every directed cycle $\gamma$ of $\Gamma^+$, we have $z([\gamma])\ge0$; we will also think of $z$ as a class in $H^1(Q)$. By \Cref{lem:cocycle}, there exists a nonnegative $\mathbb{Q}$-valued 1-cocycle $c$ on $\Gamma^+$ representing the pullback of $z$ to $H^1(\Gamma^+)$ (\Cref{fig:Nsurface}(a)). Let $n$ be the least integer such that $nc$ is integer-valued. For each $\Gamma^+$ edge $e$, place $nc(e)$ points on $e$, close enough to the terminal vertex so that they all lie in $U$ (recall the definition of $U$ preceding \Cref{lem:diskarc}).

Let $\sigma$ be a component of $B\cut \Gamma^+$, and recall that $\sigma$ is a topological disk with a top and bottom $t$ and $b$ respectively, and edges in each side of $\sigma$ oriented from $b$ to $t$. Since $\del \sigma$ is contractible in $Q$ we must have $nz([\del \sigma])=0$, so the number of points we have placed on each side of $\sigma$ is equal (\Cref{fig:Nsurface}(b)). Hence we can join the points on each side of $\sigma$ by cooriented arcs, each arc joining the two sides of $\sigma$ and with coorientation compatible with the orientation of $\Gamma^+$ (\Cref{fig:Nsurface}(c)).

Doing this for each component of $B\cut \Gamma^+$, we obtain a cooriented train track $\eta$ in $B$; at each point of intersection between the train track and $\Gamma^+$, the coorientation of $\eta$ is compatible with the orientation of $\Gamma^+$ (\Cref{fig:Nsurface}(d)). 

We now homotope $\eta$ to be positively transverse to $V|_B$. First we arrange for all branches of $\eta$ contained in transient annulus sectors to be positively transverse to $V|_B$.  This is possible by \Cref{lem:transientarc}, and because there are no cycles of adjacent transient annulus sectors (as noted in the proof of \Cref{prop:extdualgraphnonewcycles}, the union of such sectors would produce a torus carried by $B$). This homotopy does not affect the property that all intersections of $\eta$ with $\Gamma$ lie in $U$. Hence we are free to apply \Cref{lem:sourcearc} and \Cref{lem:diskarc} in order to homotope the branches of $\eta$ to be positively transverse to $V|_B$.

We now extend $\eta$ to a cooriented surface $T$ properly embedded in $N$ and positively transverse to $V$. (\Cref{fig:Nsurface}(e)). By \cite[Lemma 1]{Thu86}, $T$ is a union of $n$ cooriented surfaces that each represent the Lefschetz dual of $z$. Pick one of these components and call it $S$. We call $S$ an \textbf{$N$-surface for $z$}. 
\end{construction}

\begin{lemma}\label{lem:correctclass}
Suppose $S$ is an $N$-surface for $z$.  Then $S$ is Lefschetz dual to $i^*(z)\in H^1(N)$, where $i\colon N\to Q$ is inclusion.
\end{lemma}

\begin{proof}
Since $N$ deformation retracts to $B$ and $B$ is cut into disks by $\Gamma^+$, it suffices to show that for any cycle $\gamma$ in $\Gamma^+$ (not necessarily directed), the algebraic intersection of $S$ with $\gamma$ is equal to $i^*(z)([\gamma])$. This property is immediate from the fact that $S$ was constructed using a cocycle representing the pullback of $z$ to $H^1(\Gamma)$.
\end{proof}

We will now explain how to extend an $N$-surface for a cohomology class $z$ over $Q\cut N$ to obtain a surface properly embedded in $Q$. One can do this for any $N$-surface, but we restrict ourselves to the case when $z$ pairs positively with each $\Gamma^+$-cycle since this is the only case we need and it simplifies the presentation.

For convenience, we will denote $Q\cut N$ by $N^c$. The components of $N^c$ are in bijection with the components of $Q\cut B$, and accordingly we will refer to them as $\u$-cusped tori and $\u$-cusped products.

For the statement of this lemma, recall that $\gamma$ denotes  $\del Q\cut(R_+\cup R_-)$ (see \Cref{subsec:compactifiedmappingtorus}).

\begin{lemma}\label{lem:bdyNsurface}
 Let $z\in H^1(Q)$ be an integral class that pairs positively with the class of every directed cycle of $\Gamma^+$. Let $S$ be an $N$-surface for $z$, and let $C$ be a component of $N^c$. Then:
\begin{enumerate}[label=(\alph*)]
\item If $C$ is a $\u$-cusped torus shell, then $S\cap C$ is a collection of closed parallel curves with compatible coorientation, each essential in $\del C$ with slope distinct from that of the $\u\u$-cusp curves of $C$.
\item If $C$ is a $\u$-cusped solid torus, then $S\cap C$ is a collection of meridians of $C$.
\item If $C$ is a $\u$-cusped product, then $S\cap C$ is a collection of closed curves and arcs. If we coorient the arcs using the coorientation of $S$, they may be completed by cooriented arcs in $C\cap R_+$ to obtain a family of arcs and curves whose boundary points lie only in $R_+\cap \gamma$. 
\end{enumerate}
\end{lemma}

\begin{proof}
Let $C$ be a component of $N^c$, and let $C'$ be the corresponding complementary region of $B$.

Let $c$ be a component of $S\cap C$, and suppose that $c$ bounds a disk $D$ in $\del C$. Since $S$ was constructed to be positively transverse to $\Gamma^+$, $c$ is disjoint from any $\u\u$-cusp curves of $C$. Hence corresponding to $D$ is a disk $D'$ immersed in $B$ with boundary lying in $S\cap B$. Since $S\cap B$ is positively transverse to $V$, this disk violates the Poincaré-Hopf theorem.
We conclude that no component of $S\cap C$ bounds a disk in the $\u$-faces of $C$. 

If $C$ is a $\u$-cusped torus, $S\cap C$ must intersect the $\u\u$-cusp curves of $C$ since $z$ pairs positively with each such curve. Also, $\del S$ is positively transverse to the $\u\u$-cusp curves by construction, so the orientations of the components of $S\cap C$ must agree. This proves part (a).

Consider the long exact sequence of the triple $(Q, N^c\cup \del Q,\del Q)$, which contains the subsequence
\begin{align}\label{eqn:LES1}
&H_2 (N^c\cup\del Q, \partial Q) \to H_2(Q, \partial Q) \to H_2(Q,  N^c\cup\del Q) \\
\to & H_1(N^c\cup\del Q, \partial Q) \to \cdots\nonumber.
\end{align}
Here $H_2(Q,N^c\cup\del Q) \cong H_2(N, \partial N)$ by excision and
\[
H_i( N^c\cup\del Q, \partial Q) \cong H_i(N^c, N^c\cap \del Q)\cong \bigoplus_C H_i(C, C \cap \partial Q)
\]
where the first isomorphism is induced by the inclusion $(N^c, N^c\cap \del Q)\into ( N^c\cup\del Q, \partial Q)$, 
and the direct sum is over all complementary regions $C$ of $N$ in $Q$. By \Cref{lem:correctclass}, the image of $[S]$ in $H_1(N^c\cup Q,\del Q)$ is the image of the image of $z$, viewing $z$ as a class in $H_2(Q,\del Q)$.
Therefore $ S\cap C$ is nulhomologous in $H_1(C, C\cap \del Q)$ for each complementary region $C$.

If $C$ is a $\u$-cusped solid torus, this means that $ S\cap C$ must be either a union of meridians or an even-sized collection of parallel curves essential in $\del C$ whose coorientations cancel each other out in homology. The latter case is impossible because $z$ pairs positively with the $\u\u$-cusp curves as noted above, proving (b).

If $C$ is a $\u$-cusped product, we consider the following subsequence of the long exact sequence associated to $(C,C\cap \del Q, C\cap R_-)$:
\[
H_i(C, C \cap R_-) \to H_i(C, C \cap \partial Q) \to H_{i-1}(C \cap \partial Q, C \cap R_-) \to H_{i-1}(C, C \cap R_-).
\]
Note $H_i(C, C \cap R_-) \cong 0$ since $C$ is homeomorphic to $(C \cap R_-) \times [0,1]$. Hence 
\begin{equation}\label{eqn:LES2}
H_i(C, C \cap \partial Q) \cong H_{i-1}(C \cap \partial Q, C \cap R_-) \cong H_{i-1}(C \cap R_+, C\cap R_+\cap \gamma).
\end{equation}

Using the fact that $S\cap C$ is nulhomologous in $H_1(C,C\cap \del Q)$ as well as the isomorphism $H_1(C,C\cap \del Q)\cong H_0(C\cap R_+, C\cap R_+\cap \gamma)$, we see that the (signed) boundary points of $S\cap C$ give a nulhomologous 0-chain in $H_0(C\cap R_+, C\cap  R_+\cap \gamma)$. Hence they can be connected by paths in $C\cap R_+$ to form a collection of cooriented arcs and loops, all of whose boundary points lie in $C\cap  R_+\cap \gamma$. 
This completes the proof of (c).
\end{proof}

\begin{proposition}\label{prop:decompsurface}
Let $z\in H^1(Q)$ be an integral class which pairs positively with the homology class of every directed cycle of $\Gamma^+$. Let $S$ be an $N$-surface for $z$. Then there exists a surface  which extends $S$, is properly embedded in $Q$, is transverse to the  vector field $V$, and represents $z$ (viewed as a class in $H_2(Q,\del Q)$).
\end{proposition}

\begin{proof}
By parts (a) and (b) of \Cref{lem:bdyNsurface}, we can cap off $S$ in the $\u$-cusped solid tori and the $\u$-cusped torus shells of $N^c$ by gluing on disks and annuli. Since $V$ is circular in these $\u$-cusped tori, we can choose these disks and annuli to be transverse to $V$.

For each $\u$-cusped product $C\cong \Sigma\times[0,1]$ in $N^c$, we can extend $S\cap C$ by the family of arcs furnished by \Cref{lem:bdyNsurface}(c) to obtain a family $A$ of arcs and curves in $\Sigma\times\{1\}$. If we flow $A$ backward along $V$ until it hits $R_-$, it sweeps out a properly embedded surface tangent to $V$ that we can use to cap off $S$ in $C$. We then homotope the result slightly so that it is transverse to $V$. After doing this in each $\u$-cusped product we have produced a cooriented surface $\ol S$ which is properly embedded in $Q$ and positively transverse to $V$. Moreover if we consider the maps of \Cref{eqn:LES1} from the proof of \Cref{lem:bdyNsurface}, the images of $z$ and $[\ol S]$ in $H_2(Q,N^c\cup \del Q)\cong H_2(N,\del N)$ are equal by \Cref{lem:correctclass}. As such, they differ by an element $\alpha\in H_2(Q,\del Q)$ in the image of $H_2(N^c\cup \del Q,\del Q)\cong H_2(N^c, N^c\cap \del Q)$. Since $H_2(C, C\cap \del Q)=0$ for each $\u$-cusped torus component of $N^c$, we have $H_2(N^c, N^c\cap \del Q)\cong\bigoplus_C H_2(C,C\cap \del Q)$ where the sum is over all the $\u$-cusped product components.

We will now construct a cooriented surface $S_\alpha$ which is positively transverse to $V$ and represents $\alpha$. By \Cref{eqn:LES2} from the proof of \Cref{lem:bdyNsurface}, for each $\u$-cusped product $C$ we have $H_2(C, C \cap \partial Q) \cong H_{1}(C \cap R_+, C\cap \del R_+)$. An element of $H_{1}(C \cap R_+, C\cap \del R_+)$ is represented by a collection of cooriented closed loops and arcs in $C\cap R_+$, which can be flowed backward to $R_-$ along $V$ to sweep out a cooriented surface representing the corresponding class in $H_2(C,C\cap \del Q)$. By a small homotopy we can make the swept out surface positively transverse to $V$. Doing this for all $\u$-cusped products we obtain our surface $S_\alpha$.

To finish the proof, we perform an oriented cut-and-paste on $\ol S$ and $S_\alpha$ so that the result is positively transverse to $V$.
\end{proof}

\subsection{Uniqueness argument} \label{subsec:uniquenessarg}

In this subsection, we establish our uniqueness result. We fix some notation.
\begin{itemize}
\item $f\colon L\to L$ is a Handel-Miller endperiodic map,
\item $\mr Q:=M_f$ is the mapping torus of $f$,
\item $Q:=\ol M_f$ is the compactified mapping torus of $f$ with associated depth one foliation $\mc F$,
\item $\phi$ is the suspension semiflow of $f$ in $Q$, and
\item $\mc L$ is the unstable Handel-Miller lamination in $Q$ (recall this is the suspension by $f$ of the positive Handel-Miller lamination $\Lambda_+\subset L$).
\end{itemize}

Suppose we have an unstable veering branched surface $(B,V)$ in $Q$ compatibly carrying $\mathcal{L}$.
Let $\mr B=B\cap \mr Q$.

By \Cref{thm:folcone}, the cone spanned by cycles in the dual graph $\Gamma$ of $B$ in $H_1(Q)$ is dual to the foliation cone associated to $\mathcal{F}$ in $H_2(Q, \partial Q)$. Fix an extended dual graph $\Gamma^+$ for $B$. Let $z=[\mathcal{F}] \in H_2(M, \partial M) \cong H^1(M)$.
By \Cref{prop:extdualgraphnonewcycles}, the dual of $z$ pairs positively with every cycle of $\Gamma^+$. By \Cref{prop:decompsurface}, there exists a surface $S_z$ representing $z$ properly embedded in $Q$, which is positively transverse to $V$ and $\Gamma^+$.

\begin{lemma}\label{lemma:itsaproduct}
Let $p\in Q$. Then the forward $V$-trajectory from $p$ meets $S_z$ or $R_+(Q)$.
\end{lemma}

\begin{proof}
First let $p\in B$ and suppose that the forward trajectory from $p$ does not intersect $R_+$. Then by compactness there is a segment $\rho$ of the forward $V$-trajectory from $p$ that starts and ends in the same sector of $B\cut S_z$. By adding a segment in this sector, we can close up $\rho$ to a closed curve $\rho'$ that is positively transverse to $\brloc(B)$. By \Cref{prop:dualgraphcarries}, $\rho'$ is homotopic to a directed $\Gamma$-cycle. Since $z$ pairs positively with this directed cycle, $\rho'$ must also intersect $S_z$ positively, so $\rho$ itself must intersect $S_z$ positively.

Next suppose $p\notin B$. If $p$ lies in a $u$-cusped product then the forward trajectory from $p$ either enters $B$ which then intersects $S_z$ by above, or it terminates on $R_+$. Similarly, if $p$ lies in a $\u$-cusped torus $T$ then its forward trajectory either enters $B$ and intersects $S_z$ by above, or remains in $T$. In this case, by circularity of $V|_T$, it intersects one of the pieces we used to cap off $S_z$ inside of $T$.
\end{proof}

It follows from \Cref{lemma:itsaproduct} that the sutured manifold obtained by decomposing $Q$ along $S_z$ is a product (cf. the proof of \Cref{lem:ucusprecog} on $\u$-cusped product recognition). In particular $S_z$ can be spun around $R_\pm(Q)$ to produce a surface $L_z$ which is the fiber of a fibration 
\[
L_z\into \mr Q \onto S^1. 
\]
(see \cite[Lemma C]{Alt12}). This gives rise to a depth one foliation of $Q$, which we call $\mc F_z$. We wish to show that $L_z\cap \mc L$ is the Handel-Miller lamination associated to the fibration above, which turns out to be a surprisingly subtle point.

The cohomology class of $\mc F_z$ is equal to that of $\mc F$, so by \cite[Theorem 1.1]{CC93} there exists an ambient isotopy of $Q$ which is smooth in $\mr Q$, fixes each point in $R_\pm (Q)$, and carries $\mc F$ to $\mc F_z$ and $L$ to $L_z$. Further, it follows from the arguments in \cite{CC93} that we can require the existence of a neighborhood $N_\epsilon$ of $R_\pm(Q)$ such that the ambient isotopy moves points along flow lines of $\phi$ in $N_\epsilon$.

We now replace $\mc F$, $\phi$, and $\mc L$ by their images under this ambient isotopy. However, we leave $B$ in its original position. The reason for this is that $L_z$ is in a particularly nice position with respect to $(B,V)$ that we would like to preserve. The cost of leaving $B$ in place is that we can no longer assume $\mc L$ is in a carried position with respect to $B$.

Choose a tiling $\mc T_{\mc F}$ of all $f$-cycles of ends of $L$. 
We can choose tiled neighborhoods of all the end-cycles which are small enough so that that the associated staircase neighborhoods of components of $R_\pm(Q)$ are contained in $N_\epsilon$. Let $N_+$ and $N_-$ be the union of the positive and negative staircases respectively, and let $N_\pm=N_+\cup N_-$.

We now set some more notation, describing natural objects that live in the complement of $N_\pm$.
\begin{itemize}
\item Manifolds:
\begin{itemize}
\item Let $Q^\dag=Q\cut N_\pm$ (the dagger notation indicates that we are cutting away $N_\pm$).
\item Let $L^\dagger=(L\cap Q^\dagger)\cut(L\cap \del Q^\dag)$ be the core of $L$ complementary to the tiled neighborhoods defining $N_\pm$.
\item Let $L^\dag_{-1}=L\cup R_+(N_-)$ and $L^\dag_1=L\cup R_-(N_+)$. In words, $L_1^\dagger$ ($L_{-1}^\dag$) is obtained by adding to $L^\dagger$ the first tile in each positive (negative) end-cycle that doesn't already lie in $L^\dagger$.
\item Let $\wt {Q^\dagger}$ be the $\Z$-cover of $Q^\dagger$ associated to $L^\dagger$. This can be constructed by gluing together $\Z$-many copies of $Q^\dagger\cut L^\dagger$.
\end{itemize}
\item Maps and flows:
\begin{itemize}
     \item Let $\phi^\dag=\phi|_{Q^\dag}$.
     \item Let $f^\dagger\colon L^\dag_{-1}\to L^\dag_1$ be the homeomorphism induced by $\phi^\dag$. Note that $Q^\dag$ is the ``mapping torus" of $f^\dag$, i.e. 
\[
Q^\dag=L^\dag_{-1}\times[0,1]/\left((x,1)\sim(f^\dag(x),0) \text{ if } f^\dag(x)\in L^\dagger\right).
\]
To visualize this it may be helpful to consider \Cref{fig:Qdagger}.
\end{itemize}
\item Laminations and branched surfaces:
\begin{itemize}
    \item Let $\mc L^\dagger=\mc L|_{Q^\dagger}$.
    \item Let $B\cap L=\tau$ and let $\tau^\dagger=\tau|_{L^\dagger}$.
    \item Since $f$ is Handel-Miller, $\mc L\cap L$ is the positive Handel-Miller lamination. For convenience we will suppress the $+$ subscript and denote the positive Handel-Miller lamination by $\Lambda$. Let $\Lambda^\dag=\Lambda|_{L^\dag}$.
    \item Note that there is an isotopy of $\mc L^\dagger$, fixing its boundary, so that the result is fully carried by $B^\dagger$ (this is induced by the ambient isotopy carrying $\mc F$ to $\mc F_z$).
    Let $\mc L^\dagger_0$ denote this isotoped lamination, and let $\Lambda^\dagger_0= \mc L^\dagger_0\cap L^\dagger$.
\end{itemize}
\end{itemize}

\begin{figure}
    \centering
    \fontsize{10pt}{10pt}\selectfont
    \resizebox{!}{4cm}{
\begingroup%
  \makeatletter%
  \providecommand\color[2][]{%
    \errmessage{(Inkscape) Color is used for the text in Inkscape, but the package 'color.sty' is not loaded}%
    \renewcommand\color[2][]{}%
  }%
  \providecommand\transparent[1]{%
    \errmessage{(Inkscape) Transparency is used (non-zero) for the text in Inkscape, but the package 'transparent.sty' is not loaded}%
    \renewcommand\transparent[1]{}%
  }%
  \providecommand\rotatebox[2]{#2}%
  \newcommand*\fsize{\dimexpr\f@size pt\relax}%
  \newcommand*\lineheight[1]{\fontsize{\fsize}{#1\fsize}\selectfont}%
  \ifx\svgwidth\undefined%
    \setlength{\unitlength}{236.1173758bp}%
    \ifx\svgscale\undefined%
      \relax%
    \else%
      \setlength{\unitlength}{\unitlength * \real{\svgscale}}%
    \fi%
  \else%
    \setlength{\unitlength}{\svgwidth}%
  \fi%
  \global\let\svgwidth\undefined%
  \global\let\svgscale\undefined%
  \makeatother%
  \begin{picture}(1,0.44022005)%
    \lineheight{1}%
    \setlength\tabcolsep{0pt}%
    \put(0,0){\includegraphics[width=\unitlength,page=1]{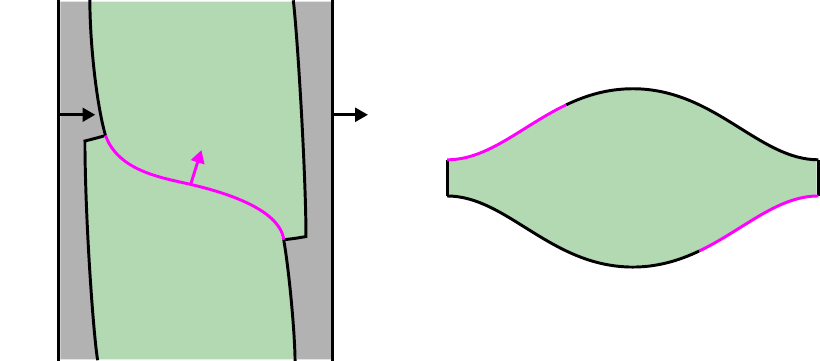}}%
    \put(0.17100632,0.36979139){\color[rgb]{0,0,0}\makebox(0,0)[lt]{\lineheight{1.25}\smash{\begin{tabular}[t]{l}$Q^\dagger$\end{tabular}}}}%
    \put(-0.00384641,0.35010885){\color[rgb]{0,0,0}\makebox(0,0)[lt]{\lineheight{1.25}\smash{\begin{tabular}[t]{l}$R_-$\end{tabular}}}}%
    \put(0.41758489,0.3501061){\color[rgb]{0,0,0}\makebox(0,0)[lt]{\lineheight{1.25}\smash{\begin{tabular}[t]{l}$R_+$\end{tabular}}}}%
    \put(0.69722741,0.22021747){\color[rgb]{0,0,0}\makebox(0,0)[lt]{\lineheight{1.25}\smash{\begin{tabular}[t]{l}$Q^\dagger \cut L^\dagger$\end{tabular}}}}%
    \put(0.73529453,0.34990688){\color[rgb]{0,0,0}\makebox(0,0)[lt]{\lineheight{1.25}\smash{\begin{tabular}[t]{l}$L^\dagger_1$\end{tabular}}}}%
    \put(0.73529536,0.06098693){\color[rgb]{0,0,0}\makebox(0,0)[lt]{\lineheight{1.25}\smash{\begin{tabular}[t]{l}$L^\dagger_{-1}$\end{tabular}}}}%
    \put(0.18707718,0.16179469){\color[rgb]{1,0,1}\makebox(0,0)[lt]{\lineheight{1.25}\smash{\begin{tabular}[t]{l}$L^\dagger$\end{tabular}}}}%
  \end{picture}%
\endgroup%
}
    \caption{A cartoon of $Q^\dagger$ and the product manifold obtained by cutting $Q^\dagger$ along $L^\dagger$.}
    \label{fig:Qdagger}
\end{figure}

\begin{lemma}\label{lem:noclosedcurves}
The train track $\tau$ carries no closed curve which is nullhomotopic in $L$.
\end{lemma}
\begin{proof}
If $\tau$ carried a curve bounding a disk $D\subset L$, the disk $D$ could be decomposed into compact complementary regions of $\tau$, at least one of which would have to have positive index. However, all the compact complementary regions of $\tau$ are pieces we used to cap off an $N$-surface in $\u$-cusped torus pieces, which all have negative index.
\end{proof}

\begin{lemma}\label{lem:itsanilam}
The laminations $\Lambda_0^\dagger$ and $\Lambda^\dagger$ are isotopic in $L^\dagger$ fixing their boundary points.
\end{lemma}

\begin{proof}
Let $\wt {\LL^\dag}$ and $\wt {\LL_0^\dag}$ be the lifts of $\LL^\dag$ and $\LL_0^\dag$ to $\wt {Q^\dag}$, respectively. 
Let $\wt {L^\dagger}$ be a lift of $L^\dagger$ to $\wt {Q^\dagger}$, and let $\wt {\Lambda^\dagger}$ and $\wt {\Lambda^\dagger_0}$ be the preimages of $\Lambda^\dagger$ and $\Lambda^\dagger_0$ in $\wt {L^\dagger}$ under the covering projection, respectively.

We can lift the isotopy between $\LL^\dag$ and $\LL^\dag_0$ in $Q^\dag$ to a proper isotopy $\iota$ in $\wt {Q^\dag}$ between $\wt{\LL^\dag}$ and $\wt{\LL^\dag_0}$ fixing their common boundary. Each leaf of $\wt {\LL^\dag}$ is properly embedded in $\wt{ Q^\dag}$ and homeomorphic to $[0,1]\times \R$, so the same is true of each leaf of $\wt{\LL^\dagger_0}$. 

Immediately we see that each leaf of $\wt{\Lambda^\dagger_0}$ is a compact 1-manifold properly embedded in $\wt {L^\dagger}$. If such a leaf were a closed curve, it would be nulhomotopic in the corresponding leaf of $\wt {\LL^\dagger_0}$, hence nulhomotopic in $\wt {L^\dagger}$ by $\pi_1$-injectivity of $\wt {L^\dagger}$; projecting to $Q^\dagger$ would then give a contradiction to \Cref{lem:noclosedcurves}. Thus we see that $\wt{ \Lambda^\dagger_0}$ is an $I$-lamination, and each leaf of $\wt{\LL^\dag}$ intersects $\wt{ L^\dag}$ in a unique leaf of the $I$-lamination.

Let $\lambda$ be a leaf of $\wt{ \Lambda^\dagger}$, which is equal to $\ell\cap \wt {L^\dagger}$ for some leaf $\ell$ of $\wt{\LL^\dagger}$. Note that $\lambda$ connects  the two components of $\del\ell$. If $\ell_0$ is the image of $\ell$ under $\iota$, $\ell_0\cap \wt {L^\dagger}$ is a union of $\wt{\LL^\dagger_0}$-leaves with boundary equal to $\del \lambda$. It cannot contain any closed curves by above, so consists of a single curve $\lambda_0$ such that $\del\lambda=\del\lambda_0$. Since $\ell$ and $\ell_0$ are isotopic rel boundary in $\wt {Q^\dagger}$ and $\lambda$ and $\lambda_0$ have the same boundary points, they are isotopic rel boundary in $\wt {Q^\dag}$. Since $\wt {L^\dag}$ is $\pi_1$-injective in $\wt {Q^\dag}$, they are in fact isotopic in $\wt{ L^\dag}$. Projecting to $L^\dagger$, we see that every leaf of the $I$-lamination $\Lambda^\dag$ is isotopic, fixing endpoints, to a leaf of the $I$-lamination $\Lambda^\dagger_0$.
\end{proof}

\begin{corollary}\label{lem:ttcarriesHMlam}
The train track $\tau^\dag$ fully carries $\Lambda^\dag$ up to an isotopy of $\Lambda^\dag$ rel boundary. Hence $\tau$ fully carries $\Lambda$. 
\end{corollary}

\begin{proof}
Since $B^\dag$ fully carries $\LL_0^\dag$, the train track $\tau^\dag$ fully carries $\Lambda_0^\dag$. The isotopy from \Cref{lem:itsanilam} between $\Lambda^\dag$ and $\Lambda^\dag_0$ rel boundary certifies that $\Lambda^\dag$ is fully carried by $B^\dag$. Since $\tau$ already carries $\Lambda^+$ outside $L^\dag$, this proves the claim.
\end{proof}

Now that we have shown that $\tau$ fully carries $\Lambda$, we wish to show that $B$ is obtained from this paper's main construction, which can be summarized in two steps:
\begin{enumerate}
    \item Given a Handel-Miller map $f\colon L\to L$ with positive Handel-Miller lamination $\Lambda$ and 2D unstable Handel-Miller lamination $\LL$, fix an efficient spiraling train track carrying $\LL\cap R_+(\ol M_f)$. Using core splits, produce an eventually $f$-periodic splitting sequence of train tracks fully carrying $\Lambda$. In \Cref{thm:splitsequnique} we prove the resulting sequence is unique up to equivalence.
    \item Suspend the periodic part of this splitting sequence in $\ol M_f$. In \Cref{lemma:vbsunique-splitseq} we prove the resulting veering branched surface is unique up to isotopy.
\end{enumerate}

Hence we must show that $B$ is obtained as the suspension of a splitting sequence obtained from performing core splits. 

\begin{definition}[Downward flip]
Let $s$ be a diamond sector of $B$, and suppose that a component of $L^\dag\cap s$ has both its boundary points along the bottom of $s$, so that the corresponding branch $b$ of $\tau$ is large. Further suppose $b$ is lowermost in $s$, meaning that $b$ is the only piece of $L^\dag\cap s$ contained in the component of $s\cut b$ containing the bottom point of $s$. Then there exists an isotopy of $L^\dag$ supported in a neighborhood of $b$, as shown in \Cref{fig:pushdownwards}, which has the effect of splitting $\tau$ along $b$. This isotopy is called a \textbf{downward flip}.
It is straightforward to see that a downward flip can always be performed in such a way that it preserves the property of $L^\dag$ being positively transverse to $V$ and $\Gamma$.
\end{definition}

\begin{figure}
    \centering
    \resizebox{!}{7cm}{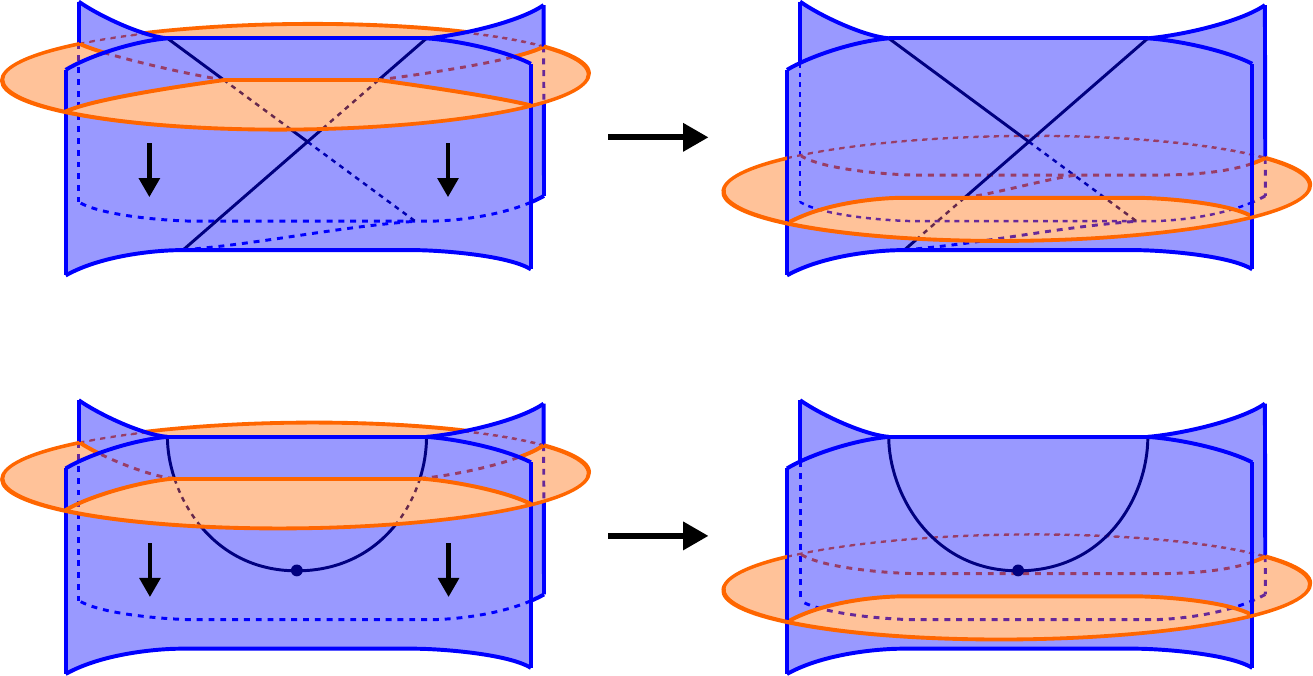}
    \caption{Downward flips.}
    \label{fig:pushdownwards}
\end{figure}

\begin{lemma}\label{lem:makeitspiral}
After finitely many downward flips, we may assume that $\tau^\dagger$ is a spiraling train track.
\end{lemma}

\begin{proof}
We first claim that whenever we see a large branch of $\tau^\dag$, it is possible to perform a downward flip on $L^\dag$.

Suppose that $b$ is a large branch of $\tau^\dag$, and let $s$ be the sector of $B^\dag$ containing $b$. Note that the part of $s$ below $b$ is disjoint from $\del Q^\dag$, since the orientation of $\brloc(B^\dag)$ points outward along $\del Q^\dag$. Note also that $s$ is not a Möbius band or annulus by the definition of veering branched surfaces (\Cref{defn:vbs}(4)) so it must be part of a diamond sector of $B$. We can now locate a bottommost arc of $L^\dag$ in $s$ (this may or may not be $b$) and perform a downward flip.

Since $\tau$ carries the $I$-lamination $\LL^\dag_0$, the train track $\tau^\dag$ will become spiraling after finitely many downward flips.
\end{proof}

\begin{lemma}
The branched surface $B$ is obtained as the suspension of a splitting sequence representing the equivalence class $\mathscr S(B\cap R_+(Q))$.
\end{lemma}

\begin{proof} 
Recall from \Cref{lemma:itsaproduct} (and the subsequent discussion) that $\mr Q \cut L_z$ is homeomorphic to a product $L_z \times [0,1]$. The strategy of the proof, similar to our proof of \Cref{thm:closeduniqueness}, is to construct a `height' function $h: \mr Q \cut L_z \to [0,1]$ such that:
\begin{enumerate}
    \item $h=0$ on the bottom face $L_z\times\{0\}$ and $h=1$ on the top face $L_z\times\{1\}$
    \item $h$ is monotonically increasing on the trajectories of $V|_{\mr Q \cut L_z}$ 
    \item $h$ is monotonically increasing on the directed paths of $\Gamma^+|_{\mr Q \cut L_z}$
\end{enumerate}
Given such a function $h$, the fibers $h^{-1}(t)$ define an isotopy from $L_z\times\{1\}$ to $L_z\times\{0\}$.
Item (3) guarantees that the intersections $h^{-1}(t)\cap B$ give a movie of train tracks which undergo splits at each $t$ for which $h^{-1}(t)$ passes through a triple point or source of $\brloc (B)$. Each train track $h^{-1}(t)\cap B$ carries $\mc L \cap h^{-1}(t)$ by \Cref{lem:ttcarriesHMlam}. 
Moreover the sequence of splits, taken together, gives a core split of $\tau|_{L^\dag_1}$. Therefore it determines a representative of $\mathscr S(B\cap R_+(Q))$.

To construct $h$, observe that by the construction of $L_z$, $\Gamma|_{\mr Q \cut L_z}$ has no directed cycles. Hence we can first define $h$ on $\Gamma|_{\mr Q \cut L_z}$ to satisfy (1) and (3). Then we can extend $h$ over the 2-cells of $\mr B \cut L_z$ to satisfy (1) and (2).

Finally we have to extend $h$ over each complementary region $C$ of $\mr B \cut L_z$ in $\mr Q \cut L_z$. If $C$ lies in a solid cusped torus component $C'$ of $Q \cut B$, then this is straightforward since $L_z$ intersects $C'$ in meridional disks. If $C$ lies in a cusped torus shell component then extending $h$ over $C$ is similarly straightforward.

If $C$ lies in a cusped product piece $C'$, we make use of the neighborhood $N$ of $B$ constructed in \Cref{subsec:dualcelldecomp}. Recall the key feature of $N$ is that along each component of $\partial N \cap C'$, $V$ points into $B$. For $N$ small enough, we can extend $h$ into $N \cap C$ so that (2) (and (1)) still hold. Now using the fact that the trajectory of every point on $C \cap (L_z \times \{0\})$ meets $N$ or $L_z \times \{1\}$ in finite time, we can extend $h$ into $C \cut N$ as well.
\end{proof}

In summary, the branched surface $B$ is obtained from this paper's main construction. Since that construction depends only on the choice of the boundary train track (\Cref{cor:uniqueuptotraintrack}), we have proven the following.

\begin{theorem} [Veering branched surface uniqueness]\label{thm:hmvbsunique}
Let $Q$ be an atoroidal sutured manifold with depth one foliation $\mathcal{F}$, and let $\mathcal{L}$ be the unstable Handel-Miller lamination associated to $\mathcal{F}$. Any veering branched surface $B$ compatibly carrying $\mc L$ is determined up to isotopy by $B\cap R_+(Q)$.
\end{theorem}

\subsection{Shifting moves} \label{subsec:shiftmoves}

As promised in the introduction of this section, we will also explain how the veering branched surfaces with different boundary train tracks in \Cref{thm:hmvbsunique} are related to each other. To do so, we introduce some moves that one can use to modify veering branched surfaces in general, then claim that the veering branched surfaces in \Cref{thm:hmvbsunique} are exactly related by these moves.

Let $B$ be a veering branched surface with boundary train track $\beta$.

\begin{definition} \label{defn:sstriangle}
A \textbf{shift-source triangle} is a triangle $t$ carried by $B$ such that:
\begin{enumerate}
    \item One side of $t$ is a mixed branch $b$ of $\beta$.
    \item The other two sides of $t$ lie along components of $\brloc(B)$.
\end{enumerate}
Note that by (1), among the two sides of $t$ lying along $\brloc(B)$, one side is cooriented outwards while the other is cooriented inwards, let these two sides be $a$ and $c$ respectively. Then $a$ must be oriented from $c$ to $b$, while $c$ must contain a source since $B$ is veering. 
\end{definition}

We claim that a shift-source triangle, with the above notation, is the union of a number of adjacent diamonds (unscalloped) with a bottom side along $c$ and a  diamond (unscalloped) with rounded bottom whose bottom side is on $c$. See \Cref{fig:shiftingmoves} top left. This follows from \Cref{prop:vbssectors} concerning the structure of sectors of veering branched surfaces as follows. Consider the sector $b$ lies on. If it is a diamond with rounded bottom, then the shift-source triangle only consists of this sector and the sector is not scalloped. If the sector is a diamond without rounded bottom, then the sector is not scalloped, one of its top sides lies on $a$ and one of its bottom sides lies on $c$. We call the other bottom side $b'$ and repeat the argument on $b'$. Since there are finitely many triple points on $c$, this process terminates eventually.

When $b$ is embedded, we define the \textbf{shifting move along a shift-source triangle $t$} to be the operation of dynamically splitting $c$ across a small neighborhood of $t$. The effect on the boundary train track is a shift on the branch $b$. We retain the source orientations on the components of the branch locus after splitting, so that the resulting branched surface $B'$ satisfies the triple point condition of being veering (\Cref{defn:vbs}(1)). See \Cref{fig:shiftingmoves} top.

\begin{figure}
    \centering
    \fontsize{6pt}{6pt}\selectfont
    \resizebox{!}{8cm}{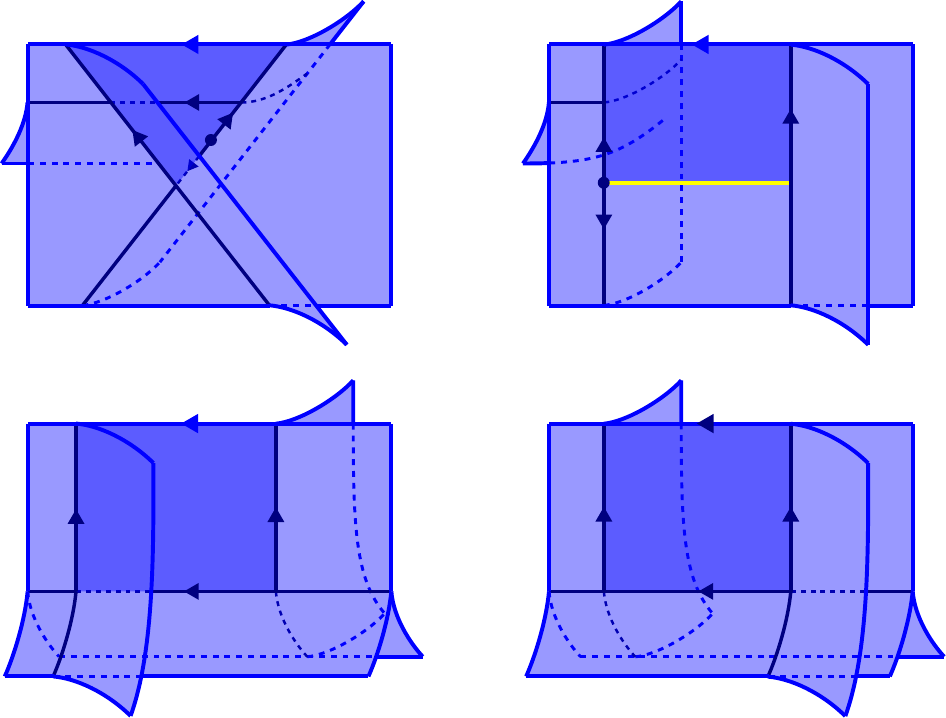}
    \caption{Top: Shifting along a shift-source triangle or rectangle. Bottom: Shifting along a shift rectangle.}
    \label{fig:shiftingmoves}
\end{figure}

We note that a shift-source triangle $t$ might not be embedded; indeed, it could pass twice through a diamond sector $s$ (see \Cref{fig:shiftlift} left). While $c$ intersects itself in this situation, the component of $\del_vN(B)$ corresponding to $c$ is embedded. This gives a canonical relative positioning to the parts of a lift of $t$ to $N(B)$ that project to $s$ such that the lift embeds in $N(B)$. Hence there is no ambiguity in what we mean by dynamically splitting along $s$. See \Cref{fig:shiftlift} middle and right, and recall the discussion of dynamic splittings at the end of \Cref{subsec:dynbranchedsurfaces} if necessary.

\begin{figure}
    \centering
    \fontsize{6pt}{6pt}\selectfont
    \resizebox{!}{5.4cm}{
\begingroup%
  \makeatletter%
  \providecommand\color[2][]{%
    \errmessage{(Inkscape) Color is used for the text in Inkscape, but the package 'color.sty' is not loaded}%
    \renewcommand\color[2][]{}%
  }%
  \providecommand\transparent[1]{%
    \errmessage{(Inkscape) Transparency is used (non-zero) for the text in Inkscape, but the package 'transparent.sty' is not loaded}%
    \renewcommand\transparent[1]{}%
  }%
  \providecommand\rotatebox[2]{#2}%
  \newcommand*\fsize{\dimexpr\f@size pt\relax}%
  \newcommand*\lineheight[1]{\fontsize{\fsize}{#1\fsize}\selectfont}%
  \ifx\svgwidth\undefined%
    \setlength{\unitlength}{416.43266512bp}%
    \ifx\svgscale\undefined%
      \relax%
    \else%
      \setlength{\unitlength}{\unitlength * \real{\svgscale}}%
    \fi%
  \else%
    \setlength{\unitlength}{\svgwidth}%
  \fi%
  \global\let\svgwidth\undefined%
  \global\let\svgscale\undefined%
  \makeatother%
  \begin{picture}(1,0.35849061)%
    \lineheight{1}%
    \setlength\tabcolsep{0pt}%
    \put(0,0){\includegraphics[width=\unitlength,page=1]{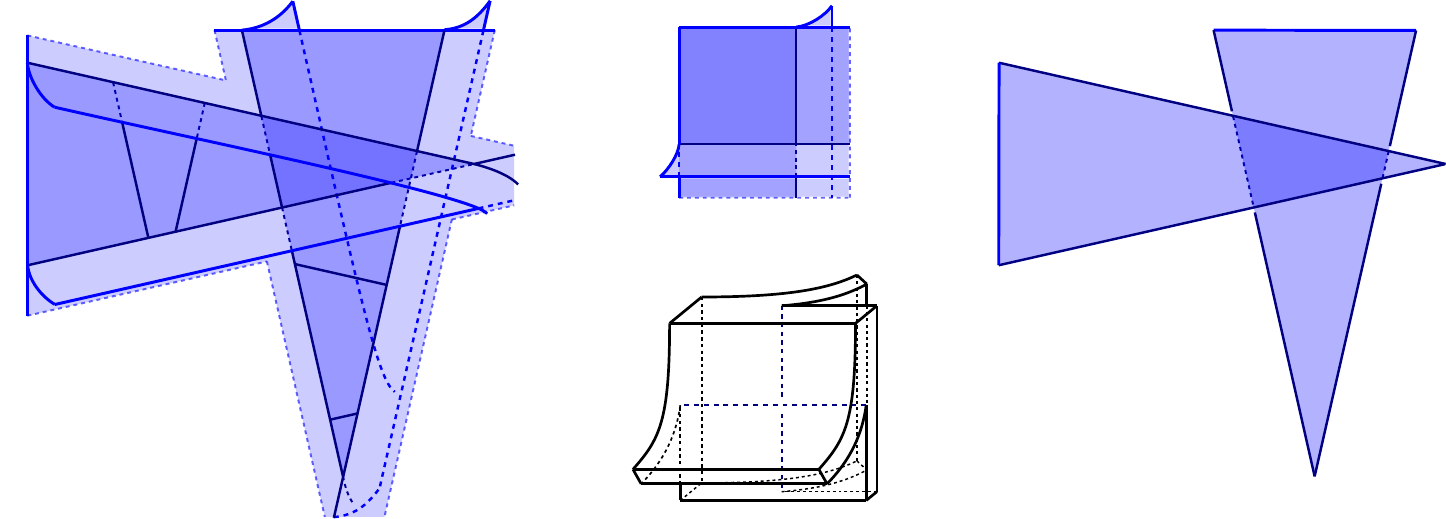}}%
    \put(0.23070855,0.34674119){\color[rgb]{0,0,0}\makebox(0,0)[lt]{\lineheight{1.25}\smash{\begin{tabular}[t]{l}$b$\end{tabular}}}}%
    \put(0.18175869,0.30357984){\color[rgb]{0,0,0}\makebox(0,0)[lt]{\lineheight{1.25}\smash{\begin{tabular}[t]{l}$a$\end{tabular}}}}%
    \put(0.28272961,0.30358041){\color[rgb]{0,0,0}\makebox(0,0)[lt]{\lineheight{1.25}\smash{\begin{tabular}[t]{l}$c$\end{tabular}}}}%
    \put(-0.00109046,0.24202081){\color[rgb]{0,0,0}\makebox(0,0)[lt]{\lineheight{1.25}\smash{\begin{tabular}[t]{l}$b$\end{tabular}}}}%
    \put(0.03830065,0.29187023){\color[rgb]{0,0,0}\makebox(0,0)[lt]{\lineheight{1.25}\smash{\begin{tabular}[t]{l}$a$\end{tabular}}}}%
    \put(0.03830304,0.18606084){\color[rgb]{0,0,0}\makebox(0,0)[lt]{\lineheight{1.25}\smash{\begin{tabular}[t]{l}$c$\end{tabular}}}}%
    \put(0.55486823,0.2951806){\color[rgb]{0,0,0}\makebox(0,0)[lt]{\lineheight{1.25}\smash{\begin{tabular}[t]{l}$c$\end{tabular}}}}%
    \put(0.50342343,0.2440436){\color[rgb]{0,0,0}\makebox(0,0)[lt]{\lineheight{1.25}\smash{\begin{tabular}[t]{l}$c$\end{tabular}}}}%
  \end{picture}%
\endgroup%
}
    \caption{There is a canonical way to lift a shift-source triangle to $N(B)$ even if it passes through a diamond sector twice.}
    \label{fig:shiftlift}
\end{figure}

We demonstrate that $B'$ is a veering branched surface by verifying the rest of the conditions in \Cref{defn:vbs}: The topology of the complementary regions is not changed, hence $B'$ is still very full. As noted above, the boundary train track of $B'$ is obtained from that of $B$ via a shift, hence is still efficient and has no large branches. The dynamic orientations of the $\u$-faces are not changed hence \Cref{defn:vbs}(3) is preserved. 

For \Cref{defn:vbs}(4), suppose $B'$ contained an annulus/Möbius band sector with both boundary components cooriented inwards. Then $B$ would carry an annulus/Möbius band with boundary components along branch loops of $\brloc(B)$ that are cooriented inwards. But then the dynamic half-plane $D(\gamma)$ determined by the one of these branch loops would not be homeomorphic to a half-plane, contradicting \Cref{prop:dynamicplanes}. Finally, $B'$ does not carry any tori or Klein bottles otherwise $B$ would carry such a closed surface as well. 

This shows that the shifting move along a shift-source triangle is an operation that transforms a veering branched surface into another veering branched surface.

\begin{definition} \label{defn:ssrectangle}
A \textbf{shift-source rectangle} is a rectangle $r$ carried by $B$ such that:
\begin{enumerate}
    \item One side of $r$ is a mixed branch $b$ of $\beta$.
    \item The two sides of $r$ adjacent to $b$ lie along components of $\brloc(B)$. As in \Cref{defn:sstriangle}, note that by (1), one of these sides is cooriented outwards while the other is cooriented inwards, let these two sides be $c$ and $a$ respectively.
    \item The side of $r$ opposite to $b$, which we call $d$, lies in the interior of a sector.
    \item The corner formed by $c$ and $d$ is a source on $\brloc(B)$.\qedhere
\end{enumerate}
\end{definition}

Similarly to the case of shift-source triangles, one can use \Cref{prop:vbssectors} to deduce that a shift-source rectangle, with the above notation, is a union of a number of adjacent (unscalloped) diamonds with a bottom side along $a$ and a rectangular neighborhood of a non-scalloped top side of a sector with a bottom side along $a$. See \Cref{fig:shiftingmoves} top right. 

Within this last sector, $d$ is an interval connecting an interior point of a bottom side to a source on a top side. Up to modifying the vector field $V$ locally, we can always assume that $V$ is transverse to $d$ pointing out of $r$.

When $b$ is embedded and with the vector field $V$ modified as described, we define the \textbf{shifting move along a shift-source rectangle $r$} to be the operation of splitting $c$ across a small neighborhood of $r$. The effect on the boundary train track is a shifting move on the branch $b$. See the top of \Cref{fig:shiftingmoves}. We retain the source orientations on the components of the branch locus after splitting. By reasoning similarly as above, we see that the resulting branched surface is veering.

As in the case of shift-source triangles, shift-source rectangles might not be embedded, but the shifting move is still well-defined.

Shifting moves along shift-source triangles are inverse to shifting moves along shift-source rectangles, in the following sense: After shifting along a shift-source triangle $t$, there is a natural shift-source rectangle $r$, shifting along which recovers the original veering branched surface. After shifting along a shift-source rectangle $r$, there is a natural shift-source triangle $t$, shifting along which recovers the original veering branched surface. See \Cref{fig:shiftingmoves} top row.

\begin{definition} \label{defn:srectangle}
A \textbf{shift rectangle} is a rectangle $r$ carried by $B$ such that:
\begin{enumerate}
    \item One side of $r$ is a mixed branch $b$ of $\beta$.
    \item The two sides of $r$ adjacent to $b$ lie along components of $\brloc(B)$. As in \Cref{defn:sstriangle}, note that by (1), one of these sides is cooriented outwards while the other is cooriented inwards, let these two sides be $a$ and $c$ respectively.
    \item The side of $r$ opposite to $b$, which we call $d$, lies along a component of $\brloc(B)$ that is cooriented into $r$.\qedhere
\end{enumerate}
\end{definition}

As above, one can use \Cref{prop:vbssectors} to deduce that a shift rectangle, with the above notation, is a union of a number of adjacent (unscalloped) diamonds with a bottom side along $c$. See \Cref{fig:shiftingmoves} bottom.

When $b$ is embedded, we define the \textbf{shifting move along $r$} to be the operation of splitting $B$ across a neighborhood of $r$. The effect on the boundary train track is a shifting move on the branch $b$. See \Cref{fig:shiftingmoves} bottom. As above, we retain the source orientations on the components of the branch locus after splitting. By similar reasoning, we see that the resulting branched surface is veering. Also as above, shift rectangles might not be embedded but the operation is still well-defined.

Shifting moves along shift rectangles are inverse to themselves, in the following sense: After shifting along a shift rectangle $r$, there is a natural shift rectangle $r'$, shifting along which recovers the original veering branched surface. See \Cref{fig:shiftingmoves} bottom row.

\begin{proposition} \label{prop:vbsshift}
For every mixed branch $b$ embedded in the boundary train track $\beta$, there is a shift-source triangle, a shift-source rectangle, or a shift rectangle with a side along $b$.
\end{proposition}

\begin{proof}
The strategy of this proof is to move iteratively downward from $b$ into $B$ using  \Cref{prop:vbssectors}, much like the reasoning employed above.

Let $a$ and $c$ be the components of $\brloc(B)$ on which the outward and inward pointing endpoints of $b$ lie, respectively.
Consider the sector $s_1$ in whose boundary $b$ lies. If $a$ has a source on $s_1$, then one can pick a shift-source rectangle for $b$ within $s_1$. If $c$ has a source on $s_1$ and $a$ does not, then $s_1$ is a diamond with rounded bottom and thus a shift triangle for $b$.  If neither $a$ nor $c$ has a source on $s_1$, then $s_1$ is an unscalloped diamond. Let $b_1$ be the side of $s_1$ opposite to $b$, which must be cooriented into $s_1$. Locally there are two sheets of $B$ converging along $b_1$. 
If $a$ and $c$ do not follow the same sheet when crossing $b_1$, then $s_1$ is a shift rectangle.

If $a$ and $c$ follow the same sheet of $B$ after crossing $b_1$, consider the sector $s_2$ meeting $s_1$ along $b_1$ and containing $a$ and $c$ in its boundary. 
We can perform the above analysis on $s_2$: if $s_1\cup s_2$ is not a a shift-source triangle, shift-source rectangle, or shift rectangle, then there is a sector $s_3$ containing $a$ and $c$ in its boundary on which we can continue our analysis. This process terminates when we reach the source of either $a$ or $c$, which is guaranteed because each branch component has finitely many triple points.
\end{proof}

This proposition implies that given any embedded mixed branch $b$ of $\beta$, we can construct a new veering branched surface $B'$ with boundary train track $\beta'$ given by $\beta$ with a shifting move done along $b$. If the original veering branched surface $B$ carried a lamination $\Lambda$, then $B'$ carries $\Lambda$ as well, since the shifting moves are defined by dynamic splittings. In particular we have the following:

\begin{proposition}
Let $B$ be a veering branched surface on a sutured manifold $Q$ with boundary train track $\beta$. Suppose $\beta'$ is another train track on $R_+(Q)$ differing from $\beta$ by shifts. Then there exists a veering branched surface $B'$ on $M$ with boundary train track $\beta'$. Furthermore, if $B$ carries a lamination $\Lambda$ in $M$, then $B'$ also carries $\Lambda$.
\end{proposition}

\begin{proof}
Each shift in the sequence of shifts from $\beta$ to $\beta'$ is performed on an embedded mixed branch. Hence by \Cref{prop:vbsshift} there exists a (2-dimensional) shifting move as defined above. Each of these moves preserves the property of being a veering branched surface, so we obtain a veering branched surface with boundary track $\beta'$.
Furthermore, as reasoned above, if $B$ carries a lamination $\Lambda$ in $M$, then $B'$ also carries $\Lambda$.
\end{proof}

By applying our uniqueness result \Cref{thm:hmvbsunique}, we obtain our goal of this subsection:

\begin{corollary} \label{cor:hmvbsshifts}
Let $f:L \to L$ be an endperiodic map, $Q$ be the compactified mapping torus with depth one foliation $\mathcal{F}$, and $\mathcal{L}$ be the unstable Handel-Miller lamination on $Q$. Then if $B$ and $B'$ are two veering branched surfaces compatibly carrying $\mathcal{L}$, then they are related by shifting moves.
\end{corollary}

\section{Examples} \label{sec:eg}

In general it is difficult to explicitly describe the Handel-Miller lamination of an endperiodic map $f$. However, by iterating $f$ and observing the images of junctures, it is sometimes possible by inspection to see that the iterated images are accumulating on a lamination carried by a specific endperiodic train track, as in \cite[Example 4.13]{CCF19}. Then one can use this track to find an $f$-periodic splitting sequence of endperiodic train tracks. We now present some examples of periodic splitting sequences of endperiodic train tracks carrying Handel-Miller laminations, which give rise to veering branched surfaces using the methods of this paper. Since the intent of this section is simply to help give intuition for our methods, we suppress the work of actually drawing the iterated junctures to find the train tracks.

\begin{example}[Translation] \label{eg:translation}
This is the simplest possible example of our construction. If $f\colon L\to L$ is a translation (recall \Cref{eg:translate}), then the Handel-Miller laminations of $f$ are empty. In this case the associated veering branched surface is empty as well.
\end{example}

\begin{example}[Stack of chairs] \label{eg:stackofchairs}
Consider the map $f$ of \Cref{ex:stackofchairsflow} and \Cref{example:stackofchairs}. In this case the positive Handel-Miller lamination $\Lambda_+$ is a single properly embedded line $\lambda$, which is preserved by the map $f$ preserving orientation. Thus $\lambda$ itself is an endperiodic train track carrying $\Lambda_+$ and there is a trivial splitting sequence from $\lambda$ to $f(\lambda)$. The associated veering branched surface is an annulus which we can identify with the unstable Handel-Miller lamination $\mc L^u$ (see \Cref{fig:stackofchairsflow}).
\end{example}

It is convenient for us to use the language of sutured manifold decompositions to describe examples of depth one foliations (see \cite{Gab83}). Roughly, a sutured manifold decomposition is the operation of cutting a sutured manifold along a suitable properly embedded surface, keeping track of boundary information.

\begin{example} \label{eg:azeg1}
Let $Q$ be the \emph{complement} in $S^3$ of the handlebody shown in \Cref{fig:azegmfd} left. Here $R_+$ is shaded in green, $R_-$ is shaded in purple, and the sutures are drawn as black lines. This is conjectured to be the minimal volume acylindrical taut sutured manifold by Zhang \cite[Conjecture 1.4]{Zha23}.

In \Cref{fig:azegmfd} right we consider a disk decomposition of $Q$, i.e. a sutured manifold decomposition along a union of disks. 
The disks are the two `holes' of the handlebody. We convey the information of how we coorient the disks by labeling a side of a decomposing disk $\pm$ if it belongs to $R_\pm$ after decomposition.

\begin{figure}
    \centering
    \fontsize{6pt}{6pt}\selectfont
    \resizebox{!}{3cm}{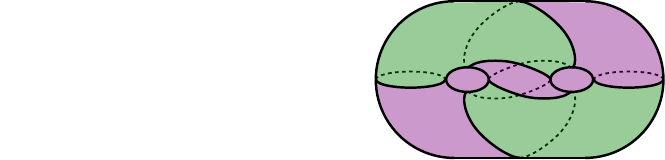}
    \caption{Left: The sutured manifold $Q$ considered in \Cref{eg:azeg1}. Right: A disk decomposition of $Q$.}
    \label{fig:azegmfd}
\end{figure}

One can check that this disk decomposition reduces $Q$ to the product sutured manifold $D^2 \times I$, hence represents $Q$ as a depth one sutured manifold and in particular determines an endperiodic map. In \Cref{fig:AZex} middle we illustrate this endperiodic map. Here we label the junctures by the same letters as \Cref{fig:azegmfd} to aid the reader's understanding.

Using the method described at the start of the section, we find a periodic splitting sequence of endperiodic train tracks that carries the positive Handel-Miller lamination, which we illustrate in \Cref{fig:AZex} bottom.

\begin{figure}
    \centering
    \fontsize{6pt}{6pt}\selectfont
    \resizebox{!}{15cm}{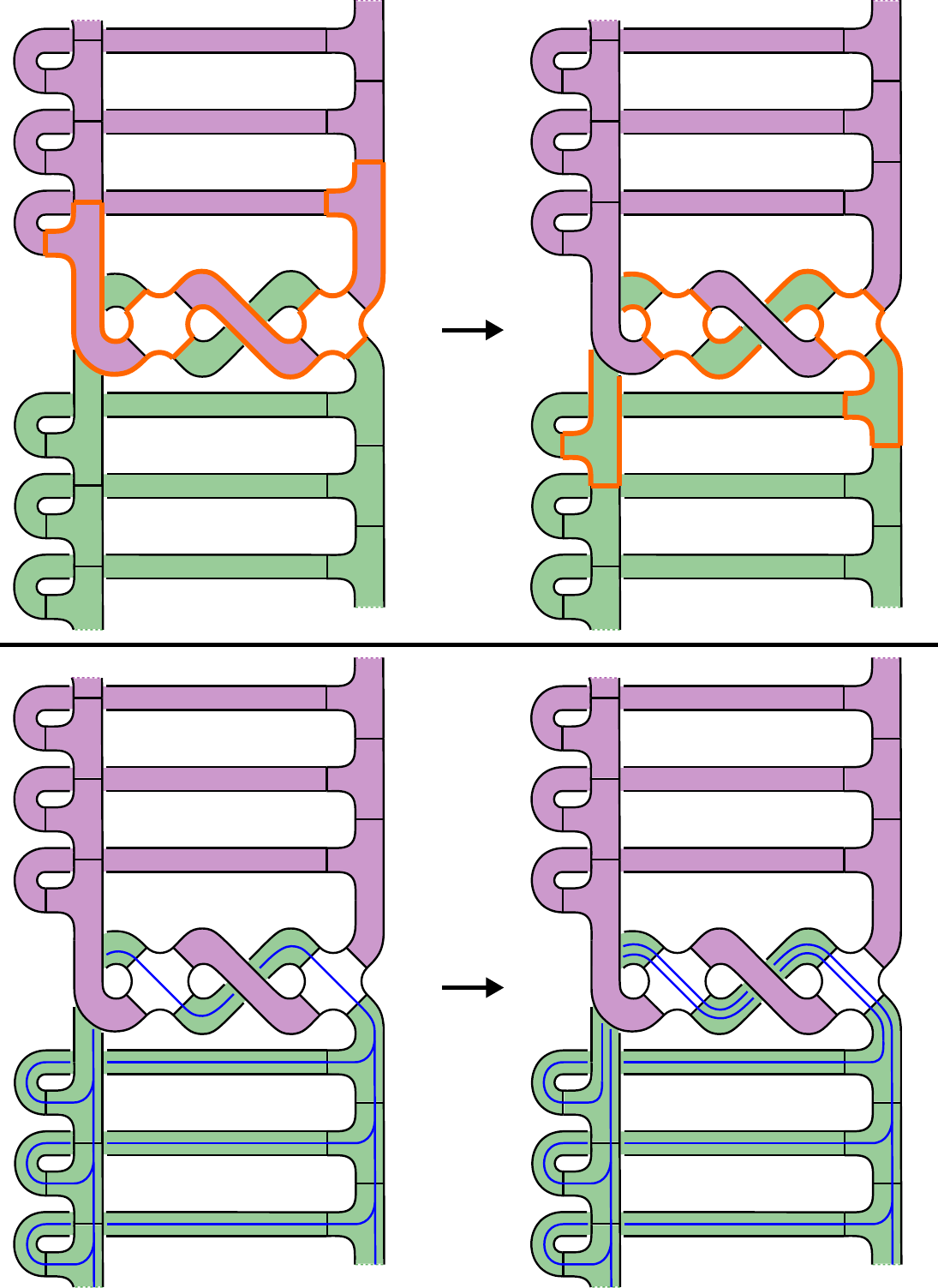}
    \caption{Top: the monodromy of the depth one foliation is determined by the fact that it maps the outlined region on the left to the outlined region on the right so that labels match. Bottom: an $f$-periodic splitting sequence of endperiodic train tracks carrying the positive Handel-Miller lamination.}
    \label{fig:AZex}
\end{figure}

By suspending this periodic splitting sequence as in \Cref{sec:bsconstruct}, we obtain an unstable veering branched surface on $Q$. We draw its boundary train track in \Cref{fig:egs} first row middle.

As pointed out in the introduction, we could have done everything in this paper for negative/stable laminations instead. In particular, one can find a periodic \emph{folding} sequence of train tracks carrying the negative Handel-Miller lamination, which suspends to a stable veering branched surface. The boundary train track of this branched surface is illustrated in \Cref{fig:egs} first row right.
\end{example}

\begin{figure}
    \centering
    \fontsize{6pt}{6pt}\selectfont
    \resizebox{!}{9cm}{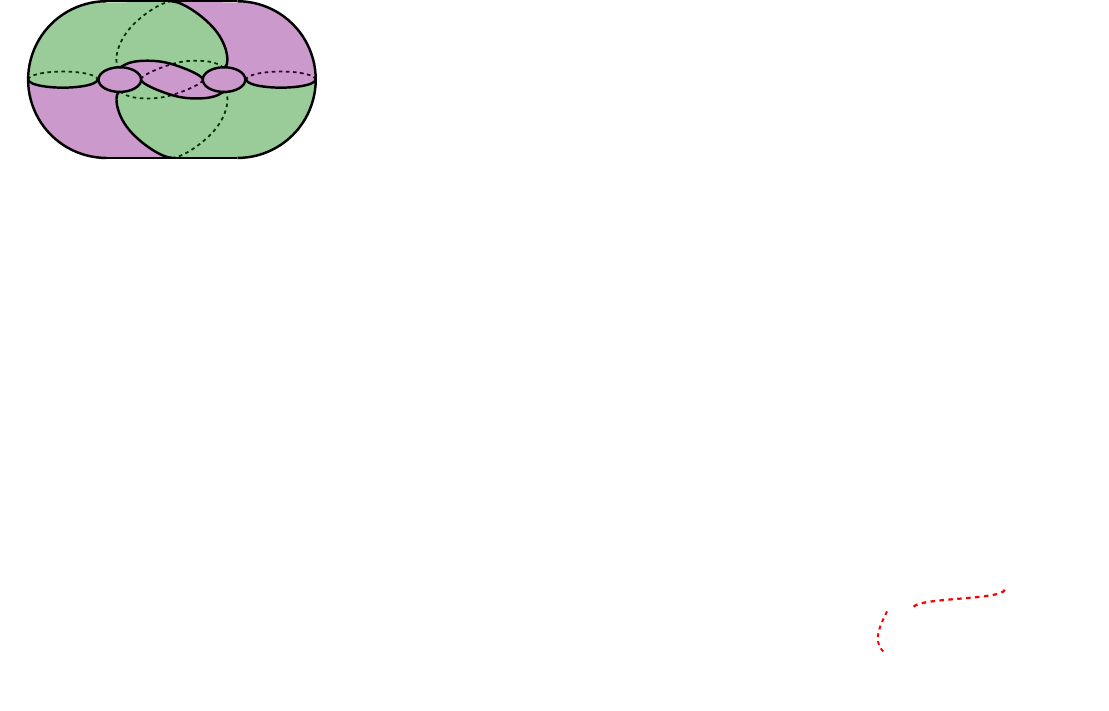}
    \caption{Boundary train tracks of unstable and stable veering branched surfaces associated to some depth one foliations.}
    \label{fig:egs}
\end{figure}

\begin{example} \label{eg:azeg2}
In \Cref{fig:egs} second row, we consider a different disk decomposition of the sutured manifold $Q$ in \Cref{eg:azeg1}, which determines another endperiodic map. 

Using a similar computation as above, we find periodic splitting/folding sequences of train tracks for this endperiodic map. These suspend to unstable/stable veering branched surfaces on $Q$, the boundary train tracks of which we record in \Cref{fig:egs} second row middle and right. 
\end{example}

\begin{example} \label{eg:cceg3}
In \Cref{fig:egs} third row, we consider a different sutured manifold. This sutured manifold, as well as the illustrated disk decomposition, was considered in \cite[Example 5]{CC99}.

Using a similar computation as above, we find periodic splitting/folding sequences of train tracks for the corresponding endperiodic map. These suspend to unstable/stable veering branched surfaces, the boundary train tracks of which we record in \Cref{fig:egs} third row middle and right. 
\end{example}

\begin{example} \label{eg:fenleyexample}
In \cite[Section 5]{Fen97}, Fenley gave an example of an endperiodic map $f: L \to L$ where the positive Handel-Miller lamination $\Lambda_+$ of $f$ does not admit an $f$-invariant transverse measure of full support. 

We present a version of Fenley's map in \Cref{fig:fenleymap}. The map is the composition of a Dehn twist $\tau_2$ of the indicated sign on $c_2$, a Dehn twist $\tau_1$ on $c_1$ of the opposite sign, and a downward shift by one unit $\sigma$, i.e. $f=\sigma \circ \tau_1 \circ \tau_2$.

\begin{figure}
    \centering
    \fontsize{6pt}{6pt}\selectfont
    \resizebox{!}{2in}{
\begingroup%
  \makeatletter%
  \providecommand\color[2][]{%
    \errmessage{(Inkscape) Color is used for the text in Inkscape, but the package 'color.sty' is not loaded}%
    \renewcommand\color[2][]{}%
  }%
  \providecommand\transparent[1]{%
    \errmessage{(Inkscape) Transparency is used (non-zero) for the text in Inkscape, but the package 'transparent.sty' is not loaded}%
    \renewcommand\transparent[1]{}%
  }%
  \providecommand\rotatebox[2]{#2}%
  \newcommand*\fsize{\dimexpr\f@size pt\relax}%
  \newcommand*\lineheight[1]{\fontsize{\fsize}{#1\fsize}\selectfont}%
  \ifx\svgwidth\undefined%
    \setlength{\unitlength}{60.96078002bp}%
    \ifx\svgscale\undefined%
      \relax%
    \else%
      \setlength{\unitlength}{\unitlength * \real{\svgscale}}%
    \fi%
  \else%
    \setlength{\unitlength}{\svgwidth}%
  \fi%
  \global\let\svgwidth\undefined%
  \global\let\svgscale\undefined%
  \makeatother%
  \begin{picture}(1,1.69546944)%
    \lineheight{1}%
    \setlength\tabcolsep{0pt}%
    \put(0,0){\includegraphics[width=\unitlength,page=1]{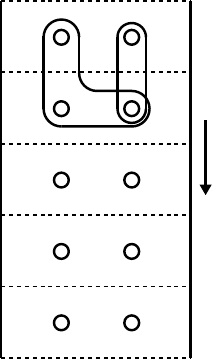}}%
    \put(0.11025928,1.59641563){\color[rgb]{0,0,0}\makebox(0,0)[lt]{\lineheight{1.25}\smash{\begin{tabular}[t]{l}$c_1$\end{tabular}}}}%
    \put(0.67225215,1.5964018){\color[rgb]{0,0,0}\makebox(0,0)[lt]{\lineheight{1.25}\smash{\begin{tabular}[t]{l}$c_2$\end{tabular}}}}%
    \put(0,0){\includegraphics[width=\unitlength,page=2]{fig_fenleymap.pdf}}%
  \end{picture}%
\endgroup%
}
    \caption{A endperiodic map $f$ where the positive Handel-Miller lamination $\Lambda_+$ of $f$ does not admit a $f$-invariant transverse measure of full support. The map is obtained by performing a twist around curve 2, then a twist around curve 1, and then shifting downward by one unit.}
    \label{fig:fenleymap}
\end{figure}

In \Cref{fig:fenleyexample}, we show a $f$-periodic splitting sequence $\tau=\tau_0 \to \tau_1 \to \tau_2 \to \tau_3 \to \tau_4=f(\tau) \to ...$ carrying the positive Handel-Miller lamination $\Lambda_+$.

\begin{figure}
    \centering
    \fontsize{12pt}{12pt}\selectfont
    \resizebox{!}{1.95in}{
\begingroup%
  \makeatletter%
  \providecommand\color[2][]{%
    \errmessage{(Inkscape) Color is used for the text in Inkscape, but the package 'color.sty' is not loaded}%
    \renewcommand\color[2][]{}%
  }%
  \providecommand\transparent[1]{%
    \errmessage{(Inkscape) Transparency is used (non-zero) for the text in Inkscape, but the package 'transparent.sty' is not loaded}%
    \renewcommand\transparent[1]{}%
  }%
  \providecommand\rotatebox[2]{#2}%
  \newcommand*\fsize{\dimexpr\f@size pt\relax}%
  \newcommand*\lineheight[1]{\fontsize{\fsize}{#1\fsize}\selectfont}%
  \ifx\svgwidth\undefined%
    \setlength{\unitlength}{366.20838142bp}%
    \ifx\svgscale\undefined%
      \relax%
    \else%
      \setlength{\unitlength}{\unitlength * \real{\svgscale}}%
    \fi%
  \else%
    \setlength{\unitlength}{\svgwidth}%
  \fi%
  \global\let\svgwidth\undefined%
  \global\let\svgscale\undefined%
  \makeatother%
  \begin{picture}(1,0.33350974)%
    \lineheight{1}%
    \setlength\tabcolsep{0pt}%
    \put(0,0){\includegraphics[width=\unitlength,page=1]{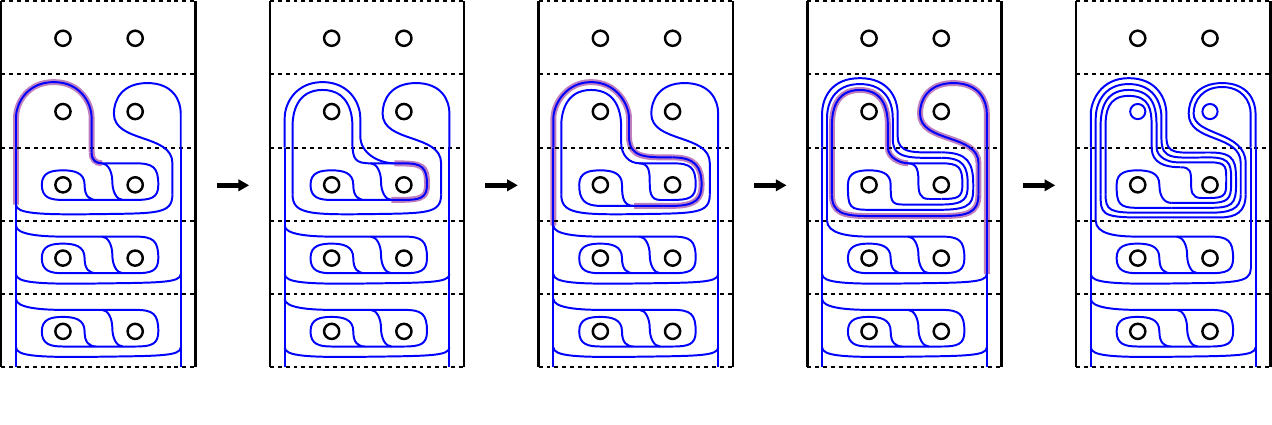}}%
    \put(0.06232361,0.00439558){\color[rgb]{0,0,0}\makebox(0,0)[lt]{\lineheight{1.25}\smash{\begin{tabular}[t]{l}$\tau$\end{tabular}}}}%
    \put(0.89879769,0.0043947){\color[rgb]{0,0,0}\makebox(0,0)[lt]{\lineheight{1.25}\smash{\begin{tabular}[t]{l}$f(\tau)$\end{tabular}}}}%
  \end{picture}%
\endgroup%
}
    \caption{A splitting sequence carrying the unstable Handel-Miller lamination $\Lambda_+$. At each step, the edge to be split next is highlighted in purple.}
    \label{fig:fenleyexample}
\end{figure}

This splitting sequence admits a sub-splitting sequence, i.e. sub-train tracks $\tau'_i \subset \tau_i$ such that each splitting move $\tau_i \to \tau_{i+1}$ takes $\tau'_i$ to $\tau'_{i+1}$. 
The corresponding sub-lamination $\Lambda'_+$ admits a transverse measure $\mu$ of full support such that $f_*(\mu)=2\mu$, which we represent as weights on the branches of $\tau'$ in \Cref{fig:fenleyexamplesub}.

The reason why $\Lambda_+$ does not admit an $f$-invariant transverse measure of full support is roughly because the branches in $\tau'$ fold up at a higher rate than those outside of $\tau'$, so any invariant measure will vanish on the latter branches. See \cite[Section 5]{Fen97} for details of this argument.
\end{example}

\begin{figure}
    \centering
    \fontsize{8pt}{8pt}\selectfont
    \resizebox{!}{2in}{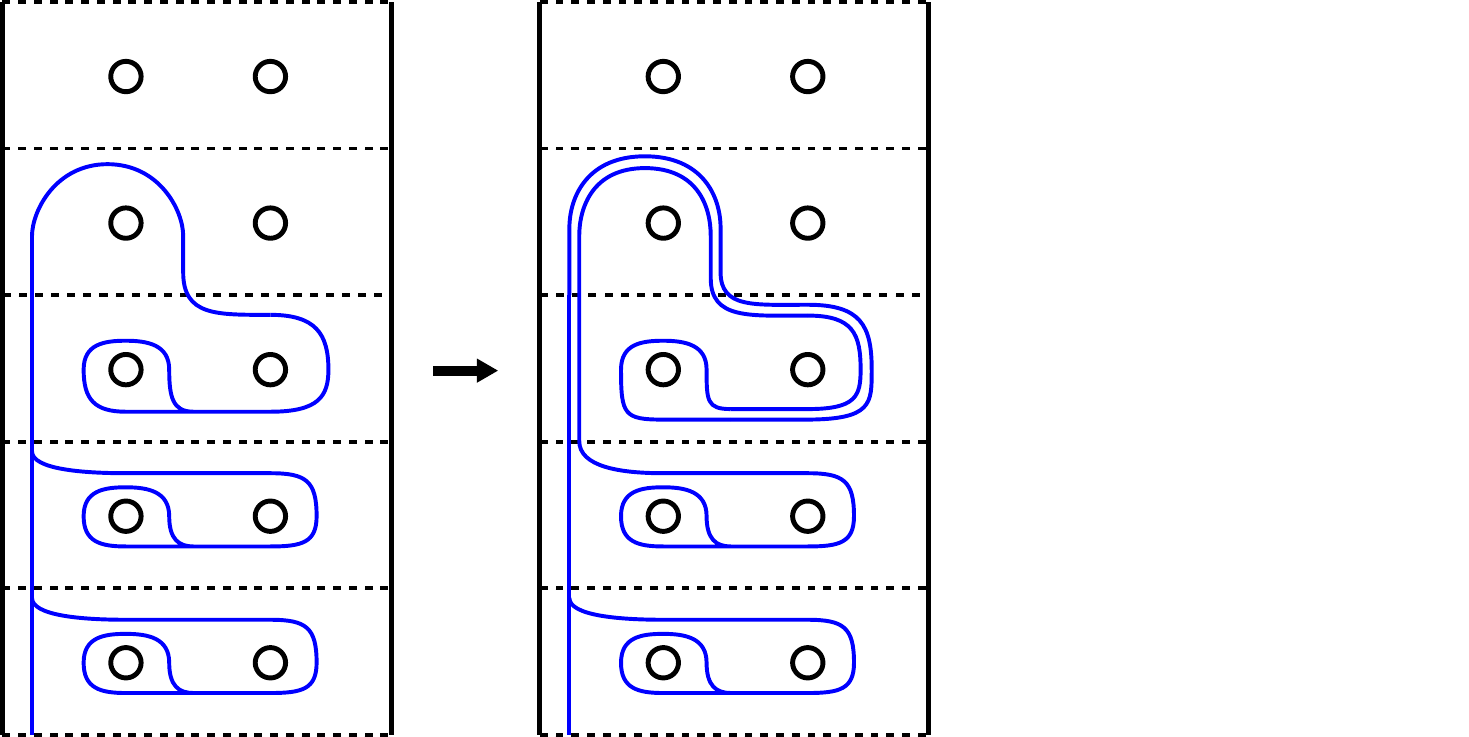}
    \caption{A sub-splitting sequence of the splitting sequence in \Cref{fig:fenleyexample} that carries an $f$-invariant sublamination of $\Lambda_+$. The sublamination admits a transverse measure $\mu$ such that $f_*(\mu)=2\mu$.}
    \label{fig:fenleyexamplesub}
\end{figure}

\begin{example} \label{eg:bookofIbundles}
Take two copies $C_1,C_2$ of the compactified mapping torus in \Cref{eg:stackofchairs} (see \Cref{fig:stackofchairsflow}), each admitting an associated disk decomposition intersecting each of the $4$ sutures once, and take the product sutured manifold $P = S_{1,2} \times [0,1]$ where $S_{1,2}$ is a torus with two boundary components. 

Note that the core of each suture of $C_i$ has a natural orientation arising from the coorientation of the decomposing disk. By fixing an orientation of $S_{1,2}$, we also get orientations on the cores of the sutures of $P$.

We define $Q$ to be the sutured manifold obtained by gluing one suture of $P$ to a suture of $C_1$, and the other suture of $P$ to a suture of $C_2$. Here the former gluing is done in a way that preserves the orientations on the cores, while the latter gluing reverses the orientations on the cores. See \Cref{fig:bookeg} top.

\begin{figure}
    \centering
    \resizebox{!}{9cm}{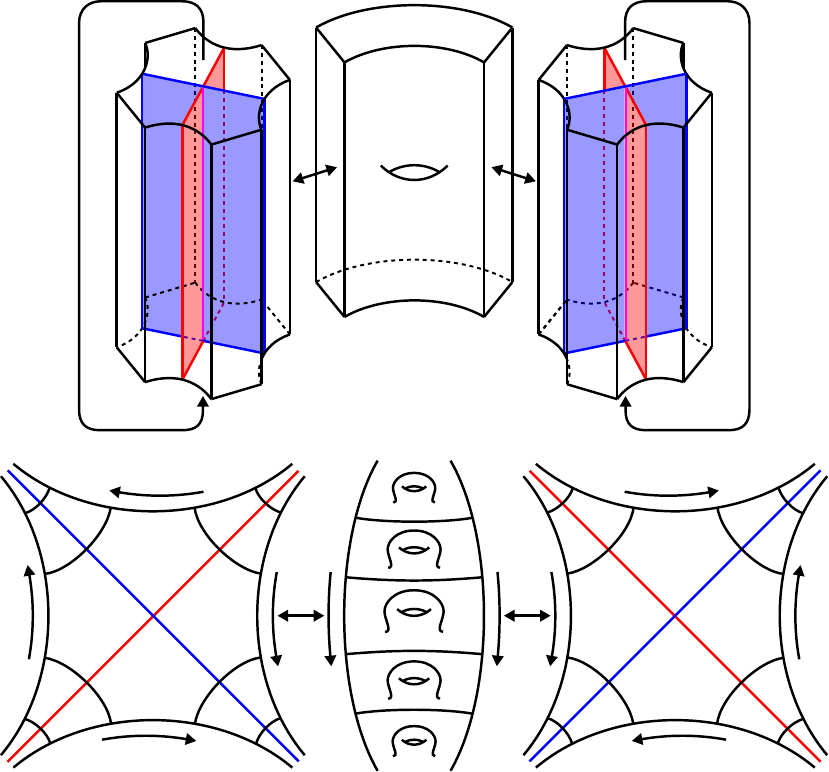}
    \caption{Top: A cartoon of the sutured manifold $Q$ considered in \Cref{eg:bookofIbundles}. Bottom: The endperiodic monodromy of a depth one foliation on $Q$.}
    \label{fig:bookeg}
\end{figure}

Up to an isotopy we can assume that the decomposing disk $D_i$ of $C_i$ intersects $P$ in an arc of the form $\{x_i\} \times [0,1]$ for $x_i \in \partial S_{1,2}$. By our orientation requirements, we can take an arc $a$ in $S_{1,2}$ connecting $x_1$ and $x_2$ and coorient $a \times [0,1] \subset P$ such that $D_1 \cup (a \times [0,1]) \cup D_2$ determines a disk decomposition of $Q$ to a product sutured manifold. 

This disk decomposition determines a depth one foliation on $Q$. We illustrate the endperiodic monodromy of this foliation, together with its positive and negative Handel-Miller laminations, in \Cref{fig:bookeg} bottom.

One can construct a veering branched surface $B$ on $Q$ compatibly carrying the suspended unstable Handel-Miller lamination $\mathcal{L}^u$ by simply taking the union of the veering branched surfaces $B_i$ constructed on $C_i$ in \Cref{eg:stackofchairs} (and extending the vector field to go from $S_{1,2} \times \{0\}$ to $S_{1,2} \times \{1\}$ on $P$).

Now notice that in the construction, if we instead glue both pairs of sutures in orientation preserving ways, we would end up with the same sutured manifold $Q$. Under this choice of gluings, we can again take the union of the veering branched surfaces $B_i$ constructed on $C_i$ in \Cref{eg:stackofchairs} to get a veering branched surface $B'$ on $Q$. 

Intuitively, $B'$ is obtained from $B$ by reversing the dynamic orientation in exactly one of $C_i$. In particular, $B$ and $B'$ have the same underlying branched surface but differ as dynamic branched surfaces. Hence $B'$ fully carries the unstable Handel-Miller lamination $\mathcal{L}^u$ considered above, but does not \emph{compatibly} carry it.

The flow graph $\Phi'$ of $B'$ consists of two cycles, which are the oriented cores of the $C_i$ respectively. By assumption on the orientations, the two cycles are mutually inverse in $H_1(Q)$, hence the dual of the cone generated by the $\Phi'$-cycles, $\mathcal{C}^{\vee}_{\Phi'}$, cannot be top-dimensional in $H_2(Q, \partial Q)$. In particular $\mathcal{C}^{\vee}_{\Phi'}$ is not a foliation cone. 

Thus $B'$ gives an example illustrating why `compatibly carrying' cannot be replaced by `fully carrying' in \Cref{thm:folcone} and \Cref{thm:hmvbsunique}.

This example can be generalized to any sutured manifold that is a \textbf{book of $I$-bundles}. These are sutured manifolds that can be written as $Q=T \bigcup_A (S \times I)$, where $T$ is a nonempty disjoint union of solid torus, called the \textbf{bindings}, $S$ is a (possibly disconnected) surface with components of negative Euler characteristic, and $S \times I$ is glued along certain components of $\partial S$ $\times I$ onto homotopically nontrivial annuli on $\partial T$.

There are $2^{\text{\#bindings}}$ ways of picking dynamic orientations on each binding of $Q$. For each choice one can construct a veering branched surface on $Q$ which is the disjoint union of suspensions of prongs inside each binding and possessing the prescribed dynamic orientation. As we vary the choice of dynamic orientations, these veering branched surfaces all have the same underlying branched surface but differ as dynamic branched surfaces. In particular, the cones generated by their dual/flow graph cycles in $H_1(Q)$ are not in general foliation cones.
\end{example}

\section{Questions} \label{sec:questions}
As explained in the introduction, the subject of our work-in-progress \cite{LT23} is an alternate approach to Mosher's gluing step. Together with the base step which we tackled in this paper, such a construction would give a way of building a veering triangulation from a sutured hierarchy $M=Q_0 \sut \cdots \sut Q_{n+1} = \text{surface} \times I$: First construct a veering branched surface on $Q_n$ then induct up the hierarchy to construct a veering branched surface on $M=Q_0$, finally take the dual triangulation (but notice that the veering triangulation is not of $M$ but rather $M$ minus some closed orbits in general).

In turn, the goal of this is to obtain a pseudo-Anosov flow without perfect fits on $M$ that is almost transverse to the finite depth foliation $\mathcal{F}$ corresponding to the sutured hierarchy. By a theorem of Mosher \cite{Mos92b} (see also \cite[Theorem A]{Lan22}), such a flow would recover the Thurston face containing the compact leaves of $\mathcal{F}$. This flow would exist by the correspondence theory between veering triangulations and pseudo-Anosov flows, provided that the final veering branched surface on $M$ has no index $0$ cusped solid tori complementary regions. This is the reason for our aversion to index $0$ cusped solid tori in \Cref{sec:dynamicpairs}. 

However, this should not be possible for all hierarchies. To see why, consider the depth $1$ case, i.e. the case when $n=1$. Applying \Cref{thm:depthonetovbs}, we get a veering branched surface $B_1$ on $Q_1$, and we want to extend this into a veering branched surface $B_0$ on $M=Q_0$, say without index $0$ cusped solid tori complementary regions. Since $B_0$ is an extension of $B_1$, the dynamics on $B_0$ should contain that of $B_1$, in particular the (isotopy classes of) closed orbits of the vector field on $B_1$ should be a subset of that on $B_0$. But the latter is the set of closed orbits of a pseudo-Anosov flow $\phi$ with no perfect fits, in particular $\phi$ has \textit{no oppositely oriented parallel orbits}, i.e. closed orbits $c_1$, $c_2$ for which $c_1$ is homotopic to $-c_2$ in $M$. So the same has to be true for the closed orbits of the vector field on $B_1$ as well.

In other words, a necessary condition for the construction to work is for the gluing map $R_+(Q_1) \cong R_-(Q_1)$ determined by the sutured decomposition to be so that no closed orbits on $B_1$ become oppositely oriented parallel in $M$. By the same reasoning, such a condition is necessary for the intermediate gluing steps as well. We believe that this condition is also sufficient, and plan to show this in the sequel to this paper.

Below we explain a few other possible lines of inquiry raised by our work here.

\subsection{Veering and taut polynomials}
In this paper we generalized some aspects of the theory of veering triangulations to veering branched surfaces on sutured manifolds, namely existence (\Cref{thm:depthonetovbs}) and uniqueness (\Cref{thm:hmvbsunique}) in the `fibered' case, and the equality between cones of cycles (\Cref{thm:folcone}).
There are more aspects of the veering triangulation theory that should admit direct generalizations. 

One of these is the veering and taut polynomials defined in \cite{LMT24}. The edge and face modules defined in \cite{LMT24} should generalize to the sutured setting, and possibly up to a slight modification accounting for annulus/Möbius band sectors, one should be able repeat much of the theory developed in \cite{LMT24}.

Here a generalization of the taut polynomial is particularly interesting. In the non-sutured layered case, it is shown in \cite{LMT24} that the taut polynomial is equal to the Teichm\"uller polynomial defined in \cite{McM00}. Thus a generalization of the taut polynomial in the sutured case would provide a generalization of the Teichm\"uller polynomial for endperiodic maps, possibly giving a polynomial invariant. 

In \cite{LMT23b}, it will be shown that any endperiodic map $f$ is isotopic to a ``spun pseudo-Anosov (spA) map," and the entropy of the restriction of an spA map to its maximal compact invariant subset generalizes the entropy of a pseudo-Anosov map. In particular, it equals the growth rate of closed orbits of the Handel-Miller suspension flow, as well as the maximal growth rate of the intersection number between $\alpha$ and $f^n(\beta)$ over all closed curves $\alpha$, $\beta$. A generalization of the Teichm\"uller polynomial is expected to contain information about this entropy, analogous to the non-sutured case as explained in \cite{LMT23a}.

\subsection{(Foliation) cones}\label{sec:conequestions}
Another particular result to hope for a generalization of is \cite[Theorem 5.15]{LMT24}, more specifically, the statement that the cone generated by cycles of the dual/flow graph is dual to a foliation cone in $H^1(Q)$ if and only if these cycles lie in an open half space of $H_1(Q)$. The backwards direction is the difficult part here. One should be able generalize the argument of \cite[Proposition 5.16]{LMT24} to show that the veering branched surface $B$ is layered, thus corresponds to a splitting sequence of endperiodic train tracks. In the non-sutured setting, we know that the train tracks in such a splitting sequence must carry the unstable foliation of the pseudo-Anosov monodromy, due to some strong uniqueness results (see for example \cite{FLP79}). However in the sutured setting we do not have a strong enough uniqueness statement that allows us to relate the dynamics of the Handel-Miller suspension flow to $B$.

In general, the cone generated by dual cycles of a veering triangulation determines the cone over some (not necessarily top dimensional) face of the Thurston norm ball, as shown in \cite{Lan22}. The cones associated to veering branched surfaces on sutured manifolds are a generalization of these cones. Meanwhile, there is a generalization of the Thurston norm to the \textit{sutured Thurston norm} on sutured manifolds. One can ask if the cones associated to veering branched surfaces are related to the cones over faces of the sutured Thurston norm ball. The answer might be rather subtle here however, since foliation cones do not generally agree with cones over sutured Thurston norm faces.

\subsection{Correspondence with pA flows}
As mentioned before, a key fact about veering triangulations is that they correspond to pseudo-Anosov flows. In the appendix, we will show that a veering branched surface forms one half of a dynamic pair \'a la Mosher, hence in particular induces what Mosher calls a pA flow on the sutured manifold, see \cite[Section 4.10]{Mos96}. One can ask whether a construction in the other direction: a pA flow to a veering branched surface is possible, and whether the two constructions are inverse to each other (in a suitable sense) as in the non-sutured case. 

Note that the technical condition of no perfect fits plays a key role in this correspondence theory in the non-sutured setting. One might have to invent a similar condition before developing a correspondence theory in the sutured case.

\subsection{Stable veering branched surfaces}
In the non-sutured setting, a veering triangulation is dual to both its stable and unstable branched surfaces. In this paper we have been dealing with unstable dynamical branched surfaces mainly. Symmetrically, we could have dealt with stable dynamical branched surfaces. However, since there is no middle triangulation to tie them together, it is not clear if the two approaches fit together exactly.

More precisely, suppose we have the compactified mapping torus of an endperiodic map. One constructs an unstable veering branched surface carrying the unstable Handel-Miller lamination as in \Cref{thm:depthonetovbs}, and symmetrically constructs a stable veering branched surface carrying the stable Handel-Miller lamination. Is there a combinatorial way to recover the stable and unstable veering branched surfaces from some cellular decomposition, as in the non-sutured case?

If one can find a positive answer to this last question than it should be possible to associate a stable veering branched surface to a general unstable veering branched surface (and vice versa). One can then ask whether such a pair of stable and unstable veering branched surface can be isotoped to form a dynamic pair. In the non-sutured case this is shown to be possible by Schleimer and Segerman in \cite{SS24}.

\subsection{Census of veering branched surfaces}
The examples of veering branched surfaces which we worked out in \Cref{sec:eg} are all quite simple. In particular many of them just have one interior vertex in their branch locus, i.e. one triple point or one source. It might be an interesting question to classify all veering branched surfaces with one interior vertex. In particular, do all of these carry the unstable Handel-Miller lamination of some endperiodic map? We note that in comparison, the veering branched surfaces with no interior vertices should be easy to classify.

In general, it might be interesting to generate a census of veering branched surfaces with a small number of interior vertices. This has been done in the non-sutured case up to 16 vertices by Giannopolous, Schleimer, and Segerman \cite{GSS}. Such a census might in particular give more examples of veering branched surfaces that are not layered, hence do not carry the unstable Handel-Miller lamination of an endperiodic map.

\appendix

\section{Dynamic pairs} \label{sec:dynamicpairs}

This appendix proves that given a veering branched surface in an atoroidal sutured manifold $Q$, one can construct a ``dynamic pair" in $Q$. We also prove that this implies the existence of a dynamic pair as defined by Mosher (called a ``Mosher pair" in what follows). This recovers the base step in Mosher's program discussed in the introduction. The ideas in this appendix will also feature prominently in the sequel paper \cite{LT23}, where we will explore the gluing step of Mosher's program. Although in some cases one might ultimately care only about the existence of a single dynamic branched surface $B^u$ at the top of a sutured hierarchy, it is necessary for us to keep track of a complementary branched surface $B^s$ to guide the construction of $B^u$.

As the name suggests, a dynamic pair consists of both a stable and an unstable branched surface. In \cite{Mos96}, Mosher introduced a tool called a ``dynamic train track" for promoting an unstable branched surface to a dynamic pair. We will prove that the flow graph of a veering branched surface is a dynamic train track. Thus by Mosher's results together with our construction of veering branched surfaces, any depth one sutured manifold contains a dynamic pair $(B^u,B^s)$ as defined by Mosher.

Crucially, however, Mosher's recipe for producing a dynamic pair involves dynamically splitting along annuli and Möbius bands, thereby introducing index 0 cusped tori. In using the correspondence between veering branched surfaces and pseudo-Anosov flows, one needs to avoid index 0 cusped tori in order to ensure a flow has no perfect fits. Since a main goal of our ongoing project is to understand when it is possible to produce a pseudo-Anosov flow with no perfect fits at the top level of a sutured hierarchy (see \Cref{sec:questions} for more discussion), we need to modify Mosher's theory of dynamic pairs and dynamic train tracks. In particular we give a definition of dynamic pair which is slightly different than Mosher's, which then requires us to rework the proof that promotes a branched surface with a dynamic train track to a dynamic pair. 
Many of the arguments liberally use ideas of Mosher, rearranged so as to be compatible with our setup.

\subsection{Dynamic pairs}

A dynamic pair consists of a pair of dynamic branched surfaces, one stable and one unstable. To give a precise definition, we first have to define a few more classes of dynamic manifolds (continuing from \Cref{defn:cuspedtorus} and \Cref{defn:cuspedproduct}) which can arise as complementary regions of such a pair of branched surfaces.

\begin{definition}[Dynamic tori] \label{defn:dyntorus}
Let $\Delta$ be a closed disk whose boundary is smooth with the exception of $2n\ge 4$ corners, 
and let $f\colon \Delta\to \Delta$ be a diffeomorphism. The mapping torus $M$ of $f$ is a 3-manifold with corners homeomorphic to a solid torus. There are $\frac{2n}{p}$ circular corner edges of $M$, where $p$ is the period of a cusp of $\Delta$ under $f$. Likewise there are $\frac{2n}{p}$ annular faces of $M$. We label these faces $\u$ and $\s$ alternatingly. Take a circular vector field $V$ that gives it the structure of a dynamic manifold. Equipped with such a vector field, $M$ is called a \textbf{dynamic solid torus}. See \Cref{fig:dynpairpieces} left.
Define the \textbf{index} of $M$ to be $1-\frac{n}{2}$.

A \textbf{dynamic torus shell} is defined similarly, but replacing $\Delta$ by a closed annulus whose boundary is smooth with the exception of $2n\ge2$ corners on a single boundary component. The annulus faces are labeled $\u$ and $\s$ alternatingly, the torus face is labeled $\b$, and circularity of the vector field is defined as for the dynamic solid torus.

We refer to a dynamic solid torus or a dynamic torus shell as a \textbf{dynamic torus}.
\end{definition}

\begin{definition}[Maw piece] \label{defn:mawpiece}
Let $\Delta$ be a closed disk whose boundary is smooth with the exception of $2$ corners and $1$ cusp. $M=\Delta \times S^1$ is a 3-manifold with corners homeomorphic to a solid torus, with $2$ corner edges and $1$ cusp edge, and $3$ annulus faces. Label the $2$ annulus faces adjacent to the cusp edge $\u$, and the remaining annulus face $\s$, and take a circular vector field $V$ that gives it the structure of a dynamic manifold. We call $M$ a \textbf{$\s\u\u$-maw piece}. See \Cref{fig:dynpairpieces} middle.
The definition of a \textbf{$\u\s\s$-maw piece} is symmetric.
\end{definition}

\begin{construction}[Pinching an edge]
If $e$ is a $\p\s$-edge connecting two $\p\s\u$-corners of a dynamic manifold, then we can \textbf{pinch} $e$ as shown in \Cref{fig:pinch}. This produces a $\u\u$-edge connecting a $\p\u\u$-gable to a $\s\u\u$-gable. Symmetrically we can pinch a $\m\u$-edge connecting two $\m\s\u$-corners to produce an $\s\s$-edge connecting an $\m\s\s$-gable to a $\u\s\s$-gable.
\end{construction}

\begin{figure}
    \centering
    \resizebox{!}{4cm}{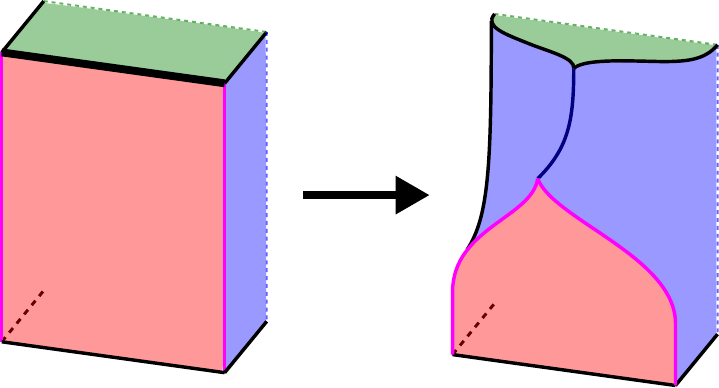}
    \caption{Pinching a $\p\s$-edge (thickened) connecting two $\p\s\u$-corners.}
    \label{fig:pinch}
\end{figure}

\begin{definition}[Drum] \label{defn:drumpiece}
Let $S$ be a surface with an even number of corners at each boundary component and with $\ind(S) \leq 0$. For each boundary component $c$ of $S$, make some choice as follows:
\begin{itemize}
    \item If $c$ has $2n \geq 2$ corners, label the sides by $\u$ and $\s$  alternatingly
    \item If $c$ has no corners, either label the side by $\b$ or choose an orientation of $c$.
\end{itemize}

Now $S \times [0,1]$ is a 3-manifold with corners. We can modify it in the following ways:
\begin{itemize}
    \item Let $c$ be a boundary component of $S$ with $2n \geq 2$ corners. Pinch each $\p\s$ edge of $c\times\{1\}$, and pinch each $\m\u$-edge of $c\times\{0\}$.
    \item For a boundary component $c$ of $S$ with no corners and labeled $\b$, label $c \times [0,1]$ by $\b$.
    \item For a boundary component $c$ of $S$ with no corners and not labeled $\b$, label $c \times [0,1/2]$ by $\s$ and $c \times [1/2,1]$ by $\u$.
    \item If $S$ is not a rectangle, label the image of $S \times \{1\}$ by $\p$ and the image of $S \times \{0\}$ by $\m$. 
    \item If $S$ is a rectangle, either label the image of $S \times \{1\}$ by $\p$ or pinch it into a $\u\u$-cusp edge, symmetrically, either label the image of $S \times \{0\}$ by $\m$ or pinch it into a $\s\s$-cusp edge.
\end{itemize}
Now choose a vector field $V$ that gives the resulting $3$-manifold with corners the structure of a dynamic manifold, and such that $V$ is circular on the annulus faces on $c \times [0,1]$ for every boundary component $c$ without corners and not labeled $\b$, and induces the dynamic orientation on them which is specified by the chosen orientation on $c$. We call this dynamic manifold a \textbf{drum}.
\end{definition}

In \Cref{fig:dynpairpieces} right we show one example of a drum. Here we start with an annulus with $6$ corners on one boundary component and no corners on the other. We label the boundary component without corners by $\b$.

Among all drums, there are some that deserve specific names. A \textbf{coherent annulus drum} is the drum obtained by taking $S$ to be an annulus without corners and with both boundary components oriented in the same direction around $S$. A \textbf{$\p$-pinched rectangle drum} is the drum obtained by taking $S$ to be a rectangle and pinching only the $\p$-face in the last step above. A \textbf{$\m$-pinched rectangle drum} is the drum obtained by taking $S$ to be a rectangle and pinching only the $\m$-face. A \textbf{pinched tetrahedron} is the drum obtained by taking $S$ to be a rectangle and pinching both the $\p$- and $\m$-faces. See \Cref{fig:ind0pieces}.

\begin{figure}
    \centering
    \resizebox{!}{5cm}{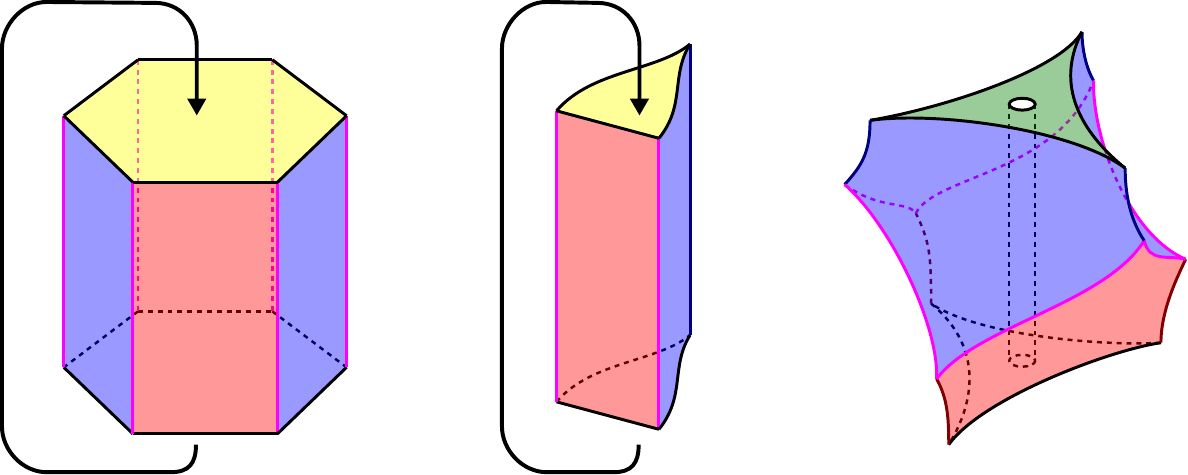}
    \caption{Left: A dynamic solid torus. Middle: An $\s\u\u$-maw piece. Right: A drum.}
    \label{fig:dynpairpieces}
\end{figure}

\begin{figure}
    \centering
    \resizebox{!}{2.5cm}{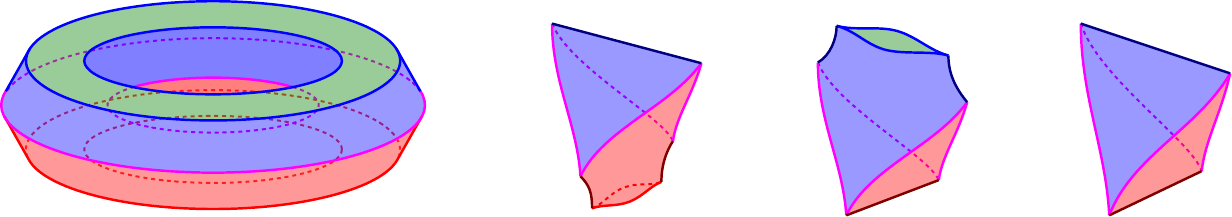}
    \caption{From left to right: A coherent annulus drum, a $\p$-pinched rectangle drum, an $\m$-pinched rectangle drum, and a pinched tetrahedron.}
    \label{fig:ind0pieces}
\end{figure}

\begin{definition} \label{defn:dynamicpair}
Let $Q$ be a sutured manifold, $B^u$ and $B^s$ be branched surfaces in $Q$, and $V$ be a $C^0$ vector field on $Q$. $(B^u,B^s,V)$ is said to be a \textbf{dynamic pair} if:
\begin{enumerate}
    \item $(Q,V)$ is a dynamic manifold
    \item $(B^u,V)$ is an unstable dynamic branched surface and $(B^s,V)$ is a stable dynamic branched surface. Note that unlike in the main text, we are not assuming that $V$ is smooth, and in fact $V$ cannot be smooth here.
    \item $V$ is smooth on $Q$ except along $\brloc(B^u)$ where it has locally unique forward trajectories and along $\brloc(B^s)$ where it has locally unique backward trajectories.
    \item  Each component of $Q \cut (B^u \cup B^s)$ is a dynamic torus, a drum, or a maw piece. (Notice that $Q \cut (B^u \cup B^s)$ is a dynamic manifold with the faces corresponding to $B^{u/s}$, $R_\pm$, and $\gamma$ labeled $\u/\s$, $\p/\m$, and $\b$ respectively).
    \item There is a collection $S^u$ of annuli and Möbius bands, called the \textbf{sinks}, which are carried by $B^u \cut B^s$ and with boundary components on $B^u \cap B^s$, such that:
    \begin{enumerate}
        \item Every $\u\s\s$-maw piece component $\mu$ of $Q \cut (B^u \cup B^s)$ is attached to an element $F$ of $S^u$, i.e. the $\u$-face of $\mu$ is identified with $F$, and is \textbf{boundary parallel}, i.e. there exists an annulus carried by $B^s$ with one boundary component on the maw circle of $\mu$ and the other boundary component on $R_-$.
        \item Every component $K$ of $B^u \cut B^s$ either contains an element $F$ of $S^u$ and $F$ is a sink of $K$, or every forward trajectory of $K \backslash (B^u \cap B^s)$ is finite and ends on $R_+$.
        \item The boundary components of elements of $S^u$ do not overlap, i.e. there does not exist boundary components $c_1, c_2$ of elements of $S^u$ which map to the same curve in $Q$, nor can a boundary component $c$ of an element of $S^u$ double cover a curve in $Q$.
    \end{enumerate}
    Similarly, there is a collection $S^s$ of annuli and Möbius bands, called the \textbf{sources}, satisfying the symmetric properties.
    \item The boundary train tracks $\beta^u=B^u \cap R_+$ and $\beta^s=B^s \cap R_-$ do not carry Reeb annuli.
    \item No component of $Q \cut (B^u \cup B^s)$ is a coherent annulus drum.\qedhere
\end{enumerate}
\end{definition}

In \Cref{defn:dynamicpair} a dynamic pair is slightly more general than what Mosher calls a ``dynamic pair" in \cite[Section 4.5]{Mos96}, which we will call a \textbf{Mosher pair}.

The following lemma explains the relationship between dynamic pairs and Mosher pairs for those interested. We do not use Mosher pairs here, so we will not provide the definition, instead referring the reader to \cite[\S 4.5]{Mos96}.

\begin{lemma}
Any Mosher pair is a dynamic pair, and any dynamic pair $(B^u, B^s, V)$ can be transformed into a Mosher pair by dynamically splitting along elements of $S^u$ and $S^s$.
\end{lemma}

\begin{proof}
If $(B^u, B^s, V)$ is a Mosher pair one can take $S^u$ to be the set of annulus/Möbius band sinks of components of $B^u \cut B^s$ and choose $S^s$ symmetrically. Then it is straightforward to check that $(B^u, B^s, V)$ satisfy \Cref{defn:dynamicpair}.

Conversely, if one has a dynamic pair $(B^u, B^s, V)$ à la \Cref{defn:dynamicpair}, the only axioms in the definition of a Mosher pair that could possibly fail are: \textit{every maw piece is attached to a torus piece}, \textit{transience of forward/backward trajectories}, and \textit{separation of torus pieces}.
\textit{Separation of torus pieces} in fact holds because if two torus pieces are glued along some pair of say $\u$-faces, then the $\s$-faces of the torus pieces must lie in $S^s$, and the boundary components of the $\s$-faces that are adjacent to the glued $\u$-faces will then be identified.

We can transform $(B^u, B^s, V)$ into a Mosher pair by splitting along sources and sinks so that all the resulting sources and sinks are faces of dynamic tori; this property implies the transience and maw piece axioms. 

Here is how the splitting works. Suppose there is a sink $F$ of $S^u$ that is not a face of a torus piece. Enlarge $F$ to a slightly larger annulus/Möbius band $F'$ still carried by $B^u$ and containing $\partial F$ in its interior. Dynamically split $B^u$ along $F'$. This creates an index $0$ dynamic solid torus and some $\s\u\u$-maw pieces and/or pinched tetrahedra, while the topology of the other components of $Q \backslash (B^u \cap B^s)$ is unchanged. One can modify $V$ so that this new $(B^u, B^s, V)$ satisfies (1)-(4) in \Cref{defn:dynamicpair}. 

The new $S^u$ is obtained by replacing $F$ in the original collection by the annuli double covering it under the splitting. That axioms (5a-c) for $S^u$ are preserved is clear from construction.
The new $S^s$ is obtained by adding in the $\s$-faces of the dynamic solid torus that is created. New $\s\u\u$-maw pieces are created on the sides of $F' \backslash F$ with no branches spiraling out. Those components of $F' \backslash F$ are contained in components of $B^u \backslash B^s$ for which every forward trajectory ends on $R_+$, otherwise (5c) will fail for $F$. Hence the new $\s\u\u$-maw pieces are attached to the new $\s$-faces and are boundary parallel, verifying (5a) for $S^s$ for the new $(B^u, B^s, V)$. (5b) is clear from construction. For (5c), the only way this could fail is if there were elements of $S^s$ sharing a boundary curve with $F$. But then $F$ is a face of a dynamic torus in the first place.

Axioms (6) and (7) are clearly preserved from construction as well. Now we repeat the argument by splitting elements in the new $S^u$ or $S^s$ that are not faces of dynamic tori. Since this kind of splitting always reduces the number of elements in $S^u$ or $S^s$ that are not faces of dynamic tori, the process terminates eventually and we get a Mosher pair.
\end{proof}

In particular, a sutured manifold contains a dynamic pair if and only if it contains a Mosher pair. 
We will use this fact implicitly going forward.

\begin{proposition}[{\cite[Proposition 4.10.1]{Mos96}}] \label{prop:dynamicpairprop}
If a sutured manifold $Q$ contains a dynamic pair, then $Q$ is irreducible, each face of $Q$ is incompressible, and no two torus components of $\gamma$ are isotopic.
\end{proposition}

When constructing a dynamic pair, it will be useful to ignore axioms (5c) and (7) during the initial stages,  later adjusting for them to hold. We record a proposition explaining this adjustment.

\begin{proposition} \label{prop:protodynamicpair}
Let $Q$ be an atoroidal sutured manifold with no torus components of $R_\pm$, let $B^u$ and $B^s$ be dynamic branched surfaces in $Q$, and let $V$ be a vector field on $Q$. Suppose $(B^u,B^s,V)$ satisfies all but (5c), (7) in \Cref{defn:dynamicpair}. If $B^u$ and $B^s$ do not carry closed surfaces, then $Q$ contains a dynamic pair $(B'^u, B'^s, V')$.

Furthermore,
\begin{itemize}
    \item if $(B^u, B^s, V)$ satisfies (7) as well, then $B'^u$ can be chosen to be a sub-branched surface of $B^u$ and $B'^s$ can be chosen to be a sub-branched surface of $B^s$;
    \item if $(B^u,B^s,V)$ satisfies (7) and $B^u$ satisfies (5c) as well, then $B'^s$ can be chosen to be $B^s$ and $B'^u$ can be chosen to be a sub-branched surface of $B^u$. The symmetric statement holds.
\end{itemize}
\end{proposition}

\begin{proof}
This is essentially proven in \cite[\S 4.12]{Mos96}. We outline an argument here, emphasizing the places that are slightly different due to our definition of a dynamic pair, and referring to the relevant sections of \cite{Mos96} for details. 

Given a coherent annulus drum $D$, we now describe how to add an annulus sector to $B^u$ and an annulus sector to $B^s$, and then eliminate the $\m$ and $\p$ faces of $D$, to create a dynamic torus and some other types of pieces.
This is depicted in \Cref{fig:cohanndrummawcase} together with an additional zipping step, and is essentially \cite[`An example' on P.193-195]{Mos96}. 
When attaching the new sectors, we need to take into consideration the branching of $B^s$ and $B^u$ on $\del D$. Let $A$ be a $\u$-face of $D$ which has a $\u\s$-circle $c$ of $D$ on its boundary. 
\begin{enumerate}[label=(\alph*)]

\item If $c\cap \brloc(B^u)$ is nonempty, then we also attach the extra $s$ sector very close to $c$ outside $D$. This creates some number of pinched tetrahedra.
\item If $c\cap \brloc(B^u)=\varnothing$ and $\brloc(B^u)\cap A=\varnothing$, then we attach the extra $B^s$ sector very close to $c$ outside $D$. The component of $Q\cut (B^u\cup B^s)$ on the other side of $A$ from $D$ is necessarily a drum in this case, so this attachment creates a boundary parallel maw piece.
\item If $c\cap \brloc(B^u)=\varnothing$, then one can show that $\brloc(B^u)\cap A$ contains some loop that is the $\u\u$-cusp circle of an $\s\u\u$-maw piece, and that there is a (possibly larger) maximal family of maw pieces whose union $w$ has the smooth structure of a maw piece as shown in \Cref{fig:cohanndrummawcase}. Let $c'$ be the other $\u\s$-curve of $w$ besides $c$. We attach the extra $B^s$ sector just outside of $w$ near $c'$ so that it passes through $w$ near its $\s$-face as shown in \Cref{fig:cohanndrummawcase}. This creates some additional index 0 dynamic solid tori and either pinched tetrahedra (as in (a)) or a boundary parallel $\u\s\s$-maw piece (as in (b)), depending on whether $c'$ intersects $\brloc(B^u)$.
\end{enumerate}

We attach the other boundary component of the new annulus sector of $B^s$ using the same recipe applied to the other $\u$-face of $D$. Then we attach a new annulus sector of $B^u$ symmetrically.

Finally we zip up the $\p$ and $\m$ faces of $D$, as shown in \Cref{fig:cohanndrummawcase} right. This is possible by our assumption that there are no torus components of $R_\pm$.

We augment $S^u$ and $S^s$ by adding the $\u$-faces and $\s$-faces, respectively, of all dynamic tori created by the above steps. This process reduces the number of coherent annulus drums by one. After performing it finitely many times, we can arrange so that $(B^u, B^s, V)$ satisfies all but (5c) in \Cref{defn:dynamicpair}.

\begin{figure}
    \centering
    \fontsize{10pt}{10pt}
    \resizebox{!}{3.5cm}{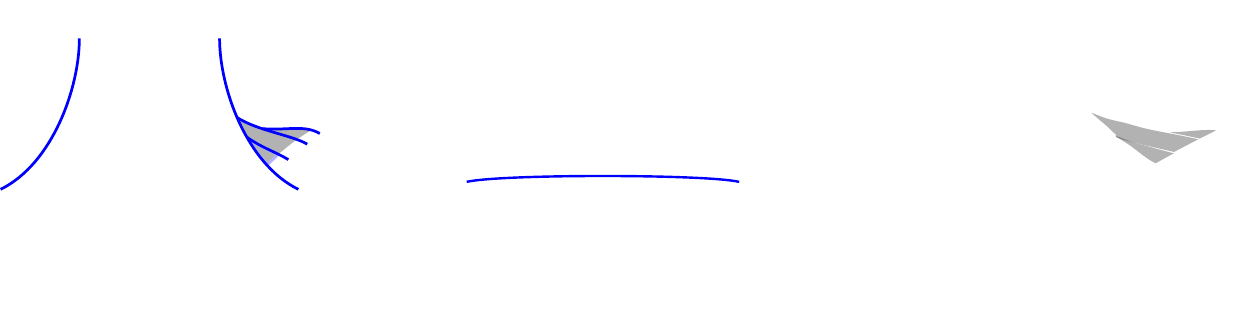}
    \caption{Modifying a coherent annulus drum to obtain a dynamic solid torus and some other pieces. Middle: Attaching a sector to $B^s$ near a $\u\s$-circle of a coherent annulus drum when a $\u$-face of the drum intersects $\brloc(B^u)$ in a collection of circles, as in case (c) from the proof of \Cref{prop:protodynamicpair}. The region $w$ is shaded.}
    \label{fig:cohanndrummawcase}
\end{figure}

Suppose there exists boundary components $c_1, c_2$ of elements of $S^u$ which map to the same curve $c$ in $Q$. Then $c$ is a core of an annulus sector $s$ of $B^s$; this uses the hypothesis that $B^s$ does not carry a closed surface, for otherwise $c$ could lie in a torus or Klein bottle sector. Also $B^u \cap s$ is a train track on $s$ with only converging switches and it contains the cycle $c$, hence it must be a union of parallel cycles with branches going from $\partial s$ to the outermost cycles. See \Cref{fig:protodynamicpairdelsector}.

\begin{figure}
    \centering
    \resizebox{!}{6cm}{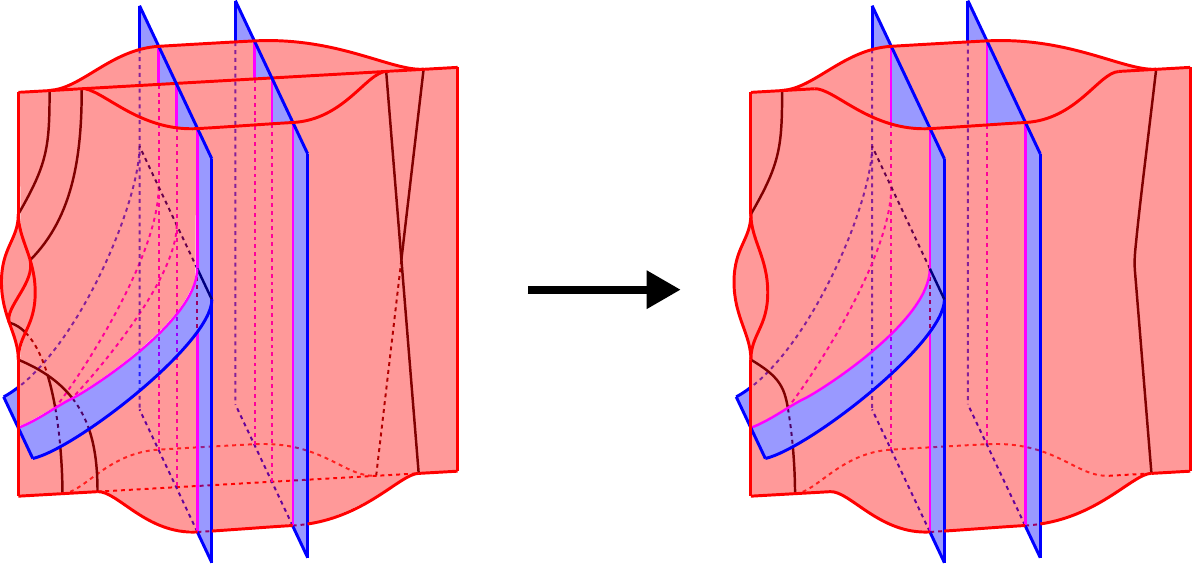}
    \caption{Locating a sector at which (5c) is violated and deleting it.}
    \label{fig:protodynamicpairdelsector}
\end{figure}

We first claim that each of the annulus regions between the cycles abuts a torus piece on both of its sides. This is true for those regions that meet $c$, since the components of $Q \cut (B^u \cup B^s)$ meeting those regions has a sequence of $\u, \s, \u, \s$-annulus faces, 
and only dynamic tori satisfy this property. Then we can repeat the argument on the annulus regions of $s$ that meet the $\s$-faces of these torus pieces, and induct outwards.

Then we claim that a region of $s$ outside of an outermost cycle containing no branches of $\tau$ is adjacent to maw pieces on both of its sides. This follows from a similar argument as above: each of the components of $Q \cut (B^u \cup B^s)$ meeting those regions has a sequence of $\s, \u, \s$-annulus faces, meaning it is either a maw piece or dynamic torus; they cannot both be torus pieces because then the cycle would not have been outermost. Hence at least one is a maw piece.
If the other component of $Q \cut (B^u \cup B^s)$ is a dynamic solid torus, then there would be a cusp circle on an annulus $\s$-face of a dynamic torus, contradicting (5b) in \Cref{defn:dynamicpair}. Hence the other component of $Q \cut (B^u \cup B^s)$ is also a maw piece.

For a region of $s$ outside of an outermost cycle containing branches of $\tau$, we claim that it is adjacent to pinched tetrahedra or $\p$-pinched rectangle drums on both of its sides. This follows from the same line of argument: the components of $Q \cut (B^u \cup B^s)$ meeting those regions have two $\u$-faces meeting at a $\u\u$-cusp and some $V$-trajectories starting on those $\u$-faces enter $\u$-faces of dynamic tori hence never end at $R_+$. Only pinched tetrahedra and $\p$-pinched rectangle drums satisfy these properties. 

Now remove $s$ from $B^s$, as indicated in \Cref{fig:protodynamicpairdelsector}. 
This glues up maw pieces to form bigger maw pieces, and glues up pinched tetrahedra and $\p$-pinched rectangle drums to form bigger pinched tetrahedra/$\p$-pinched rectangle drums. (This uses the hypothesis that $B^u$ does not carry closed surfaces, otherwise at some point a maw piece/pinched tetrahedron/$\p$-pinched rectangle drum might be glued onto itself.) Dynamic torus pieces are glued up along $\s$-faces. Any dynamic torus admits a Seifert fibration with at most 1 singular fiber such that the restriction to any $\u$- or $\s$-face is a foliation by circles, and hence the pieces obtained gluing the dynamic tori admit Seifert fibrations.

If there exists a boundary component $c$ of an element of $S^u$ which double covers a curve in $Q$, we can locate a sector and remove it similarly as above. And symmetrically the same thing can be done with $S^u$ replaced by $S^s$.
This idea of removing sectors is essentially \cite[Proposition 4.12.1 Step 2]{Mos96}.

Let $T_i$ be a disjoint collection of tori parallel to the boundary components of the Seifert fibered pieces. We may modify $V$ so that it is tangent to the $T_i$. Cutting along $\bigcup T_i$ returns sutured manifolds with the restricted $(B^u, B^s, V)$ satisfying all but (5c) in \Cref{defn:dynamicpair} and some Seifert fibered spaces disjoint from $B^u$ and $B^s$. This is \cite[Proposition 4.12.1 Step 3]{Mos96}.

On the components of $Q \cut \cup T_i$ with nonempty $B^u$ and $B^s$, we inductively repeat the above argument, deleting sectors and cutting along tori at each stage. Eventually the process will stop since there are only finitely many sectors of $B^u$ and $B^s$. When the process terminates, we get a disjoint collection of tori $T_i$ such that each component of $Q \cut \cup T_i$ either contains a dynamic pair or is a Seifert fibered space. 
We claim that each $T_i$ corresponds to the boundary of a dynamic torus. Indeed, otherwise $T_i$ would be essential in $Q$ by \Cref{prop:dynamicpairprop} and an innermost disk argument, violating atoroidality. Hence we have in fact produced a dynamic pair in $Q$.

For the additional statements: If there are no coherent annulus drums, then we do not have to modify the branched surfaces as in the first part of the proof. Thus we only have to carry out the part of the proof where we remove sectors. If furthermore (5c) is satisfied for $S^u$, then we do not even need to carry out that part of the proof for $B^s$, and so $B^s$ is left unchanged throughout. The symmetric statement is of course true for $B^u$ as well. 
\end{proof}

We remark that without the atoroidality assumption, the above proof still produces a family of essential tori separating $Q$ into a collection of pieces admitting dynamic pairs, and a collection of pieces admitting Seifert fibrations.

\subsection{Dynamic train tracks}

\begin{definition} \label{defn:dtt}
Let $B$ be an unstable dynamic branched surface. A oriented train track $\tau$ embedded in $B$ is said to be a \textbf{dynamic train track} if $\tau$ is disjoint from $\partial Q$ and if there exists a dynamic vector field $V$ such that:
\begin{enumerate}
    \item $V$ is tangent to $\tau$
    \item The set of converging switches of $\tau$ is equals to $\tau \cap \brloc(B)$
    \item $V$ is smooth on $B-\brloc(B)$ except at diverging switches of $\tau$.
    \item Every component $K$ of $B \cut \tau$ either carries an annulus or Möbius band $A$ such that $\partial A \subset \partial K$ and $A$ is a sink of $K$, or every forward trajectory of $K \backslash \tau$ is finite and ends at a point of $B \cap R_+$.
\end{enumerate}

A dynamic train track on a stable dynamic branched surface is defined symmetrically.
\end{definition}

Our terminology differs from \cite{Mos96} slightly. In \cite{Mos96}, dynamic train tracks were only required to satisfy (1)-(3) in \Cref{defn:dtt}, and those that in addition satisfy (4) are said to be \textbf{filling}. However, since we do not need to deal with non-filling dynamic train tracks, we include (4) in our definition for brevity.

Before stating the next proposition, we explain a construction called `splitting for goodness'. This construction is a slightly streamlined version of the material in \cite[Section 4.7 and Proposition 4.11.1 Step 2e]{Mos96}.

\begin{construction}[Splitting for goodness]

Let $B$ be an unstable dynamic branched surface and $\tau$ be a dynamic train track on $B$. Let $K$ be a component of $B \cut \tau$ that meets $R_+$. By definition all forward trajectories starting in the interior of $K$ are finite and end on $R_+$. Take a small regular neighborhood of $\tau \cap K$ in $K$ and let $\tau'$ be the boundary of this regular neighborhood in $K$. Meanwhile let $\beta'$ be $K \cap R_+$. The forward trajectories in the interior of $K$ induce a map from $\tau'$ to $\beta'$. 

For each interval component of $\brloc(K)$ with one endpoint on $\tau$ and one endpoint on $R_+$, dynamically split $K$ along a triangle so that this component of $\brloc(K)$ now lies close to the forward trajectory of a point on the component close to the endpoint on $\tau$. This is \cite[Figure 4.10]{Mos96}. See \Cref{fig:splittingforgoodness} first row.

For each interval component of $\brloc(K)$ with both endpoints on $R_+$, dynamically split $K$ along the bigon with one side on this component, so that this component of $\brloc(K)$ disappears. See \Cref{fig:splittingforgoodness} second row.

We refer to the operation of performing these dynamic splittings as \textbf{splitting $B$ for goodness}.
\end{construction}

\begin{figure}
    \centering
    \resizebox{!}{6cm}{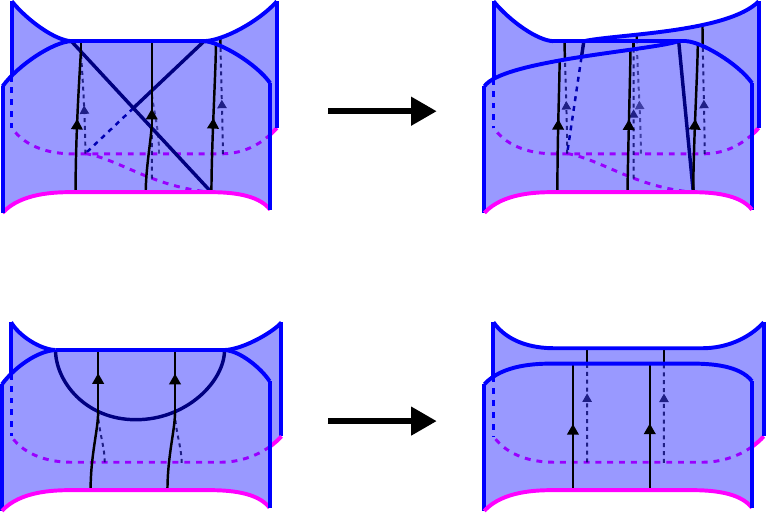}
    \caption{Splitting for goodness refers to the operation of doing these two types of splittings.}
    \label{fig:splittingforgoodness}
\end{figure}

\begin{proposition} \label{prop:dtttodynamicpair}
Let $(B^u,V)$ be a very full unstable dynamic branched surface equipped with a dynamic train track $\tau$. If $\beta^u=B^u \cap R_+$ does not carry Reeb annuli, $B^u$ does not carry closed surfaces, and $Q$ is atoroidal, then there is a dynamic pair $(B'^u, B^s,V')$ on $Q$.

Furthermore, if:
\begin{itemize}
    \item $\beta^u$ has no annulus or cusped bigon complementary regions,
    \item the $\u$-cusped product complementary regions of $B^u$ have no cusp circles, 
    and
    \item there is no annulus or Möbius band sector of $B^u$ which meets $\u$-cusped torus pieces on both sides,
\end{itemize} 
then $B'^u$ can be chosen to be $B^u$ after splitting for goodness. 
\end{proposition}
\begin{proof}
Again, this is essentially proven in \cite{Mos96}. We outline an argument here, emphasizing the places that are slightly different due to our definition of a dynamic pair and splitting for goodness, and referring to the relevant sections of \cite{Mos96} for details.

We first modify $\tau$ by removing any ``extraneous sinks." An extraneous sink is a sink $\gamma$ of $\tau$ for which there exists an annulus or Möbius band $R$ carried by $B^u$ containing $\gamma$ in its interior, with boundary components lying on $\tau$. As Mosher describes in \cite[Proposition 4.11.1 Step 2b]{Mos96}, there is a process for removing $\gamma$ along with branches of $\tau$ spiraling into $\gamma$, such that the result still satisfies our definition of dynamic train track. We remove $\gamma$ in this way and denote the new dynamic train track by $\tau$ also.

Then we split $B^u$ for goodness. As remarked above, this step is essentially \cite[Proposition 4.11.1 Step 2e]{Mos96}.

Next, consider a $\u$-cusped product complementary region of $B^u$, homeomorphic to $S \times [0,1]$. Suppose there are annulus $\u$-faces on $S \times \{1 \}$ that are parallel, have parallel dynamic orientations, and cobound a collection of cusped bigon $\p$-faces. Consider a maximal collection of such annulus $\u$-faces and add an annulus sector $s$ to $B^u$ with boundary components on the outermost faces. Also augment $\tau$ by adding 2 copies of the core of $s$, oriented by the dynamic orientation of the $\u$-faces. See \Cref{fig:dividebreachtorus}.

\begin{figure}
    \centering
    \resizebox{!}{6cm}{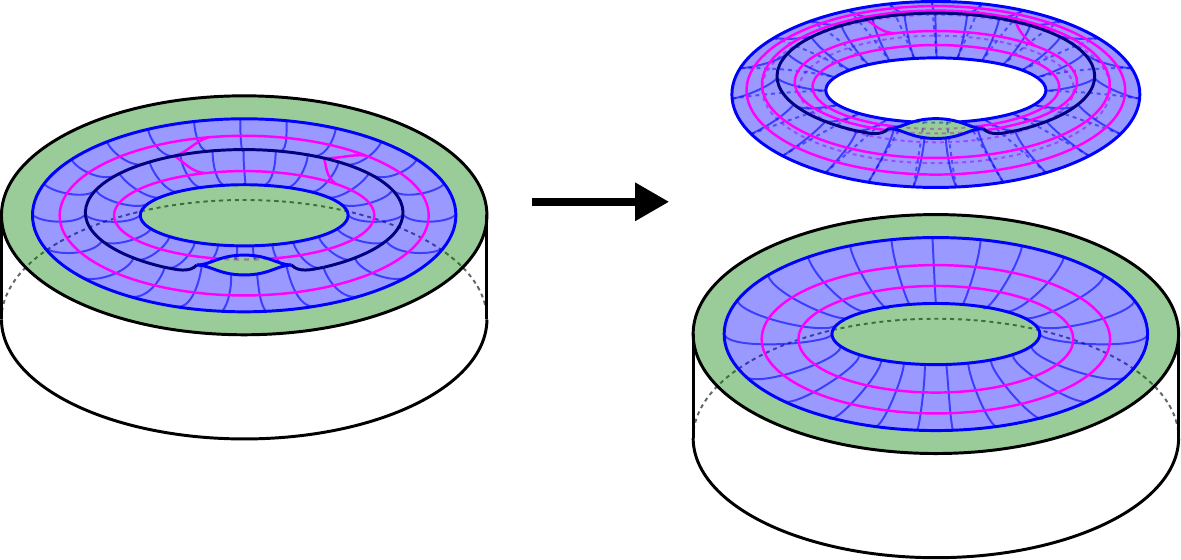}
    \caption{Cutting off breached $\u$-cusped torus pieces from $\u$-cusped product pieces (and splitting lone cycles on annulus faces of the breached $\u$-cusped torus).}
    \label{fig:dividebreachtorus}
\end{figure}

This cuts off pieces from the $\u$-cusped product that we call breached $\u$-cusped tori, and are defined as follows.
\begin{definition}
A \textbf{breached $\u$-cusped torus} is a $\u$-cusped torus with some cusped bigon $\p$-faces added along the $\u\u$-cusp edges. \textbf{Breached $\s$-cusped tori} are similarly defined.

A \textbf{breached $\s\u\u$-maw piece} is a dynamic manifold which is a $\s\u\u$-maw piece with some cusped bigon $\p$-faces added along the $\u\u$-cusp edge. \textbf{Breached $\u\s\s$-maw pieces} are similarly defined.
\end{definition}

Now modify $V$ so that $B^u$ is still a very full unstable dynamic branched surface, with $\tau$ a dynamic train track on $B^u$.

Finally, for every complementary region of $B^u$ which is a (breached) $\u$-cusped torus piece, we inspect the restriction of $\tau$ to each of its annulus faces. Since we have removed all extraneous sinks, the restriction of $\tau$ to each face contains one or two cycles. If the restriction of $\tau$ only has one cycle, split $\tau$ along the cycle. The result is still a dynamic train track, which we denote again by $\tau$. This is exactly as described in \cite[Proposition 4.11.1 Step 2d]{Mos96}.

We record some useful facts about $B^u$ and $\tau$ after these modifications:
\begin{enumerate}[label=(\roman*)]
    \item The complementary regions of $B^u$ are (breached) $\u$-cusped torus pieces and $\u$-cusped product pieces, and $\tau$ is still a dynamic train track on $B^u$.
    \item Every loop component of $\brloc(B^u)$ that does not meet $\tau$ is boundary parallel. This follows from the definition of dynamic train tracks.
    \item Every component of $\brloc(B^u)$ which is not a loop meets $\tau$. This is a consequence of splitting for goodness.
    \item No $\u$-cusped product complementary region of $B^u$ has parallel dynamically oriented $\u$-annulus faces that meet along $\u\u$-cusp circles or cobound cusped bigons. This follows from the axiom that cusp circles have to be incoherent in a cusped product piece (\Cref{defn:cuspedproduct}(g)), and also from the annulus sectors we chose to add to $B^u$.
\end{enumerate}

Next, we construct the stable branched surface $B^s$. We will do so by capping off $\tau$ within each complementary region of $B^u$.

For every complementary region of $B^u$ which is a (breached) $\u$-cusped torus, place annuli with boundary components lying along cycles of $\tau$ on the faces to divide the $\u$-cusped torus into a dynamic torus and (breached) $\s\u\u$-maw pieces. Then attach tongues to the annuli to subdivide the (breached) $\s\u\u$-maw pieces into $\s\u\u$-maw pieces, pinched tetrahedra, and $\m$-pinched rectangle drums. 
See \Cref{fig:dtttodynpairtorus}. This is essentially as described in \cite[Proposition 4.11.1 Step 3]{Mos96}.

\begin{figure}
    \centering
    \resizebox{!}{7cm}{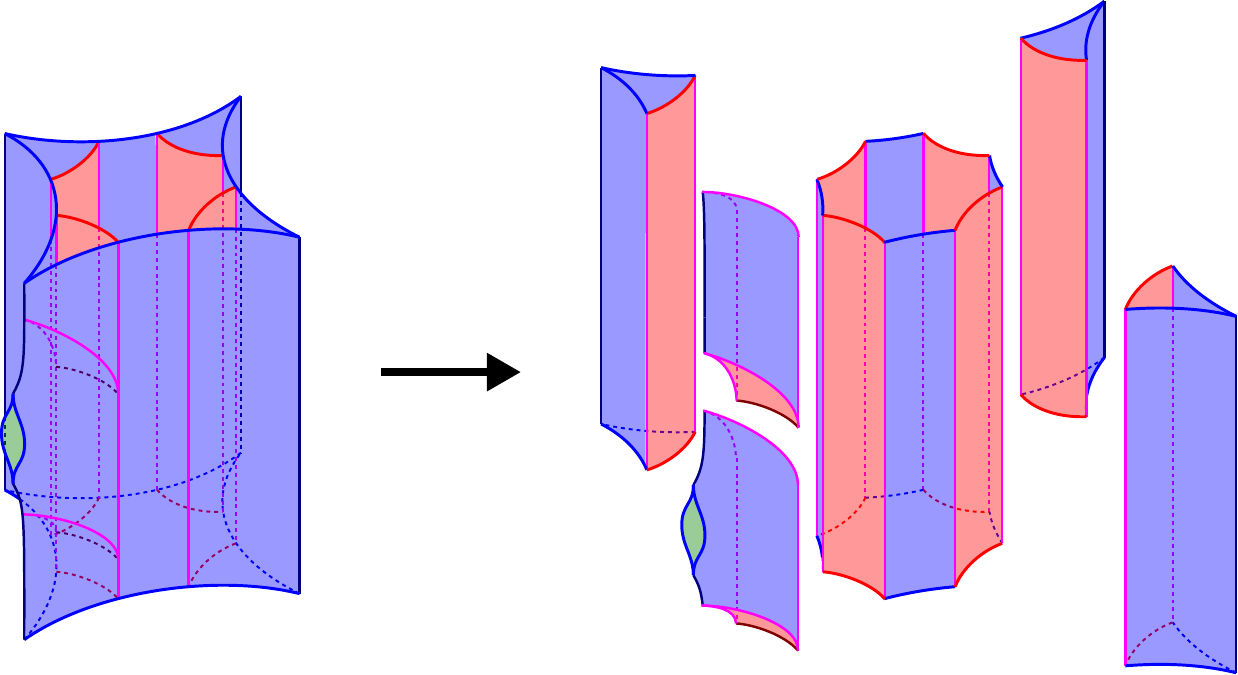}
    \caption{Cutting (breached) $\u$-cusped torus pieces into dynamic torus pieces, $\s\u\u$-maw pieces, pinched tetrahedra and $\m$-pinched rectangle drums.}
    \label{fig:dtttodynpairtorus}
\end{figure}

For every complementary region of $B^u$ which is a $\u$-cusped product homeomorphic to $D \times [0,1]$, consider the restriction of $\tau$ to $D \times \{1\}$, which we temporarily denote as $\tau'$. Take $\tau' \times [0,1]$ inside $D \times [0,1]$. If an annulus $\u$-face on $D \times \{1\}$ contains two cycles $c_1, c_2$ of $\tau'$, zip together $c_1 \times [0,\epsilon]$ and $c_2 \times [0,\epsilon]$ for small $\epsilon$. This divides the $\u$-cusped product piece into drums and $\u\s\s$-maw pieces. See \Cref{fig:dtttodynpairproduct}. This is exactly as described in \cite[Proposition 4.11.1 Step 3]{Mos96}, and uses properties (iii) and (iv) above.

\begin{figure}
    \centering
    \resizebox{!}{13cm}{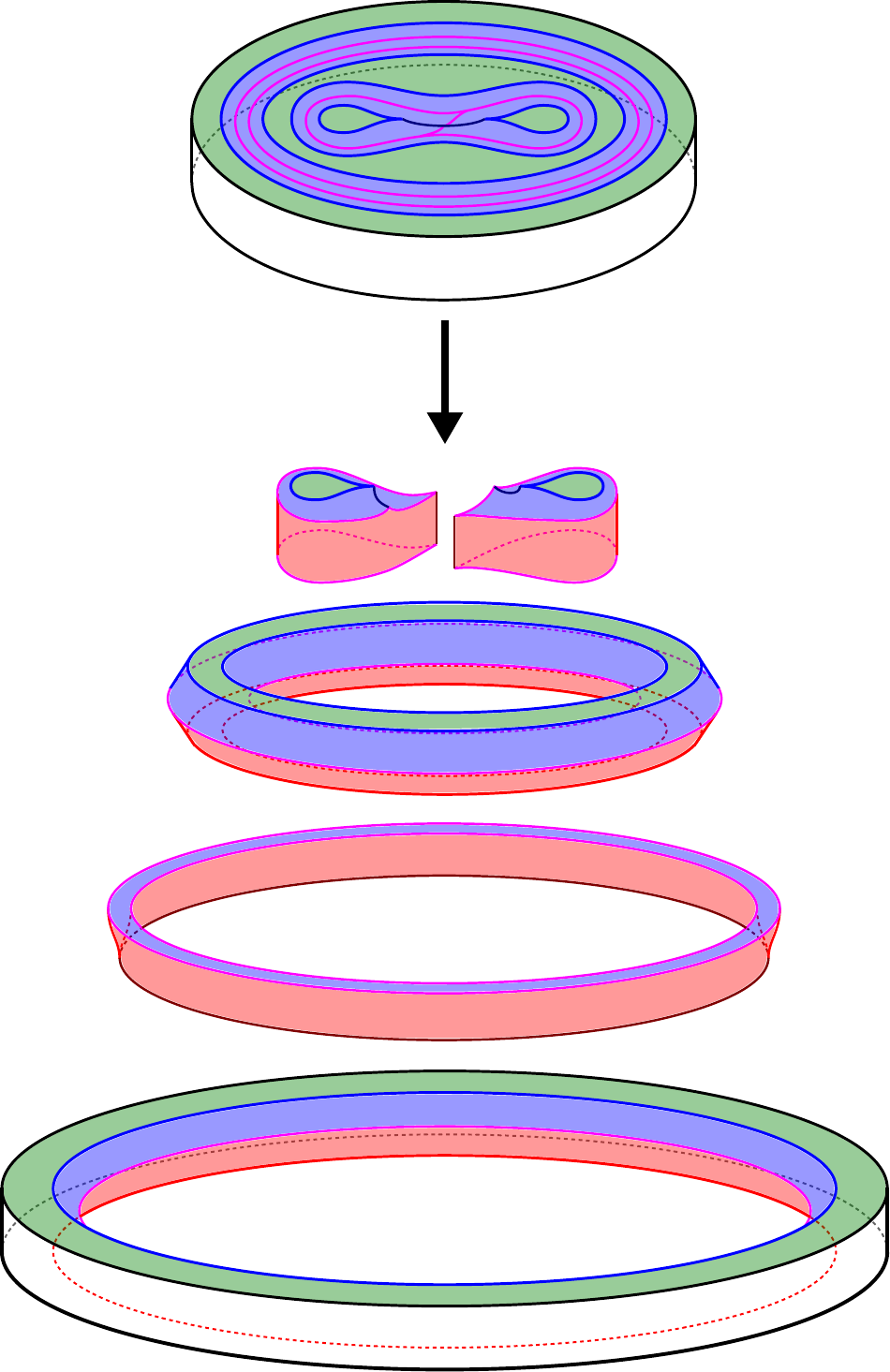}
    \caption{Cutting $\u$-cusped product pieces into drums and $\u\s\s$-maw pieces.}
    \label{fig:dtttodynpairproduct}
\end{figure}

Now take the union over all the complementary regions of $B^u$ to get $B^s$, and modify the vector field $V$ such that $(B^u, B^s, V)$ satisfies (1)-(4) of \Cref{defn:dynamicpair}.

The rest of the construction is checking (5a), (5b), (6) of \Cref{defn:dynamicpair} on $(B^u, B^s, V)$, as well as showing that $B^s$ does not carry closed surfaces, for then we can apply \Cref{prop:protodynamicpair}.

Take $S^u$ to be the set of annulus and M\"obius band sinks of the components of $B^u \cut \tau$. Since every $\u\s\s$-maw piece arises from parallel cycles of $\tau$ on the boundary of a $\u$-cusped product piece, it is clear that these are attached to elements of $S^u$ and are boundary parallel. The components of $B^u \cut B^s$ are exactly the components of $B^u \cut \tau$, hence (5b) for $S^u$ follows from the axioms of a dynamic train track.

Take $S^s$ to be the set of annuli added to (breached) $\u$-cusped torus pieces. By construction, every $\s\u\u$-maw piece is attached to an element of $S^s$. By property (ii) above, every $\s\u\u$-maw piece is boundary parallel as well. Property (5b) for $S^s$ is clear from construction.

That $B^u \cap R_+$ does not carry Reeb annuli is part of the hypothesis. If $B^s \cap R_-$ carries a Reeb annulus, then the restriction of $\tau$ to the boundary of some $\u$-cusped product piece must carry a Reeb annulus. The boundary components of such a Reeb annulus will then lie on coherently dynamically oriented annulus $\u$-faces that meet along a $\u\u$-cusp circle or bound cusped bigons, contradicting property (iii) above.

Now suppose $B^s$ carries a closed surface $C$. By existence of the vector field $V$ and atoroidality of $Q$, $C$ is a torus which is either peripheral or bounds a solid torus.

If $C$ bounds a solid torus, call it $T$. Otherwise let $T$ be the component of $Q\cut C$ containing a component of $\del Q$ parallel to $C$ (necessarily a component of $\gamma$).

Since $C \cap B^u$ is an oriented train track on $C$ with only converging switches, it is in fact a union of disjoint loops that cut $C$ up into a collection of annuli. These annuli must be elements of $S^s$, which we defined above to be the set of annuli added to (breached) $\u$-cusped torus pieces. The complementary regions of $B^u \cup S^s$ are (breached) $\s\u\u$-maw pieces, dynamic tori, and $\u$-cusped product pieces, and $T$ is a union of such regions. The $\u$-cusped product pieces cannot appear inside $T$ since $T$ has no $\m$-faces, and the breached $\s\u\u$-maw pieces cannot appear inside $T$ since $T$ has no $\p$-faces. Thus $T$ is a union of $\s\u\u$-maw pieces and dynamic tori. But by an index argument on a meridional disk (if $T$ is a solid torus) or annulus (if $T$ is a thickened torus) we see that this is impossible.

We conclude by applying \Cref{prop:protodynamicpair} to obtain a dynamic pair in $Q$. 

For the ``furthermore" statement, if the $\u$-cusped product complementary regions of $B^u$ have no cusp circles and if $\beta^u$ has no cusped bigon complementary regions, then we do not need to add sectors to $B^u$ before constructing $B^s$. If $\beta^u$ has no annulus complementary region, then (7) in \Cref{defn:dynamicpair} is automatically satisfied by the constructed $(B'^u, B^s)$.  If $B^u$ does not have annulus or Möbius band sectors which meet $\u$-cusped torus pieces on both sides, then (5c) for $S^s$ in \Cref{defn:dynamicpair} is satisfied, for by construction the only places where the boundary components of these sources can meet are along such sectors. We apply the second bullet point of \Cref{prop:protodynamicpair} to finish the proof.
\end{proof}

\subsection{From veering branched surfaces to dynamic pairs}
\label{subsec:vbstodp}

Recall the definition of veering branched surfaces in \Cref{defn:vbs} and the definition of flow graphs in \Cref{defn:flowgraph} and \Cref{fig:vbssectors}.

\begin{proposition} \label{prop:flowgraphdtt}
Let $B$ be a veering branched surface in a sutured manifold $Q$. The flow graph $\Phi$ of $B$ is a dynamic train track on $B$.
\end{proposition}
\begin{proof}
From construction, it is clear that $\Phi$ is disjoint from $\partial Q$ and the set of converging switches of $\Phi$ is equal to $\Phi \cap \brloc(B)$. Recall that $\Phi$ is obtained from the semiflow graph $\Phi_+$ by deleting the branches whose forward trajectories all end on $R_+$ (\Cref{defn:flowgraph}). We will show the remaining properties by first analyzing the situation for the semiflow graph $\Phi_+$. 

Recall from \Cref{defn:dynamicplanepieces} that for each sector $s$, the components of $s \cut \Phi_+$ are triangles, rectangles, and tongues.

We define a vector field $V_+$ on $B$ by requiring:
\begin{itemize}
    \item $V_+$ is tangent to $\Phi_+$.
    \item On a triangle, $V_+$ flows from the side cooriented inwards to the side cooriented outwards.
    \item On a rectangle, $V_+$ flows from the side cooriented inwards to the side cooriented outwards.
    \item On a tongue, $V_+$ flows from the cusp to the side on $\brloc(B)$ or $R_+$.
    \item $V_+$ is smooth on $B$ except at the diverging switches of $\Phi_+$ where backward trajectories are unique, and on $\brloc(B)$ where forward trajectories are  unique.
\end{itemize}

We can recover the components of $B \cut \Phi_+$ by gluing together components of $s \cut \Phi_+$. To describe this gluing, construct a directed graph $G$ by setting the vertices to be the set of components of $s \cut \Phi_+$ for all sectors $s$, and an edge from a component $c_1$ to another $c_2$ if $c_1$ is followed by $c_2$. The definitions of $\Phi_+$ and $V_+$ imply that $G$ has the following properties:
\begin{itemize}
    \item Each vertex has at most one outgoing edge; it has no outgoing edges if and only if it has a side on $R_+$
    \item A vertex has no incoming edges if and only if it is a tongue
    \item Only tongues and rectangles can have edges entering a rectangle
\end{itemize}
Moreover, the components of $G$ are in one-to-one correspondence with the components of $B \cut \Phi_+$.

The first property above implies that each component $H$ of $G$ either has a vertex or a cycle as a sink. In the former case, every forward trajectory in the corresponding component of $B \cut \Phi_+$ is finite and ends on $R_+$. In the latter case, the corresponding component of $B \cut \Phi_+$ carries an annulus or Möbius band, formed by the union of triangles that make up the sink in $H$, for which every forward trajectory in the component ends on. Moreover, in this latter case, the boundary of the component of $B \cut \Phi_+$ lies on $\Phi_+$ (and does not meet $R_+$). Indeed, none of the vertices in the corresponding component of $G$ can have a side along $R_+$, or else it would be a sink.

We use this information to analyze the components of $B \cut \Phi$. The components of $B \cut \Phi_+$ that contain an annulus or Möbius band sink remain as components of $B \cut \Phi$, since points of $\Phi_+$ on their boundary have paths that go around the annulus/Möbius band hence do not end on $R_+$. See \Cref{fig:flowgraphdttglue} right.

\begin{figure}
    \centering
    \resizebox{!}{6cm}{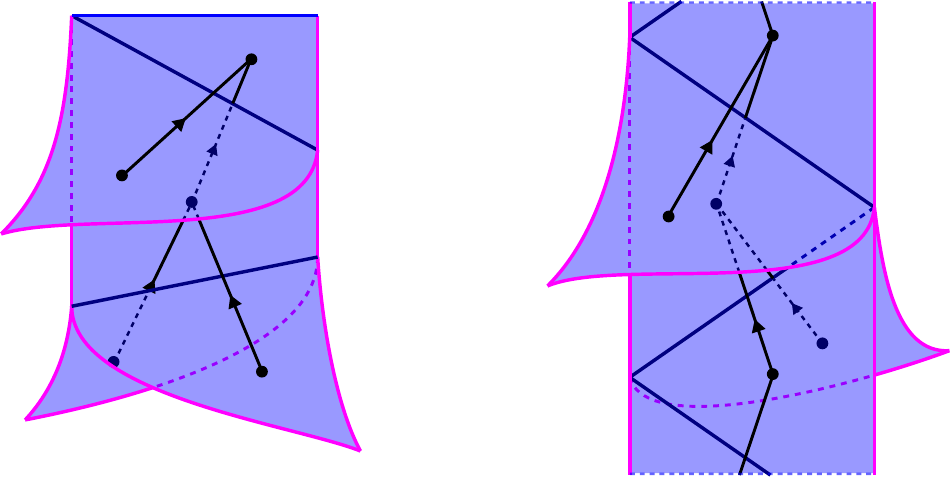}
    \caption{Analyzing the components of $B \cut \Phi_+$ by gluing up the components of $s \cut \Phi_+$.}
    \label{fig:flowgraphdttglue}
\end{figure}

The other components of $B \cut \Phi_+$ are glued together to form components of $B \cut \Phi$, and every forward $V_+$-trajectory in the interior of the latter is finite and ends on $R_+$. See \Cref{fig:flowgraphdttglue} left. We remark that this property is not necessarily true at the boundary: along $\Phi_+$, forward $V_+$-trajectories are not unique in general and certain $V_+$-trajectories may glue up to form a loop.

In any case, it only remains to smooth $V_+$ near the switches of $\Phi_+$ that are no longer switches of $\Phi$.
\end{proof}

\begin{theorem}\label{thm:VBStoDP}
If an atoroidal sutured manifold $Q$ admits an unstable veering branched surface $B$, then it admits a dynamic pair $(B^u, B^s, V)$.
\end{theorem}
\begin{proof}
By \Cref{prop:flowgraphdtt}, the branched surface contains a dynamic train track. By \Cref{prop:dtttodynamicpair}, $Q$ contains a dynamic pair.
\end{proof}

\begin{corollary} \label{cor:hmveeringdynamicpair}
Let $Q$ be the compactified mapping torus of an endperiodic map $f:L \to L$. If $Q$ is atoroidal, then there is a dynamic pair $(B^u, B^s)$ on $Q$ such that $B^u$ is an unstable veering branched surface compatibly carrying the unstable Handel-Miller lamination of $f$.
\end{corollary}

\begin{proof}
\Cref{thm:depthonetovbs} provides an unstable veering branched surface compatibly carrying the Handel-Miller lamination. We claim that when we apply \Cref{prop:dtttodynamicpair} and \Cref{prop:flowgraphdtt} to $B^u$ to get a dynamic pair, we do not need to modify $B^u$. To establish this, it suffices to check the hypotheses in the additional statement of \Cref{prop:dtttodynamicpair}, and to show that splitting for goodness does not change $B^u$.

The boundary train track $B^u\cap R_+$ is efficient by the definition of veering branched surfaces, and hence has no annulus or cusped bigon complementary regions. The fact that the $\u$-cusped product complementary regions of $B^u$ do not have cusp circles is shown in the proof of \Cref{thm:depthonetovbs}, specifically in \Cref{lem:ucuspedproduct}. 

As for the last additional condition of \Cref{prop:dtttodynamicpair}, suppose there is an annulus or Möbius band sector $s$ of $B^u$ which meets $\u$-cusped torus pieces on both sides. Notice that $s$ cannot have corners, otherwise there will be some source on $\partial s$, but the cusp circles of a $\u$-cusped torus piece cannot have sources. With this in mind, the condition follows from \Cref{lem:sectorannulus} and \Cref{lem:sectormobius}.

Finally, when one splits for goodness, one changes the boundary train track by a splitting move. But this cannot occur since the boundary train track of a veering branched surface is already spiraling.
\end{proof}

We remark that the above does not claim that $B^s$ carries the stable Handel-Miller lamination, although presumably this is true. We would be interested to see a proof of this. 

\section*{Acknowledgements}
We thank John Cantwell and Larry Conlon for telling us about their work on foliation cones and Handel-Miller theory, David Gabai for explaining some of the history of his work on pseudo-Anosov flows transverse to finite depth foliations, and Ian Agol and Eriko Hironaka for helpful conversations. We are additionally grateful to the anonymous referee for their numerous helpful comments and corrections.

This collaboration traces back to a topic group run by James Farre at the first Nearly Carbon Neutral Geometric Topology Conference in June 2020, organized by Martin Bobb and Allison N. Miller. We are grateful to James, Martin, and Allison for putting on a great event in a creative way at a stressful time.

M.L. was supported by the NSF under DMS 2013073 and thanks the Mathematisches Forschungsinstitut Oberwolfach for providing an ideal working environment during the final stages of this project.
C.C.T. was supported by the Simons Foundation \#376200.

\bibliographystyle{alphaurl}
\bibliography{bibliography}
\label{sec:biblio}

\end{document}